\documentclass[11pt]{amsart}
\usepackage{amsmath,amssymb}
\usepackage{graphicx}
\usepackage{amsfonts,amscd,latexsym,epsfig, psfrag}

\usepackage{hyperref}
\usepackage[all]{xy}

\newcommand{\er}{{\Diamond}}

\newcommand{\Aa}{{\mathcal A}}

\newcommand{\Mor}{{\rm Mor}}
\newcommand{\Obj}{{\rm Obj}}

\renewcommand{\Hat}{\widehat}

\newcommand{\pbar}{{\ov {\p}_J}}
\newcommand{\Cc}{{\mathcal C}}
\newcommand{\Kk}{{\mathcal K}}

\newcommand{\im}{{\rm im\,}}
\newcommand{\less}{{\smallsetminus}}

\newcommand{\p}{{\partial}}
\newcommand{\al}{{\alpha}}

\newcommand{\be}{{\beta}}

\newcommand{\eps}{{\varepsilon}}
\newcommand{\de}{{\delta}}
\newcommand{\De}{{\Delta}}

\newcommand{\Ga}{{\Gamma}}
\newcommand{\io}{{\iota}}
\newcommand{\ka}{{\kappa}}

\newcommand{\si}{{\sigma}}

\newcommand{\Uu}{{\mathcal U}}
\newcommand{\Bb}{{\mathcal B}}
\newcommand{\Ww}{{\mathcal W}}

\newcommand{\Mm}{{\mathcal M}}

\newcommand{\Tt}{{\mathcal T}}
\newcommand{\oMm}{{\overline {\Mm}}}
\newcommand{\ov}{\overline}

\newcommand{\id}{{\rm id}}
\newcommand{\rd}{{\rm d}}

\newcommand{\Ii}{{\mathcal I}}

\newcommand{\Qq}{{\mathcal Q}}
\newcommand{\Vv}{{\mathcal V}}

\newcommand{\N}{{\mathbb N}}

\newcommand{\Q}{{\mathbb Q}}
\newcommand{\R}{{\mathbb R}}

\newcommand{\E}{{\mathbb E}}

\newcommand{\Z}{{\mathbb Z}}

\newcommand{\Nn}{{\mathcal N}}

\newcommand{\ev}{{\rm ev}}

\newcommand{\bB}{{\bf B}}
\newcommand{\bC}{{\bf C}}

\newcommand{\bE}{{\bf E}}
\newcommand{\bK}{{\bf K}}
\newcommand{\bX}{{\bf X}}

\newcommand{\bZ}{{\bf Z}}

\newcommand{\s}{{\mathfrak s}}

\newtheorem{theorem}{Theorem}[subsection]
\newtheorem{thm}[theorem]{Theorem}

\newtheorem{lemma}[theorem]{Lemma}

\newtheorem{prop}[theorem]{Proposition}

\newtheorem{definition}[theorem]{Definition}
\newtheorem{defn}[theorem]{Definition}
\newtheorem{example}[theorem]{Example}

\newtheorem{remark}[theorem]{Remark}
\newtheorem{rmk}[theorem]{Remark}

\numberwithin{figure}{subsection}
\numberwithin{equation}{subsection}

\newcommand{\MS}{{\medskip}}

\newcommand{\NI}{{\noindent}}

\newcommand{\pr}{{\rm pr}}

   \newcounter{qcounter}

\newenvironment{itemlist}
   { \begin{list} {$\bullet$}
         {  \setlength{\itemsep}{.5ex} \setlength{\leftmargin}{2.5ex} } }
   { \end{list} }

\newcommand*{\longhookrightarrow}{\ensuremath{\lhook\joinrel\relbar\joinrel\rightarrow}}

\newcommand\quotient[2]{
        \mathchoice
            {
                \text{\raise1ex\hbox{$#1$}\Big/\lower1ex\hbox{$#2$}}
            }
            {
                #1\,/\,#2
            }
            {
                #1\,/\,#2
            }
            {
                #1\,/\,#2
            }
    }

\newcommand\quot[2]{
                \text{\raise1ex\hbox{$#1$}/\lower1ex\hbox{$\scriptstyle#2$}}
  }

\newcommand\quo[2]{
                \text{\raise1ex\hbox{$#1\!\!$}/\lower1ex\hbox{$\!\scriptstyle#2$}}
  }

\newcommand\qu[2]{
                \text{\raise.8ex\hbox{$\scriptstyle#1\!$}/\lower.8ex\hbox{$\!\scriptstyle#2$}}
  }

\newcommand\qq[2]{
                \text{\raise.8ex\hbox{$#1\!$}/\lower.8ex\hbox{$#2$}}
}

\newenvironment{myquote}[1]%
 {\begin{list}{}%
         {\setlength{\leftmargin}{#1}}%
         {\setlength{\rightmargin}{#1}}%
         \item[]%
 }
 {\end{list}}

 \title[The topology of Kuranishi atlases]{The topology of Kuranishi atlases}
 \author{Dusa McDuff}
\address{Department of Mathematics,
 Barnard College, Columbia University}
\email{dusa@math.columbia.edu}
\author{Katrin Wehrheim}
\address{Department of Mathematics, UC Berkeley}
\email{katrin@math.berkeley.edu}
\thanks{partially supported by NSF grants  
DMS 0905191, DMS 1308669 and DMS 0844188}

\begin{document}
\maketitle

\begin{abstract}
Kuranishi structures were introduced in the 1990s by Fukaya and Ono for the purpose of assigning a virtual cycle to moduli spaces of pseudoholomorphic curves that cannot be regularized by geometric methods. Starting from the same core idea (patching local finite dimensional reductions) we develop a theory of topological Kuranishi atlases and cobordisms that transparently resolves algebraic and topological challenges in this virtual regularization approach. 
It applies to any Kuranishi-type setting, e.g.\ atlases with isotropy, boundary and corners, or lack of differentiable structure. 
\end{abstract}

\tableofcontents
\section{Introduction}

\NI {\bf Historical Context:}
Abstract regularization methods were developed in symplectic geometry in the 1990s for the purpose of assigning a virtual fundamental cycle to moduli spaces of pseudoholomorphic curves that cannot be regularized by geometric methods. In purely topological terms, such {\it regularization methods} aim to answer the following question:

\begin{myquote}{0ex}
{\it 
Given a surjection $\pr: E \to B$ with a canonical zero section $0:B\to E$ (i.e.\ $\pr\circ 0 = \id_B$) and a designated section $s:B\to E$ with compact zero set $s^{-1}(0)=\{b\in B \,|\, s(b)=0(b) \}$, what extra structure is needed to induce a homology class on $s^{-1}(0)$ that generalizes the Euler class of vector bundles?
}
\end{myquote}

\NI
The current abstract regularization approaches in symplectic geometry roughly fall into three classes:
The {\it Kuranishi approach}, introduced by Fukaya-Ono \cite{FO} and implicitly Li-Tian \cite{LT}, and further developed by Fukaya-Oh-Ohta-Ono \cite{FOOO} and Joyce \cite{J1} in the 2000s, works with local finite dimensional reductions -- classically smooth sections $s_\io:U_\io\to E_\io$ of orbibundles in which both base and fiber are finite dimensional -- and obtains the global zero set $s^{-1}(0)=\qu{\cup s_\io^{-1}(0)}{\sim}$ as a quotient by transition data.
The {\it obstruction bundle approach} introduced by Liu-Tian \cite{LiuT},
Ruan~\cite{Ruan}, and Siebert \cite{Sieb} works with the zero set of a continuous map $s:B\to E$ which has local obstruction bundles $E_\io\to U_\io\subset B$ that cover the cokernel of $\rd s|_{U_\io}$ in local, but generally not compatible, smooth structures.
The {\it polyfold approach} developed by Hofer-Wysocki-Zehnder [HWZ1--4] in the 2000s also works with a global section $s:B\to E$ but equips the bundle with a global generalized smooth structure in which $s$ is Fredholm.

The original goal of our discussions, beginning with \cite{w:msritalk} in 2009, which have motivated much new work in this field \cite{FOOO12,J2,pardon}, was to clarify how these approaches resolve the fundamental difficulty in regularizing pseudoholomorphic curve moduli spaces: Elements of $s^{-1}(0)$ are equivalence classes of  pseudoholomorphic maps modulo reparametrization by a finite dimensional group of automorphisms. This action is smooth on finite dimensional sets of smooth maps, but it is nowhere differentiable as action on a classical Banach space of maps; see e.g.\ \cite[\S3]{MW1}. 
While one of the cornerstones of the polyfold approach is a new notion of smoothness that includes infinite dimensional reparametrization actions, the other approaches need to deal with this issue only in the notion of compatibility between different local obstruction bundles resp.\ local finite dimensional reductions. 
We communicated in \cite{MW0} a variety of fundamental issues: The obstruction bundle approach as explained in \cite{Mcv} makes differentiability assumptions on the transition data that do guarantee a virtual fundamental class, but which generally do not  hold in the analytic setting for pseudoholomorphic curves. 
In the Kuranishi approach we found that geometric construction of local finite dimensional reductions following \cite{LT} does yield smooth transition data. 
However, as of 2011, algebraically inconsistent notions of Kuranishi structure were still in use (see \cite[\S2.5]{MW1}), and fundamental topological issues were not addressed in generalizing
the perturbative construction of the Euler class to the Kuranishi context (see \cite[\S2.6]{MW1}).

\MS
\NI {\bf Aim of this paper: } We develop a theory of topological Kuranishi atlases and cobordisms that transparently resolves the algebraic and topological issues of Kuranishi regularization. To demonstrate the universal applicability of our theory -- originally developed in \cite{MW0} in the context of smooth atlases with trivial isotropy -- we formulate it in a context of atlases that consist of continuous maps $s_\io:U_\io\to E_\io$ between locally compact metric spaces.
Thus our theory allows bases $U_\io$ with boundary and corners, sections $s_\io$ with weak differentiability properties, or arising as quotients by isotropy group actions. 
Our theory then provides a framework for the regularization, in particular a global ambient space and section $s:B\to E$ made up of local sections modulo transition data. To obtain a virtual fundamental class $[s^{-1}(0)]$ from this framework, one still has to add suitable differentiability and orientation assumptions, and in that context construct compatible transverse perturbations of the local zero sets $(s_\io+\nu_\io)^{-1}(0)$, or more abstract compatible local Euler classes as in \cite{pardon}.

However, our framework allows for a straightforward perturbation theory in which compatibility,  Hausdorffness, and compactness of the perturbed zero set are automatic. We demonstrate this in the case of smooth Kuranishi atlases with trivial isotropy in \cite{MW1} and extend it to finite but not necessarily effective isotropy in \cite{MW2}. Moreover, \cite{MW:GW,Cast,Mcn} show that Gromov-Witten moduli spaces fit into this framework. The following describes the new abstract framework and scope of applications in more detail.

\MS
\NI {\bf Moduli Spaces of Pseudoholomorphic Curves:}
In the applications to symplectic geometry, one wishes to assign a fundamental cycle to a compactified moduli space $\oMm$ of pseudoholomorphic curves in a symplectic manifold $M$.
Since in the algebro-geometric setting (when $M$ carries a complex structure) the space $\oMm$ carries a virtual fundamental class, one expects well defined homological information also in the symplectic setting. In that case, $M$ carries a contractible space of almost complex structures $J$.
Thus elements of $\oMm$ are not described algebraically but analytically: For fixed $J$, any element of the moduli space $\oMm$ (i.e.\ $J$-curve) has a neighbourhood $F\subset \oMm$ that is homeomorphic $F\cong \qu{\pbar^{-1}(0)}{\Ga}$ to the zero set of a Fredholm section $\pbar$ that is equivariant under a finite isotropy group $\Ga$.
In the best case scenario (equivariant transversality of $\pbar$) this makes $\oMm$ a smooth oriented orbifold, which has a fundamental class in the rational homology of $\oMm$. Moreover, if transversality can be achieved in families, then this homological information is independent of the choice of $J$ via cobordisms. 
However, while any choice of $J$ yields an equivariant Fredholm section, many situations do not allow for a $J$ that also achieves transversality.
This is a pervasive problem in symplectic geometry, where most modern tools -- from the Floer Theory used in the proof of the Arnold Conjecture to the Fukaya categories involved in Mirror Symmetry -- 
require the regularization of various moduli spaces. However, unless a suitable type of injectivity of the $J$-curves is known a priori, no variation of geometric structure (e.g.\ the choice of $J$) will yield equivariant transverse perturbations.

In such cases, the Kuranishi approach aims to implement the perfect obstruction theory of algebraic geometry by choosing obstruction spaces $E$ which locally cover the cokernel of the Fredholm sections $\pbar$ and yield finite dimensional reductions $F\simeq s^{-1}(0) / \Ga$. Such Kuranishi charts roughly consist of a $\Ga$-equivariant section $s=\pbar|_U:U\to E$ of a finite rank bundle $E\to U=\pbar^{-1}(E)$. 
Depending on the specific context, the base $U$ might have additional structure such as boundary and corners (on which the chart is given as fiber product of Kuranishi charts on other moduli spaces), evaluation maps to $M$, or forgetful maps to Deligne-Mumford moduli spaces, and the section $s$ might only be stratified smooth.
However, the key difference to algebraic geometry is in the uniqueness: Already the local Fredholm description involves non-canonical choices, and a transition map between different charts is generally just an identification of zero sets $s_1^{-1}(0) / \Ga \cong s_2^{-1}(0) / \Ga$. 
An extension to a map $U_1\to U_2$ between the base spaces has to be constructed and is not unique either.

\MS

\NI {\bf Ideas and Challenges of Kuranishi Regularization:}
Given a cover of a compact metric space $X= \bigcup F_\io$ (e.g.\ $X=\oMm$) by Kuranishi charts $F_\io \cong\qu{s_\io^{-1}(0)}{\Ga_\io}$ involving $\Ga_\io$-equivariant sections $s_\io:U_\io\to E_\io$, and a suitable compatibility notion between these charts, one wishes to regularize such a space $X$ with Kuranishi atlas $\Kk=\{(U_\io,E_\io,\Ga_\io,s_\io)\}$ by means of abstract perturbations of the local sections. The expected degree of this class is the fixed dimension $D:= \dim U_\io - {\rm rank}\, E_\io$ of charts in the atlas, and pre-2011 regularization results claim virtual moduli cycles in the following sense:
{\it Given suitable evaluation maps $\ev:U_\io\to M$ there exists a class of perturbations $\nu=\{\nu_\io\}$ such that the push-forwards $\ev:(s_\io + \nu_\io)^{-1}(0)\to M$ define cycles $\ev_*[X]_\nu \in C_D(M)$ whose homology class is independent of $\nu$.}

Our framework allows for a more direct regularization in terms of cycles that represent a homology class on $X$, which can then be pushed forward by evaluation maps. In the case of trivial isotropy, \cite{MW1} builds on the present paper to prove the following:

\begin{myquote}{0ex}
{\it 
Let  $X$ be a compact metrizable space.
Then any cobordism class of $D$-dimensional weak additive Kuranishi atlases $\Kk$ on $X$ determines uniquely
\begin{itemize}
\item
a {\bf virtual moduli cycle}, that is a cobordism class of smooth, compact manifolds, and 
\item
a {\bf virtual fundamental class} $[X]^{vir}_\Kk\in \check{H}_D(X;\Q)$ in \v{C}ech homology.
\end{itemize}
}
\end{myquote}

\NI
In the case of nontrivial isotropy, the analogous result -- with ``weighted branched manifolds" in place of manifolds -- is proven in \cite{MW2} by combining the present paper with the techniques in \cite{MW1} applied to a more complicated categorical setting.
Under reasonable differentiability assumptions on the local sections $s_\io:U_\io\to E_\io$, transverse (multi-valued) perturbations and hence a smooth (weighted branched) structure 
are easily achieved locally, leaving the following challenges:
\begin{itemlist}
\item[(a)]  
Capture compatibility between overlapping local models appropriately;
\item[(b)]
construct compatible transverse (possibly multivalued) perturbations;
\item[(c)]
ensure that the resulting locally smooth zero sets remain compact and Hausdorff (resp.\ in the case of nontrivial isotropy have consistent weights on their branches);
\item[(d)]
capture orientation data that induces coherent orientations on the perturbed zero sets;
\item[(e)]
establish suitable independence from the choice of perturbations;
\item[(f)]
capture the effect of varying $J$ in a cobordism theory;
\item[(g)]
appropriately describe the invariant information obtained from perturbed zero sets.
\end{itemlist}

In the case of a section $s:U\to E$ of an (orbi)bundle (i.e.\ a Kuranishi atlas with a single chart), challenges (a) and (b) above are absent so that generic (multi-valued) perturbations yield locally smooth solution sets. Challenge (c) is resolved by the ambient space $U$ that is Hausdorff and locally compact.
Straightforward cobordism theories deal with challenges (d) and (e), and finally a metric on the ambient space $U$ allows for the construction of $[s^{-1}(0)]^{vir}\in H_{\scriptscriptstyle{\rm rk}\, E - \dim U}(s^{-1}(0);\Q)$ as inverse limit in  \v{C}ech 
homology.
However, once several charts are involved, their base spaces will usually have different dimensions and hence are at best related by embeddings $\phi_{IJ}:U_{IJ}\hookrightarrow U_J$ of open subsets $U_{IJ}\subset U_I$ into a larger dimensional domain $U_J$. 
This immediately causes fundamental challenges:

\begin{itemlist}
\item[(a)]  
The cocycle condition $\phi_{JK}\circ\phi_{IJ}=\phi_{IK}$ holds automatically only on the zero sets $s_I^{-1}(0)$, which are related via embeddings to the moduli space $X$.
In order to be able to patch local perturbations of these zero sets together, this condition must be extended to open subsets of the base spaces $U_I$ as an axiom on the transition maps $\phi_{IJ}$. 
However, the natural cocycle condition that is obtained from analytical constructions -- equality on the overlap -- does not even induce an ambient set $\bigsqcup U_I /\!\! \sim$ since the transition maps only induce a transitive relation $\sim$ if each $\phi_{IK}$ extends $\phi_{JK}\circ\phi_{IJ}$.

\item[(b)]
While one could easily add a transverse perturbation $\nu_I:U_I\to E_I$ to each section $s_I$, patching of the perturbed zero sets $(s_I+\nu_I)^{-1}(0)$ requires compatible perturbations in the sense that $\nu_J|_{\im\phi_{IJ}} = (\phi_{IJ})_*\nu_I$. 
Unless charts are totally ordered, compatible with the direction of transition maps, this yields more conditions than degrees of freedom.
Moreover, there will be obstructions to transversality and even smoothness of perturbations 
unless images of transition maps intersect nicely.

\item[(c)]
Given compatible transverse perturbations, the perturbed space 
$\bigsqcup (s_I+\nu_I)^{-1}(0) /\!\! \sim$ has smooth charts, but there is no good reason for the quotient topology to be Hausdorff other than an embedding into a Hausdorff ambient space $\bigsqcup U_I /\!\! \sim$.
While there are types of atlases that guarantee Hausdorffness, this requires even stronger compatibility of the charts, in particular the domains of $\phi_{IK}$ and $\phi_{JK}\circ\phi_{IJ}$ need to agree exactly.

Moreover, the quotient topology on such an ambient space is not locally compact or metrizable if there is a nontrivial transition map between domains of different dimension. Thus there is neither a natural notion of ``small perturbations'' nor a good reason why they should yield a compact perturbed space
$\bigsqcup (s_I+\nu_I)^{-1}(0) / \!\! \sim$.
\end{itemlist}

\noindent
On the upside, it turns out that our approaches to overcoming these obstacles also naturally yield all the tools necessary for the direct construction of the virtual fundamental class. 

\MS

\NI {\bf Key Notions and Results:}
In order to obtain a coherent theory that resolves these challenges transparently, we develop in Section~\ref{s:chart} notions of topological Kuranishi charts and coordinate changes that underlie all existing notions in e.g.\ \cite{FO,FOOO,J1,MW1,MW2,pardon}, see Remark~\ref{rmk:topchart}.
Our resolutions of the challenges (a)-- (g) then proceed as follows.

\begin{itemlist}
\item[(a)]  
Definition~\ref{def:Ku} introduces the notion of {\bf topological Kuranishi atlas} on a compact metrizable space $X$, consisting of a covering $X=\bigcup_{i=1}^N F_i$ by basic charts with footprints $F_i\simeq s_i^{-1}(0)$, whose compatibility is captured by transition charts with footprints $F_I=\bigcap_{i\in I} F_i$ and directional coordinate changes $U_I\supset U_{IJ}\hookrightarrow U_J$ for each $I\subset J$ that satisfy a cocycle condition. Such an atlas $\Kk$ can equivalently be seen as a bundle functor between topological categories $\pr_\Kk: \bB_\Kk\to\bE_\Kk$ together with a section functor $\s_\Kk:\bB_\Kk\to\bE_\Kk$ and an embedding $\iota_\Kk: X \to |\Kk|$ to the zero set $|\s_\Kk^{-1}(0)|$.
This categorical language allows us to succintly express compatibility conditions between maps and spaces of interest, and it yields a {\bf virtual neighbourhood} $|\Kk|:= \bigsqcup U_I / \!\! \sim$ of $X\cong\io_\Kk(X)$ in Definition~\ref{def:Knbhd}.

Definition~\ref{def:Ku3} introduces the notion of {\bf filtered weak topological Kuranishi atlas}, which is the structure that naturally arises from local Fredholm descriptions of a moduli space. As we demonstrate for Gromov--Witten moduli spaces in \cite{MW:GW}, these satisfy a weaker form of the cocycle condition (equality on the overlap) but have a natural filtration arising from ``additive choices of obstruction spaces''.

Theorem~\ref{thm:K} then shows that every filtered weak topological Kuranishi atlas has a {\bf tame shrinking} which satisfies the stronger cocycle and filtration conditions stated in Definition~\ref{def:tame}.
For such a tame atlas $\Kk$, the virtual neighbourhood $|\Kk|$ is Hausdorff with embeddings $\pi_\Kk: U_I \to |\Kk|$. Moreover $|\Kk|$ can be equipped with a metric topology that is compatible with metrics on the domains $U_I$, and this tame metric shrinking is unique up to a notion of metric cobordism in Definition~\ref{def:mCKS}.

\item[(b)]
Since the category $\bB_\Kk$ has too many morphisms for us to be able to construct a nontrivial perturbation functor $\nu: \bB_{\Kk}\to \bE_{\Kk}$, we introduce in Definition~\ref{def:vicin} the notion of a {\bf reduction} $\bB_{\Kk}|_{\Vv}$.
It is the full subcategory determined by a precompact subset of objects $\Vv=\bigsqcup V_I\sqsubset\Obj_{\bB_\Kk}$ whose realization $\pi_\Kk(\Vv)\subset |\Kk|$ still contains the zero set $|\s_\Kk|^{-1}(0)$, but whose domains $V_I\sqsubset U_I$ only have overlaps $\pi_\Kk(V_I)\cap \pi_\Kk(V_J)\neq\emptyset$ if there is a direct coordinate change $U_I\supset U_{IJ}\hookrightarrow V_J$.
Theorem~\ref{thm:red} shows that reductions exist, are unique up to a notion of cobordism reduction, and have various subtle 
refinements.
In particular, a reduction $\Vv$ induces a precompact subset $\pi_\Kk(\Vv)\subset|\Kk|$. While this generally is not an open subset, it is the closest analogue to a precompact neighbourhood of the zero set $|\s_\Kk|^{-1}(0)\simeq X$
-- it controls compactness of perturbed zero sets, see Section~\ref{ss:pert}.

\item[(c)]
Given {\bf nested reductions} $\Cc\sqsubset\Vv$, which exist by Lemma~\ref{le:delred}, the remaining challenge -- given suitable smooth structure -- is to construct perturbations $\nu: \bB_{\Kk}|_\Vv\to \bE_{\Kk}|_\Vv$ of $\s_\Kk|_\Vv$ so that the local zero sets $(s_I|_{V_I} + \nu_I)^{-1}(0)\subset V_I$ are cut out transversely and contained in $\pi_\Kk^{-1}(\pi_\Kk(\Cc))$.
In order to obtain uniqueness, one also needs to interpolate different choices $\nu_0,\nu_1$ by a perturbation over the Kuranishi concordance $[0,1]\times \Kk$.
(For trivial isotropy this is performed in \cite{MW1}; for nontrivial isotropy essentially the same construction is used in \cite{MW2} to obtain multi-valued perturbations.)
Then the realization $\big|(\s_\Kk|_\Vv+\nu)^{-1}(0)\big|$ of the zero set is sequentially compact by Theorem~\ref{thm:zeroS0}.
Since it also is a locally smooth subset of a Hausdorff space, it forms a closed manifold (resp.\ branched manifold). Moreover, this is unique up to cobordism by interpolation of perturbations.

\item[(d)]
Basic notions of determinant bundles and orientations of smooth Kuranishi structures are introduced in \cite{MW1} and extended to nontrivial isotropy in \cite{MW2}.

\item[(e)]
Towards proving independence of the virtual fundamental class from the multitude of choices (in particular shrinking, metric, and perturbation), Section~\ref{s:Kcobord} develops notions of {\bf concordance between topological Kuranishi atlases} with various extra structures.
Theorem~\ref{thm:cobord2} then shows that the metric shrinkings constructed by Theorem~\ref{thm:K} are unique up to concordance.
This framework also allows one to make the perturbation constructions (e.g.\ in \cite{MW1}) unique up to perturbations over the concordance.
 
\item[(f)]
We generalize the notion of concordance to a notion of {\bf topological Kuranishi cobordism} between two topological Kuranishi atlases on different compact moduli spaces in Definition~\ref{def:CKS}.
Our cobordisms only have boundary, no corners, in the sense that they are topological Kuranishi atlases on compact spaces such as a union of compact moduli spaces ${\textstyle \bigcup_{t\in [0,1]}} \{t\}\times \oMm(J_t)$, in which all data has boundary collar structure near $t=0,1$.
Section~\ref{s:Kcobord} develops the theory of (a) for Kuranishi cobordisms, and since Section~\ref{s:red} resp.\ \cite{MW1,MW2} treat Kuranishi cobordisms in parallel with Kuranishi atlases, 
the approach described here will allow us to conclude, for example, that the the Gromov-Witten virtual moduli cycles for different choices of $J$ are cobordant in a suitable sense.

\item[(g)]
In any setting in which existence of transverse perturbations (c) and a compatible notion of orientations (d) is established, one can now capture the invariant information both as a smooth cobordism class and as a \v{C}ech homology class on $X$.
This hinges on the topological model $|\Kk|$ for a metrizable neighbourhood of $X$ in which $X$ appears as the zero set of a section $|\s_\Kk|$ of a finite-dimensional ``bundle" $\pr: |\bE_\Kk|\to |\Kk|$.
For any $\eps>0$, Theorem~\ref{thm:red} yields nested reductions with $\pi_\Kk(\Cc)\subset B_\eps(X)$ contained in the $\eps$-neighbourhood of $X$, so that the perturbation construction yields a smooth compact zero set $|\s_\Kk+\nu|\subset B_\eps(X)$, independent of $\nu$ up to cobordism. The resulting cobordism class of closed (possibly weighted branched) submanifolds $|\bZ_\nu|$ forms the virtual moduli cycle and 
(if oriented) represents an element $[\bZ_\nu]\in H_D(B_\eps(X);\Q)$.  
The virtual fundamental class $[X]^{vir}_\Kk\in H_D(X;\Q)= \underset{\leftarrow }\lim\, \check{H}_D(B_\eps(X);\Q)$ is then obtained as inverse limit as $\eps\to0$.\footnote{
Here we cannot work with integral \v{C}ech homology since it does not even satisfy the exactness axiom. However, in the case of trivial isotropy an integral theory could be obtained using Steenrod homology \cite{Mi}. 
}
\end{itemlist}

\medskip
\noindent
{\bf Acknowledgements:}
We would like to thank
Mohammed Abouzaid,
Tom Coates, 
Kenji Fukaya,
Tom Mrowka,
Kaoru Ono,
Yongbin Ruan,
Dietmar Salamon,
Bernd Siebert,
Cliff Taubes,
Gang Tian,
and
Aleksey Zinger
for encouragement and enlightening discussions about this project,
and Jingchen Niu for pointing out some gaps in an earlier version.
We moreover thank MSRI, IAS, BIRS and SCGP for hospitality.

\section{Topological Kuranishi atlases
}
\label{s:chart}

Throughout, $X$ is assumed to be a compact and metrizable space.
The first two sections of this chapter introduce the notions of topological Kuranishi charts for $X$ and topological coordinate changes between them. Our definitions differ significantly from e.g.\ \cite{FO,J1,pardon} but are motivated by capturing the topological essence of all these notions. Moreover, our insistence on specifying the domains of coordinate changes is new, as is our notion of Kuranishi atlas (rather than Kuranishi structure) and its interpretation in terms of categories in Section~\ref{ss:Ksdef}.
The main result of this chapter is the novel construction of a virtual neighbourhood of $X$ from a topological Kuranishi atlas in Section~\ref{ss:vnbhd}.

\subsection{Charts and restrictions}\label{ss:chart} \hspace{1mm}\\ \vspace{-3mm}

The most general notions of Kuranishi charts consist of 
\begin{itemize}
\item
a manifold $B$ with boundaries and corners and an action by a finite group $\Ga$;
\item
a $\Ga$-equivariant finite rank obstruction bundle $E \to B$;
\item
a $\Ga$-equivariant section $s: B\to E$;
\item
a homeomorphism $\qu{\s^{-1}(0)}{\Ga} \cong F$ to an open subset $F\subset X$.
\end{itemize}
In order to capture the topological information, we may forget any smoothness and replace the $\Ga$-equivariant bundle and section by the induced ``bundle map'' $\qu{E}{\Ga}\to\qu{B}{\Ga}$ and ``section'' $\underline{s}:\qu{B}{\Ga}\to\qu{E}{\Ga}$. If we moreover capture the linear structure of the bundle by the induced zero section $0:\qu{B}{\Ga}\to\qu{E}{\Ga}, [u]\mapsto [0\in \pr^{-1}(u)]$, then the zero set
$\underline{s}^{-1}(0)=\{x\in \qu{B}{\Ga} \,|\, \underline{s}(x)=0(x)\}$ is identified with $\qu{\s^{-1}(0)}{\Ga}$. 
Hence we obtain a topological Kuranishi chart in the following sense with the same footprint 
 $\underline{s}^{-1}(0)\cong\qu{\s^{-1}(0)}{\Ga}\cong F\subset X$.

\begin{defn}\label{def:tchart}
A {\bf topological Kuranishi chart} for $X$ with open footprint $F\subset X$ is a tuple $\bK = (U,\E,
\s,\psi)$ consisting of
\begin{itemize}
\item
the {\bf domain} $U$, which is a separable, locally compact metric space;
\item
the {\bf obstruction ``bundle''} which is a continuous map 
$\pr :\E \to U$
from a separable, locally compact metric space $\E$,
together with a 
{\bf zero section} ${0: U \to \E}$, which is a continuous map with $\pr \circ 0 = \id_{U}$;
\item
the {\bf section} $\s: U\to \E$, which is a continuous map with $\pr \circ \s = \id_{U}$;
\item
the {\bf footprint map} $\psi : \s^{-1}(0) \to X$, which is a homeomorphism 
between the {\bf zero set} $\s^{-1}(0):=
\s^{-1}(\im 0)=\{x\in U \,|\, \s(x)=0(x)\}$ and 
the {\bf footprint} ${\psi(\s^{-1}(0))=F}$.
\end{itemize}
\end{defn}
   
\begin{rmk}\rm    \label{rmk:topchart}
(i)   According to \cite{MW2}, a {\bf smooth Kuranishi chart} is a tuple as above, where $\E = U\times E/\Ga$ is the  product of a finite dimensional manifold with a finite dimensional vector space, modulo a smooth action by a finite group $\Ga$ (which is linear on $E$). Then $\pr:\E\to U$ is the obvious projection and $0: [u]\to [(u,0)]$ is the zero section.
To see that these and other notions of Kuranishi charts induce topological Kuranishi charts, note that finite dimensional manifolds and their quotients by the action of a finite group are automatically separable (i.e.\ contain a countable dense subset), locally compact, and metrizable. 
\MS

\NI (ii)
We will not use particular choices of metrics on topological Kuranishi charts, until we consider compatible metrics for topological Kuranishi atlases in Definition~\ref{def:metric}. 
So it would be more appropriate (but more cumbersome) to say that the domain $U$ of a topological Kuranishi chart is a separable, locally compact, metrizable topological space.
\MS

\NI (iii)
Both the domain $U$ and the bundle $\E$ are also second countable and hence Lindel\"of: Every open cover has a countable subcover. Indeed, these properties are equivalent to separability in metric spaces; see \cite[Ex.~4.5]{Mun}.
\MS

\NI (iv)
The zero section $0:U\to\E$ is a homeomorphism to a closed subset of $\E$. Indeed, its inverse is the projection $\pr|_{\im 0}$, and to check that $\im 0 = \{ 0(x) \,|\, x\in U\}\subset \E$ is closed consider a sequence $0(x_i)\to e_\infty$.
Its limit must be $e_\infty=\lim 0(x_i) = 0(\pr(e_\infty))$ since continuity of $\pr$ implies $x_i=\pr(0(x_i)) \to \pr(e_\infty)$ and $0$ is continuous.
It follows that the zero set $\s^{-1}(0) \subset U$ is a closed subset since it is the preimage of a closed subset under a continuous map.
$\hfill\er$
\end{rmk}

Since we aim to define a regularization of $X$, the most important datum of a Kuranishi chart is its footprint.
So, as long as the footprint is unchanged, we can vary the domain $U$ and section $\s$ without changing the chart in any important way. Nevertheless, we will always work with charts that have a fixed domain and section. In fact, our definition of a coordinate change between Kuranishi charts will crucially involve these domains.
Moreover, it will require the following notion of restriction of a chart to a smaller subset of its footprint.

\begin{defn} \label{def:restr}
Let $\bK$ be a topological Kuranishi chart and $F'\subset F$
an open subset of the footprint.
A {\bf restriction of $\bK$ to $\mathbf{\emph F\,'}$} is a topological Kuranishi chart of the form
$$
\bK' = \bK|_{U'} := \bigl(\, U' \,,\, \E'= \E|_{U'}  \,,\, \s'=\s|_{U'} \,,\, \psi'=\psi|_{U'\cap\s^{-1}(0)}\, \bigr)
$$
given by a choice of open subset $U'\subset U$ of the domain such that $U'\cap \s^{-1}(0)=\psi^{-1}(F')$.
In particular, $\bK'$ has footprint $\psi'(\s'^{-1}(0'))=F'$.
Here $\E|_{U'}$ denotes the obstruction bundle with total space $\E':=\pr^{-1}(U')$, projection $\pr':=\pr|_{\E'}$, and zero section $0':=0|_{U'}$.
\end{defn}

Note here that we will not always put the ``topological'' prefix in front of all Kuranishi data though we always mean to refer to the topological Kuranishi framework, unless we are explicitly discussing smooth structures or constructions of perturbations. However, if we had formally introduced a (smooth) Kuranishi framework, it would for the most part make sense to speak about both. 
For example, the following lemma constructs a restriction of a topological Kuranishi chart whose footprint is any given open subset of the original footprint. If the topological chart underlies a smooth Kuranishi chart, then a simple pullback of the restricted domain will yield the corresponding restriction of the smooth Kuranishi chart. 

The next lemma provides a tool for restricting to precompact domains, which we require for refinements of Kuranishi atlases in Sections~\ref{ss:shrink} and \ref{s:red}.
Here and throughout we will use the notation $V'\sqsubset V$ to mean that the inclusion $V'\hookrightarrow V$ is {\it precompact}. That is, ${\rm cl}_V(V')$ is compact, where ${\rm cl}_V(V')$ denotes the closure of $V'$ in the relative topology of $V$.  If both $V'$ and $V$ are contained in a compact space $X$, then $V'\sqsubset V$ is equivalent to the inclusion $\ov{V'}:={\rm cl}_X(V')\subset V$ of the closure of $V'$ with respect to the ambient topology.

\begin{lemma}\label{le:restr0}
Let $\bK$ be a topological Kuranishi chart. Then for any open subset $F'\subset F$ there exists a restriction $\bK'$ to $F'$ whose domain $U'$ is such that $\ov{U'}\cap \s^{-1}(0) = \psi^{-1}(\ov{F'})$. If moreover $F'\sqsubset F$ is precompact, then $U'$ can be chosen to be precompact.
\end{lemma}

\begin{proof}
Since $F'\subset F$ is open and $\psi:\s^{-1}(0)\to F$ is a homeomorphism in the relative topology of $\s^{-1}(0)\subset U$, there exists an open set $V\subset U$ such that $\psi^{-1}(F')=U'\cap \s^{-1}(0)$.
If $F'\sqsubset F$ then we claim that $U'\subset V$ can be chosen  so that  in addition its closure 
intersects $\s^{-1}(0)$ in $\psi^{-1}(\ov{F'})$.
To arrange this we define
$$
U' := \bigl\{x\in V \,\big|\, d(x, \psi^{-1}(F')) < d(x, \psi^{-1}(F\less F'))\bigr\},
$$
where
$d(x,A): = \inf_{a\in A} d(x,a)$
denotes the distance between the point $x$ and the subset $A\subset U$ with respect to any metric $d$ on 
$U$. Then $U'$ is an open subset of $V$.
By construction, its intersection with $s^{-1}(0)=\psi^{-1}(F)$ is $V\cap \psi^{-1}(F')=\psi^{-1}(F')$.
To see that $\ov{U'} \cap \s^{-1}(0) \subset \psi^{-1}(\ov{F'})$, consider a sequence 
$x_n\in U'$ that converges to $x_\infty\in \psi^{-1}(F)$.
If $x_\infty\in \psi^{-1}(F\less F')$ then by definition of $U'$ there are points $y_n\in \psi^{-1}(F')$ such that $d(x_n,y_n)< d(x_n,x_\infty)$.
This implies $d(x_n,y_n)\to 0$, hence we also get convergence $y_n\to x_\infty$, which proves $x_\infty\in \ov{\psi^{-1}(F')} = \psi^{-1}(\ov{F'})$, where the last equality is by the homeomorphism property of $\psi$.
Finally, the same homeomorphism property implies the inclusion $\psi^{-1}(\ov{F'}) = \ov{\psi^{-1}(F')} \subset \ov{U'} \cap \s^{-1}(0)$, and thus equality.
This proves the first statement.

The second statement will hold if we show that for precompact $F$ we may choose $V$, and hence $U'\subset V$ to be precompact in $U$.
For that purpose we use the homeomorphism property of $\psi$ and closedness of $\s^{-1}(0)\subset U$ to deduce that $\psi^{-1}(F')\subset U$ is a precompact set in a locally compact space.
So it has a precompact open neighbourhood $V\sqsubset U$, since each point in $\ov{\psi^{-1}(F')}$ 
has a precompact neighbourhood by local compactness of $U$ and a compactness argument provides a finite covering of $\ov{\psi^{-1}(F')}$ by such precompact neighbourhoods. 
\end{proof}

\subsection{Coordinate changes} \label{ss:coord}   \hspace{1mm}\\ \vspace{-3mm}

The following notion of coordinate change is key to the definition of Kuranishi atlases.
Here we begin using notation that will also appear in our definition of Kuranishi atlases. For now, $\bK_I=(U_I,\E_I,
\s_I,\psi_I)$ and $\bK_J=(U_J,\E_J,\s_J,\psi_J)$ just denote different Kuranishi charts for the same space $X$.
Also recall that the data of the obstruction bundles $\E_I$,  resp.\ $\E_J$, implicitly contains projections $\pr_I$,
resp.\ $\pr_J$,  and zero sections $0_I$, resp.\ $0_J$. 

\begin{figure}[htbp] 
   \centering
   \includegraphics[width=4in]{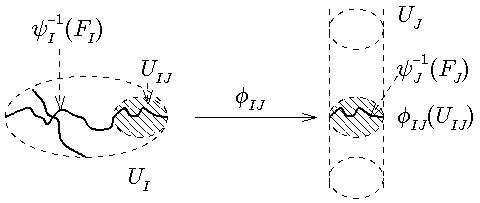}
   \caption{A coordinate change in which $\dim U_{J} =\dim U_I+ 1$.
   Both $U_{IJ}$ and its image $\phi_{IJ}(U_{IJ})$ are shaded.}
   \label{fig:2}
\end{figure}

\begin{defn}\label{def:tchange}
Let $\bK_I$ and $\bK_J$ be topological Kuranishi charts such that $F_I\cap F_J$ is nonempty.
A {\bf topological coordinate change} from $\bK_I$ to $\bK_J$ is a map $\Hat\Phi: \bK_I|_{U_{IJ}}\to \bK_J$ defined on a restriction of $\bK_I$ to $F_I\cap F_J$. More precisely:  
\begin{itemize}
\item
The {\bf domain} of the coordinate change is an open subset $U_{IJ}\subset U_I$ such that
$\s_I^{-1}(0_I)\cap U_{IJ} = \psi_I^{-1}(F_I\cap F_J)$.
\item 
The {\bf map} of the coordinate change is a topological embedding (i.e.\ homeomorphism to its image) 
$\Hat\Phi: \E_I|_{U_{IJ}}:=\pr_I^{-1}(U_{IJ}) \to \E_J$ that satisfies the following.
\begin{enumerate}
\item
It is a bundle map, i.e.\ we have $\pr_J \circ \Hat\Phi = \phi \circ\pr_I |_{\pr_I^{-1}(U_{IJ})}$ for a topological embedding $\phi: U_{IJ}\to U_J$, and it is linear in the sense that $0_J \circ \phi  = \Hat\Phi \circ 0_I|_{U_{IJ}}$.
\item
It intertwines the sections in the sense that $\s_J \circ \phi  = \Hat\Phi \circ \s_I|_{U_{IJ}}$.
\item
It restricts to the transition map induced from the footprints in $X$ in the sense that
$\phi|_{\psi_I^{-1}(F_I\cap F_J)}=\psi_J^{-1} \circ\psi_I : U_{IJ}\cap \s_I^{-1}(0_I) \to \s_J^{-1}(0_J)$.
\end{enumerate}
\end{itemize}

In particular, the following diagrams commute:
\begin{align} 
& \qquad\quad\;\;
 \begin{array} {ccc}
{\E_I|_{U_{IJ}}}& \stackrel{\Hat\Phi} \longrightarrow &
{\E_J} 
\phantom{\int_Quark}  \\
\phantom{sp} \downarrow {\pr_I}&&\downarrow {\pr_J} \phantom{spac}\\
\phantom{s}{U_{IJ}} & \stackrel{\phi} \longrightarrow &{U_J} \phantom{spacei}
\end{array} 
\qquad
 \begin{array} {ccc}
{\E_I|_{U_{IJ}}}& \stackrel{\Hat\Phi} \longrightarrow &
{\E_J} 
\phantom{\int_Quark}  \\
\phantom{sp} \uparrow {0_I}&&\uparrow {0_J} \phantom{spac}\\
\phantom{s}{U_{IJ}} & \stackrel{\phi} \longrightarrow &{U_J} \phantom{spacei}
\end{array}
 \notag \\
\label{eq:map-square}
& \qquad\qquad
 \begin{array} {ccc}
{\E_I|_{U_{IJ}}}& \stackrel{\Hat\Phi} \longrightarrow &
{\E_J} 
\phantom{\int_Quark}  \\
\phantom{sp} \uparrow {\s_I}&&\uparrow {\s_J} \phantom{spac}\\
\phantom{s}{U_{IJ}} & \stackrel{\phi} \longrightarrow &{U_J} \phantom{spacei}
\end{array} 
\qquad
 \begin{array} {ccc}
{U_{IJ}\cap \s_I^{-1}(0_I)} & \stackrel{\phi} \longrightarrow &{\s_J^{-1}(0_J)} \phantom{\int_Quark} \\
\phantom{spa} \downarrow{\psi_I}&&\downarrow{\psi_J} \phantom{space} \\
\phantom{s}{X} & \stackrel{{\rm Id}} \longrightarrow &{X}. \phantom{spaceiiii}
\end{array}
\end{align}
\end{defn}

The map $\Hat\Phi$ is not required to be locally surjective.
Indeed, the rank of the obstruction bundles $\E$ will typically be different for different charts.

\begin{rmk}\rm   \label{rmk:index}
A {\bf smooth coordinate change} between smooth Kuranishi charts, as defined in \cite{MW1} for trivial isotropy $\Ga=\{\id\}$, is a topological coordinate change $\Hat\Phi=(\phi,\Hat\phi)$ given by a smooth embedding of domains $\phi:U_{IJ}\to U_J$ and a linear injection of fibers $\Hat\phi:E_I\to E_J$, which satisfy an {\bf index condition}: $\rd\phi$ and $\Hat\phi$ identify the kernels and cokernels of $\rd s_I$ and $\rd s_J$. This is an essential requirement for the perturbative construction of the virtual fundamental cycle of a smooth Kuranishi atlas. However, the only consequence relevant in the topological context is the openness condition (iv) in Definition~\ref{def:Ku3}.
\hfill$\er$
\end{rmk}

\begin{example}\label{ex:change}  \rm
Here is  a basic example of a topological coordinate change with $I=\{1\}\subset J=\{1,2\}$ 
between charts on the finite set $X = \{-1,0,1\} $.
Consider the two Kuranishi charts with $U_1=(-2,2)$, $\E_1=U_1\times \R$,  $\s_1(x)=(x,(x^2-1)x^2)$ and 
$U_{12}=(-1,2)\times (-1,1)$, $\E_{12}=U_{12}\times \R^2$, $\s_{12} (x,y) = \bigl(x,y, (x^2-1)x^2 , y \bigr)$, with the obvious projections and zero maps and footprint maps given by the obvious identification.
Their footprints are $\{0,1\}= F_{12}\subset F_1= \{-1,0,1\} = X$, and both charts
have dimension $0$ although their domains and obstruction spaces are not locally diffeomorphic.
A natural coordinate change that extends the identification $\s_1^{-1}(0_1)\supset \{0,1\}\cong \{(0,0),(1,0)\} =\s_{12}^{-1}(0_{12})$ is the inclusion $\phi: x\mapsto (x,0)$ of $U_{1,12}:=(-1,2)$  onto $(-1,2)\times \{0\}\subset U_{12}$ together with $\Hat\Phi: (x,v)\mapsto (x,0,v,0)$.
$\hfill\er$
\end{example}

Note that coordinate changes are in general {\it unidirectional} since the map $\E_I|_{U_{IJ}}\to\E_J$ is not assumed to have open image.
Note also that the footprint of the intermediate chart $\bK_I|_{U_{IJ}}$ is always the full intersection $F_I\cap F_J$. 
Moreover, in Kuranishi atlases we will only have coordinate changes when $F_J\subset F_I$, so that this intersection is $F_I\cap F_J=F_J$.
By abuse of notation, we often denote a coordinate change by $\Hat\Phi: \bK_I\to \bK_J$, thereby indicating the choice of a domain $U_{IJ}\subset U_I$ and a map $\Hat\Phi:\bK_I|_{U_{IJ}}\to\bK_J$.
Further, for clarity we usually add subscripts, writing $\Hat\Phi_{IJ} = (\phi_{IJ},\Hat\phi_{IJ}): \bK_I\to \bK_J$.
The next lemmas provide restrictions and compositions of coordinate changes.

\begin{lemma} \label{le:restrchange}
Let $\Hat\Phi:\bK_I|_{U_{IJ}}\to \bK_J$ be a topological coordinate change from $\bK_I$ to $\bK_J$, and let $\bK'_I=\bK_I|_{U'_I}$, $\bK'_J=\bK_J|_{U'_J}$ be restrictions to open subsets $F'_I\subset F_I, F'_J\subset F_J$ with $F_I'\cap F_J'\ne \emptyset$.
Then a {\bf restricted topological coordinate change} 
$\Hat\Phi|_{U'_{IJ}}: \bK'_I \to \bK'_J$
is given by any choice of open subset $U'_{IJ}\subset U_{IJ}$ of the domain such that
$$
U'_{IJ} \subset U'_I\cap \phi^{-1}(U'_J) , \qquad \psi_I(\s_I^{-1}(0_I)\cap U'_{IJ}) = F'_I \cap F'_J ,
$$
and the map
$\Hat\Phi|_{\pr_I^{-1}(U'_{IJ})} : \E_I|_{U'_{IJ}} \to \E_J$.
\end{lemma}

\begin{proof}
First note that restricted domains $U'_{IJ}\subset U_{IJ}$ always exist
 since we can choose e.g.\ $U'_{IJ} = U'_I \cap\phi^{-1}(U'_J)$, which is open in $U_{IJ}$ by the continuity of $\phi$ and has the required footprint since
$$
\psi_I\bigl(\s_I^{-1}(0_I) \cap U'_I \bigr) \cap  \psi_I\bigl(\s_I^{-1}(0_I) \cap \phi^{-1}(U'_J) \bigr)
\;=\; F'_I\cap \psi_J\bigl(\s_J^{-1}(0_J) \cap U'_J \bigr)
\;=\; F'_I\cap F'_J .
$$
Next, $\bK'_I|_{U'_{IJ}}=\bK_I|_{U'_{IJ}}$ is a restriction of $\bK'_I$ to $F'_I\cap F'_J$ since it has the required footprint
$$
\psi'_I({\s_I'}\,\!^{-1}(0_I)\cap U'_{IJ}) \;=\; \psi_I\bigl(\s_I^{-1}(0_I) \cap U'_{IJ} \bigr)
\;=\; F'_I\cap F'_J .
$$
Finally, $\Hat\Phi' := \Hat\Phi|_{\pr_I^{-1}(U'_{IJ})}$
is the required map since it satisfies the conditions of Definition~\ref{def:tchange} with the induced embedding $\phi':=\phi|_{U'_{IJ}}$,
\begin{enumerate}
\item
$\pr'_J \circ \Hat\Phi' 
=\pr_J \circ \Hat\Phi |_{\pr_I^{-1}(U'_{IJ})}
= \phi |_{U'_{IJ}} \circ\pr_I  |_{\pr_I^{-1}(U'_{IJ})}
= \phi' \circ\pr'_I|_{{\pr'_I}\,\!^{-1}(U'_{IJ})}$, \\
$0'_J \circ \phi'
= 0_J|_{U'_J} \circ \phi|_{U'_{IJ}}
= \Hat\Phi  |_{\pr_I^{-1}(U'_{IJ})}\circ 0_I|_{U'_{IJ}}
= \Hat\Phi' \circ 0'_I|_{U'_{IJ}}$;
\item
$\s'_J \circ \phi'
= \s_J|_{U'_J} \circ \phi|_{U'_{IJ}}
= \Hat\Phi  |_{\pr_I^{-1}(U'_{IJ})}\circ \s_I|_{U'_{IJ}}
= \Hat\Phi' \circ \s'_I|_{U'_{IJ}}$;
\item
$\phi' |_{{\psi'_I}\,\!^{-1}(F'_I\cap F'_J)} = \phi |_{\psi_I^{-1}(F'_I\cap F'_J)}
= \psi_J^{-1} \circ \psi_I |_{\psi_I^{-1}(F'_I\cap F'_J)}
= {\psi'_J}\,\!^{-1} \circ \psi'_I |_{\psi'_I\,\!^{-1}(F'_I\cap F'_J)}$.
\end{enumerate}
This completes the proof.
\end{proof}

\begin{lemma} \label{le:cccomp}
Let $\bK_I,\bK_J,\bK_K$ be topological Kuranishi charts such that
$ F_I\cap F_K \subset F_J$, and let $\Hat\Phi_{IJ}: \bK_I\to \bK_J$ and $\Hat\Phi_{JK}: \bK_J\to \bK_K$ be topological coordinate changes.
(That is, we are given restrictions
$\bK_I|_{U_{IJ}}$ to $F_I\cap F_J$ and $\bK_J|_{U_{JK}}$ to $F_J\cap F_K$ and maps $\Hat\Phi_{IJ}: \bK_I|_{U_{IJ}}\to \bK_J$, $\Hat\Phi_{JK}: \bK_J|_{U_{JK}}\to \bK_K$.)
Then the following holds.
\begin{enumerate}
\item
The domain $U_{IJK}:=\phi_{IJ}^{-1}(U_{JK}) \subset U_I$ defines a restriction $\bK_I|_{U_{IJK}}$
to $F_I \cap F_K$.
\item
The composition $\Hat\Phi_{JK}\circ\Hat\Phi_{IJ}: \E_I|_{U_{IJK}}\to \E_K$ 
defines a map ${\Hat\Phi_{IJK}:\bK_I|_{U_{IJK}}\to\bK_K}$
as in Definition~\ref{def:tchange}, which covers the embedding 
${\phi_{IJK}:=\phi_{JK}\circ\phi_{IJ}: U_{IJK}\to U_K}$.
\end{enumerate}
We denote the induced {\bf composite coordinate change} $\Hat\Phi_{IJK}$ by
$$
\Hat\Phi_{JK}\circ \Hat\Phi_{IJ}
:=\Hat\Phi_{IJK} : \; \bK_I|_{U_{IJK}} \; \to\; \bK_K.
$$
\end{lemma}
\begin{proof}
In order to check that $\bigl(U_{IJK},\E_I|_{U_{IJK}},
\s_I|_{U_{IJK}}, \psi_I|_{\s_I^{-1}(0_I)\cap U_{IJK}}\bigr)$
is the required restriction,
we need to verify that it has footprint $F_I\cap F_K$.
 Indeed, $\psi_I\bigl(\s_I^{-1}(0_I)\cap U_{IJK}\bigr) =F_I\cap F_K$ holds since we may decompose
$\psi_I=\psi_J\circ\phi_{IJ}$ on $\s_I^{-1}(0_I)\cap U_{IJ}$ with $U_{IJK}\subset U_{IJ}$,
and then combine the identities
\begin{align*}
\phi_{IJ}( \s_I^{-1}(0_I) \cap U_{IJK} )
&=\phi_{IJ}( \s_{IJ}^{-1}(0_I) ) \cap U_{JK} \subset \s_J^{-1}(0_J),
\\
\psi_J\bigl( \phi_{IJ}( \s_{IJ}^{-1}(0_I) ) \bigr)
&= \psi_I( \s_{IJ}^{-1}(0_I) ) = F_I\cap F_J ,
\\
\psi_J( \s_{J}^{-1}(0_J)\cap U_{JK}) &= F_J\cap F_K .
\end{align*}
Next, our assumption $F_I\cap F_K\subset F_J$
ensures that ${(F_I\cap F_J)\cap ( F_J\cap F_K) = F_I\cap F_K}$.
This proves the first assertion.
To prove the second claim, first note that both compositions and restrictions (to open subsets) of topological embeddings are again topological embeddings.
Moreover, we check the conditions of Definition~\ref{def:tchange}, 
\begin{itemize}
\item[-]
The composition is a bundle map, i.e.\ on $\pr_I^{-1}(U_{IJK})$ we have
$$
\pr_K\circ\Hat\Phi_{IJK}
= \pr_K\circ\Hat\Phi_{JK}\circ\Hat\Phi_{IJ} 
= \phi_{JK}\circ \pr_J\circ \Hat\Phi_{IJ} 
= \phi_{JK}\circ \phi_{IJ} \circ\pr_I
= \phi_{IJK}\circ\pr_I 
$$
and the weak form of linearity
$$
0_K \circ \phi_{IJK}
= 0_K \circ \phi_{JK} \circ \phi_{IJ}
= \Hat\Phi_{JK} \circ 0_J \circ \phi_{IJ}
= \Hat\Phi_{JK}\circ \Hat\Phi_{IJ} \circ 0_{I}
= \Hat\Phi_{IJK} \circ 0_{IK} .
$$
\item[-]
The sections are intertwined, i.e.\ on $U_{IJK}$ we have
$$
\s_K \circ \phi_{IJK}
= \s_K \circ \phi_{JK} \circ \phi_{IJ}
= \Hat\Phi_{JK} \circ \s_J \circ \phi_{IJ}
= \Hat\Phi_{JK}\circ \Hat\Phi_{IJ} \circ \s_{I}
= \Hat\Phi_{IJK} \circ \s_{IK} .
$$
\item[-]
On the zero set $U_{IJK}\cap \s_I^{-1}(0_I) = \psi_I^{-1}(F_I\cap F_K)$ we have
$$
\phi_{IJK}
= \phi_{JK} \circ \phi_{IJ}
= \bigl( \psi_K^{-1}\circ\psi_{J} \bigr)
\circ \bigl(\psi_{J}^{-1} \circ \psi_{I}\bigr)
= \psi_K^{-1}\circ \psi_{I} .
$$
\end{itemize}
This completes the proof.  \end{proof}

Finally, we introduce  two notions of equivalence between coordinate changes that may not have the same domain.
One can easily shown that these equivalence relations are compatible with composition, but we will not make use of this fact.

\begin{defn} \label{def:overlap}
Let $\Hat\Phi^\al :\bK_I|_{U^\al_{IJ}}\to \bK_J$ and  $\Hat\Phi^\be:\bK_I|_{U^\be_{IJ}}\to \bK_J$ be coordinate changes.
\begin{itemlist}
\item
We say the coordinate changes are {\bf equal on the overlap} and write $\Hat\Phi^\al\approx\Hat\Phi^\be$, if the restrictions of Lemma~\ref{le:restrchange} to $U'_{IJ}:=U^\al_{IJ}\cap U^\be_{IJ}$ yield equal maps
$\Hat\Phi^\al|_{U'_{IJ}} = \Hat\Phi^\be|_{U'_{IJ}} $.
\item
We say that $\Hat\Phi^\be$ {\bf extends} $\Hat\Phi^\al$ and write $\Hat\Phi^\al\subset\Hat\Phi^\be$,
if $U_{IJ}^\al\subset U_{IJ}^\be$ and the restriction of Lemma~\ref{le:restrchange} yields equal maps
$\Hat\Phi^\be|_{U_{IJ}^\al} = \Hat\Phi^\al $.
\end{itemlist}
\end{defn}

\subsection{Covering families and transition data}\label{ss:Ksdef} \hspace{1mm}\\ \vspace{-3mm}

This section defines the notion of topological Kuranishi atlas on $X$ and describes it in categorical terms.
There are various notions of ``Kuranishi structure", but in practice every such structure on a compact moduli space of holomorphic curves is constructed from basic building blocks as follows.

\begin{defn}\label{def:Kfamily}
\begin{itemlist}
\item
A {\bf covering family of basic charts} for $X$ is a finite collection $(\bK_i)_{i=1,\ldots,N}$ of topological Kuranishi charts for $X$ whose footprints cover $X=\bigcup_{i=1}^N F_i$.
\item
{\bf Transition data} for a covering family $(\bK_i)_{i=1,\ldots,N}$ is a collection of topological Kuranishi charts $(\bK_J)_{J\in\Ii_\Kk,|J|\ge 2}$ and coordinate changes $(\Hat\Phi_{I J})_{I,J\in\Ii_\Kk, I\subsetneq J}$ as follows:
\begin{enumerate}
\item
$\Ii_\Kk$ denotes the set of subsets $I\subset\{1,\ldots,N\}$ for which the intersection of footprints is nonempty, $F_I:= \; {\textstyle \bigcap_{i\in I}} F_i  \;\neq \; \emptyset \;$;
\item
$\bK_J$ is a topological Kuranishi chart for $X$ with footprint $F_J=\bigcap_{i\in J}F_i$ for each $J\in\Ii_\Kk$ with $|J|\ge 2$, and for one element sets $J=\{i\}$ we denote $\bK_{\{i\}}:=\bK_i$;
\item
$\Hat\Phi_{I J}$ is a topological coordinate change $\bK_{I} \to \bK_{J}$ for every $I,J\in\Ii_\Kk$ with $I\subsetneq J$.
\end{enumerate}
\end{itemlist}
 \end{defn}

The transition data for a covering family automatically satisfies a cocycle condition on the zero sets, where due to the footprint maps to $X$ we have for $I\subset J \subset K$
$$
\phi_{J K}\circ \phi_{I J}
= \psi_K^{-1}\circ\psi_J\circ\psi_J^{-1}\circ\psi_I
= \psi_K^{-1}\circ\psi_I
= \phi_{I K}
\qquad \text{on}\; s_I^{-1}(0)\cap U_{IK} .
$$
Since there is no natural ambient topological space into which the entire domains of the Kuranishi charts map, the cocycle condition on the complement of the zero sets has to be added as an axiom. 
For the embeddings between the domains of the charts however, there are three natural notions of cocycle condition with varying requirements on the domains of the coordinate changes.

\begin{defn}  \label{def:cocycle}
Let $\Kk=(\bK_I,\Hat\Phi_{I J})_{I,J\in\Ii_\Kk, I\subsetneq J}$ be a tuple of basic charts and transition data. Then for any $I,J,K\in\Ii_K$ with  $I\subsetneq J \subsetneq K$ we define the composed coordinate change $\Hat\Phi_{J K}\circ \Hat\Phi_{I J} : \bK_{I}  \to \bK_{K}$ as in Lemma~\ref{le:cccomp} with domain $\phi_{IJ}^{-1}(U_{JK})\subset U_I$.
Using the notions of Definition~\ref{def:overlap} we then say that the triple of coordinate changes
$\Hat\Phi_{I J}, \Hat\Phi_{J K}, \Hat\Phi_{I K}$ satisfies 
\begin{itemlist}
\item the {\bf weak cocycle condition}
if $\Hat\Phi_{J K}\circ \Hat\Phi_{I J} \approx \Hat\Phi_{I K}$, i.e.\ the coordinate changes are equal on the overlap;
\item the {\bf cocycle condition}
if $\Hat\Phi_{J K}\circ \Hat\Phi_{I J} \subset \Hat\Phi_{I K}$, i.e.\  $\Hat\Phi_{I K}$ extends the composed coordinate change;
\item the {\bf strong cocycle condition}
if $\Hat\Phi_{J K}\circ \Hat\Phi_{I J} = \Hat\Phi_{I K}$ are equal as coordinate changes.
\end{itemlist}
More explicitly, 
\begin{itemlist}
\item 
 the weak cocycle condition requires
 \begin{equation} \label{eq:wc}
\qquad
\Hat\Phi_{J K}\circ \Hat\Phi_{I J} = \Hat\Phi_{I K}
\qquad \text{on}\;\;
\pr_I^{-1}\bigl(\phi_{IJ}^{-1}(U_{JK}) \cap U_{IK} \bigr).
\end{equation}
\item 
 the cocycle condition requires \eqref{eq:wc} and $U_{IJK}:=\phi_{IJ}^{-1}(U_{JK}) \subset U_{IK}$; 
\item
the strong cocycle condition requires \eqref{eq:wc} and $U_{IJK}:=\phi_{IJ}^{-1}(U_{JK}) = U_{IK}$.
\end{itemlist}
\end{defn}

\begin{remark}\label{rmk:cocycle}\rm
For topological coordinate changes satisfying the weak cocycle condition, the second identity follows 
from the bundle map property (i) in Definition~\ref{def:tchange}.
The cocycle condition resp.\ strong cocycle condition require in addition $\phi_{IJ}^{-1}(U_{JK}) \subset U_{IK}$ resp.\ $\phi_{IJ}^{-1}(U_{JK}) =U_{IK}$.
$\hfill\er$
\end{remark}

The relevance of these versions is that the weak cocycle condition can be achieved in practice by constructions of finite dimensional reductions for holomorphic curve moduli spaces, whereas the strong cocycle condition is needed for our construction of a virtual moduli cycle in \cite{MW1} from perturbations of the sections in the Kuranishi charts.
The cocycle condition is an intermediate notion which is too strong to be constructed in practice and too weak to induce a virtual moduli cycle, but it does allow us to formulate Kuranishi atlases categorically. This in turn gives rise, via a topological realization of a category, to a virtual neighbourhood of $X$ into which all Kuranishi domains map.

\begin{defn}\label{def:Ku}
A {\bf topological Kuranishi atlas} on a compact metrizable space
$X$ is a tuple
$
\Kk=\bigl(\bK_I,\Hat\Phi_{I J}\bigr)_{I, J\in\Ii_\Kk, I\subsetneq J}
$
of a covering family of basic charts $(\bK_i)_{i=1,\ldots,N}$ 
and transition data $(\bK_J)_{|J|\ge 2}$, $(\Hat\Phi_{I J})_{I\subsetneq J}$ for $(\bK_i)$ as in Definition~\ref{def:Kfamily}, 
that consist of topological Kuranishi charts and topological coordinate changes, which 
satisfy the {\it cocycle condition} $\Hat\Phi_{J K}\circ \Hat\Phi_{I J} \subset \Hat\Phi_{I K}$ for every triple $I,J,K\in\Ii_K$ with $I\subsetneq J \subsetneq K$.
\end{defn}

\begin{rmk}\label{rmk:Ku}\rm
We have assumed from the beginning that $X$ is compact and metrizable.
Some version of  compactness is essential in order for $X$ to define a virtual fundamental class, but one might hope to weaken the metrizability assumption.  
However, we claim that any compact space $X$ that is covered by topological Kuranishi charts is automatically metrizable.
Indeed, one of the most basic properties of topological Kuranishi charts is that the footprint maps $\psi_i: \s_i^{-1}(0_i)\to X$ are homeomorphisms between subsets $\s_i^{-1}(0_i)\subset U_i$ of metrizable spaces and open subsets $F_i \subset X$. If $X$ is covered by such charts, it will be locally metrizable. Since compactness also implies paracompactness, such $X$ is metrizable by the Smirnov metrization theorem \cite[Thm.42.1]{Mun}.
$\hfill\er$
\end{rmk}

It is useful to think of the domains and obstruction spaces of a topological Kuranishi atlas as forming the following categories.

\begin{defn}\label{def:catKu}
Given a topological Kuranishi atlas $\Kk$ we define its {\bf domain category} $\bB_\Kk$ to consist of
the space of objects\footnote{
When forming categories such as $\bB_\Kk$, we take always the space of objects 
to be the disjoint union of the domains 
$U_I$, even if we happen to have defined the sets $U_I$ 
as subsets of some larger space such as $\R^2$ 
or a space of maps as in the Gromov--Witten case.
Similarly, the morphism space is a disjoint union of the $U_{IJ}$ even though $U_{IJ}\subset U_I$ for all $J\supset I$.}
$$
\Obj_{\bB_\Kk}:= \bigsqcup_{I\in \Ii_\Kk} U_I \ = \ \bigl\{ (I,x) \,\big|\, I\in\Ii_\Kk, x\in U_I \bigr\}
$$
and the space of morphisms
$$
\Mor_{\bB_\Kk}:= \bigsqcup_{I,J\in \Ii_\Kk, I\subset J} U_{IJ} \ = \ \bigl\{ (I,J,x) \,\big|\, I,J\in\Ii_\Kk, I\subset J, x\in U_{IJ} \bigr\}.
$$
Here we denote $U_{II}:= U_I$ for $I=J$, and for $I\subsetneq J$ use
the domain $U_{IJ}\subset U_I$ of the restriction $\bK_I|_{U_{IJ}}$ to $F_J$
that is part of the coordinate change $\Hat\Phi_{IJ} : \bK_I|_{U_{IJ}}\to \bK_J$.

Source and target of these morphisms are given by
$$
(I,J,x)\in\Mor_{\bB_\Kk}\bigl((I,x),(J,\phi_{IJ}(x))\bigr),
$$
where $\phi_{IJ}: U_{IJ}\to U_J$ is the  embedding given by $\Hat\Phi_{I J}$, and we denote $\phi_{II}:={\rm id}_{U_I}$.
Composition\footnote
{
Note that we write compositions in the categorical ordering here.} is defined by
$$
\bigl(I,J,x\bigr)\circ \bigl(J,K,y\bigr)
:= \bigl(I,K,x\bigr)
$$
for any $I\subset J \subset K$ and $x\in U_{IJ}, y\in  U_{JK}$ such that $\phi_{IJ}(x)=y$.

The {\bf obstruction category} $\bE_\Kk$ is defined in complete analogy to $\bB_\Kk$ to consist of
the spaces of objects $\Obj_{\bE_\Kk}:=\bigsqcup_{I\in\Ii_\Kk} \E_I$ and morphisms
$$
\Mor_{\bE_\Kk}: = \bigl\{ (I,J,e) \,\big|\, I,J\in\Ii_\Kk, I\subset J,  e\in \E_I|_{U_{IJ}} \bigr\},
$$
with source and target maps
$$
 (I,J,e) \mapsto (I,e) , \qquad  (I,J,e) \mapsto (J,\Hat\Phi_{IJ}(e)).
$$
\end{defn}

We also express the further parts of a topological Kuranishi atlas in categorical terms:

\begin{itemlist}
\item
The obstruction category $\bE_\Kk$ is a bundle over $\bB_\Kk$ in the sense that there is a functor
$\pr_\Kk:\bE_\Kk\to\bB_\Kk$ that is given on objects and morphisms by projection 
$(I,e)\mapsto (I,\pr_I(e))$ and $(I,J,e)\mapsto(I,J,\pr_I(e))$.

\item
The zero sections $0_I$ and sections $\s_I$ induce two continuous sections of this bundle, i.e.\ functors $0_\Kk:\bB_\Kk\to \bE_\Kk$ and $\s_\Kk:\bB_\Kk\to \bE_\Kk$ which act continuously on the spaces of objects and morphisms, and whose composite with the projection $\pr_\Kk: \bE_\Kk \to \bB_\Kk$ is the identity. More precisely, $\s_\Kk$ is given by $(I,x)\mapsto (I,\s_I(x))$ on objects and by $(I,J,x)\mapsto (I,J,\s_I(x))$ on morphisms, and analogously for $0_\Kk$.

\item
The zero sets of the sections $\bigsqcup_{I\in\Ii_\Kk} \{I\}\times \s_I^{-1}(0_I)\subset\Obj_{\bB_\Kk}$ form a very special strictly full subcategory $\s_\Kk^{-1}(0_\Kk)$ of $\bB_\Kk$. Namely, $\bB_\Kk$ splits into the subcategory $\s_\Kk^{-1}(0_\Kk)$ and its complement (given by the full subcategory with objects  $\{ (I,x) \,|\, \s_I(x)\neq 0_I(x) \}$) in the sense that there are no morphisms of $\bB_\Kk$ between the underlying sets of objects.
(This holds because, given any morphism $(I,J,x)$,
the injectivity of $\Hat\Phi_{IJ}$ ensures that we have 
$\s_I(x)=0_I(x) \;\Leftrightarrow\; \Hat\Phi_{IJ}(\s_I(x))=\Hat\Phi_{IJ}(0_I(x))
\;\Leftrightarrow\; \s_J(\phi_{IJ}(x))= 0_J(\phi_{IJ}(x))$.)

\item
The footprint maps $\psi_I$ give rise to a surjective functor $\psi_\Kk: \s_\Kk^{-1}(0_\Kk) \to \bX$ 
to the category $\bX$ with object space $X$ and trivial morphism spaces.
It is given by $(I,x)\mapsto \psi_I(x)$ on objects and by $(I,J,x)\mapsto {\rm id}_{\psi_I(x)}$ on morphisms.
\end{itemlist}

\begin{lemma}\label{le:Kcat}  
The categories $\bB_{\Kk}$ and $\bE_{\Kk}$  are well defined. 
\end{lemma}
\begin{proof}  
We must check that the composition of morphisms in $\bB_{\Kk}$ is well defined and associative; 
the proof for $\bE_\Kk$ is analogous.
To see this, note that the composition $\bigl(I,J,x\bigr)\circ \bigl(J,K,y\bigr)$ only needs to be defined for $x=\phi_{IJ}^{-1}(y)\in\phi_{IJ}^{-1}(U_{JK})$, i.e.\ for $x\in U_{IJK}$ in the domain of the composed coordinate change $\Hat\Phi_{JK}\circ\Hat\Phi_{IJ}$, which by the cocycle condition
is contained in the domain of $\Hat\Phi_{IK}$, and hence $\bigl(I,K,x\bigr)$ is a well defined morphism.
With this said, identity morphisms are given by $\bigl(I,I,x\bigr)$ for all $x\in U_{II}=U_I$, and the composition is associative since for any $I\subset J \subset K\subset L$, and $x\in U_{IJ}, y\in U_{JK}, z\in U_{KL}$ the three morphisms
$\bigl(I,J,x\bigr), \bigl(J,K,y\bigr),\bigl(K,L,z\bigr)$ are composable iff
$y=\phi_{IJ}(x)$ and $z=\phi_{JK}(y)$. In that case we have
$$
\bigl(I,J,x\bigr)\circ \Bigl(\bigl(J,K,y\bigr) \circ \bigl(K,L,z)\bigr) \Bigr) \\
\;=\;
\bigl(I,J,x\bigr)\circ \bigl(J,L,\phi_{IJ}(x)\bigr)
\;=\;
\bigl(I,L,x\bigr)
$$
and $z=\phi_{JK}(\phi_{IJ}(x))=\phi_{IK}(x)$, hence
$$
\Bigl( \bigl(I,J,x\bigr)\circ \bigl(J,K,y\bigr) \Bigr) \circ \bigl(K,L,z\bigr)
\;=\; \bigl(I,K,x\bigr)\circ \bigl(K,L,\phi_{IK}(x)\bigr)
\;=\; \bigl(I,L,x\bigr),
$$
which proves associativity.
\end{proof}

\subsection{The virtual neighbourhood} \label{ss:vnbhd} \hspace{1mm}\\ \vspace{-3mm}

The categorical formulation of a topological  Kuranishi atlas $\Kk$ allows us to construct a topological space  $|\Kk|$ which contains a homeomorphic copy $\io_{\Kk}(X)\subset |\Kk|$ of $X$ and hence may be viewed as a virtual neighbourhood of $X$.

\begin{defn}  \label{def:Knbhd}
Let $\Kk$ be a  topological Kuranishi atlas for the compact space $X$.
Then the {\bf virtual neighbourhood} of $X$,
$$
|\Kk| := \Obj_{\bB_\Kk}/{\scriptstyle\sim}
$$
is the topological realization\footnote
{
As is usual in the theory of \'etale groupoids we take the realization of
the category $\bB_\Kk$ to be a quotient of its space of objects rather than the classifying space
of the category $\bB_\Kk$ (which is also sometimes called the topological realization).} 
of the category $\bB_\Kk$, that is the quotient of the object space $\Obj_{\bB_\Kk}$ by the equivalence relation generated by
$$
\Mor_{\bB_\Kk}\bigl((I,x),(J,y)\bigr) \ne \emptyset \quad \Longrightarrow \quad
(I,x) \sim (J,y) .
$$
We denote by  $\pi_\Kk:\Obj_{\bB_\Kk}\to |\Kk|$ the natural projection $(I,x)\mapsto [I,x]$, where $[I,x]\in|\Kk|$ denotes the equivalence class containing $(I,x)$.
We moreover equip $|\Kk|$ with the quotient topology, in which $\pi_\Kk$ is continuous.
Similarly, we define
$$
|\bE_\Kk|:=\Obj_{\bE_\Kk} /{\scriptstyle\sim}
$$
to be the topological realization of the obstruction category $\bE_\Kk$.  The natural projection $\Obj_{\bE_\Kk}\to |\bE_\Kk|$ is denoted $\pi_{\bE_\Kk}$.
\end{defn}

\begin{lemma} \label{le:Knbhd1}
The functor ${\rm pr}_\Kk:\bE_\Kk\to\bB_\Kk$ induces a continuous map
$$
|{\rm pr}_\Kk|:|\bE_\Kk| \to |\Kk|,
$$
which we call the {\bf obstruction bundle} of $\Kk$, although its fibers generally do not have the structure of a vector space.
However, the functors $0_\Kk$ and $\s_\Kk$ induce continuous maps
$$
|0_\Kk| : \; |\Kk| \to |\bE_\Kk| , \qquad\qquad |\s_\Kk|:|\Kk|\to |\bE_\Kk| .
$$
These maps are sections in the sense that $|\pr_\Kk|\circ|\s_\Kk| =  |\pr_\Kk|\circ |0_\Kk|= {\rm id}_{|\Kk|}$.
Moreover, there is a natural homeomorphism from the realization of the subcategory $\s_\Kk^{-1}(0_\Kk)$
(with quotient topology) to the zero set of the section
$|\s_\Kk|$, with the relative topology induced from $|\Kk|$,
$$
\bigr| \s_\Kk^{-1}(0_\Kk)\bigr| \;=\; 
\quotient{\s_\Kk^{-1}(0_\Kk)}{\sim_{\scriptscriptstyle \s_\Kk^{-1}(0_\Kk)}}
\;\overset{\cong}{\longrightarrow}\;
|\s_\Kk|^{-1}(|0_\Kk|), 
$$
where $|\s_\Kk|^{-1}(|0_\Kk|)  \,:=\;  \bigl\{p \in |\Kk| \,\big|\, |\s_\Kk|(p) = |0_\Kk|(p)  \bigr\}  \;\subset\; |\Kk|$.
Moreover, the footprint functor $\psi_\Kk: \s_\Kk^{-1}(0_\Kk) \to \bX$ descends to a homeomorphism $|\psi_\Kk| :  |\s_\Kk|^{-1}(|0_\Kk|) \to X$.   Its inverse is
$$
\io_{\Kk}:= |\psi_\Kk|^{-1} : \; X\;\longrightarrow\;  |\s_\Kk|^{-1}(|0_\Kk|) \;\subset\; |\Kk|, \qquad
p \;\mapsto\; [(I,\psi_I^{-1}(p))] ,
$$
where $[(I,\psi_I^{-1}(p))]$ is independent of  
the choice of $I\in\Ii_\Kk$ with $p\in F_I$.
\end{lemma}
\begin{proof}
The existence, continuity, and identities for $|\pr_\Kk|$, $|0_\Kk|$, and $|\s_\Kk|$ follow from the continuity of, and identities between, the maps induced by $\pr_\Kk$, $0_\Kk$, and $\s_\Kk$ on the object space, together with the following general fact: Any functor $f:A\to B$, which is continuous on the object space, induces a continuous map between the realizations
(where these are given the quotient topology of each category).
Indeed, $|f|:|A|\to|B|$ is well defined since the functoriality of $f$ ensures $a\sim a' \Rightarrow f(a)\sim f(a')$. Then by definition we have $\pi_B\circ f = |f| \circ \pi_A$ with the projections $\pi_A:A\to |A|$ and $\pi_B:B\to |B|$.
To prove continuity of $|f|$ we need to check that for any open subset $U\subset |B|$ the preimage $|f|^{-1}(U)\subset|A|$ is open, i.e.\ by definition of the quotient topology, $\pi_A^{-1}\bigl(|f|^{-1}(U)\bigr)\subset A$ is open. But
$\pi_A^{-1}\bigl(|f|^{-1}(U)\bigr) = f^{-1}\bigl(\pi_B^{-1}(U)\bigr)$, which is open by the continuity of $\pi_B$ (by definition) and $f$ (by assumption).

Towards the last statement, first note that $|\s_\Kk|^{-1}(|0_\Kk|)= \qu{\s_\Kk^{-1}(0_\Kk)}{\sim}  \subset |\Kk|$ is given by the $\sim$ equivalence classes for which one and hence every representative lies in the subcategory $\s_\Kk^{-1}(0_\Kk)$.
Next, recall that the equivalence relation $\sim$ on $\Obj_{\bB_\Kk}$ that defines $|\Kk|$ is given by the embeddings $\phi_{IJ}$, their inverses, and compositions. Since these generators intertwine the zero sets $\s_I^{-1}(0_I)$ and the footprint maps $\psi_I:\s_I^{-1}(0_I)\to F_I$, we have the useful observations
\begin{align} \label{eq:useful1}
\psi_I(x)=\psi_J(y) \quad &\Longrightarrow \quad \;  (I,x)\sim (J,y) , \\
 \label{eq:useful2}
(I,x)\sim (J,y) , \; \s_I(x)=0_I(x) \quad &\Longrightarrow \quad \;  \s_J(y)=0_J(y), \; \psi_I(x)=\psi_J(y).
\end{align}
In particular, \eqref{eq:useful2} implies that the equivalence relation $\sim$ on $\Obj_{\bB_\Kk}$ that defines $|\Kk|=\qu{\Obj_{\bB_\Kk}}{\sim}$, restricted to the objects $\{(I,x) \,|\, \s_I(x)=0_I(x)\}$ of the subcategory $\s_\Kk^{-1}(0_\Kk)$, coincides with the equivalence relation $\sim_{\scriptscriptstyle \s_\Kk^{-1}(0_\Kk)}$ generated by the morphisms of $\s_\Kk^{-1}(0_\Kk)$.
Hence the map $\bigr| \s_\Kk^{-1}(0_\Kk)\bigr| \to |\s_\Kk|^{-1}(|0_\Kk|)$, $[(I,x)]_{\scriptscriptstyle \s_\Kk^{-1}(0_\Kk)} \mapsto [(I,x)]_{\Kk}$ is a bijection. It also is continuous because it is the realization of the functor $\s_\Kk^{-1}(0_\Kk)\to \bB_\Kk$ given by the continuous embedding of the object space.

To check that the inverse is continuous, consider an open subset $Z\subset \bigr| \s_\Kk^{-1}(0_\Kk)\bigr|$, 
that is with open preimage $\pi_\Kk^{-1}(Z)\subset \{(I,x) \,|\, \s_I(x)=0_I(x)\}$.
The latter is given the relative topology induced from $\Obj_{\bB_\Kk}$, hence we have
$\pi_\Kk^{-1}(Z) = 
\Ww\cap \{(I,x) \,|\, \s_I(x)=0_I(x)\}$ for some open subset $\Ww\subset\Obj_{\bB_\Kk}$. 
Now we 
claim that $\Ww$ can be chosen so that 
$\Ww=\pi_\Kk^{-1}(\pi_\Kk(\Ww))$, 
and hence $\pi_\Kk(\Ww)\subset|\Kk|$ is open, and
$\pi_\Kk(\Ww)\cap |\s_\Kk|^{-1}(|0_\Kk|)=Z$.
For that purpose note that each footprint $\psi_I(\pi_\Kk^{-1}(Z)\cap U_I)\subset X$ is open since $\psi_I$ is a homeomorphism from $\s_I^{-1}(0_I)$ to an open subset of $X$ and $Z_I:= \pi_\Kk^{-1}(Z)\cap U_I \subset \s_I^{-1}(0_I)$ is open by assumption.
Hence the finite union $\bigcup_{I\in\Ii_\Kk}\psi_I(Z_I)\subset X$ is open, thus has a closed complement, so that each preimage $C_J := \psi_J^{-1}\bigl( X\less \bigcup_I \psi_I(Z_I)\bigr)\subset U_J$ is also closed by the homeomorphism property of the footprint map $\psi_I$. 
Moreover, by \eqref{eq:useful1} and \eqref{eq:useful2}  the morphisms in $\bB_\Kk$ on the zero sets are determined by the footprint functors, so that we have $\psi_J^{-1}(\psi_I(Z_I))=\pi_\Kk^{-1}(\pi_\Kk(Z_I))\cap U_J=\pi_\Kk^{-1}(Z \cap \pi_\Kk(U_I))\cap U_J$ for each $I,J\in\Ii_\Kk$, and thus
$C_J = \s_J^{-1}(0_J) \less \pi_\Kk^{-1}(Z)$.
With that we obtain an open set
$\Ww: = {\textstyle \bigsqcup_{I\in\Ii_\Kk}} \bigl( U_I \less C_I \bigr) \subset \Obj_{\bB_\Kk}$ such that $\pi_\Kk(\Ww)\subset|\Kk|$ is open since $\Ww$ is invariant under the equivalence relation by $\pi_\Kk$, namely
$$
\pi_\Kk(\Ww) \;=\;  {\textstyle \bigcup_{I\in\Ii_\Kk}} \pi_\Kk( U_I ) \less \bigl( \pi_\Kk(\s_I^{-1}(0_I))  \less Z\bigr)  \;=\; |\Kk| \less  \bigl( |\s_\Kk^{-1}(0_\Kk)| \less Z\bigr) 
$$
so that its preimage is $\Ww$ and hence open, by the identity
$$
\pi_\Kk^{-1}(\pi_\Kk(\Ww))
\;=\; {\textstyle \bigsqcup_{I\in\Ii_\Kk}}  U_I \cap \pi_\Kk^{-1}\bigl( |\Kk| \less  ( |\s_\Kk^{-1}(0_\Kk)| \less Z ) \bigr) 
 \;=\; {\textstyle \bigsqcup_{I\in\Ii_\Kk}}  U_I \less \bigl( \s_I^{-1}(0_I) \less Z \bigr).
$$
Finally, using the above, we check that $\pi_\Kk(\Ww)$ has the required intersection
\begin{align*}
\pi_\Kk(\Ww) \cap  |\s_\Kk|^{-1}(|0_\Kk|)
\;=\; \bigl(  |\Kk| \less  \bigl( |\s_\Kk^{-1}(0_\Kk)| \less Z \bigr) \bigr)  \cap  |\s_\Kk|^{-1}(|0_\Kk|)
\;=\; Z .
\end{align*}
This proves the homeomorphism between the quotient space $\bigl|\s_\Kk^{-1}(0_\Kk)\bigr|$ and 
the subspace $ |\s_\Kk|^{-1}(|0_\Kk|)$.

Next, recall that $\psi_\Kk$ is a surjective functor from $\s_\Kk^{-1}(0_\Kk)$ to $\bX$ with objects $X$ (i.e.\ the footprints $F_I = \psi_I (\s_I^{-1}(0_I))$ cover $X$).
Hence the above general argument for realizations of functors shows that $|\psi_\Kk|$ is well defined, surjective, and continuous when $|\psi_\Kk|$ is considered as a map from the quotient space $|\s_\Kk^{-1}(0_\Kk)|$ to $X$.

The map $|\psi_\Kk|=\io_\Kk^{-1}$ considered here is given by composing this realization of the functor $\psi_\Kk$ with the natural homeomorphism $\bigl|\s_\Kk^{-1}(0_\Kk)\bigr|\overset{\cong}{\to} |\s_\Kk|^{-1}(|0_\Kk|)$.
So it remains to check continuity of its inverse $\io_{\Kk}$ with respect to the subspace topology on $|\s_\Kk|^{-1}(|0_\Kk|)\subset|\Kk|$. 
For that purpose we need to consider an open subset $V\subset|\Kk|$, that is $\pi_\Kk^{-1}(V)\subset \Obj_{\bB_\Kk}$ is open. Since $\Obj_{\bB_\Kk}$ is a disjoint union that means $\pi_\Kk^{-1}(V)=\bigsqcup_{I\in\Ii_\Kk} \{I\}\times W_I$ is a union of open subsets $W_I\subset U_I$.  So in the relative topology $W_I\cap \s_I^{-1}(0_I)\subset \s_I^{-1}(0_I)$ is open, as is its image under the homeomorphism $\psi_I: \s_I^{-1}(0_I) \to F_I \subset X$.
Therefore
$$
\io_{\Kk}^{-1}(V)
\;=\;  |\psi_{\Kk}|(V)
\;=\;\psi_\Kk\left(\s_\Kk^{-1}(0_\Kk)\cap
{\textstyle
\bigsqcup_{I\in\Ii_\Kk}} \{I\}\times W_I\right)
\,=\;
{\textstyle \bigcup_{I\in\Ii_\Kk}} \psi_I(W_I\cap \s_I^{-1}(0_I))
$$
is open in $X$ since it is a union of open subsets.  This completes the proof.
\end{proof}

Note that the injectivity of $\io_{\Kk}:X\to|\Kk|$ could be seen directly from the injectivity property \eqref{eq:useful2} of the equivalence relation $\sim$ on $\s_\Kk^{-1}(0_\Kk)\subset\Obj_{\bB_\Kk}$.
In particular, this property implies injectivity of the projection of the zero sets in fixed charts, $\pi_\Kk :\s_I^{-1}(0_I) \to |\Kk|$. This injectivity however only holds on the zero set.
On $U_I\less \s_I^{-1}(0_I)$, the projections $\pi_\Kk: U_I\to |\Kk|$ need not be injective,
as Example~\ref{ex:Knbhd} below shows.

\MS
The remainder of this section is a collections of examples which show that -- beyond the embedding of $X$ -- 
the virtual neighbourhood generally only has undesirable properties:
The maps from the domains $U_I$ of the charts to $|\Kk|$ need not be injective, $|\Kk|$ need not be Hausdorff, and $|\Kk|$ is -- except in very simple cases -- neither metrizable nor locally compact.

\begin{example}[Failure of Injectivity]\label{ex:Knbhd}\rm
The circle $X=S^1=\R/\Z$ can be covered by a single ``global'' smooth Kuranishi chart $\bK_0$ of dimension $1$ with domain $U_0= S^1\times \R$, obstruction space $E_0 =\R$, section map $s_0 = \pr_\R: U_0\to E_0 = \R$, and footprint map $\psi_i= \pr_{S^1}$.
A slightly more complicated smooth Kuranishi atlas (involving transition charts but still no cocycle conditions) can be obtained by the open cover $S^1=F_1\cup F_2\cup F_3$ with $F_i=(\frac i3,\frac{i+2}3) \subset \R/\Z$ such that all pairwise intersections $F_{ij}:=F_i\cap F_j \neq \emptyset$ are nonempty, but the triple intersection $F_1\cap F_2\cap F_3$ is empty. We obtain a covering family of basic charts $\bigl(\bK_i:=\bK_0|_{U_i}\bigr)_{i=1,2,3}$ with these footprints by restricting $\bK_0$ to the open domains $U_i:=F_i\times (-1,1)\subset S^1\times \R$.
Similarly, we obtain transition charts $\bK_{ij}:=\bK_0|_{U_i\cap U_j}$ and coordinate changes $\Hat\Phi_{i, ij}:= \Hat\Phi_{0,0}|_{U_i\cap U_j}$ by restricting the identity map $\Hat\Phi_{0,0}=({\rm id}_{U_0}, {\rm id}_{E_0}):\bK_0 \to \bK_0$ to the overlap $U_{ij}:= U_i\cap U_j$.
These are well defined for any pair $i,j\in\{1,2,3\}$ (and satisfy all cocycle conditions), but for a Kuranishi atlas it suffices to restrict to $i<j$. That is, the transition charts $\bK_{12},\bK_{13},\bK_{23}$ and corresponding coordinate changes
$\Hat\Phi_{1,12}, \Hat\Phi_{2,12}, \Hat\Phi_{1,13}, \Hat\Phi_{3,13}, \Hat\Phi_{2,23}, \Hat\Phi_{3,23}$ form transition data, for which the cocycle condition is vacuous.
The realization of this 
Kuranishi atlas is $|\Kk|=U_1\cup U_2\cup U_3\subset S^1\times \R$, and the maps $U_i\to |\Kk|$ are injective.

However, keeping the same basic charts $\bK_1, \bK_2$, and transition data for $i,j\in\{1,2\}$, we may choose $\bK_3$ to have the same form as $\bK_0$ but with domain $U_3\subset (0,2)\times \R$ such that the projection $\pi:\R\times \R \to S^1\times \R$ embeds $U_3\cap (\R\times\{0\})=(1,\frac 23)\times\{0\}$ to $F_3\times\{0\}$. We can moreover choose $U_3$ so large that the inverse image of $U_1\cap U_2$ meets $U_3$ in two components $\pi^{-1}(U_1\cap U_2)\cap U_3 = V_3^1 \sqcup V_3^2$ with $\pi(V_3^1)=\pi(V_3^2)$, but there are continuous lifts $\pi^{-1}: U_i \cap \pi(U_3) \to U_3$ with $V^i_3\subset\pi^{-1}(U_i)$; cf.\ Figure~\ref{fig:3}.
These intersections $V^i_3\subset U_3$ necessarily lie outside of the zero section $s_3^{-1}(0)=F_3\times\{0\}$, though their closure might intersect it.
Then it remains to construct transition data from $\bK_i$ for $i=1,2$ to $\bK_3$.
We choose the transition charts as restrictions $\bK_{i3}:= \bK_3|_{U_{i3}}$ of $\bK_3$ to the domains $U_{i3}:= \pi^{-1}(U_1)\cap U_3$, with transition maps $\Hat\Phi_{3,i3}:=\Hat\Phi_{3,3}|_{U_{i3}}$. Finally, we construct the transition maps $\Hat\Phi_{i,i3}:\bK_i|_{U_{i,i3}}\to\bK_3$ for $i=1,2$ by the identity $\Hat\phi_{i,i3}:={\rm id}_{E_i}$ on the identical obstruction spaces $E_i=E_3=E_0$ and the lift
$\phi_{i,i3}:= \pi^{-1}$ on the domain $U_{i,i3}: = U_1\cap \pi(U_3)$.

\begin{figure}[htbp] 
   \centering
   \includegraphics[width=4in]{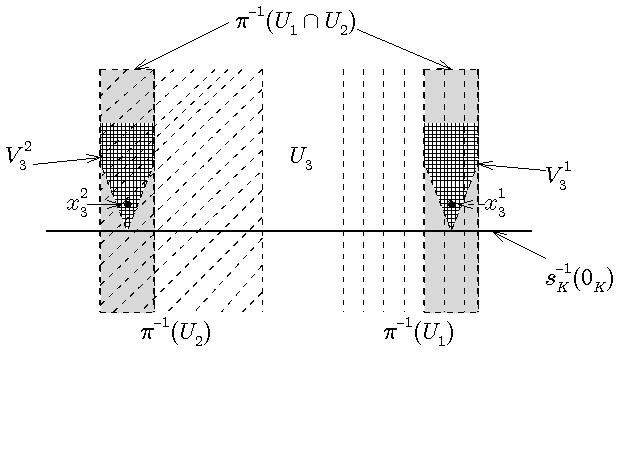}
   \vspace{-2cm}
   \caption{
The lift $\pi^{-1}(U_1\cap U_2)$ is shown as two light grey strips in $\R\times \R$, intersecting the dark grey region $U_3$ in the two shaded sets $V_3^1, V_3^2$.  The domains $U_1,U_2\subset S^1\times\R$ lift injectively to the dashed sets. The points $x_3^1\neq x_3^2 \in U_3$ have the same image in $|\Kk|\subset S^1\times\R$.
}
   \label{fig:3}
\end{figure}

This again defines a smooth Kuranishi atlas with vacuous cocycle condition, but the map $\pi_\Kk: U_3\to |\Kk|$ is not injective.
Indeed any point $x_3^1\in V_3^1\subset U_3$ is identified $[x_3^1]=[x_3^2]\in|\Kk|$ with the corresponding point $x_3^2\in V_3^2$ with $\pi(x_3^1)=\pi(x_3^2)= y \in S^1\times\R$.
Indeed, denoting by $(ij, z)$ the point $z$ considered as an element of $U_{ij}$ (which is just a simplified version of the previous notation $(I,x)$ for a point $x\in U_I$), we have
\begin{equation}\label{eq:equiv123}
(3,x_3^1)\sim (13, x_3^1)\sim (1,y)\sim(12,y)\sim(2,y)\sim (23,x_3^2)\sim (3,x_3^2),
\end{equation}
 where each equivalence is induced by the relevant coordinate change.
Since there are such points $x_3^1$ arbitrarily close to the zero set $s_3^{-1}(0)=F_3\times\{0\}$, the projection $\pi_\Kk:U_3\to |\Kk|$ is not injective on any neighborhood of the zero set $s_3^{-1}(0)$.
$\hfill\er$
\end{example}

Next, we give a simple example where $|\Kk|$ is not Hausdorff in any neighbourhood of $\io_\Kk(X)$ even though the map $s\times t:\Mor_{\bB_\Kk}\to \Obj_{\bB_\Kk}\times \Obj_{\bB_\Kk}$ is proper.

\begin{example}[Failure of Hausdorff property]  \label{ex:Haus}\rm
We construct a smooth Kuranishi atlas for $X := \R$, starting with a basic chart whose footprint $F_1=\R$ already covers $X$,
$$
\bK_1 := \bigl(\, U_1=\R^2 \,,\, E_1=\R \,,\, s_1(x,y)=y \,,\, \psi_1(x,0)=x \,\bigr) .
$$
We then construct a second basic chart $\bK_2:=\bK_1|_{U_2}$ with footprint $F_2=(0,\infty)\subset\R$ and the transition chart $\bK_{12}:=\bK_1|_{U_{12}}$ as restrictions of $\bK_1$ to the domains
$$
U_2: = \{-y<x\le0\}\cup \{x>0\}
, \qquad
U_{12} := \{ x>0\} .
$$
This induces coordinate changes $\Hat\Phi_{i,12}:=\Hat\Phi_{1,1}|_{U_{i,12}} : \bK_i|_{U_{i,12}}\to\bK_{12}$ for $i=1,2$ given by restriction of the trivial coordinate change $\bigl(\phi_{1,1}={\rm id}_{\R^2} , \Hat\phi_{1,1}={\rm id}_\R\bigr)$ to $U_{i,12} := U_{12}$.
This defines a Kuranishi atlas since there are no compositions of coordinate changes for which a cocycle condition needs to be checked.
Moreover, $s\times t$ is proper because
on each of the finitely many connected components of $\Mor_{\bB_\Kk}$ the target map $t$ restricts to a homeomorphism to a connected component of $\Obj_{\bB_\Kk}$. (For example, $t:\Mor_{\bB_\Kk} \supset U_{i,12} \to U_{12} \subset \Obj_{\bB_\Kk}$ is the identity.)

On the other hand the images in $|\Kk|$ of the points $(0,y)\in U_1$ and $(0,y)\in U_2$
for $y>0$ have no disjoint neighbourhoods
since for every $x>0$
$$
\bigl( 1, (x,y)\bigr) \sim
\bigl( 12, (x,y)\bigr) \sim
\bigl( 2, (x,y)\bigr) .
$$
Therefore $\io_\Kk(X)$ does not have a Hausdorff neighbourhood in $|\Kk|$.
$\hfill\er$
\end{example}

In Sections~\ref{ss:Haus}  and \ref{ss:shrink} below we will achieve both the injectivity and the Hausdorff property by a subtle shrinking of the domains of charts and coordinate changes. However, we are still unable to make the virtual neighbourhood $|\Kk|$ locally compact or even metrizable, due to the following natural example.

\begin{example}
[Failure of metrizability and local compactness]  \label{ex:Khomeo}
\rm
For simplicity we will give an example with noncompact $X = \R$. (A similar example can be constructed with $X = S^1$.)
We construct a smooth Kuranishi atlas $\Kk$ on $X$ by two basic charts, 
$\bK_1 = (U_1=\R, E_1=\{0\}, s=0,\psi_1=\id)$ and
$$
\bK_2 = \bigl(U_2=(0,\infty)\times \R,\ E_2=\R, \ s_2(x,y)= y,\ \psi_2(x,y)= x\bigr),
$$
one transition chart $\bK_{12} = \bK_2|_{U_{12}}$ with domain $U_{12} := U_2$, and the coordinate changes $\Hat\Phi_{i,12}$ induced by the natural embeddings of the domains $U_{1,12} := (0,\infty)\hookrightarrow (0,\infty)\times\{0\}$ and $U_{2,12} := U_2\hookrightarrow U_2$.
Then as a set $|\Kk| = \bigl(U_1\sqcup U_2\sqcup U_{12}\bigr)/\sim$
can be identified with $\bigl(\R\times\{0\}\bigr) \cup \bigl( (0,\infty)\times\R\bigr) \subset \R^2$.
However, the quotient topology at $(0,0)\in|\Kk|$ is strictly stronger than the subspace topology.
That is, for any $O\subset\R^2$ open the induced subset $O\cap|\Kk|\subset|\Kk|$ is open, but some open subsets of $|\Kk|$ cannot be represented in this way.
In fact,  
for any $\eps>0$ and continuous function $f:(0,\eps)\to (0,\infty)$,
the set
$$
U_{f,\eps} \, :=\; \bigl\{ [x] \,\big|\, x\in U_1, |x|< \eps \}  \;\cup\; \bigl\{ [(x,y)] \,\big|\, (x,y)\in U_2,  |x|< \eps , |y|<f(x)\} \;\subset\; |\Kk|
$$
is open in the quotient topology.  Moreover these sets form a basis for the 
neighbourhoods of 
$[(0,0)]$ in the quotient topology. 
To see this, let $V\subset |\Kk|$ be open in the quotient topology. Then,
since $\pi_\Kk^{-1}(V)\cap U_1$ is a neighbourhood of $0$, there is 
$\eps>0$ so that 
$\{ (x,0) \,|\,  |x|<\eps \} \subset V$.  Further, define 
$f:\{x\in\R \,|\, 0<x<\eps \} \to (0,\infty)$ by
$f(x) := \sup \{\de \,|\, B_\de(x,0)\subset V\}$, where $B_\de(x,0)$ is the open ball in $\R^2$ with radius $\de$. Then 
$f(x)>0$ for all 
$0<x<\eps$ because $\pi_\Kk^{-1}(V)\cap U_2$ is a neighbourhood of $(0,x)$.
The triangle inequality implies that 
$f(x')\ge f(x) - |x'-x|$ for all  $0<x,x'<\eps$. Hence $|f(x)- f(x')|\le |x'-x|$, so that $f$ is continuous.  
Thus we have constructed a neighbourhood $U_{f,\eps}\subset |\Kk|$ of $[(0,0)]$ of the above type with $U_{f,\eps}\subset V$.

We will use this to see that the point $[(0,0)]$ does not have a countable neighbourhood basis in the quotient topology.
Indeed, suppose by contradiction that $(U_k)_{k\in\N}$ is such a basis, then by the above we can iteratively find $1>\eps_k>0$ and $f_k:(0,\eps_k)\to(0,\infty)$ so that $U_{f_k,\eps_k}\subset U_k\cap U_{f_{k-1},\frac 12 \eps_{k-1}}$ (with $U_{f_0,\frac 12 \eps_0}$ replaced by $|\Kk|$). In particular, the inclusion $U_{f_k,\eps_k}\subset U_{f_{k-1},\frac 12 \eps_{k-1}}$ implies $\eps_k < \eps_{k-1}$.
Now there exists a continuous function $g:(0,1)\to (0,\infty)$ such that $g(\frac 12\eps_k) < f_k(\frac 12 \eps_k)$ for all $k\in\N$. 
Then the neighbourhood $U_{g,1}$ does not contain any of the $U_k$ because $U_{g,1}\supset U_k \supset U_{f_k,\eps_k}$ implies that $g(\frac 12\eps_k) \geq f_k(\frac 12 \eps_k)$.
This contradicts the assumption that $(U_k)_{k\in\N}$ is a neighbourhood basis of $[(0,0)]$, hence there exists no countable neighbourhood basis.

Note also that the point $[(0,0)]\in|\Kk|$ has no compact neighbourhood with respect to the subspace topology from $\R^2$, and hence neither  with respect to the stronger quotient topology on $|\Kk|$. 
The same failure of local compactness and metrizability occurs for any Kuranishi atlas that involves coordinate changes between charts with domains of different dimension (more precisely the issue arises from an embedding $U_{IJ}\to U_J$ if $U_{IJ}\subset U_I$ is not just a connected component and $\dim U_I < \dim U_J$). In particular, the 
tame Kuranishi atlases that we will work with to achieve the Hausdorff property, will -- except in trivial cases -- always exclude local compactness or metrizability.
$\hfill\er$
\end{example}

\begin{rmk}\label{rmk:Khomeo}\rm  
For the Kuranishi atlas in Example~\ref{ex:Khomeo} there exists an exhausting sequence $\ov{\Aa^n}\subset \ov{\Aa^{n+1}}$ of closed subsets of $\bigsqcup_{I\in \Ii_\Kk} U_I$ with the properties
\begin{itemize}
\item  
each $\pi_\Kk(\ov{\Aa^n})$ contains $\iota_\Kk(X)$;
\item  
each $\pi_\Kk(\ov{\Aa^n})\subset |\Kk|$ is metrizable and locally compact in the subspace topology;
\item 
$\bigcup_{n\in\N} \ov{\Aa^n} = \bigsqcup_{I\in \Ii_\Kk} U_I$.
\end{itemize}
For example, we can take $\ov{\Aa^n}$ to be the disjoint union of the closed sets 
$$
\ov{A_1^n}= [-n,n]\subset U_1, \qquad 
\ov{A_{2}^n} : = \{(x,y)\in U_2 \,\big|\,  x \geq \tfrac 1n, |y| \leq  n\},
$$
and any closed subset $\ov{A_{12}^n} \subset \ov{A_2^n}$.
However, in the limit $[(0,0)]$ becomes a ``bad point'' because its neighbourhoods have to involve open subsets of $U_2$.  

In fact, if we altered Example~\ref{ex:Khomeo} to a Kuranishi atlas for the compact space $X=S^1$, then we could choose $\ov{\Aa^n}$ compact, so that the subspace and quotient topologies on  $\pi_\Kk(\ov{\Aa^n})$
coincide by Proposition~\ref{prop:Ktopl1}~(ii). We emphasize the subspace topology above because that is the one inherited by (open) subsets of $\ov{\Aa^n}$.  For example, the quotient topology on $\pi_\Kk(\Aa^n)$, where $\Aa^n: = \bigcup_I {\rm int}(\ov{A_I^n})$ has the same bad properties at $[(\frac 1n,0)]$ as the quotient topology on $|\Kk|$ has at $[(0,0)]$, while the subspace topology on $\pi_\Kk(\Aa_n)$ is metrizable.
 We prove in Proposition~\ref{prop:Ktopl1} that a similar statement holds for all $\Kk$,
 though there we only consider a fixed set $\ov\Aa$ since we have no need for an exhaustion of the domains.
$\hfill\er$
 \end{rmk}

\section{Topological taming of Kuranishi atlases}\label{s:Ks}

Having defined the notion of topological Kuranishi atlas on a compact metrizable space $X$ (which we fix throughout), the previous chapter constructed a virtual neighbourhood $|\Kk|$ for $X$. However, Examples~\ref{ex:Haus} and~\ref{ex:Knbhd} showed that $|\Kk|$ need not be Hausdorff and that the maps from the domains $U_I$ of the charts to $|\Kk|$ need not be injective.
Moreover, in practice one can construct only weak Kuranishi atlases in the sense of Definition~\ref{def:Kwk}, although they do often have the filtration property of Definition~\ref{def:Ku3}.
The main result of this chapter is then Theorem~\ref{thm:K}, which states that given a filtered weak
topological Kuranishi atlas one can construct a topological Kuranishi atlas $\Kk$, whose neighborhood $|\Kk|$ is Hausdorff and has the injectivity property, and that moreover is well defined up to a natural notion of cobordism.  This cobordism theory is developed in Section~\ref{s:Kcobord}.

\subsection{Filtrations, Metrics, and Tameness} \label{ss:tame} \hspace{1mm}\\ \vspace{-3mm}

We begin by introducing the notion of a filtered weak topological Kuranishi atlas, which can be constructed 
in practice on compactified holomorphic curve moduli spaces, as outlined in \cite{MW1,Mcn}.
We then introduce tameness conditions for topological Kuranishi atlases that imply the Hausdorff property of the virtual neighbourhood. Finally, we provide tools for refining topological Kuranishi atlases to achieve the tameness condition.

\begin{defn}\label{def:Kwk}
A {\bf weak topological Kuranishi atlas} on a compact metrizable space $X$ is a covering family of basic charts with transition data $\Kk=\bigl(\bK_I,\Hat\Phi_{I J}\bigr)_{I, J\in\Ii_\Kk, I\subsetneq J}$ as in Definition \ref{def:Kfamily}, that satisfy the {\it weak cocycle condition} of Definition~\ref{def:cocycle}, that is for every triple $I,J,K\in\Ii_K$ with $I\subsetneq J \subsetneq K$ we have equality on the overlap $\Hat\Phi_{J K}\circ \Hat\Phi_{I J} \approx \Hat\Phi_{I K}$.
\end{defn}

\begin{rmk}\rm
In practice one does not need to construct a metric on the underlying topological space $X$. Rather, it suffices to check that $X$ is compact and 
has a weak topological Kuranishi atlas.
Then $X$ is automatically metrizable as in Remark~\ref{rmk:Ku}.
$\hfill\er$
\end{rmk}

This weaker notion of Kuranishi atlas is crucial for two reasons. Firstly, in the application to moduli spaces of holomorphic curves, it is not clear how to construct Kuranishi atlases that satisfy the cocycle condition.
Secondly, it is hard to preserve the cocycle condition while manipulating Kuranishi atlases, for example by shrinking as we do below.
Note that if $\Kk$ is only a weak Kuranishi atlas then we cannot define its domain category $\bB_{\Kk}$ precisely as in Definition~\ref{def:catKu} since the given set of morphisms is not closed under composition.  We will deal with this by simply not considering this category unless $\Kk$ is a Kuranishi atlas, i.e.\ satisfies the standard cocycle condition in Definition~\ref{def:cocycle}.

On the other hand, the constructions of transition data in practice, e.g.\ in \cite{MW:GW,Mcn},  use a sum construction which has the effect of adding the obstruction bundles 
and thus yields an additivity property $\E_I = {\textstyle \bigoplus_{i\in I}} \; \Hat\Phi_{iI}(\E_i)$ in the smooth context. It generalizes to the following filtration property in the topological context.
Here we simplify the notation by writing $\Hat\Phi_{i I}:= \Hat\Phi_{\{i\} I}$ for the coordinate change $\bK_i =\bK_{\{i\}} \to \bK_I$ where $i\in I$.
This notion uses subsets $\E_{IJ}\subset \E_J$  that play the role of $\Hat\Phi_{IJ} (\E_I) \subset \E_J$, and allows us to formulate a topological version of the index condition, which can only be formulated for Kuranishi charts and coordinate changes whose structure maps have well defined differentials. (Then the index condition requires the differentials to identify kernel and cokernel of the charts.)
This generalized index property will be crucial in the taming construction of Proposition~\ref{prop:proper}.

\begin{defn}\label{def:Ku3}  
Let $\Kk$ be a  weak topological Kuranishi atlas.  We say that $\Kk$ is {\bf filtered} if 
it is equipped with a {\bf filtration}, that is a tuple of closed subsets $\E_{IJ}\subset \E_J$ for each $I,J\in \Ii_\Kk$ with $I\subset J$, that satisfy the following conditions:

\begin{enumerate}
\item $\E_{JJ}= \E_J$ and $\E_{\emptyset J} = \im 0_J$ for all $J\in\Ii_\Kk$;
\item 
$\Hat\Phi_{JK}\bigl(
\pr_J^{-1}(U_{JK})\cap
\E_{IJ}\bigr) = \E_{IK}\cap \pr_K^{-1}(\im \phi_{JK})$ for all $I,J,K\in\Ii_\Kk$ with
${I\subset J\subsetneq K}$;
\item  $\E_{IJ}\cap \E_{HJ} = \E_{(I\cap H)J}$ for all $I,H,J\in\Ii_\Kk$ with $I, H \subset J$;
\item  $\im \phi_{IJ}$ is an open subset of $\s_J^{-1}(\E_{IJ})$
for all $I,J\in \Ii_\Kk$ with $I\subsetneq J$.
\end{enumerate}
\end{defn}

\begin{rmk}\rm  
Applying condition (ii) above to any triple $I=I \subsetneq J$ for $J\in\Ii_\Kk$ gives
\begin{equation}\label{eq:filt}
\im \Hat\Phi_{IJ} = \Hat\Phi_{IJ} (\pr_I^{-1}(U_{IJ}))
=\E_{IJ}\cap \pr_J^{-1}(\im\phi_{IJ}).
\end{equation}
In particular, by the compatibility of coordinate changes $\s_J\circ\phi_{IJ}=\Hat\Phi_{IJ}\circ\s_I|_{U_{IJ}}$
we obtain
$$
\s_J\bigl(\phi_{IJ}(U_{IJ})\bigr) = \Hat\Phi_{IJ}(\s_I(U_{IJ})) \subset  
\im\Hat\Phi_{IJ}
\subset \E_{IJ}.
$$
In other words, the first three conditions above imply the inclusion $\phi_{IJ}(U_{IJ}) \subset \s_J^{-1}(\E_{IJ})$.
Condition (iv) strengthens this by saying the image is open.  
This can be viewed as topological version of the index condition, since for smooth coordinate changes it follows from the index condition.
$\hfill\er$
\end{rmk}

\begin{lemma} \label{le:Ku3}
For any filtration $(\E_{IJ})_{I\subset J}$ on a weak topological Kuranishi atlas $\Kk$ we have for any $H,I,J\in\Ii_\Kk$
\begin{align} \label{eq:CIJ}
 H,I\subset J
\quad\Longrightarrow\quad &
\s_J^{-1}(\E_{IJ}) \cap \s_J^{-1}(\E_{HJ}) = \s_J^{-1}(\E_{(I\cap H)J}), 
\end{align}
in particular
\begin{align}\label{eq:CIJ0}
H\cap I=\emptyset \quad\Longrightarrow\quad & \s_J^{-1}(\E_{IJ}) \cap \s_J^{-1}(\E_{HJ}) = \s_J^{-1}(0_J).
\end{align}
\end{lemma}

\begin{proof}    
These statements all follow from applying $\s_J$ to the defining property 
(iii), and making use of (i) in case $H\cap I=\emptyset$.
\end{proof}

\begin{example}\label{ex:Ku3} \rm
The prototypical example of a smooth atlas that is not filtered  is one that has two basic charts $\bK_1, \bK_2$ with overlapping (but distinct) footprints, sections $\s_1,\s_2$ whose zero sets $\s_1^{-1}(0_1), \s_2^{-1}(0_2)$ have empty interiors, and the same obstruction space,  
which we understand to mean that all bundles $\E_i=U_i\times E$ and $\E_{12}=U_{12}\times E$
have the same fiber~$E$. 
In this case the index condition of Remark~\ref{rmk:index} implies that each $\im\phi_{i(12)}$ is an open submanifold of $U_{12}$ that contains $\s_{12}^{-1}(0_{12})$, see \cite{MW1}.
If this atlas were also filtered, then conditions (iii) and (iv) in Definition~\ref{def:Ku3} would imply that $\s_{12}^{-1}(\E_{1(12)})\cap \s_{12}^{-1}(\E_{2(12)})=\s_{12}^{-1}(0_{12})$ contains the open neighbourhood $\im \phi_{1(12)} \cap \im \phi_{2(12)}$ of the zero set $\s_{12}^{-1}(0_{12})$.  
This is possible only if $\im \phi_{1(12)} \cap \im \phi_{2(12)}=\s_{12}^{-1}(0_{12})$, which implies that $\psi_i^{-1}(F_{12})$ coincides with the open set $U_{i(12)}$,  contradicting the assumption on the sections $\s_i$.
$\hfill\er$
\end{example}

Next, we introduce a notion of metrics on topological Kuranishi atlases that will be part of the main result.

\begin{defn}\label{def:metric}  
A topological Kuranishi atlas $\Kk$ is said to be {\bf metrizable} if there is a bounded metric $d$ on the set $|\Kk|$ such that for each $I\in \Ii_\Kk$ the pullback metric $d_I:=(\pi_\Kk|_{U_I})^*d$ on $U_I$ induces the given topology on the 
domain $U_I$.
In this situation we call $d$ an {\bf admissible metric} on $|\Kk|$. 
A {\bf metric topological Kuranishi atlas} is a pair $(\Kk,d)$ consisting of a metrizable topological Kuranishi atlas together with a choice of  admissible metric $d$.
For a metric topological Kuranishi atlas, we denote the $\de$-neighbourhoods of subsets $Q\subset |\Kk|$ resp.\ $A\subset U_I$ for $\de>0$ by
\begin{align*}
B_\de(Q) &\,:=\; \bigl\{w\in |\Kk|\ | \ \exists q\in Q : d(w,q)<\de \bigr\}, \\
B^I_\de(A) &\,:=\; \bigl\{x\in U_I\ | \ \exists a\in A : d_I(x,a)<\de \bigr\}.
\end{align*}
\end{defn}

It is important to note that an admissible metric generally does not induce the quotient topology on $|\Kk|$, since this may not be not metrizable by Example~\ref{ex:Khomeo}.
However, the following shows that the metric topology on $|\Kk|$ is weaker (has fewer open sets) than the quotient topology.

\begin{lemma}\label{le:metric}  
Suppose that $d$ is an admissible metric on the virtual neighbourhood $|\Kk|$ of a
topological Kuranishi atlas $\Kk$.
Then the following holds.
\begin{enumerate}
\item
The identity $\id_{|\Kk|} :|\Kk| \to (|\Kk|,d)$ is continuous as a map from the quotient topology to the metric topology on $|\Kk|$.
\item
In particular, each set $B_\de(Q)$ is open in the quotient topology on $|\Kk|$, so that the 
existence of an admissible metric implies that
$|\Kk|$ is Hausdorff.
\item 
The embeddings $\phi_{IJ}$ that are part of the coordinate changes 
for $I\subsetneq J\in\Ii_\Kk$
are isometries when considered as maps $(U_{IJ},d_I)\to (U_J,d_J)$.
\end{enumerate}
\end{lemma}

\begin{proof} 
Since the neighbourhoods of the form $B_\de(Q)$ define the metric topology, it suffices to prove that these are also open in the quotient topology, i.e.\ that each subset $U_I\cap \pi_\Kk^{-1}(B_\de(Q))$ is open in $U_I$. 
So consider $x\in U_I$ with $\pi_\Kk(x)\in B_\de(Q)$. By hypothesis there is $q\in Q$ and $\eps>0$ such that $d(\pi_\Kk(x),q)<\de-\eps$, and compatibility of metrics and the triangle inequality then imply the inclusion $\pi_\Kk(B^I_\eps(x))\subset B_\eps(Q)\subset B_\de(Q)$.
Thus $B^I_\eps(x)$ is a neighbourhood of $x\in U_I$ contained in $U_I\cap \pi_\Kk^{-1}(B_\de(Q))$. 
This proves the openness required for (i) 
and (ii).
Since every metric space is Hausdorff, $|\Kk|$ is therefore Hausdorff in the quotient topology as stated in (ii).   
Claim (iii) follows from the construction of $d_I,d_J$ as pullback of $d$ under $\pi_\Kk$ and the fact that $\pi_\Kk(\phi_{IJ}(x))=\pi_\Kk(x)$ for $x\in U_{IJ}$. 
\end{proof}

One might hope to achieve the Hausdorff property by constructing an admissible metric, but the existence of the latter is highly nontrivial. Instead, in a refinement process that will take up the remainder of this chapter, we will first construct a Kuranishi atlas whose virtual neighbourhood has the Hausdorff property, then prove metrizability of certain subspaces, and finally obtain an admissible metric by pullback to a further refined Kuranishi atlas. 
This process will prove the following theorem whose formulation uses the notions of shrinking from Definition~\ref{def:shr}, tameness from Definition~\ref{def:tame}, preshrunk tame shrinking from Proposition~\ref{prop:metric}, and concordance from Definition~\ref{def:Kcobord}.

\begin{thm}\label{thm:K}
Let $\Kk$ be a filtered weak topological Kuranishi atlas.
Then there exists a preshrunk tame shrinking of $\Kk$. It provides a metrizable tame topological Kuranishi atlas $\Kk'$ with domains $(U'_I\subset U_I)_{I\in\Ii_{\Kk'}}$ such that the realizations $|\Kk'|$ and $|\bE_{\Kk'}|$ are Hausdorff in the quotient topology.
In addition, for each $I\in \Ii_{\Kk'} = \Ii_\Kk$ the projection maps $\pi_{\Kk'}: U_I'\to |\Kk'|$ and
$\pi_{\Kk'}:
\E'_I:=\pr_I^{-1}(U'_I)
\to |\bE_{\Kk'}|$ are homeomorphisms onto their images and fit into a commutative diagram
$$
\begin{array}{ccc} 
\E'_I & \stackrel{\pi_{\Kk'}}\longhookrightarrow & |\bE_{\Kk'}|  \quad \\
\;\; \downarrow \scriptstyle \pr_I    & & \;\; \downarrow \scriptstyle |\pr_{\Kk'}| \\
U_I' &
\stackrel{\pi_{\Kk'}} \longhookrightarrow  &|\Kk'| \quad 
\end{array}
$$
Any two such preshrunk tame shrinkings with choices of admissible metrics are concordant by a metric tame topological Kuranishi concordance whose realization also has the above Hausdorff and homeomorphism properties.
\end{thm}

\begin{proof} 
The key step is Proposition~\ref{prop:proper}, which establishes the existence of a tame shrinking. 
As we show in Proposition~\ref{prop:metric}, the existence of a metric tame shrinking is an easy consequence. 
Uniqueness up to metric tame concordance is proven by applying Theorem~\ref{thm:cobord2} to the 
product concordance $[0,1]\times \Kk$ with product filtration.
By Proposition~\ref{prop:Khomeo} and Lemma~\ref{le:cob0} for the concordance, tameness implies the Hausdorff and homeomorphism properties.
The diagram commutes since it arises as the realization of commuting functors to $\pr_{\Kk'}:\bE_{\Kk'}\to\bB_{\Kk'}$.
\end{proof}

The Hausdorff property for the virtual neighbourhood $|\Kk|$ will require the following control of the domains of coordinate changes, which we will achieve in Section~\ref{ss:shrink} by a shrinking from a filtered weak Kuranishi atlas.

\begin{defn}\label{def:tame}
A weak topological Kuranishi atlas is {\bf tame} if it is equipped with a filtration $(\E_{IJ})_{I\subset J}$ such that for all $I,J,K\in\Ii_\Kk$ we have
\begin{align}\label{eq:tame1}
U_{IJ}\cap U_{IK}&\;=\; U_{I (J\cup K)}\qquad\qquad
\quad\;
\forall I\subset J,K ;\\
\label{eq:tame2}
\phi_{IJ}(U_{IK}) &\;=\; U_{JK}\cap 
\s_J^{-1}(\E_{IJ})
 \qquad\forall I\subset J\subset K.
\end{align}
Here we allow equalities, using the notation $U_{II}:=U_I$ and $\phi_{II}:={\rm Id}_{U_I}$.
Further, to allow for the possibility that $J\cup K\notin\Ii_\Kk$, we define
$U_{IL}:=\emptyset$ for $L\subset \{1,\ldots,N\}$ with $L\notin \Ii_\Kk$.
Therefore \eqref{eq:tame1} includes the condition
$$
U_{IJ}\cap U_{IK}\ne \emptyset
\quad \Longrightarrow \quad F_J\cap F_K \ne \emptyset  \qquad \bigl( \quad \Longleftrightarrow\quad
J\cup K\in \Ii_\Kk \quad\bigr).
$$
\end{defn}

The notion of tameness generalizes the identities $F_J\cap F_K=F_{J\cup K}$ and $\psi_J^{-1}(F_{K}) = U_{JK}\cap \s_J^{-1}(0_J)$ between the footprints and zero sets, which we can include into \eqref{eq:tame1} and \eqref{eq:tame2} as the case $I = \emptyset$, by using the notation 
\begin{equation}\label{eq:empty}
U_{\emptyset J}: = F_J,\qquad \phi_{\emptyset J}:=\psi_J^{-1}, 
\qquad \E_{\emptyset J}:= \im 0_J .
\end{equation}
Indeed, the first tameness condition \eqref{eq:tame1} extends the identity for intersections of footprints -- which is equivalent to $\psi_I^{-1}(F_J)\cap \psi_I^{-1}(F_K) = \psi_I^{-1}(F_{J\cup K})$ for all $I\subset J,K$ 
-- to the domains of the transition maps in $U_I$. 
In particular, with $I\subset J\subset K$ it implies nesting of the domains of the transition maps,
\begin{equation}\label{eq:tame4}
U_{IK}\subset U_{IJ} \qquad\forall I\subset J \subset K.
\end{equation}
(This in turn generalizes the $I=\emptyset$ case $F_K\subset F_J$ for $J \subset K$.)
The second tameness condition \eqref{eq:tame2} extends the relation between footprints and zero sets -- equivalent to $\phi_{IJ}(\psi_I^{-1}(F_K)) = U_{JK}\cap \s_J^{-1}(0_J)$ for all $I\subset J\subset K$ --
to a relation between domains of transition maps and preimages of corresponding subbundles by the section.
In particular, with $J=K$ it controls the image of the transition maps,
generalizing the $I=\emptyset$ case $\psi_J^{-1}(F_J) =  \s_J^{-1}(0_J)$ to
$\phi_{IJ}(U_{IJ}) =  \s_J^{-1}(\E_{IJ})$ for all $I\subset J$.
This strengthens the inclusion $\im\phi_{IJ}\subset  \s_J^{-1}(\E_{IJ})$ from Lemma~\ref{le:Ku3}.
As a result, tameness controls the topology and intersections of images of the coordinate changes as follows.

\begin{lemma}\label{le:phitrans} 
If $\Kk$ is a tame topological Kuranishi atlas, then 

the images of the transition maps $\phi_{IJ}$ and $\Hat\Phi_{IJ}$ are closed subsets of the Kuranishi domain $U_J$ resp.\ bundle $\E_J$,
\begin{equation}\label{eq:tame3}
\im\phi_{IJ} =  \s_J^{-1}(\E_{IJ}) \subset U_J , \qquad
\im\Hat\Phi_{IJ}=\E_{IJ}\cap (s_J\circ \pr_J)^{-1}(\E_{IJ}) \subset \E_J ,
\end{equation}
for any $I,J\in\Ii_\Kk$, $I\subset J$.
Moreover,
for any  $J\in\Ii_\Kk$, $H, I\subset J$ with $H \cap I\ne \emptyset$ we have
\begin{equation}\label{eq:tame5}
\im \phi_{H J}\cap\im \phi_{IJ}=\im \phi_{(H\cap I) J} .
\end{equation}
In case $H\cap I=\emptyset$ we have the intersection identity\footnote{This intersection identity is consistent with \eqref{eq:empty} since $F_H\cap F_I \supset F_J$ so that $\s_J^{-1}(0_J) = \psi_J^{-1}(F_J) = \psi_J^{-1}(F_H\cap F_I) =  \im \phi_{\emptyset J}$.} 
$\im \phi_{H J} \cap \im \phi_{IJ} = \s_J^{-1}(0_J)$.
\end{lemma}
\begin{proof} 
The identities in \eqref{eq:tame3} follow from \eqref{eq:tame2} with $J=K$ and \eqref{eq:filt}. Then the subsets $\im\phi_{IJ}\subset U_J$ and $\im\Hat\Phi_{IJ}\subset\E_J$ are closed since they are the preimages of closed subsets $\E_{IJ}\subset\E_J$ under the continuoussection $\s_J$ and projection $\pr_J$.

Moreover, tameness \eqref{eq:tame3} identifies $\im \phi_{L J}=\s_J^{-1}(\E_{LJ})$ for $L=H,I,H\cap I$, so that the intersection identity holds by the filtration property \eqref{eq:CIJ}.
In case $H\cap I=\emptyset$ it gives $\s_J^{-1}(\E_{HJ})\cap \s_J^{-1}(\E_{IJ}) = \s_J^{-1}(0_J)$ since $\E_{\emptyset J}=\im 0_J$.
\end{proof}

 The next lemma shows that every tame weak topological Kuranishi atlas satisfies the strong cocycle condition, and so is a topological Kuranishi atlas.

\begin{lemma}\label{le:tame0}
Suppose that the filtered weak topological Kuranishi atlas $\Kk$ satisfies the tameness conditions \eqref{eq:tame1}, \eqref{eq:tame2} for all $I,J,K\in\Ii_\Kk$ with $|I|\leq k$. Then for all $I\subset J\subset K$ with $|I|\leq k$ the strong cocycle condition in Definition~\ref{def:cocycle} is satisfied, i.e.\
$\Hat\Phi_{JK}\circ \Hat\Phi_{IJ}=\Hat\Phi_{IK}$ with equality of domains, in particular
$$
U_{IJ}\cap \phi_{IJ}^{-1}(U_{JK}) = U_{IK} .
$$
\end{lemma}
\begin{proof}
From the tameness conditions \eqref{eq:tame2} and \eqref{eq:tame3} we obtain for all $I\subset J\subset K$ with $|I|\leq k$
$$
\phi_{IJ}(U_{IK})
= U_{JK}\cap \s_J^{-1}(\E_{IJ}) 
= U_{JK}\cap \phi_{IJ}(U_{IJ}) .
$$
Applying $\phi_{IJ}^{-1}$ to both sides and using \eqref{eq:tame4} implies equality of the domains.
Then the weak cocycle condition $\phi_{JK}\circ \phi_{IJ}=\phi_{IK}$ on the overlap of domains is identical to the strong cocycle condition.
\end{proof}

Finally, the key feature of the tameness property is that it implies the topological properties claimed in Theorem~\ref{thm:K}.

\begin{prop}\label{prop:Khomeo}
Suppose that the topological Kuranishi atlas $\Kk$ is tame. Then $|\Kk|$ and $|\bE_\Kk|$ are Hausdorff, and for each $I\in\Ii_\Kk$ the quotient maps $\pi_{\Kk}|_{U_I}:U_I\to |\Kk|$ and $\pi_{\bE_\Kk}|_{\E_I}:\E_I\to |\bE_\Kk|$ are homeomorphisms onto their image.
\end{prop}

The proof will take up the following section.
We end this section with further topological properties of the virtual neighbourhood of a tame Kuranishi atlas that will be useful when constructing an admissible metric in Section~\ref{ss:shrink} and eventually 
the virtual fundamental class in e.g.\ \cite{MW1}.  For that purpose we need to be careful in differentiating between the quotient and subspace topology on subsets of the virtual neighbourhood, as follows.

\begin{definition} \label{def:topologies}
For any subset $\Aa\subset \Obj_{\bB_\Kk}$ of the union of domains of a topological Kuranishi atlas $\Kk$, we denote by 
$$
\|\Aa\|:=\pi_\Kk(\Aa)\subset|\Kk| , 
\qquad \mbox{ resp. } \quad
|\Aa|:=\pi_\Kk(\Aa)\cong \quot{\Aa}{\sim}\ ,
$$
the set $\pi_\Kk(\Aa)$ equipped with its subspace topology induced from the inclusion $\pi_\Kk(\Aa)\subset|\Kk|$, resp.\ its quotient topology induced from the inclusion $\Aa\subset \Obj_{\bB_\Kk}$ and the equivalence relation $\sim$ on $\Obj_{\bB_\Kk}$ (which is generated by all morphisms in $\bB_\Kk$, not just those between elements of $\Aa$).
\end{definition}

\begin{remark} \label{rmk:hom} \rm
In many cases we will be able to identify different topologies on subsets of the virtual neighbourhood $|\Kk|$ by appealing to the following elementary {\bf nesting uniqueness of compact Hausdorff topologies}:

Let $f:X\to Y$ be a continuous bijection from a compact topological space $X$ to a Hausdorff space $Y$. Then $f$ is in fact a homeomorphism. 
Indeed, it suffices to see that $f$ is a closed map, i.e.\ maps closed sets to closed sets, since that implies continuity of $f^{-1}$. But any closed subset of $X$ is also compact, and its image in $Y$ under the continuous map $f$ is also compact, hence closed since $Y$ is Hausdorff.

In particular, if $Z$ is a set with nested compact Hausdorff topologies $\Tt_1\subset\Tt_2$, then $\id_Z: (Z,\Tt_2)\to (Z,\Tt_1)$ is a continuous bijection, hence homeomorphism, i.e.\ $\Tt_1=\Tt_2$.
$\hfill\er$
\end{remark}

\begin{prop}\label{prop:Ktopl1}  
Let $\Kk$ be a tame topological Kuranishi atlas.
\begin{enumerate}
\item
For any subset $\Aa\subset \Obj_{\bB_\Kk}$ the identity map $\id_{\pi_\Kk(\Aa)}: |\Aa| \to \|\Aa\|$ is continuous.
\item 
If $\Aa \sqsubset \Obj_{\bB_\Kk}$ is precompact, then both $|\ov\Aa|$ and $\|\ov\Aa\|$ are compact. In fact, the quotient and subspace topologies on $\pi_\Kk(\ov\Aa)$ coincide, that is $|\ov\Aa|=\|\ov\Aa\|$ as topological spaces.
\item
If $\Aa \sqsubset \Aa' \subset \Obj_{\bB_\Kk}$, then $\pi_\Kk(\ov{\Aa}) = \ov{\pi_\Kk(\Aa)}$ and $\pi_\Kk(\Aa) \sqsubset \pi_\Kk(\Aa')$ in the topological space $|\Kk|$.
\item
If  $\Aa \sqsubset \Obj_{\bB_\Kk}$ is precompact, then $\|\ov{\Aa}\|=|\ov\Aa|$ is metrizable; in particular this implies that $\|\Aa\|$ is metrizable.
\end{enumerate}

\end{prop}
\begin{proof}
To prove (i) recall that openness of $\Uu\subset\pi_\Kk(\Aa)$ in the subspace topology implies the existence of an open subset $\Ww \subset |\Kk|$ with $\Ww \cap \pi_\Kk(\Aa)=\Uu$. Then we have $\Aa \cap \pi_\Kk^{-1}(\Uu)=\Aa \cap \pi_\Kk^{-1}(\Ww)$, where $\pi_\Kk^{-1}(\Ww) \subset \bigsqcup_{I\in\Ii_\Kk}U_I$ is open by definition of the quotient topology on $|\Kk|$. However, that exactly implies openness of $\Aa \cap \pi_\Kk^{-1}(\Uu)\subset\Aa$ and thus of $\Uu$ in the quotient topology. This proves continuity.

The compactness assertions in (ii) follow from the compactness of $\ov\Aa$ together with the fact that both $\pi_\Kk: \Aa \to |\Kk|$ and $\pi_\Kk: \Aa \to {\Aa}/{\sim}$ are continuous maps.
Moreover, $\|\Aa\|$ is Hausdorff because its topology is induced by the Hausdorff topology on $|\Kk|$.
Therefore the identity map $|\ov\Aa|\to\|\ov\Aa\|$ is a continuous bijection from a compact space to a Hausdorff space, and hence a homeomorphism by Remark~\ref{rmk:hom}, which proves the equality of topologies.

In (iii), the continuity of $\pi_\Kk$ implies $\pi_\Kk(\ov \Aa)\subset \ov{\pi_\Kk(\Aa)}$ for the closure in $|\Kk|$. On the other hand, compactness of $\ov\Aa$ implies that $\pi_\Kk(\ov \Aa)$ is compact by (ii), in particular it is closed and contains $\pi_\Kk(\Aa)$, hence also contains $ \ov{\pi_\Kk(\Aa)}$. This proves equality $\pi_\Kk(\ov \Aa)=\ov{\pi_\Kk(\Aa)}$.  The last claim of (iii) then holds because $\ov{\pi_\Kk(\Aa)}= \pi_\Kk(\ov \Aa)\subset \pi_\Kk(\Aa')$, and $\pi_\Kk(\ov \Aa)$ is compact by (ii). 

To prove the metrizability in (iv), we will use Urysohn's metrization theorem, which says that any regular and second countable Hausdorff space is metrizable.
Here $\|\ov{\Aa}\|\subset |\Kk|$ is regular (i.e.\ points and closed sets have disjoint neighbourhoods) since it is a compact subset of a Hausdorff space. 
So it remains to establish second countability, i.e.\ to find a countable base for the topology, namely a countable collection of open sets, such that any other open set can be written as a union of part of the collection.

For that purpose first recall that each $U_I$ is second countable by 
Remark~\ref{rmk:topchart}~(iii).
This property is inherited by the subsets $\ov{A}_I\subset U_I$ for $I\in\Ii_\Kk$, and by their images $\pi_\Kk(\ov{A}_I)\subset|\Kk|$ via the homeomorphisms $\pi_\Kk|_{U_I}$ of Proposition~\ref{prop:Khomeo}.
Moreover, each $\pi_\Kk(\ov{A}_I)$ is compact since it is the image under the continuous map $\pi_\Kk$ of the closed subset $\ov{A}_I=\ov{\Aa}\cap U_I$ of the compact set $\ov{\Aa}$.
So, in order to prove second countability of the finite union $\|\ov{\Aa}\|=\bigcup_{I\in\Ii_\Kk} \pi_\Kk(\ov{A}_I)$ iteratively, it remains to establish second countability for a union of two compact second countable subsets, as follows.

\MS
\NI
{\bf Claim:} Let $B,C \subset Y$ be compact subsets of a Hausdorff space $Y$ such that $B,C $ are second countable in the subspace topologies. Then $B \cup  C$ is second countable in subspace topology.

\MS
To prove this claim, let $(V_i^B)_{i\in\N}$, resp.\ $(V_i^C)_{i\in\N}$, be countable  
bases of open neighbourhoods 
for $B$, resp.\ $C$. Then $(V_i^B \less C)_{i\in\N}$, resp.\ $(V_i^C \less B)_{i\in\N}$, are countable neighbourhood bases for 
$B\less C$,  resp.\ $C\less B$.
Moreover, $V_i^B \less C \subset B\cup C$,  and similarly $V_i^C \less B\subset B\cup C$,
is open since $(B\cup C)\less V_i^B = C \cup (B\less V_i^B)$ is a union of compact and hence closed sets.
To finish the construction of a countable neighborhood basis for $B\cup C$ we add to $(V_i^B \less C)_{i\in\N}$ and $(V_i^C \less B)_{i\in\N}$ the sets for $i,j\in\N$
$$
V_{ij}: = \bigl(V_i^B\less C\bigr) \cup (V_i^B\cap V_j^C)\cup \bigl(V_j^C\less B\bigr).
$$
To check that these are open in $BC:=B\cup C$ we rewrite their complement
\begin{align*}
(B\cup C)\less V_{ij} 
&=  \bigl(BC \less (V_j^C\less B) \bigr) \cap \bigl(BC \less (V_i^B\less C) \bigr) \cap \bigl(BC \less (V_i^B\cap V_j^C)\bigr)  \\
&= \bigl(B \cup (C\less V_j^C) \bigr)
\cap  \bigl( (B\less V_i^B)\cup C \bigr) \cap   \bigl(BC \less (V_i^B\cap V_j^C)\bigr) \\
&= \bigl( (B\less V_i^B) \cup (C\less V_j^C) \cup (B\cap C) \bigr)
 \less (V_i^B\cap V_j^C) \\
&= (B\less V_i^B) \cup (C\less V_j^C) 
\cup \bigl((B\cap C) \less (V_i^B\cap V_j^C)\bigr).
\end{align*}
Here each of $B$, $C$, and $B\cap C$ is compact, with relatively open subset $V_i^B$ resp.\ $V_j^C$ resp.\ $V_i^B\cap V_j^C= (V_i^B\cap C)\cap (V_j^C\cap B)$, so that $(B\cup C)\less V_{ij}$ is the union of compact sets $B\less V_i^B$, $C\less V_j^C$, and $(B\cap C) \less (V_i^B\cap V_j^C)$. This shows that $V_{ij}$ is open since its complement is closed.

Finally, the sets $V_{ij}$ together with the $V_i^B \less C$ and $V_i^C \less B$ form the required neighbourhood basis since any $x\in B\cup C$ either lies in $B\less C$ (hence in some $V_i^B \less C$),  in $C\less B$ (hence in some $V_i^C \less B$), or in $B\cap C$. In the latter case we find $i,j$ so that $x\in V_i^B\cap V_j^C \subset V_{ij}$. This proves the claim and thus proves that $\|\ov\Aa\|$ is metrizable.

In particular, $\|\Aa\|$ is metrizable in the subspace topology, by restriction of a metric on $\|\ov\Aa\|$, which finishes the proof of (iv).
\end{proof}

\begin{rmk}\label{rmk:Steen}\rm  
As a final topological remark, we note that when $\Kk$ is tame so that $|\Kk|$ is Hausdorff, then it follows from 
\cite[2.6]{STEEN} that the quotient topology on $|\Kk|$ is compactly generated. 
Thus, although its topology is not in general metrizable, $|\Kk|$ does belong to a well understood and well studied category of topological spaces.
\hfill$\er$
\end{rmk}

\subsection{Proof of Hausdorff property} \label{ss:Haus}  \hspace{1mm}\\ \vspace{-3mm}

The filtration conditions (i) -- (iv) in Definition~\ref{def:Ku3} interact with the tameness identities \eqref{eq:tame1} and \eqref{eq:tame2} in a rather subtle way.  
As we will see in the proof of Proposition~\ref{prop:proper} (cf.\  the discussion following \eqref{eq:UJK(k)}), condition (iv) for filtrations is just strong enough to allow the inductive construction of a tame shrinking of an atlas that also satisfies filtration conditions (i)--(iii).
On the other hand, if conditions (i)--(iii) hold as well as the tameness identities, then the equivalence relation on $\Obj_{\bB_\Kk}$, given in Definition~\ref{def:Knbhd} by abstractly inverting the morphisms, simplifies in a way that we will establish in Lemma~\ref{le:Ku2}. 
This is the key tool in the proof in Proposition~\ref{prop:Khomeo} that tame atlases have Hausdorff realizations.
We begin the proof by reformulating the equivalence relation $\sim$ with the help of a partial order given by the morphisms -- more precisely the embeddings $\phi_{IJ}$ that are part of the coordinate changes.

\begin{definition} \label{def:preceq}
Let $\preceq$ denote the partial order on $\Obj_{\bB_\Kk}$ given by
$$
(I,x)\preceq(J,y) \quad :\Longleftrightarrow \quad \Mor_{\bB_\Kk}((I,x),(J,y))\neq\emptyset .
$$
That is, we have $(I,x)\preceq (I,y)$ iff $x\in U_{IJ}$ and $y=\phi_{IJ}(x)$.
Moreover, for any $I,J\in\Ii_\Kk$ and subset $S_I\subset U_I$ we denote the subset of points in $U_J$ that are equivalent to a point in $S_I$ by
$$
\eps_J(S_I) \,:=\; U_J \cap \pi_\Kk^{-1}(\pi_\Kk(S_I))  \;=\;
\bigl\{y\in U_J  \,\big|\, \exists\, x\in S_I : (I,x)\sim (J,y) \bigr\} \;\subset\; U_J.
$$
The partial order $\preceq$ on $\Obj_{\bE_\Kk}$ is defined analogously.
\end{definition}

The relation $\preceq$ on $\Obj_{\bE_\Kk}$ is very similar to that on $\Obj_{\bB_\Kk}$.  Indeed,
$(I,e)\preceq (J,f)$ implies  $(I,\pr_I(e))\preceq (J,\pr_J(f))$. Conversely, if $(I,x)\preceq (J,y)$ then for every $e\in \pr_I^{-1}(x)\subset\E_I$ there is a unique $f=\Hat\Phi_{IJ}(e)\in \E_J$ such that $(I,e)\preceq (J,f)$. 
Thus to ease notation we mostly work with the relation on $\Obj_{\bB_\Kk}$ though any statement about this has an immediate analog for the relation on $\Obj_{\bE_\Kk}$ (and vice versa).
Now we can reformulate the equivalence relation $\sim$ in terms of the partial order $\preceq$.

\begin{lemma} \label{lem:eqdef}
The equivalence relation $\sim$ on $\Obj_{\bB_\Kk}$ of Definition~\ref{def:Knbhd} is equivalently defined by $(I,x)\sim (J,y)$ iff there is a finite tuple of objects $(I_0, x_0), \ldots, (I_k, x_k)\in\Obj_{\bB_\Kk}$ such that
\begin{align}\label{eq:ch}
&(I,x) = (I_0,x_0)\preceq  (I_1,x_1) \succeq (I_2,x_2) \preceq  \ldots (I_k, x_k)=(J,y)  \\
\text{or}\qquad &
(I,x) = (I_0,x_0)\succeq  (I_1,x_1) \preceq (I_2,x_2) \succeq  \ldots (I_k, x_k)=(J,y) . \nonumber
\end{align}
\end{lemma}
\begin{proof}
The relation $\preceq$ is transitive by the cocycle condition \ref{def:Ku}
and antisymmetric since the transition maps are directed. In particular, we have $(I,x)\preceq (I,y)$ iff $x=y$.
The two definitions of $\sim$ are equivalent since, if \eqref{eq:ch} had consecutive morphisms $(I_{\ell-1},x_{\ell-1})\preceq (I_{\ell},x_{\ell}) \preceq (I_{\ell+1},x_{\ell+1})$, these could be composed to a single morphism $(I_{\ell-1},x_{\ell-1})\preceq (I_{\ell+1},x_{\ell+1})$ by the cocycle condition.  Similarly, any consecutive morphisms $(I_{\ell-1},x_{\ell-1})\succeq (I_{\ell},x_{\ell}) \succeq (I_{\ell+1},x_{\ell+1})$ can be composed to a single morphism $(I_{\ell-1},x_{\ell-1})\succeq (I_{\ell+1},x_{\ell+1})$.
\end{proof}

The following lemma for tame Kuranishi atlases further simplifies the definition of $\sim$ in terms of $\preceq$, and thus provide a more useful characterization of the sets $\eps_J(S_I)$, as well as good topological properties of the relation $\sim$ and the corresponding projection~$\pi_\Kk$.

\begin{lemma} \label{le:Ku2} 
Let $\Kk$ be a tame topological Kuranishi atlas.
\begin{enumerate}
\item [(a)] 
For $(I,x),(J,y)\in\Obj_{\bB_\Kk}$ the following are equivalent.
\begin{enumerate}
\item[(i)] $(I,x)\sim (J,y)$;
\item[(ii)] $(I,x)\preceq (I\cup J,z) \succeq (J,y)$ for some $z\in U_{I\cup J}$ (in particular $I\cup J\in\Ii_\Kk$);
\item[(iii)] $(I,x)\succeq (I\cap J,w) \preceq (J,y)$ for some $w\in U_{I\cap J}$ (in particular $\emptyset\ne I\cap J\in \Ii_\Kk$), or $I\cap J=\emptyset$, $\s_I(x)=0_I(x)$, $\s_J(y)=0_J(y)$, and $\psi_I(x)=\psi_J(y)$.
\end{enumerate}

\item[(b)] 
For $(I,e),(J,f)\in\Obj_{\bE_\Kk}$ the following are equivalent.
\begin{enumerate}
\item[(i)] $(I,e)\sim (J,f)$;
\item[(ii)] $(I,e)\preceq (I\cup J,g) \succeq (J,f)$ for some $g\in \E_{I\cup J}$ (in particular $I\cup J\in\Ii_\Kk$);
\item[(iii)] $(I,e)\succeq (I\cap J,d) \preceq (J,f)$ for some $d\in \E_{I\cap J}$ (in particular $\emptyset\ne I\cap J\in \Ii_\Kk$), or $I\cap J=\emptyset$ and $e=\s_I(x)=0_I(x)$, $f=\s_J(y)=0_J(y)$ for some $x\in U_I$, $y\in U_J$ with $\psi_I(x)=\psi_J(y)$.
\end{enumerate}

\item[(c)]
$\pi_\Kk:U_I \to |\Kk|$ and $\pi_\Kk: 
\E_I \to |\bE_\Kk|$ are injective for each $I\in\Ii_\Kk$, that is 
$(I,p)\sim (I,q)$ implies $p=q$ in both cases $p,q\in U_I$ or $p,q\in \E_I$.
In particular, the elements $z$ and $w$ in {\rm (a)} resp.\ $g$ and $d$ in {\rm (b)} are automatically unique.
\item[(d)]
For any $I,J\in\Ii_\Kk$ and any subset $S_I\subset U_I$ we have
\begin{align*}
\eps_J(S_I)  \;:=\; U_J\cap \pi_\Kk^{-1}\bigl(\pi_\Kk(S_I)\bigr)
&\;=\; \phi_{J(I\cup J)}^{-1}\bigl( \phi_{I (I\cup J)}(U_{I(I\cup J)}\cap S_I) \bigr) \\
&\;=\; \phi_{(I\cap J) J}\bigl( U_{(I\cap J)I} \cap \phi_{(I\cap J)I}^{-1}(S_I) \bigr) ,
\end{align*}
where in case $I\cap J=\emptyset$ the second identity is 
$\eps_J(S_I) = \psi_J^{-1}\bigl(\psi_I(\s_I^{-1}(0_I)\cap S_I) \bigr)$, 
consistently with \eqref{eq:empty}.
In particular we have
$$
\eps_J(U_I) \;:=\; U_J\cap \pi_\Kk^{-1}\bigl(\pi_\Kk(U_I)\bigr) \;=\; 
 U_{J(I\cup J)}\cap 
  \s_J^{-1}(\E_{(I\cap J)J}).
$$
Analogously for any $I,J\in\Ii_\Kk$ and any subset $S_I\subset \E_I$ we have
\begin{align*}
\Hat\eps_J(S_I)  \;:=\; \pi_{\bE_\Kk}^{-1}\bigl(\pi_{\bE_\Kk}(S_I)\bigr)
&\;=\; \Hat\Phi_{J(I\cup J)}^{-1}\bigl( \Hat\Phi_{I (I\cup J)}\bigl(\pr_I^{-1}(U_{I(I\cup J)})\cap S_I \bigr) \bigr) \\
&\;=\; \Hat\Phi_{(I\cap J) J}\bigl( \pr_I^{-1}(U_{I(I\cup J)}) \cap \Hat\Phi_{(I\cap J)I}^{-1}(S_I) \bigr) ,
\end{align*}
which in case $I\cap J=\emptyset$ is
$\; \Hat\eps_J(S_I)= 0_J\bigl(\psi_J^{-1}\bigl(\psi_I(\pr_I (\im 0_I\cap S_I) ) \bigr)\bigr)$.
\end{enumerate}
\end{lemma}

\begin{proof}
We first prove the following claims which generalize the implications 
(iii) $\Rightarrow$ (ii) and (ii) $\Rightarrow$ (iii) in (a) and (b).

\medskip
\noindent
{\bf Claim 1:} {\it Suppose that $(I,e)\preceq (K,g) \succeq (J,f)$ for some $(K,g)\in\Obj_{\bE_\Kk}$.
If in addition $I\cap J\neq\emptyset$, then there exists $d\in \E_{I\cap J}$ such that  $(I,e)\succeq (I\cap J,d) \preceq (J,f)$.

Moreover, 
denote $x:=\pr_I(e)$ and $y:=\pr_J(f)$. 
Then $\s_I(x)\neq 0_I(x)$ or $e\neq 0_I(x)$ both imply $I\cap J\neq\emptyset$, whereas in case $\s_I(x)=e=0_I(x)$ we may have $I\cap J=\emptyset$ but always $f=\s_J(y)=0_J(y)$ and $y=\psi_J^{-1}\bigl(\psi_I(x)\bigr)$.
}

\MS
\NI
Before we begin, note that $\Hat\Phi_{IK}(e)=g=\Hat\Phi_{JK}(f)$ implies $\phi_{IK}(x)=z=\phi_{JK}(y)$ for the projection to the base $z:=\pr_K(g)$. 

Now suppose first that $e\ne 0_I(x)$. Then $(I,e)\preceq (K,g)$ implies  $g=\Hat\Phi_{IK}(e)\ne \Hat\Phi_{IK}(0_I(x)) = 0_K(z)$, and $(K,g) \succeq (J,f)$ implies $g\in \Hat\Phi_{JK}(\E_J)$, so that $\Hat\Phi_{IK}(\E_I)\cap \Hat\Phi_{JK}(\E_J)\not\subset \im 0_K$.
The same follows in case $\s_I(x)\ne 0_I(x)$ from $0_K(z) =\Hat\Phi_{IK}(0_I(x)) \ne \Hat\Phi_{IK}(\s_I(x)) = \s_K(z)= \Hat\Phi_{JK}(\s_J(y))$.
On the other hand, the filtration property \eqref{eq:filt} gives
$\Hat\Phi_{IK}(\E_I)\cap \Hat\Phi_{JK}(\E_J) \subset \E_{IK}\cap\E_{JK}$, which in case
$I\cap J = \emptyset$ by the filtration properties equals to $\E_{(I\cap J) K}=\E_{\emptyset K}=\im 0_K$. This shows $I\cap J \neq \emptyset$ as claimed.

Next, assuming $I\cap J \neq \emptyset$ for any reason, the filtration properties, \eqref{eq:filt}, and tameness \eqref{eq:tame5} give 
\begin{align*}
 g \;\in \; \Hat\Phi_{IK}(\E_I) \cap \Hat\Phi_{JK}(\E_J) 
&\;=\;  
\E_{IK}\cap\pr_K^{-1}(\im\phi_{IK})\cap\E_{JK}\cap\pr_K^{-1}(\im\phi_{JK}) \\
& \;=\;
\E_{(I\cap J)K} \cap \pr_K^{-1}(\im\phi_{(I\cap J) K})
\;=\;
\Hat\Phi_{(I\cap J) K}(\E_{(I\cap J) K}) .  
\end{align*}
Therefore we have $g=\Hat\Phi_{(I\cap J)K}(d)$ for some $d\in \E_{(I\cap J)K}$.
We also have $g=\Hat\Phi_{IK}(e)$ by assumption, and Lemma \ref{le:tame0} implies that $\Hat\Phi_{(I\cap J)K}(d) = \Hat\Phi_{IK}\bigl( \Hat\Phi_{(I\cap J)I}(d)\bigr)$, so the elements $e$ and $\phi_{(I\cap J)I}(d)$ of  $\E_I$ have the same image under $\Hat\Phi_{IK}$. Since the latter is injective we deduce $e=\Hat\Phi_{(I\cap J) I}(d)$.
Similarly, $f=\Hat\Phi_{(I\cap J) J}(d)$ follows from $\Hat\Phi_{(I\cap J)K} = \Hat\Phi_{JK}\circ \Hat\Phi_{(I\cap J)J}$.

Finally, in case $\s_I(x)=e=0_I(x)$ the assertions $\s_J(y)=0_J(y)$ and $\psi_I(x)=\psi_J(y)$ hold in general weak Kuranishi atlases, as noted in \eqref{eq:useful2}. Moreover, in that case we have $\Hat\Phi_{JK}(0_J(y))=0_K(z) = \Hat\Phi_{IK}(e) = g= \Hat\Phi_{JK}(f)$, so that injectivity of $\Hat\Phi_{JK}$ implies $f=0_J(y)$.

\medskip

\noindent
{\bf Claim 2:} {\it Suppose $(I,e)\succeq (H,d) \preceq (J,f)$ for some $(H,d)\in\Obj_{\bE_\Kk}$. Then there exists $g\in \E_{I\cup J}$ such that $(I,e)\preceq (I\cup J,g) \succeq (J,f)$.
The same conclusion holds under the assumption $e=0_I(x)$ and $f=0_J(y)$ for $x\in\s_I^{-1}(0_I)$, $y\in\s_J^{-1}(0_J)$ with $\psi_I(x)=\psi_J(y)$.
}

\MS\NI
The fact that both $\Hat\Phi_{HI}(d)=e$ and $\Hat\Phi_{HJ}(d)=f$ are defined implies $d\in \pr_H^{-1}(U_{HI}\cap U_{HJ})$, 
which by \eqref{eq:tame1} means $d\in \pr_H^{-1}( U_{H (I\cup J)})$ so that $g:=\Hat\Phi_{H (I\cup J)}(d)\in \E_{I\cup J}$ is defined.
The strong cocycle condition in Definition~\ref{def:cocycle}, proved in Lemma~\ref{le:tame0}, then implies
$$
g = \Hat\Phi_{H (I\cup J)}(d) = \Hat\Phi_{I(I\cup J)}\bigl(\Hat\Phi_{H I}(d)\bigr) =  \Hat\Phi_{I(I\cup J)}(e) .
$$
Similarly,  $g= \Hat\Phi_{J(I\cup J)}(f)$ follows from $f=\Hat\Phi_{HJ}(d)$ and the strong cocycle condition.
Finally, in case $e=0_I(x)=\s_I(x)$ and $f=0_J(y)=\s_J(y)$ with $\psi_I(x)=\psi_J(y)$ we conclude $\emptyset\neq F_I\cap F_J = F_{I\cup J}$ so that $K:=I\cup J\in \Ii_\Kk$. In particular there exists $z\in U_K$ with $\psi_K(z)=\psi_I(x)$, so that the property of coordinate changes $\phi_{IK}|_{\psi_I^{-1}(F_K)}=\psi_I^{-1}\circ\psi_K$ (from Definition~\ref{def:tchange}~(iii) for $I\subset K$) implies $x\in U_{IK}$ with $\phi_{IK}(x)=z$. Then the linearity of coordinate changes (Definition~\ref{def:tchange}~(i)) implies $\Hat\Phi_{IK}(0_I(x))=0_K(z)$ and hence $(I,0_I(x))\preceq (K,0_K(z))$, and we analogously obtain $(J,0_J(y))\preceq (K,0_K(z))$, which proves the claim with $g=0_K(z)$.

\medskip

The now established Claims 1 and 2 proves (ii) $\Leftrightarrow$ (iii) in the setting of (b) as well as (a), in the latter case by applying the claims to $e=0_I(x), f=0_J(y), g=0_{I\cup J}(z), d=0_{I\cap J}(w)$.
Moreover, (ii) $\Rightarrow$ (i) holds by definition of the equivalence relation $\sim$ for both (a) and (b).
Finally, the implication (i) $\Rightarrow$ (ii) for (b) is proven by considering a chain of morphisms as in \eqref{eq:ch}, applying Claim 2 to replace every other occurrence of 
$\ldots  \succeq  (I_\ell,e_\ell) \preceq \ldots$ with $\ldots  \preceq  (I'_\ell,e'_\ell) \succeq \ldots$, 
and then composing consecutive morphisms in the same direction (using the cocycle condition as in Lemma~\ref{lem:eqdef}). 
Unless we started with $(I,e) \preceq(J,f)$ or $(I,e) \succeq (J,f)$ which are both special cases of (ii), this shortens the chain and can be iterated until we reach a chain of type $(I,e) \preceq (K,g) \succeq (J,f)$. Then Claim 1 either provides a chain $(I,e) \succeq (I\cap J,d) \preceq (J,f)$ or shows that $x=\s_I(x)=0_J(x)$ and $f=\s_J(y)=0_J(y)$ for some $\psi_I(x)=\psi_J(y)$, upon which Claim 2 provides the required chain of type  (ii).
Repeating the same argument for the zero sections applied to $x,y,z,w$ to prove (i) $\Rightarrow$ (ii) in (a) completes the proofs of (a) and (b).
 
Next, part (c) is a consequence of (i)$\Rightarrow$(ii) since $\phi_{II}={\rm Id}_{U_I}$ and 
$\Hat\Phi_{II}={\rm Id}_{\E_I}$.  
The uniqueness of $z$ follows from the representation $z\in \pi_\Kk^{-1}(\pi_\Kk(x))\cap U_{I\cup J}$  in terms of the projection $\pi_\Kk$, and similar for the other elements.

The formulas for $\eps_J(S_I)$ 
and $\E_J\cap\pi_{\bE_\Kk}^{-1}\bigl(\pi_{\bE_\Kk}(S_I)\bigr)$
in (d) follow from the equivalent definitions of $\sim$ in~(a)
resp.\
(b). 
In case $S_I=U_I$ we start from 
$$
\eps_J(U_I) = \phi_{J(I\cup J)}^{-1}\bigl( \phi_{I (I\cup J)}(U_{I (I\cup J)}\cap U_I) \cap \im\phi_{J(I\cup J)} \bigr)
$$
and then use \eqref{eq:tame5}, the strong cocycle condition $\phi_{(I\cap J) (I\cup J)}=\phi_{J (I\cup J)}\circ\phi_{(I\cap J) J}$, and \eqref{eq:tame3} to obtain the claimed
\begin{align*}
\eps_J(U_I) &
\;=\;  
\phi_{J(I\cup J)}^{-1}\bigl( \im \phi_{I (I\cup J)} \cap \im\phi_{J(I\cup J)} \bigr) 
\;=\;  
\phi_{J(I\cup J)}^{-1}\bigl( \im \phi_{(I\cap J) (I\cup J)} \bigr) \\
&\;=\;  
U_{J(I\cup J)}\cap  \im \phi_{(I\cap J) J}
\;=\;  U_{J(I\cup J)} \cap \s_J^{-1}(\E_{(I\cap J)J}).
\end{align*}
This completes the proof.
\end{proof}

With these preparations we can finally show that the 
tameness conditions imply the Hausdorff and homeomorphism properties claimed in Theorem~\ref{thm:K}.

\begin{proof}[Proof of Proposition~\ref{prop:Khomeo}]
We first prove the statement for $|\Kk|$, and then give the necessary extensions of the argument for $|\bE_\Kk|$.
To see that $|\Kk|$ is Hausdorff, note first that the equivalence relation on $O: =  \Obj_{\bB_\Kk}=  \bigsqcup_{I\in \Ii_\Kk} U_I$ 
 is closed, i.e.\ that the subset 
$$
R: = \bigl\{\bigl((I,x), (J,y)\bigr) \ | \ (I,x)\sim (J,y)\bigr\}\subset O\times O
$$
is closed. Since $O\times O$ is the disjoint union of the metrizable sets $U_I\times U_J$
for the finite index set $I,J\in \Ii_\Kk$, closedness of $R$ will follow if we show that for all pairs $I,J$ and all convergent sequences $x^\nu\to x^\infty$ in $U_I$, $y^\nu\to y^\infty$ in $U_J$ with $(I,x^\nu)\sim (J,y^\nu)$ for all $\nu$, we have $(I,x^\infty)\sim (J,y^\infty)$. 
For that purpose denote $H:=I\cap J$.
In case $H=\emptyset$, Lemma~\ref{le:Ku2}~(a) implies 
$\s_I(x^\nu)=0_I(x^\nu)$, $\s_J(y^\nu)=0_J(y^\nu)$, 
and $\psi_I(x^\nu)=\psi_J(y^\nu)$, all of which persist in the limit by continuity of $\s_\bullet$ and $\psi_\bullet$. Thus we have $\psi_I(x^\infty)=\psi_J(y^\infty)$, which by \eqref{eq:useful1} implies $(I,x^\infty)\sim (J,y^\infty)$. 
In case $H\neq\emptyset$, Lemma~\ref{le:Ku2}~(a)  provides a sequence $w^\nu\in U_H$ such that $x^\nu = \phi_{HI}(w^\nu)$ and $y^\nu = \phi_{HJ}(w^\nu)$.
Since $\phi_{HI}(U_{HI})=\s_I^{-1}(\E_H)\subset U_I$ is closed by Lemma~\ref{le:phitrans}, it contains the limit $x^\infty$, and since $\phi_{HI}$ is a homeomorphism to its image we deduce convergence $w^\nu\to w^\infty\in U_{HI}$ to a preimage of $x^\infty=\phi_{HI}(w^\infty)$.  Then by continuity of the transition map we obtain $\phi_{HJ}(w^\infty) = y^\infty$, so that $(I, x^\infty)\sim (J, y^\infty)$ as claimed.  

With closedness of $R\subset O\times O$ established, we can appeal to the well known fact
that whenever a space $O$ is a separable, locally compact, metric space, then its quotient $\qu{O}{R}$ by a closed relation $R\subset O\times O$ is Hausdorff.  
For completeness, we give a proof in Lemma~\ref{le:bourb} below.
In our case, $O=\bigsqcup_{I\in\Ii_\Kk} U_I$ satisfies these assumptions since it is the finite disjoint union of spaces $U_I$ of this kind
by Definition~\ref{def:tchart}.
So this proves that $|\Kk|=\qu{O}{R}$ is Hausdorff.

To make the same argument for $|\bE_\Kk|$, we replace the $U_I$ by  the spaces $\E_I$,
each of which is a separable, locally compact, metric space by Definition~\ref{def:tchart} as well.
We then establish closedness of the relation by the same arguments as above, 
applied to sequences $e^\nu\to e^\infty$ in $\E_I$, $f^\nu\to f^\infty$ in $\E_J$ with $(I,e^\nu)\sim (J,f^\nu)$ whose base points $x^\nu=\pr_I(e^\nu)$, $y^\nu=\pr_J(f^\nu)$ converge to $(I,x^\infty)\sim (J,y^\infty)$ by the above.
In case $I\cap J=\emptyset$, Lemma~\ref{le:Ku2}~(b) implies $e^\nu=0_I(x^\nu)$, $f^\nu=0_J(y^\nu)$ which by continuity of $0_\bullet$ implies $(I,e^\infty=0_I(x^\infty))\sim (J,f^\infty=0_J(y^\infty))$.
In case $H:=I\cap J \neq\emptyset$, Lemma~\ref{le:Ku2}~(b) provides $g^\nu\in \E_H$ such that $e^\nu = \Hat\Phi_{HI}(g^\nu)$. 
In this case \eqref{eq:filt} together with 
Lemma~\ref{le:phitrans} 
implies that the limit $e^\infty$ is contained in the intersection of closed subsets $\Hat\Phi_{HI}(\E_H)=\E_{HI}\cap \pr_I^{-1}(\s_I^{-1}(\E_H))$ of $\E_I$.
Then the homeomorphism properties of $\Hat\Phi_{HI}$ and $\Hat\Phi_{HJ}$ imply $g^\nu\to g^\infty\in \E_{HI}$ with $e^\infty=\Hat\Phi_{HI}(g^\infty)$ and $\Hat\Phi_{HJ}(g^\infty) = f^\infty$, so that $(I, e^\infty)\sim (J, f^\infty)$ confirms closedness of the relation.
The Hausdorff property of the quotient then follows as before.

To show that $\pi_\Kk|_{U_I}$ is a homeomorphism onto its image, first recall that it is injective by Lemma \ref{le:Ku2}~(c).
It is moreover continuous since $|\Kk|$ is equipped with the quotient topology. Hence it remains to show that $\pi_\Kk|_{U_I}$ is an open map to its image, i.e.\ for a given open subset $S_I\subset U_I$ we must find an open subset $\Ww\subset |\Kk|$ such that $\Ww\cap \pi_{\Kk}(U_I) = \pi_\Kk(S_I)$. 
Equivalently, we set $\Ww:=|\Kk|\less\Qq$ and construct the complement 
$$
\Qq := \bigl(\io_\Kk(X) \cup {\textstyle \bigcup_{H\subset I}} \pi_\Kk(U_H)\bigr) \less \pi_\Kk(S_I)\;\subset\; |\Kk| .
$$
With that the intersection identity follows from $\Qq\cap \pi_{\Kk}(U_I) = \pi_\Kk(U_I)\less \pi_\Kk(S_I)$, so it remains to show that $U_J\cap \pi_\Kk^{-1}(\Qq)$ is closed for each $J$.
In case $J\subset I$ we have $U_J\cap \pi_\Kk^{-1}(\Qq) = U_J\less \eps_J(S_I)$,
which is closed iff $\eps_J(S_I)$ is open.
Indeed, Lemma~\ref{le:Ku2}~(d) gives $\eps_J(S_I)=\phi_{JI}^{-1}(S_I)\subset U_J$, and this is open since $S_I\subset U_I$ and hence $S_I\cap U_{IJ}\subset U_{IJ}$ is open and $\phi_{JI}$ is continuous.

In case $J\not\subset I$ 
to show that $U_J\cap \pi_\Kk^{-1}(\Qq)$ is closed, we 
express it 
as the union of 
$Q^0_J:=U_J\cap \pi_\Kk^{-1}(\io_\Kk(X)\less \pi_\Kk(S_I))$ with the union over $H\subset I$ of 
\begin{align*}
Q_{JH} \,:=\;
U_J\cap \pi_\Kk^{-1} \bigl( \pi_\Kk(U_H)\less \pi_\Kk(S_I)\bigr) 
& = U_J\cap \pi_\Kk^{-1} \bigl( \pi_\Kk(U_H\less \eps_H(S_I))\bigr) \\
& = \eps_J\bigl(U_H\less \phi_{HI}^{-1}(S_I)\bigr) \;=\; \eps_J(C_H) .
\end{align*}
Note here that $C_H:=U_H\less \phi_{HI}^{-1}(S_I) \subset U_H$ is closed since as above 
$\phi_{HI}^{-1}(S_I) \subset U_H$ is open.
We moreover claim that this union can be simplified to
\begin{equation}\label{qqquark}
U_J\cap \pi_\Kk^{-1}(\Qq) \; = \; Q^0_J \; \cup \bigcup_{H\subset I\cap J} \eps_J(C_H) .
\end{equation}
Indeed, in case $H\cap J= \emptyset$ we have
$\eps_J(C_H) \subset  Q^0_J$ since Lemma~\ref{le:Ku2}~(d) gives 
\begin{align*}
\eps_J\bigl(U_H\less \phi_{HI}^{-1}(S_I)\bigr) 
&= \psi_J^{-1}\bigl(\psi_{H} \bigl(\s_H^{-1}(0_H)\less \phi_{HI}^{-1}(S_I) \bigr)\bigr) \\
&= \psi_J^{-1}\bigl(F_H \less \bigl(F_I \cap \pi_\Kk(S_I) \bigr)\bigr)
\subset U_J\cap \pi_\Kk^{-1}(\io_\Kk(X)\less \pi_\Kk(S_I)) = Q^0_J.
\end{align*}
For $H\not\subset J$ with $H\cap J\neq \emptyset$ 
we have $\eps_J(C_H)\subset \eps_J(C_{H\cap J})$ since
Lemma~\ref{le:Ku2}~(d) gives 
$$
\eps_J\bigl(U_H\less \phi_{HI}^{-1}(S_I)\bigr) 
= \phi_{(H\cap J)J}\bigl(
U_{(H\cap J)J} \cap
\phi_{(H\cap J)H}^{-1} (U_H\less \phi_{HI}^{-1}(S_I))\bigr) 
\subset
U_J \cap \pi_\Kk^{-1}( C_{H\cap J} ),
$$
where $\phi_{(H\cap J)H}^{-1}\bigl( U_H\less \phi_{HI}^{-1}(S_I) \bigr) \subset  U_{H\cap J}\less \phi_{(H\cap J)I}^{-1}(S_I) = C_{H\cap J}$ by the cocycle condition.
This confirms \eqref{qqquark}.

It remains to show that $Q^0_J$ and $\eps_J(C_H)$ for $H\subset I\cap J$ are closed.
For the latter, Lemma~\ref{le:Ku2}~(d) gives 
$\eps_J(U_H\less \phi_{HI}^{-1}(S_I))=\phi_{HJ}(U_{HJ}\cap C_H)$,
which is closed in $U_J$ since closedness of $C_H\subset U_H$ implies relative closedness of $U_{HJ}\cap C_H \subset U_{HJ}$, and $\phi_{HJ}: U_{HJ} \to U_J$ is a homeomorphism onto a closed subset of $U_J$ by 
Lemma~\ref{le:phitrans}.
Finally,
$$
Q^0_J = 
U_J\cap \pi_\Kk^{-1}(\io_\Kk(X)\less \pi_\Kk(S_I))  \;= \; 
s_J^{-1}(0_J) \less \psi_J^{-1}\bigl( F_J \cap \psi_I(\s_I^{-1}(0_I)\cap S_I)\bigr)
$$
is closed in $s_J^{-1}(0_J)$ and hence in $U_J$, since both $F_J$ and $\psi_I(\s_I^{-1}(0_I)\cap S_I)$ are open in $X$, and $\psi_J:s_J^{-1}(0_J)\to X$ is continuous.
Thus $U_J\cap \pi_\Kk^{-1}(\Qq)$ is closed, as claimed, which finishes the proof of the homeomorphism property of $\pi_\Kk|_{U_I}$.

The analogous homeomorphism statement for $|\bE_\Kk|$ is proven by the same constructions, using part (b) of Lemma~\ref{le:Ku2} instead of part (a) and replacing $|\Kk|$, $U_L$, $U_{KL}$, $\phi_{KL}$, $\eps_{KL}$ by $|\bE_\Kk|$, $\E_L$, $\pr_K^{-1}(U_{KL})$, $\Hat\Phi_{KL}$, $\Hat\eps_{KL}$, respectively,  and instead of $\iota_\Kk(X)$ using the subset $|0_\Kk|\bigl(\iota_{\Kk}(X)\bigr)= \bigl| \bigsqcup_{I\in\Ii_\Kk} 0_I(\psi_I^{-1}(F_I)) \bigr| \subset |\bE_\Kk|$.
Then 
$$
\Hat Q^0_J := \E_J\cap \pi_{\bE_\Kk}^{-1}\bigl(|0_\Kk|\bigl(\iota_{\Kk}(X)\bigr)\less \pi_{\bE_\Kk}(S_I) \bigr)  \;= \; 
0_J \bigl( Q^0_J \bigr)
$$
is closed in $\E_J$ since $0_J$ is a homeomorphism to its closed image by Remark~\ref{rmk:topchart}~(iv), and $Q^0_J=\s_J^{-1}(0_J) \less \psi_J^{-1}\bigl( F_J \cap \psi_I(\s_I^{-1}(0_I)\cap S_I)\bigr)\subset U_J$ was shown to be closed above.
\end{proof}

\begin{lemma}[\cite{Bourb} Exercise 19, \S10, Chapter 1]\label{le:bourb}
Let $O$ be a separable, locally compact, metric space, and suppose that an equivalence relation $\sim$ on $O$ is given by a closed subset $R\subset O\times O$. 
Then the quotient topology on $\qu{O}{\sim}$ is Hausdorff.  
\end{lemma}
 \begin{proof}
In the following let $\pi:O\to \qu{O}{\sim}$ denote the projection to the quotient, which is continuous by definition of the quotient topology.
We begin by considering the special case when $O$ and hence $\qu{O}{\sim}$ is compact.

\MS
\NI  {\bf Step 1:}
{\it If in addition $O$ is compact, then $\qu{O}{\sim}$ is compact, normal, and Hausdorff.
In particular, any two disjoint closed sets $A_1,A_2\subset \qu{O}{\sim}$ have open neighhourhoods $A_1\subset U_1$, $A_2\subset U_2$ in $\qu{O}{\sim}$ with disjoint closures $\ov{U_1}\cap\ov{U_2}=\emptyset$.
}

\MS
\NI
The key in this case is the fact that the projection $\pi:O\to \qu{O}{\sim}$ is a closed map (maps closed sets to closed sets).
To check this, let $C\subset O$ be a closed subset. Then $\pi(C)\subset \qu{O}{\sim}$ is closed in the quotient topology iff $\pi^{-1}\bigl(\qu{O}{\sim}\less \pi(C)\bigr)= O \less \pi^{-1}\bigl(\pi(C)\bigr)$ is open, i.e.\ $\pi^{-1}\bigl(\pi(C)\bigr)\subset O$ is closed.
Now $\pi^{-1}\bigl(\pi(C)\bigr) = {\rm pr}_2 \bigl((C\times O)\cap R\bigr)$ is indeed closed since it is the projection of an intersection of two closed sets, and the projection $\pr_2:O\times O\to O$ to the second factor is a continuous map between compact Hausdorff spaces, so maps closed hence compact sets to compact hence closed sets.

Now the quotient of a Hausdorff space with closed projection is Hausdorff by \cite[Exercise~31.7]{Mun}. Moreover, $\qu{O}{\sim}=\pi(O)$ is compact since it is the image of a compact set $O$ under the continuous projection $\pi$, and compact Hausdorff spaces are normal by \cite[Theorem~32.3]{Mun}.  Finally, in a normal space, disjoint closed sets can be separated by neighbourhoods with disjoint closures by \cite[Exercise~31.2]{Mun}.  

\MS
\NI  {\bf Step 2:}
{\it  
In the general case, $O=\bigcup_{k\in\N} O_k$  is exhausted by a nested sequence 
$O_k \subset O_{k+1}$ of open subsets with compact closures $\ov {O_k}\subset O_{k+1}$.
Moreover, the quotients $X_k:=\qu{\ov{O_k}}{\sim}$ are compact and Hausdorff with 
continuous injections $\iota_{k \ell}: X_k\to X_\ell$  for $k<\ell$
that are induced by the injective inclusions $j_k: X_k\to \qu{O}{\sim}$ in the sense that 
$j_\ell\circ\iota_{k \ell}= j_k$.}

\MS

By hypothesis, each point $x\in O$ has a metric neighbourhood $B_{\eps}(x)$ that is contained in a precompact neighbourhood $V$ and so has compact closure $\ov B_{\eps}(x)$.
Recall that separable metric spaces are also second countable and hence Lindel\"of: Every open cover has a countable subcover.
Hence we may cover $O$ by a countable number of the precompact 
neighbourhoods $\bigl(B_{\eps_i}(x_i)\bigr)_{i\in\N}$.  
Then define $O_1: = B_{\eps_1}(x_1)$, and, given $O_k$ with compact closure $\ov O_k$, choose a finite subset $A_k$ so that $\ov O_k\subset \bigcup_{i\in A_k} B_{\eps_i}(x_i)$ and then define   
$O_{k+1}: = B_{\eps_{k+1}}(x_{k+1}) \cup \bigcup_{i\in A_k} B_{\eps_i}(x_i)$.  Then 
 $O=\bigcup_{k\in\N} O_k$ because $B_{\eps_k}(x_k)\subset O_k$ for each $k$.

Now the Hausdorff property of $\pi_k:\ov{O_k}\to X_k$ follows from Step 1, and compactness 
of $X_k$ follows from compactness of $\ov{O_k}$ and continuity of the projection $\pi_k:\ov{O_k}\to X_k$.
The inclusions $j_k$ are induced by the operation $\pi\circ\pi_k^{-1}$ on sets. These take points to points, and they are injective by transitivity of the relation $\sim$.
The inclusions 
$\iota_{k\ell}$  are induced by the operation $\pi_\ell\circ\pi_k^{-1}$, which take points to points due to $\ov{O_k}\subset\ov{O_\ell}$ for $\ell>k$, and are injective by transitivity. 
The identity  
$j_\ell\circ\iota_{k\ell}= j_k$ 
follows from the representation in terms of the projections and transitivity.
Finally,  $\iota_{k\ell }$  is continuous since $V\subset X_\ell$ open means $\pi_\ell^{-1}(V) = \Uu \cap \ov {O_\ell}$ for some $\Uu\subset O$ open, and then $\pi_k^{-1}\bigl(\iota_{k\ell}^{-1} (V)\bigr)=\pi_\ell^{-1}(V) \cap \ov O_k = \Uu \cap \ov O_k$ is open in $\ov O_k$ as well.

\MS
\NI  
{\bf Step 3:}
{\it 
Given disjoint points $x^1, x^2 \in \pi_1(O_1)\subset X_1$, there are open neighbourhoods $B_k^i\subset X_k$ of $\iota_{1k}(x^i)$ for $i=1,2$ and $k\in\N$, satisfying}
\begin{equation} \label{topinduct}
\ov {B_k^1} \cap \ov {B_k^2} = \emptyset \qquad\quad\text{and}\qquad\quad
\iota_{k(k+1)}(B_k^i)= \iota_{k(k+1)}(X_k) \cap B_{k+1}^i \quad\text{for}\; i=1,2 .
\end{equation}

\MS
\NI
To prove this we can use the normal Hausdorff property for the compact quotients $X_k$ proved in Step 1.
We start the inductive construction by choosing $B_1^1, B_1^2\subset X_1$ to be open neighbourhoods of $x^1\neq x^2$ in $X_1$ with disjoint closures.
For the inductive step, suppose that we have found open neighbourhoods $B_\ell^i\subset X_\ell$ of $x^i_\ell:=\iota_{1\ell}(x^i)$ for $\ell=1,\ldots,k$.
Then the continuous inclusion $\iota_k:=\iota_{k(k+1)}$ maps the disjoint compact sets $\ov{B_k^1}, \ov{B_k^2}\subset X_k$ to disjoint compact sets  $\iota_k\bigl(\ov{B_k^1}\bigr)\cap \iota_k\bigl(\ov{B_k^2}\bigr) = \emptyset$ in $X_{k+1}$ which contain $\iota_k(x^i_k)=x^i_{k+1}$ for $i=1,2$. By the normal Hausdorff property of $X_{k+1}$, we find open neighbourhoods $U^i_{k+1}\subset X_{k+1}$ of $\iota_k\bigl(\ov{B_k^i}\bigr)$ for $i=1,2$ with disjoint closures.

Then $B_{k+1}^i := U^i_{k+1} \less \iota_k\bigl(X_k \less B^i_k\bigr)$ for $i=1,2$ also have disjoint closures, still contain $x^i_{k+1}\in \iota_k\bigl(B^i_k\bigr) \subset U^i_{k+1}$, and are open since $\iota_k\bigl(X_k \less B^i_k\bigr)$ is the compact and hence closed image of a compact set under a continuous map.
Finally, they satisfy the second iteration condition $ \iota_k(X_k) \cap B_{k+1}^i = \iota_k(B_k^i)$ in \eqref{topinduct} since
$\iota_k(X_k) = \iota_k(B_k^i) \cap \iota_k\bigl(X_k \less B^i_k\bigr) \bigr)$ with $\iota_k(B_k^i) \subset U^i_{k+1}$.

\MS
\NI  
{\bf Step 4:}
{\it We prove the Hausdorff property of $\qu{O}{\sim}$ in the general case.}

\MS
\NI
Given disjoint points $y_1\neq y_2$ in $\qu{O}{\sim}$, we can without loss of generality forget about the first few exhausting subsets so that $y_1,y_2\in \pi(O_1)$, and hence we can view $y_i=j_1(x^i)$ as disjoint points $x^1,x^2\in X_1$ and $\iota_{1k}(x^i)=j_k^{-1}(y_i)\in X_k$ for each $k\in\N$, and apply Step 3.
We then claim that
$$
U_i := {\textstyle \bigcup_{ k\in\N}}  \pi\bigl( O_k \cap \pi_k^{-1}(B_k^i)\bigr) \;\subset\;\qu{O}{\sim}
$$
are disjoint open neighbourhoods of $y_i$ for $i=1,2$. 
For that purpose note that iteration of the second part of \eqref{topinduct} implies 
$\iota_{k\ell }(B_k^i)= \iota_{k \ell} (X_k) \cap B_\ell^i$ for $\ell>k$.
For $\ell\geq k$ this identity and the inclusion $O_k\subset O_\ell$ implies
$$
O_\ell \cap \pi^{-1}\bigl( \pi\bigl( O_k \cap \pi_k^{-1}(B_k^i)\bigr)\bigr)
= O_\ell \cap \pi_\ell^{-1} \bigl( \iota_{k\ell } \bigl( \pi_k( O_k ) \cap B_k^i\bigr)\bigr)
\subset O_\ell \cap \pi_\ell^{-1} \bigl( B_\ell^i\bigr) ,
$$
whereas for $\ell<k$ we use the equivalent formulation $\iota_{\ell k}^{-1}(B_k^i)=B_\ell^i$
and the inclusion $\ov{O_\ell}\subset O_k$ to obtain
$$
O_\ell \cap \pi^{-1}\bigl( \pi\bigl( O_k \cap \pi_k^{-1}(B_k^i)\bigr)\bigr)
= 
O_\ell \cap \pi_\ell^{-1}\bigl( \iota_{\ell k}^{-1} \bigl( \pi_k(O_k) \cap B_k^i \bigr)\bigr)
= 
O_\ell \cap \pi_\ell^{-1}\bigl( B_\ell^i \bigr) .
$$
Together this implies $O_\ell \cap \pi^{-1}(U_i) = O_\ell \cap \pi_\ell^{-1}\bigl( B_\ell^i \bigr)$, where $\pi_k^{-1}(B_k^i)$ is open in the relative topology on $\ov{O_k}$ (because $B_k^i$ is open and $\pi_k$ continuous), i.e.\ $\pi_k^{-1}(B_k^i)=\Bb^i_\ell\cap \ov{O_k}$ for some open subset $\Bb^i_\ell\subset O$. This implies that each $O_\ell \cap \pi^{-1}(U_i) = O_\ell \cap \Bb^i_\ell$ is an intersection of open subsets, and hence their union $\pi^{-1}(U_i)$ is open in $O=\bigcup_{\ell\in\N} O_\ell$.

Note also that each $U_i$ contains $y_i \in \pi(O_1) \cap \pi\bigl(\pi_1^{-1}(B^i_1))$, so we have found open neighborhoods of $y_1,y_2\in\qu{O}{\sim}$. They are disjoint since for w.l.o.g. $\ell\ge k\in\N$ we have
$$
\pi\bigl( O_k \cap \pi_k^{-1}(B_k^1)\bigr)
\cap 
\pi\bigl( O_\ell \cap \pi_\ell^{-1}(B_\ell^2)\bigr)
\;\subset\;
\pi\bigl( \pi_\ell^{-1}\bigl(
\iota_{k\ell }\bigl( B_k^1 \bigr)  \cap B_\ell^2 \bigr)\bigr) ,
$$
which is empty since the construction \eqref{topinduct}  implies $\iota_{k\ell }\bigl( B_k^1 \bigr)\subset B_\ell^1$ and $B_\ell^1\cap B_\ell^2=\emptyset$.
This concludes the proof of Step 4 and thus the Lemma.
\end{proof}

\subsection{Tame shrinkings} \label{ss:shrink} \hspace{1mm}\\ \vspace{-3mm}

The purpose of this section is to prove the existence part of Theorem~\ref{thm:K} by giving a general construction of a metric tame topological Kuranishi atlas starting from a filtered weak topological Kuranishi atlas in Propositions~\ref{prop:proper} and \ref{prop:metric}.
The construction uses a shrinking of the footprints along with the domains of charts and transition maps, as follows.

\begin{defn}\label{def:shr0}
Let $(F_i)_{i=1,\ldots,N}$ be an open cover of a compact space $X$.  We say that $(F_i')_{i=1,\ldots,N}$ is a {\bf shrinking} of $(F_i)$ if $F_i'\sqsubset F_i$ are
precompact open subsets which cover $X= \bigcup_{i=1,\ldots,N} F'_i$, and are such that for all subsets $I\subset \{1,\ldots,N\}$ we have
\begin{equation} \label{same FI}
F_I: = {\textstyle\bigcap_{i\in I}} F_i \;\ne\; \emptyset
\qquad\Longrightarrow\qquad
F'_I: = {\textstyle\bigcap_{i\in I}} F'_i \;\ne\; \emptyset .
\end{equation}
\end{defn}

Recall here that precompactness $V'\sqsubset V$ is defined as the relative closure of $V'$ in $V$ being compact.  If $V$ is contained in a compact space $X$, then $V'\sqsubset V$ is equivalent to the closure $\ov{V'}$ in the ambient space $X$ being contained in $V$.

\begin{defn}\label{def:shr}
Let $\Kk=(\bK_I,\Hat\Phi_{I J})_{I, J\in\Ii_\Kk, I\subsetneq J}$ be a weak topological Kuranishi atlas.   We say that a weak topological Kuranishi atlas $\Kk'=(\bK_I',\Hat\Phi_{I J}')_{I, J\in\Ii_{\Kk'}, I\subsetneq J}$ is a {\bf shrinking} of $\Kk$~if
\begin{enumerate}
\item  
the footprint cover $(F_i')_{i=1,\ldots,N'}$ is a shrinking of the cover $(F_i)_{i=1,\ldots,N}$,
in particular the numbers $N=N'$ of basic charts and index sets $\Ii_{\Kk'} = \Ii_\Kk$ agree;
\item
for each $I\in\Ii_\Kk$ the chart $\bK'_I$ is the restriction of $\bK_I$ to a precompact domain 
$U_I'\sqsubset U_I$
as in Definition \ref{def:restr};
\item
for each $I,J\in\Ii_\Kk$ with $I\subsetneq J$ the coordinate change $\Hat\Phi_{IJ}'$ is the restriction of $\Hat\Phi_{IJ}$  to the open subset $U'_{IJ}: =  \phi_{IJ}^{-1}(U'_J)\cap U'_I$
 as in Lemma~\ref{le:restrchange}.
\end{enumerate}
\end{defn}

Note that a shrinking is determined by the choice of domains $U'_I\sqsubset U_I$,
thus can be considered as the restriction of $\Kk$ to the subset $\bigsqcup_{I\in\Ii_\Kk} U_I'\subset
\Obj_{\bB_\Kk}$. 
Any such restriction of a weak topological Kuranishi atlas preserves the weak cocycle condition (since it only requires equality on overlaps), and hence is a shrinking.
The next lemma shows that shrinking preserves the filtration. More precisely, any shrinking of a filtered atlas will be equipped with a canonical filtration as follows.

\begin{lemma}\label{le:filter}  
Any shrinking of a filtered weak topological Kuranishi atlas $\Kk$ is canonically filtered: 
Given a filtration $(\E_{IJ}\subset \E_J)_{I\subset J}$ of $\Kk$ and a shrinking 
$\Kk'$ induced by domains $U'_I\sqsubset U_I$, the canonical filtration on $\Kk'$ is given by $\E_{IJ}' := \E_{IJ}\cap \pr_J^{-1}(U_J')$.
\end{lemma}
\begin{proof}  
The conditions in Definition~\ref{def:Ku3} are preserved by the shrinking as follows:
\begin{enumerate}
\item 
$\E'_{JJ}= \E'_J$ and $\E'_{\emptyset J} = \im 0'_J$ holds since $\pr_J^{-1}(U_J')=\E'_J$.
\item 
For ${I\subset J\subsetneq K}$ we have
\begin{align*}
\Hat\Phi'_{JK}\bigl({\pr'_J}^{-1}(U'_{JK})\cap \E'_{IJ}\bigr) 
& =\Hat\Phi_{JK}\Bigl( \pr_J^{-1}\bigl( U_{JK} \cap U'_J \cap \phi_{JK}^{-1}(U'_K)  \bigr) \cap \E_{IJ} \Bigr)  \\
&= \Hat\Phi_{JK}\Bigl( \pr_J^{-1}\bigl(U'_J \cap \phi_{JK}^{-1}(U'_K)  \bigr) \Bigr) \cap \E_{IK}\cap \pr_K^{-1}(\im \phi_{JK}) \\
&= 
\pr_K^{-1}( U'_K) \cap \E_{IK} \cap \pr_K^{-1}\bigl(\phi_{JK}(U'_{JK})\bigr) \\
&= \E'_{IK}\cap {\pr'}_K^{-1}(\im \phi'_{JK}) ,
\end{align*} 
where the second and third equalities hold because $\Hat\Phi_{JK}$ is a bundle map;
\item 
for $I, H \subset J$ we have
$$
\E'_{IJ}\cap \E'_{HJ} 
= \E_{IJ}\cap \pr_J^{-1}(U_J') \cap \E_{HJ}
= \E_{(I\cap H)J} \cap \pr_J^{-1}(U_J') = \E'_{(I\cap H)J}.
$$
\item 
For $I\subsetneq J$ we can write the image $\im \phi'_{IJ}= \phi_{IJ}(U_{IJ}')= \phi_{IJ}(U_{IJ}\cap U'_I) \cap U'_J$ as the image under the homeomorphism $\phi_{IJ}$ of an open subset $U_{IJ}'\subset U_{IJ}$, which shows that $\im \phi'_{IJ}$ is an open subset of $\im\phi_{IJ}\cap U'_J$ and thus -- since $(\E_{IJ})$ is a filtration -- of 
${\s}_J^{-1}(\E_{IJ}) \cap U'_J = {\s}_J^{-1}\bigl(\E_{IJ} \cap \pr_J^{-1}(U'_J) \bigr) = {\s'}_J^{-1}(\E'_{IJ})$.
\end{enumerate}
This shows that filtrations are inherited by shrinkings.
\end{proof}

\begin{rmk}\rm \label{rmk:shrink}
Given two tame shrinkings $\Kk^0$ and $\Kk^1$ of the same weak topological Kuranishi atlas, one might hope to obtain a ``common refinement'' $\Kk^{01}$ by intersection $U^{01}_{IJ}:=U^0_{IJ}\cap U^1_{IJ}$ of the domains.
This could in particular simplify the proof of compatibility of the atlases $\Kk^0,\Kk^1$ in 
Theorem~\ref{thm:cobord2}.
However, for this to be a valid approach, the footprint covers $(F^0_i)_{i=1,\ldots,N}$ and $(F^1_i)_{i=1,\ldots,N}$ 
would have to be comparable in the sense that their intersections still cover $X=\bigcup_{i=1,\ldots,N} (F^0_i\cap F^1_i)$ and have the same index set, i.e.\ $F^0_I \cap F^1_I \neq \emptyset$ for all $I\in\Ii_\Kk$.
Once this is satisfied, one can check that $\Kk^{01}$ defines another tame shrinking of $\Kk$.
$\hfill\er$
\end{rmk}

We can now prove the main result of this section.

\begin{prop}  \label{prop:proper}
Every filtered weak topological Kuranishi atlas $\Kk$ has a
{\bf tame shrinking}:  a shrinking $\Kk'$ that is a tame topological Kuranishi atlas.
\end{prop}

Here the main challenge is to achieve the tameness conditions \eqref{eq:tame1}, \eqref{eq:tame2}, which relate the domains $U_{IJ}$ of the transition data with each other and with the filtration $\E_{IJ}$.
So although a shrinking is determined by the choice of the sets $U_I'\sqsubset U_I$, our goal is to achieve relations between the domains $U'_{IJ}:= \phi_{IJ}^{-1}(U'_J)\cap U'_I$ of the coordinate changes, and hence 
we cannot choose the sets $U_I'$ independently of each other.
Since the relevant conditions are expressed in terms of the $U_{IJ}'$, we will iteratively construct such domains $U_{IJ}^{(k)}$, including in particular $U_I^{(k)}:=U_{II}^{(k)}$.
The latter will form shrinkings in each step, though we prove it only up to the level $k$ of the iteration. 
Our construction is made possible only because of the filtration conditions on $\Kk$.

\begin{proof}[Proof of Proposition~\ref{prop:proper}]
In a first step, we will choose a shrinking $(F_i')_{i=1,\ldots,N}$ of the footprint cover $(F_i)_{i=1,\ldots,N}$ in the sense of Definition~\ref{def:shr0}. For that purpose we use the fact that $X=\bigcup_{I\in\Ii_\Kk} F_I$ is a finite open cover to choose $\de>0$ so that every nonempty $F_I$ contains some ball $B_\de(x_I)$. 
Then for every $i\in I$ the footprint $F_i$ contains the compact ball $\ov{B}_{\de/2}(x_I) \subset F_I \subset F_i$. Now we iteratively choose $F'_i\sqsubset F_i$ for $i=1,\ldots,N$ as open subset containing the compact set $C_i:=\bigl( X \less \bigl( \bigcup_{i<j} F'_i \cup \bigcup_{j>i} F_j \bigr) \cup \bigcup_{I\ni i} \ov{B}_{\de/2}(x_I)$. 
This is possible since $X=\bigcup_{i<j} F'_i \cup \bigcup_{j\ge i} F_j$ is an open cover in each step, and $X$ is normal. The ball construction ensures that $F'_I\bigcup_{i\in I} F'_i\neq\emptyset$ since it contains $\ov{B}_{\de/2}(x_I)$ for each $I\in\Ii_\Kk$.

In another preliminary step we find precompact open subsets $U_I^{(0)}\sqsubset U_I$ and open sets $U_{IJ}^{(0)}\subset U_{IJ}\cap U_I^{(0)}$ for all $I,J\in\Ii_\Kk$ such that 
\begin{equation}\label{eq:U(0)}
U_I^{(0)}\cap \s_I^{-1}(0_I) = \psi_I^{-1}(F_I'),\qquad U_{IJ}^{(0)}\cap \s_I^{-1}(0_I) = \psi_I^{-1}(F'_I\cap F_J').
\end{equation}
The restricted domains $U_I^{(0)}$ are provided by Lemma~\ref{le:restr0}. 
Next, we may take
\begin{equation}\label{eq:UIJ(0)}
U_{IJ}^{(0)}:= U_{IJ}\cap U_I^{(0)}\cap \phi_{IJ}^{-1}( U_J^{(0)}) 
\; \subset U_{IJ}\cap U_I^{(0)},
\end{equation}
which is open because $U_\bullet^{(0)}\subset U_\bullet$ is open and $\phi_{IJ}:U_{IJ}\to U_J$ is continuous.
It has the required footprint
$$
U_{IJ}^{(0)}\cap \s_I^{-1}(0_I) =U_{IJ}\cap \psi_I^{-1}(F_I')\cap \phi_{IJ}^{-1} \bigl(\psi_J^{-1}(F_J')\bigr) = \psi_I^{-1}(F_I'\cap F_J') =  \psi_I^{-1}(F_J').
$$
Therefore, this defines a weak topological Kuranishi atlas with footprints $F_I'$, which satisfies the conditions of Definition~\ref{def:shr} and so is a shrinking of $\Kk$.
Moreover, it inherits a canonical filtration by Lemma~\ref{le:filter}.

Now we will construct the required shrinking $\Kk'$ by choosing possibly smaller domains  $U_I'\subset U_I^{(0)}$ and $U_{IJ}'\subset U_{IJ}^{(0)}$ but will preserve the footprints $F_I'$.
We will also arrange $U_{IJ}' = U_I'\cap \phi_{IJ}^{-1}(U_J')$, so that $\Kk'$ is a shrinking of the original $\Kk$.  Since $\Kk'$ is automatically filtered by Lemma~\ref{le:filter}, we just need to make sure that it satisfies the tameness conditions~\eqref{eq:tame1}  and~\eqref{eq:tame2}.
By Lemma \ref{le:tame0} it will then satisfy the cocycle condition and hence will be a topological Kuranishi atlas.
We will construct the domains $U_I', U_{IJ}'$ by a finite iteration, starting with $U_I^{(0)}, U_{IJ}^{(0)}$.
Here we streamline the notation by setting $U_{I}^{(k)}:=U_{II}^{(k)}$ and extend the notation to all pairs of subsets $I\subset J\subset\{1,\ldots,N\}$ by setting $U_{IJ}^{(k)}=\emptyset$ if $J\notin\Ii_\Kk$. (Note that $J\in\Ii_\Kk$ and $I\subset J$ implies $I\in\Ii_\Kk$.)
Then in the $k$-th step we will construct open subsets $U_{IJ}^{(k)}\subset U_{IJ}^{(k-1)}$ for all $I\subset J\subset\{1,\ldots,N\}$ such that the following holds.

\begin{enumerate}
\item
The zero set conditions $U_{IJ}^{(k)}\cap \s_I^{-1}(0_I) = \psi_I^{-1}(F_J')$ hold for all $I\subset J$.
\item
The first tameness condition \eqref{eq:tame1} holds for all $I\subset J,K$ with $|I|\le k$, that is
$$
U_{IJ}^{(k)}\cap U_{IK}^{(k)}= U_{I (J\cup K)}^{(k)} .
$$
In particular, we have $U_{IK}^{(k)} \subset U_{IJ}^{(k)}$ for $I\subset J \subset K$ with $|I|\le k$.
\item
The second tameness condition \eqref{eq:tame2} holds for all $I\subset J\subset K$ with $|I|\le k$, 
with respect to the induced filtration $\E_{IJ}^{(k)}:=\E_{IJ}\cap \pr_J^{-1}(U_J^{(k)})$.
Since $\pr_J\circ\s_J =\id_{U_J}$ and $U_{JK}^{(k)}\subset U_J^{(k)}$ by (ii), this tameness is equivalent to
$$
\phi_{IJ}(U_{IK}^{(k)}) = U_{JK}^{(k)}\cap \s_J^{-1}(\E_{IJ}) .
$$
In particular we have $\phi_{IJ}(U_{IJ}^{(k)}) = U_{J}^{(k)}\cap  \s_J^{-1}(\E_{IJ})$  for all $I\subset J$ with $|I|\le k$.
\end{enumerate}\MS
Note that after the $k$-th step, the domains $U^{(k)}_{IJ}$ form a shrinking ``up to order $k$'' in the sense that
\begin{equation}\label{kclaim}
U_{IJ}^{(k)} = U_I^{(k)} \cap \phi_{IJ}^{-1}(U_J^{(k)}) \qquad \forall\; |I|\le k , I\subsetneq J .
\end{equation}
Indeed, for any such pair $I\subsetneq J$, property (iii) with $J=K$ implies 
$$
\phi_{IJ}(U_{IJ}^{(k)}) \;=\; U_{J}^{(k)}\cap \s_J^{-1}(\E_{IJ}) 
\;=\; U_{J}^{(k)}\cap \im\phi_{IJ},
$$ 
where the second equality holds because the first implies 
$U_{J}^{(k)}\cap \s_J^{-1}(\E_{IJ})  \subset \im\phi_{IJ}$,  
while also $\im \phi_{IJ}\subset \s_J^{-1}(\E_{IJ})$ 
because $\s_J\circ\phi_{IJ}=\Hat\phi_{IJ}\circ \s_I$.
Since $\phi_{IJ}$ is injective, this implies $U_{IJ}^{(k)}=\phi_{IJ}^{-1}(U_J^{(k)})$.
Now  \eqref{kclaim} follows since (ii) with $K=I$ implies $U_{IJ}^{(k)} \subset U_{II}^{(k)} = U_I^{(k)}$.

Thus, when the iteration is complete, that is when $k=M: =\max_{I\in \Ii_\Kk} |I|$, then $\Kk'$ is a shrinking of $\Kk$.
Moreover, the tameness conditions hold on $\Kk'$ by (ii) and (iii), and Lemma~\ref{le:tame0} implies that $\Kk'$ satisfies the strong cocycle condition. Hence $\Kk'$ is the desired tame topological Kuranishi atlas.
So it remains to implement the iteration. 

Our above choice of the domains $U_{IJ}^{(0)}$ completes the $0$-th step since conditions (ii) (iii) are vacuous. Now suppose that the $(k-1)$-th step is complete for some $k\geq 1$.
Then we define $U_{IJ}^{(k)}:=U_{IJ}^{(k-1)}$ for all $I\subset J$ with $|I|\leq k-1$. For $|I|=k$ we also set 
$U_{II}^{(k)}:=U_{II}^{(k-1)}$.
This ensures that (i) and (ii) continue to hold for $|I|<k$. In order to preserve (iii) for triples $H\subset I\subset J$ with $|H|<k$ we then require that the intersection 
$U_{IJ}^{(k)}\cap \s_I^{-1}(\E_{HI})= U_{IJ}^{(k-1)}\cap \s_I^{-1}(\E_{HI})$ is fixed.
In case
$H=\emptyset$, this is condition (i), and since $U_{IJ}^{(k)}\subset U_{IJ}^{(k-1)}$ it can generally be phrased as inclusion (i$'$) below.
With that it remains to construct the open sets $U_{IJ}^{(k)}\subset U_{IJ}^{(k-1)}$ as follows.
\begin{itemize}
\item[(i$'$)]
For all $H\subsetneq I\subset J$ with $|H|<k$ and $|I|\geq k$ we have $U_{IJ}^{(k-1)}\cap \s_I^{-1}(\E_{HI})
\subset U_{IJ}^{(k)}$. Here we include $H=\emptyset$, in which case the condition says that
$U_{IJ}^{(k-1)}\cap  \s_I^{-1}(0_I) \subset U_{IJ}^{(k)}$ (which implies $U_{IJ}^{(k)}\cap \s_I^{-1}(0) = \psi_I^{-1}(F_J')$, as explained above).
\item[(ii$'$)]
For all $I\subset J,K$ with $|I|= k$ we have
$U_{IJ}^{(k)}\cap U_{IK}^{(k)}= U_{I (J\cup K)}^{(k)}$.
\item[(iii$'$)]
For all $I\subsetneq J\subset K$ with $|I|=k$ we have
$\phi_{IJ}(U_{IK}^{(k)}) = U_{JK}^{(k)}\cap \s_J^{-1}(\E_{IJ})$. 
\end{itemize}

Here we also used the fact that (iii) is automatically satisfied for $I=J$ and so stated (iii$'$) only for $J\supsetneq I$.
By construction, the domains $U^{(k)}_{II}$ for $|I|=k$ already satisfy (i$'$), so we may now do this iteration step in two stages:

\MS\NI
{\bf Step A}\; will construct $U_{IK}^{(k)}$ for $|I|=k$ and $I\subsetneq K$ satisfying (i$'$),(ii$'$) and

\begin{itemize}
\item[(iii$''$)]
$U_{IK}^{(k)} \subset \phi_{IJ}^{-1}(U_{JK}^{(k-1)})$
for all $I\subsetneq J\subset K$ .
\end{itemize}
{\bf Step B} will construct $U_{JK}^{(k)}$ for $|J|>k$ and $J\subset K$ satisfying (i$'$) and (iii$'$).
\MS

\bigskip\NI
{\bf Step A:} We will accomplish this construction by applying Lemma \ref{le:set} below for fixed $I$ with $|I|=k$ to the metric space $U:=U_I$, its precompact open subset $U':=U^{(k)}_{II}$, the relatively closed subset
$$
Z :=
\bigcup_{H\subsetneq I} \bigl( U_{II}^{(k-1)} \cap  \s_I^{-1}(\E_{HI}) \bigr)  
\;=\; \bigcup_{H\subsetneq I} \phi_{HI}\bigl(U_{HI}^{(k-1)}\bigr)
\;\subset\; U'
$$
and the relatively open subsets for all $i\in\{1,\ldots,N\}\less I$
$$
Z_i := U^{(k-1)}_{I (I\cup\{i\})} \cap Z
\;=\; \bigcup_{H\subsetneq I} \bigl( U_{I(I\cup\{i\})}^{(k-1)} \cap  \s_I^{-1}(\E_{HI}) \bigr) 
\;=\; \bigcup_{H\subsetneq I} \phi_{HI}\bigl(U^{(k-1)}_{H (I\cup\{i\})} \bigr) .
$$
Here, by slight abuse of language, we define $\phi_{\emptyset I}\bigl(U_{\emptyset J}^{(k-1)}\bigr): =
U_{IJ}^{(k-1)} \cap \s_I^{-1}(0_I) $. 
Note that $Z\subset U'$ is relatively closed since $U_{II}^{(k-1)}=U_{II}^{(k)}=U'$ 
and $\E_{HI}\subset \E_I$ is closed for all $H$, and $Z_i\subset Z$ is relatively open since 
 $U^{(k-1)}_{I (I\cup\{i\})} \subset U'$ is open.
Also the above identities  for $Z$ and $Z_i$ in terms of $\phi_{HI}$
follow from (iii) for the triples $H\subset I \subset I$ and $H\subset I \subset I\cup\{i\}$ with $|H|<|I|=k$.

To understand the choice of these subsets, note that in case $k=1$ with $I=\{i_0\}$ the set $Z$ is given by $H=\emptyset$ and $U^{(0)}_{II}=U^{(0)}_{i_0} \sqsubset U_{i_0}$, hence $Z = U^{(0)}_{i_0} \cap \s_{i_0}^{-1}(0_{i_0}) = \psi_{i_0}^{-1}(F_{i_0}')$ is the preimage of the shrunk footprint and for $i\neq{i_0}$ we have $Z_{i}
= \psi_{i_0}^{-1}(F_{i_0}'\cap F'_i)$.
When $k\ge 1$, the index sets $K\subset\{1,\ldots,N\}$ containing $I$ as proper subset are in one-to-one correspondence with the nonempty index sets $K'\subset\{1,\ldots,N\}\less I$ via $K=I\cup K'$ and give rise to the relatively open sets
$$
Z_{K'} \,:=\; \bigcap_{i\in K'} Z_i
\;=\; Z \cap \bigcap_{i\in K'} U^{(k-1)}_{I (I\cup\{i\})}
\;=\; Z \cap U^{(k-1)}_{IK} .
$$
Here we used (ii) for $|H| < k$. We may also use the identity $Z_i=\bigcup_{H\subsetneq I} \ldots$ together with (ii) and
(iii) for $|H|<k$ to identify these sets with
\begin{equation}\label{eq:ZzK}
Z_{K'}
= \bigcup_{H\subsetneq I} \phi_{HI}\Bigl(\; \underset{i\in K'}{\textstyle \bigcap} U^{(k-1)}_{H (I\cup\{i\})} \Bigr)
= \bigcup_{H\subsetneq I} \phi_{HI}\bigl( U^{(k-1)}_{H (I\cup K')} \bigr)
= \bigcup_{H\subsetneq I} \bigl( U_{IK}^{(k-1)} \cap  \s_I^{-1}(\E_{HI}) \bigr) ,
\end{equation}
which explains their usefulness:
If we construct the new domains such that $Z_{K'}\subset U_{IK}^{(k)}$, then this implies the inclusion $U_{IK}^{(k-1)}\cap \s_I^{-1}(\E_{HI}) \subset Z_{K'} \subset U_{IK}^{(k)}$ required by (i$'$).

Finally, in order to achieve the inclusion condition
$U_{IK}^{(k)} \subset \phi_{IJ}^{-1}(U_{JK}^{(k-1)})$ of (iii$''$),
we fix the open subsets $W_{K'}$ for all $I\subsetneq K = I \cup K'$ as
\begin{equation}\label{eq:WwK}
W_{K'}: = \bigcap_{I\subset J\subset K} \bigl( \phi_{IJ}^{-1} (U^{(k-1)}_{JK}) \cap U^{(k-1)}_{IJ} \bigr) \;\subset\; U' .
\end{equation}
If we require $U_{IK}^{(k)}\subset W_{K'}$ then this ensures (iii$''$)
as well as $U_{IK}^{(k)}\subset U_{IK}^{(k-1)}$.
The latter follows from the inclusion $W_{K'}\subset U_{IK}^{(k-1)}$, which holds by definition \eqref{eq:WwK} with $J=K$.
Now if we can ensure that $W_{K'}\cap Z = Z_{K'}$, then Lemma \ref{le:set} provides choices of open subsets $U_{IK}^{(k)}\subset U'$ satisfying (ii$'$) and the inclusions $Z_{K'}\subset U_{IK}^{(k)}\subset W_{K'}$. The latter imply (i$'$) and the desired inclusion (iii$''$), as discussed above.

Hence it remains to check that the sets $W_{K'}$ in \eqref{eq:WwK} do satisfy the conditions
$W_{K'}\cap Z = Z_{K'}$.
To verify this, first note that $W_{K'}$ is contained in $U^{(k-1)}_{IJ}$ for all $J\supset I$, in particular
for
$J=K$, and hence
$$
W_{K'}\cap Z \;\subset\; U^{(k-1)}_{IK} \cap Z \;=\; Z_{K'}.
$$
It remains to check $Z_{K'} \subset W_{K'}$.
By~\eqref{eq:WwK} and the  expression for $Z_{K'}$ in the middle of~\eqref{eq:ZzK},
it suffices to show that for all $H\subsetneq I \subset J \subset K$
$$
\phi_{HI}\bigl( U^{(k-1)}_{H K} \bigr) \;\subset\; \phi_{IJ}^{-1} (U^{(k-1)}_{JK}) \cap U^{(k-1)}_{IJ} .
$$
But, (ii) for $H\subset J\subset K$ and (iii) for $H\subset I \subset J$ imply
$$
\phi_{HI}( U^{(k-1)}_{H K} ) \subset \phi_{HI}( U^{(k-1)}_{H J} ) \subset U^{(k-1)}_{IJ},
$$
so it remains
to check that
$\phi_{HI}\bigl( U^{(k-1)}_{H K} \bigr) \subset \phi_{IJ}^{-1} ( U^{(k-1)}_{JK} )$.
For that purpose we will use the weak cocycle condition $\phi_{IJ}^{-1}\circ\phi_{HJ}=\phi_{HI}$ on
$U_{H J}\cap \phi_{HI}(U_{IJ})$. Note that $U^{(k-1)}_{H K}$ lies in this domain since, by (ii) for $|H|<k$, it is a subset of $U^{(k-1)}_{H J}$, which by (iii) for $H\subset I\subset J$ is contained in $\phi_{HI}^{-1}(U_{IJ})$. This proves the first equality in
$$
\phi_{HI}\bigl( U^{(k-1)}_{H K} \bigr)
=\phi_{IJ}^{-1}\bigl( \phi_{HJ}\bigl( U^{(k-1)}_{H K} \bigr) \bigr)
\subset \phi_{IJ}^{-1} ( U^{(k-1)}_{JK} ),
$$
and the last inclusion holds by (iii) for $H\subset J\subset K$.
This finishes Step A.

\MS\NI
{\bf Step B:} The crucial requirement on the construction of the open sets $U_{JK}^{(k)}\subset U_{JK}^{(k-1)}$ for $|J|\geq k+1$ and $J\subset K$ is (iii$'$), that is
$$
U_{JK}^{(k)}\cap \s_J^{-1}(\E_{IJ}) 
 = \phi_{IJ}(U_{IK}^{(k)})
$$
 for all $I\subsetneq J$ with $|I|=k$.
Here $U_{IK}^{(k)}$ is fixed by Step A and satisfies $\phi_{IJ}(U_{IK}^{(k)})\subset U_{JK}^{(k-1)}\cap 
 \s_J^{-1}(\E_{IJ}) $  by (iii$''$),
where the second part of the inclusion is automatic by $\phi_{IJ}$ mapping to  $\s_J^{-1}(\E_{IJ})$.  
Hence the maximal subsets $U_{JK}^{(k)}\subset U_{JK}^{(k-1)}$ satisfying (iii$'$) are
\begin{equation}\label{eq:UJK(k)}
U_{JK}^{(k)}\,:=\; U_{JK}^{(k-1)}\less \bigcup_{I\subset J, |I|= k}
 \bigl(  \s_J^{-1}(\E_{IJ})   \less \phi_{IJ}(U^{(k)}_{IJ}) \bigr) .
\end{equation}
It remains to check that these subsets are open and satisfy (i$'$).
Here $U_{JK}^{(k)}$ is open since $ \s_J^{-1}(\E_{IJ})  
\subset U_J$ is closed and $\phi_{IJ}(U^{(k)}_{IJ})\subset  \s_J^{-1}(\E_{IJ})$  
 is relatively open by condition (iv) in Definition~\ref{def:Ku3}.
Finally, condition (i$'$), namely
$$
 U_{JK}^{(k-1)}\cap  \s_J^{-1}(\E_{HJ})  \subset U_{JK}^{(k)},
$$
follows from the following inclusions for all $H\subsetneq I \subsetneq J \subset K$ with $|I|=k$.
On the one hand, we have $U_{JK}^{(k-1)}\cap \s_J^{-1}(\E_{HJ}) \subset \s_J^{-1}(\E_{IJ})$ 
from the properties of 
the filtration on $\Kk$; on the other hand, (iii) for $|H|<k$ together with the weak cocycle condition on 
$U_{HK}^{(k-1)}\subset U_{HJ}\cap\phi_{HI}^{-1}(U_{IJ})$
and (i$'$) for $|I|=k$ imply
\begin{align*}
U_{JK}^{(k-1)}\cap \s_J^{-1}(\E_{HJ}) 
= \phi_{HJ}( U_{HK}^{(k-1)} )
&= \phi_{IJ}\bigl(\phi_{HI}( U_{HK}^{(k-1)} )\bigr) \\
&= \phi_{IJ}\bigl(U_{IK}^{(k-1)}\cap \s_I^{-1}(\E_{HI}) \bigr) 
\subset  \phi_{IJ}(U^{(k)}_{IJ}) .
\end{align*}
Hence no points of $ U_{JK}^{(k-1)}\cap \s_J^{-1}(\E_{HJ})$ 
 are removed from
$U_{JK}^{(k-1)}$ when we construct $U_{JK}^{(k)}$.
This finishes Step B and hence the $k$-th iteration step.

\MS
Since the order of $I\in\Ii_\Kk$ is bounded $|I|\leq N$ by the number of basic charts, this iteration provides a complete construction of the shrinking domains after at most $N$ steps. In fact, we obtain $U'_I=U^{(|I|-1)}_{II}$ and $U'_{IJ}=U^{(|I|)}_{IJ}$ since the iteration does not alter these domains in later steps.
\end{proof}

\begin{lemma} \label{le:set}
Let $U$ be a metric space, $U'\subset U$ a precompact open set, and $Z\subset U'$ a relatively closed subset. Suppose we are given a finite collection of relatively open subsets $Z_i\subset Z$ for $i=1,\ldots,N$ and open subsets $W_K\subset U'$ with
$$
W_K\cap Z= Z_K: = {\textstyle\bigcap_{i\in K}} Z_i
$$
 for all index sets $K\subset\{1,\ldots,N\}$.
Then there exist open subsets $U_K\subset W_K$ with $U_K\cap Z=Z_K$ and
$U_J\cap U_K = U_{J\cup K}$ for all $J,K\subset\{1,\ldots,N\}$. 
\end{lemma}

\begin{proof}
Let us first introduce a general construction of an open set $U_f\subset U'$ associated to any lower semi-continuous function $f:\overline{Z}\to [0,\infty)$, where $\overline{Z}\subset U$ denotes the closure in $U$.
By assumption, $\overline{Z}$ is compact, hence the distance function $\overline{Z}\to [0,\infty), z\mapsto d(x,z)$ for fixed $x\in U'$ achieves its minimum on a nonempty compact set $M_x\subset \overline{Z}$. Hence we have
$$
d(x,Z) := \inf_{z\in Z} d(x,z) = d(x,\overline{Z}) = \min_{z\in\overline Z} d(x,z) = \min_{z\in M_x} d(x,z)
$$
for the distance between $x$ and the set $Z$, or equivalently the closure $\overline{Z}$.
\MS

\NI {\bf Claim:}  {\it For any lower semi-continuous function $f:\overline{Z}\to [0,\infty)$ the
set}
$$
U_f := \bigl\{ x\in U' \,\big|\, d(x,Z) < \inf f|_{M_x} \bigr\} \subset U'
$$
{\it is open (in $U'$ or equivalently in $U$) and satisfies}
\begin{equation} \label{eq:supp}
U_f \cap Z  
 = \bigl\{ z\in Z \,\big|\, f(z)>0 \bigr\}.
\end{equation}

\NI {\it Proof of Claim.}
For $x\in Z$ we have $d(x,Z)=0$ and $M_x=\{x\}$, so $d(x,Z) < \inf f|_{M_x}$ is equivalent to $0<f(x)$. We prove openness of $U_f\subset U'$ by checking closedness of $U'\less U_f$.
Thus, we consider a convergent sequence $U'\ni x_i\to x_\infty\in U'$ with $d(x_i, Z)\geq \inf f|_{M_{x_i}}$ and aim to prove $d(x_\infty, Z)\geq \inf f|_{M_{x_\infty}}$.
Since $f$ is lower semi-continuous and each $M_{x_i}$ is compact, we may choose a sequence $z_i\in\overline{Z}$ with $z_i\in M_{x_i}$ and $f(z_i)=\inf f|_{M_{x_i}}$.
(Indeed, for fixed $i$ any minimizing sequence $z^\nu_i\in M_{x_i}$ with $\lim_{\nu\to\infty}f(z^\nu_i) = \inf f|_{M_{x_i}}$ has a convergent subsequence $z^\nu_i\to z_i\in M_{x_i}$ and the limit satisfies $f(z_i) \leq \lim f(z^\nu_i) = \inf f|_{M_{x_i}}$, hence $f(z_i) = \inf f|_{M_{x_i}}$.)
Since $\overline{Z}$ is compact, we may moreover choose a subsequence, again denoted by $(x_i)$ and $(z_i)$, such that $z_i\to z_\infty\in\overline{Z}$ converges. Then by continuity of the distance functions we deduce $z_\infty\in M_{x_\infty}$ from
$$
d(x_\infty, z_\infty) = \lim d(x_i,z_i) = \lim d(x_i, Z) = d(x_\infty, Z) ,
$$
and finally the lower semi-continuity of $f$ implies the claim
$$
d(x_\infty, Z) = \lim d(x_i,Z) \geq \lim f(z_i) \geq f(z_\infty) \geq \inf f|_{M_{x_\infty}} .
$$
\medskip

We now use this general construction to define the sets $U_K:=\bigcap_{i\in K} U_{f_i}$ as intersections of the subsets $U_{f_i}\subset U'$ arising from functions $f_i: \overline{Z}\to [0,\infty)$ defined
by
$$
f_i(z) := \min \bigl\{  d( z, U' \less W_J) \,\big|\, J\subset\{1,\ldots,N\} : \;  i\in J, \; d(z,Z\less Z_i) = d(z,Z\less Z_J) \bigr\}.
$$
To check that $f_i$ is indeed lower semi-continuous, consider a sequence $z_\nu\to z_\infty\in \overline{Z}$. Then $f_i(z_\nu)= d(z_\nu,U'\less W_{J^\nu})$ for some index sets $J^\nu$ with $i\in J^\nu$ and $d(z_\nu,Z\less Z_i) = d(z_\nu,Z\less Z_{J^\nu})$. Since the set of all index sets is finite, we may choose a subsequence, again denoted $(z_\nu)$, for which $J^\nu=J$ is constant. Then in the limit we also have $d(z_\infty,Z\less Z_i) = d(z_\infty,Z\less Z_J)$ and hence
$$
f_i(z_\infty) \leq d(z_\infty, U'\less W_J) = \lim d(z_\nu, U'\less W_J) = \lim f_i(z_\nu).
$$
Thus $f_i$ is lower semi-continuous.
Therefore, the above Claim implies that each $U_{f_i}$ is open, and hence also that each $U_K$ is open as the finite intersection of open sets.

The intersection property holds by construction:
$$
U_J\cap U_K = \bigcap_{i\in J} U_{f_i} \cap \bigcap_{i\in K} U_{f_i}
= \bigcap_{i\in J\cup K} U_{f_i} = U_{J\cup K}.
$$
To obtain $U_K\cap Z = \bigcap_{i\in K} \bigl( U_{f_i} \cap Z \bigr) = Z_K$ it suffices to verify that $U_{f_i}\cap Z = Z_i$.
In view of \eqref{eq:supp}, and unravelling the meaning of $f_i(z)>0$ for $z\in Z$, that means we have to prove the
following equivalence for $z\in Z$,
$$
z\in Z_i \quad \Longleftrightarrow \quad  d(z, U'\less W_J)>0  \quad \forall J\subset\{1,\ldots N\} : i\in J ,\; d(z,Z\less Z_i) = d(z,Z\less Z_J).
$$
Assuming the right hand side, we may choose $J=\{i\}$ to obtain $d(z,U'\less W_{\{i\}})>0$, and hence, since $U'\less W_{\{i\}}$ is closed, $z\in Z \less (U'\less W_{\{i\}})= Z_i$.
On the other hand, $z\in Z_i$ implies $d(z,Z\less Z_i)>0$ since $z\in Z$ and $Z\less Z_i \subset Z$ is relatively closed. So for any $J$ with $d(z,Z\less Z_i) = d(z,Z\less Z_J)$ we obtain
$d(z,Z\less Z_J)>0$.  Hence $z\in Z_J\subset W_J$, so that $d(z, U'\less W_J)>0$.
This proves the desired equivalence, and hence $U_K\cap Z = \bigcap_{i\in K}Z_i = Z_K$.

Finally, we need to check that $U_K\subset W_K$.
Unravelling the construction, note that $U_K$ is the set of all $x\in U'$ that satisfy
\begin{equation}\label{eq:UK}
d(x, Z) <  d( z , U' \less W_J)
\end{equation}
for all $z\in M_x$ and all $J\subset\{1,\ldots,N\}$ such that there exists $i\in J\cap K $ satisfying
$d(z,Z\less Z_i) = d(z,Z\less Z_J)$.
Now suppose by contradiction that there exists a point $x\in U_K\less W_K$, and pick $z\in M_x$.
Then $d(x,Z) = d(x, z) \geq d( z , U' \less W_K)$ since $x\in U'\less W_K$.
This contradicts \eqref{eq:UK} with $J=K$. On the other hand, the condition $d(z,Z\less Z_i) = d(z,Z\less Z_J)$ for $J=K$ is always satisfied for some $i\in K$ since we have $d(z,Z\less Z_K) = \min_{j\in K} d(z,Z\less Z_j)$.
This provides the contradiction and hence proves $U_K\subset W_K$.
\end{proof}

This lemma completes the proof that every weak filtered topological Kuranishi atlas has a tame shrinking.
We will return to these ideas in Section~\ref{s:Kcobord} when discussing cobordisms.
We end this section by constructing admissible metrics on certain tame shrinkings by pullback with the map in the following lemma. The proof uses the following lemma.

\begin{lemma}\label{le:injtameshr}  
Let $\Kk'$ be a tame shrinking of a tame topological Kuranishi atlas $\Kk$. Then the natural map $\io:|\Kk'|\to |\Kk|$ induced by the inclusion of domains $\io_I:U'_I\to U_I$ is injective.
\end{lemma}
\begin{proof} 
We write $U_I, U_{IJ}$ for the domains of the charts and coordinate changes of $\Kk$ and $U_I', U_{IJ}'$  for those of $\Kk'$, so that $U_I'\subset U_I, U_{IJ}'\subset U_{IJ}$ for all $I,J\in \Ii_\Kk = \Ii_{\Kk'}$.
Suppose that $\pi_\Kk(I,x) = \pi_\Kk(J,y)$ where $x\in U_I', y\in U_J'$.  Then we must show that $\pi_{\Kk'}(I,x) = \pi_{\Kk'}(J,y)$.  Since $\Kk$ is tame, 
Lemma~\ref{le:Ku2}~(a) implies that if $I\cap J\ne \emptyset$ there is $w\in U_{I\cap J}$ such that
$\phi_{(I\cap J)I} (w)$ is defined and equal to $x$.  Hence $x\in \s_I^{-1}(\E_{I\cap J})\cap U_I' =\phi_{(I\cap J)I} (U'_{(I\cap J)I})$ by the tameness equation \eqref{eq:tame3}  for $\Kk'$.
Therefore $w\in U'_{(I\cap J)I}$. 
Similarly, because $\phi_{(I\cap J)J} (w)$ is defined and equal to $y$, we have $w\in U'_{(I\cap J)J}$.
Then by definition of $\pi_{\Kk'}$ we deduce $\pi_{\Kk'}(I,x) = \pi_{\Kk'}(I\cap J,w) =\pi_{\Kk'}(J,y)$.
On the other hand, if $I\cap J = \emptyset$ then Lemma~\ref{le:Ku2}~(a) implies that $\psi_I(x) = \psi_J(y)$ so that  $\pi_{\Kk'}(I,x) = \pi_{\Kk'}(J,y)$ by the injectivity of $\io_\Kk: X\to |\Kk|$ proved in Lemma~\ref{le:Knbhd1}.
\end{proof}

In order to construct metric tame topological Kuranishi atlases, we will find it useful to consider tame shrinkings $\Kk_{sh}$ of a weak topological Kuranishi atlas $\Kk$ that are obtained as shrinkings of an intermediate tame shrinking $\Kk'$ of $\Kk$. For short we will call such $\Kk_{sh}$ a {\bf preshrunk tame shrinking} of $\Kk$.

\begin{prop}\label{prop:metric}  
Let $\Kk$ be a filtered weak topological Kuranishi atlas.
Then every preshrunk tame shrinking of $\Kk$ is metrizable. In particular, $\Kk$ has a metrizable tame shrinking.
\end{prop}

\begin{proof}  
First use Proposition~\ref{prop:proper} to construct a  tame shrinking $\Kk'$ of $\Kk$ with domains $(U_I'\subset U_I)_{I\in \Ii_\Kk}$, and then use this result again to construct a tame shrinking $\Kk_{sh}$ of $\Kk'$ with domains $(U_I^{sh}\sqsubset U_I')_{I\in \Ii_\Kk}$. We claim that $\Kk_{sh}$ is metrizable. 

For that purpose we apply Proposition~\ref{prop:Ktopl1}~(iv) to the precompact subset $\Aa: = \bigsqcup_{I\in\Ii_\Kk} U_I^{sh}$ of $\Obj_{\bB_{\Kk'}}$ to obtain a metric $d'$ on $\pi_{\Kk'}(\ov\Aa)$ that induces the relative topology on the subset $\pi_{\Kk'}(\ov\Aa)$ of $|\Kk'|$, that is $\bigl( \pi_{\Kk'}(\ov\Aa) , d'\bigr) = \|\ov\Aa\|$.
Further, since $\pi_{\Kk'}(\ov\Aa)$ is compact, the metric $d'$ must be bounded. 
Now, by Lemma~\ref{le:injtameshr} the natural map $\io: |\Kk_{sh}|\to |\Kk'|$ is injective, with image $\pi_{\Kk_{sh}}(\Aa)$, so that the pullback $d:=\io^* d'$ is a bounded metric  on $|\Kk_{sh}|$ that is compatible with the relative topology induced by $|\Kk'|$; in other words $\io : \bigl(|\Kk_{sh}|, d \,\bigr) \to \|\Aa\|$ is an isometry.

Next, note that the pullback metric $d_I$ on $U^{sh}_I$ does give the usual topology since 
$\pi_{\Kk^{sh}}: U^{sh}_I \to \pi_{\Kk^{sh}}(U^{sh}_I) \subset \bigl( |\Kk^{sh}| , d \bigr)$  is a homeomorphism to its image.  
Indeed, by Lemma~\ref{le:injtameshr} it can also be written as $\pi_{\Kk^{sh}}|_{U^{sh}_I} = \io^{-1} \circ \pi_{\Kk'} \circ \io_I$ with the embedding $\io_I: U^{sh}_I \to U'_I$. The latter is a homeomorphism to its image, as is $\pi_{\Kk'} : U'_I \to \pi_{\Kk'}(U'_I) \subset |\Kk'|$ by Proposition~\ref{prop:Khomeo}, and $\io^{-1}$ by the definition of the metric topology on $|\Kk^{sh}|$.  
\end{proof}  

%
%
%
%
%
%
%

\section{Cobordisms of topological Kuranishi atlases}\label{s:Kcobord}

Because there are many choices involved in constructing a Kuranishi atlas, and holomorphic curve moduli spaces in addition depend on the choice of an almost complex structure, it is important to have suitable notions of equivalence.
Since we are only interested here in constructing the virtual moduli cycle as cobordism class, resp.\ the virtual fundamental class as a homology class, a notion of uniqueness up to cobordism will suffice for our purposes.   
We will introduce (see Definition~\ref{def:CKS}) a general notion of {\bf topological Kuranishi cobordism}
to be a topological Kuranishi atlas on a compact metrizable space $Y$ with two collared boundary components $\p^0Y$  and $\p^1Y$.
However, if we are considering atlases over a fixed space $X$, then we will mostly work with
the  stronger equivalence relation of {\bf concordance}, which is a Kuranishi cobordism on the product $[0,1]\times X$. 

It turns out that, although the construction of a Kuranishi atlas on a fixed Gromov--Witten moduli space $X$ in \cite{Mcn,MW2} depends on many choices (for example of slicing conditions and obstruction spaces for the basic charts), the resulting atlas is unique up to concordance.
  
  \subsection{Definitions and properties} \label{ss:Kcobord} \hspace{1mm}\\ \vspace{-3mm}

In this section we first define these notions and then develop the abstract theory of cobordisms.  
The main result is Theorem~\ref{thm:cobord2}, which in particular implies that tame shrinkings are unique up to tame concordance. 
  
We begin with the notion of cobordism between topological Kuranishi atlases.
In order to define this so that it is transitive, we will need a special form of charts and coordinate changes at the boundary that allows for gluing of cobordisms.
Thus we will define a Kuranishi cobordism to be a Kuranishi atlas over a space $Y$ whose designated ``boundary components" $\p^0Y, \p^1Y \subset Y$ have collared neighbourhoods as follows.

\begin{defn} \label{def:Ycob}  
A {\bf collared cobordism} $(Y, \io^0_Y,\io^1_Y)$ is a separable, locally compact, 
metrizable space $Y$ together with disjoint (possibly empty)
closed subsets $\p^0 Y,$ $ \p^1 Y\subset Y$ and maps
$$
\io_Y^0:  [0,\eps)\times  \p Y^0  \to Y, \qquad  \io_Y^1:  (1-\eps, 1]\times  \p Y^1  \to Y
$$
for some $\eps>0$ that are {\bf collared neighbourhoods} in the following sense: 
They extend the inclusions $\io_Y^0(0,\cdot) : \p^0 Y\hookrightarrow Y$, resp.\ $\io_Y^1(1,\cdot) : \p^1 Y\hookrightarrow Y$, and are homeomorphisms onto disjoint open neighbourhoods of $\p^0 Y\subset Y$, resp.\ $\p^1 Y\subset Y$. 

We call $\p^0 Y$ and $\p^1 Y$ the {\bf boundary components} of $(Y, \io^0_Y,\io^1_Y)$.
\end{defn}

For the next definition, it is useful to introduce the notation
\begin{equation}\label{eq:Naleps}
A_\delta^0: = [0,\delta)  \qquad\text{and} \qquad A_\delta^1: = (1-\delta,1] \qquad\text{ where }\  0<\delta<\tfrac 12
\end{equation}
for collar neighbourhoods of 
$0$ resp.\ $1$ in $[0,1]$.

\begin{defn}\label{def:collarset}
If $(Y, \io_Y^0, \io_Y^1)$ is a collared cobordism, we say that an open subset $F\subset Y$ is {\bf collared} 
if there is $0<\delta\le\eps$ such that for $\al\in \{0,1\}$ we have
$$
F \cap \im (\io_Y^\al)\ne \emptyset
\;\; \Longleftrightarrow\;\;
F \cap \im (\io_Y^\al)
= \io_Y^\al( A^\al_\delta\times \p^\al F) .
$$
Here we denote by
$
\partial^\al F :=  F \cap \p^\al Y 
$
the intersection with the ``boundary component'' $\p^\al Y$ and allow one or both of $\p^\al F$ to be empty.
\end{defn}

Note that a collared subset $F\subset Y$ with empty 
``boundary'' $\p^\al F=\emptyset$ is in fact disjoint from the open neighbourhood $\im\io^\al_Y$ of the corresponding ``boundary component'' $\p^\al Y$.

\begin{example}\label{ex:natcol}\rm
In general the ``boundary components" $\p^\al Y$ 
are by no means uniquely determined by $Y$ or 
topological boundaries
of $Y$
in any sense, though the main example
of a collared cobordism 
is $Y = [0,1]\times X$, which has the natural ``boundary components'' $\{0\}\times X$ and $\{1\}\times X$. 
In this case we always take $\io_Y^\al$ to be the canonical extensions of the 
inclusions $\io_Y^\al(\al,\cdot): \{\al\}\times X \to [0,1]\times X$, for some choice of $0<\eps<\frac 12$.
More generally, we might  
consider a union of moduli spaces
$$
Y = {\textstyle \bigcup_{t\in [0,1]}} \{t\}\times \oMm_{0,k}(M,A,J_t) ,
$$
where $J_t$ is a family of
almost complex structures that are constant for $t$ near $0$ and~$1$. 
This again has canonical ``boundary components'' $\p^\al Y=\{\al\}\times \oMm_{0,k}(M,A,J_\al)$ and collared neighbourhoods $\io_Y^\al\bigl(t,(\al,p)\bigr)=(t,p)$ for sufficiently small choice of $\eps>0$.
$\hfill\er$
\end{example}

When constructing 
Kuranishi cobordisms
we will require all charts and coordinate changes in a sufficiently small collar to be of a compatible product form as introduced below.

\begin{defn} \label{def:Cchart}  
Let $(Y, \io_Y^0, \io_Y^1)$ be a compact collared cobordism.
\begin{itemlist}
\item
Let $\bK^\al=(U^\al,\E^\al,\s^\al,\psi^\al)$ be a topological Kuranishi chart for $\p^\al Y$, and let $A\subset[0,1]$ be a relatively open interval. Then we define the {\bf product chart} for $[0,1] \times \p^\al Y$ with footprint $A\times F^\al$ 
as
$$
A\times \bK^\al  :=\bigl(A \times U^\al, A\times \E^\al, \, \id_{A}\times (\s^\al\circ{\rm pr}_{U^\al} ) ,\, \id_{A}\times \psi^\al \bigr) ,
$$
where $A\times \E^\al$ denotes the obstruction bundle with 
projection $\id_{A}\times\pr^\al:A\times \E^\al\to A \times U^\al$ and 
 zero section $\id_{A}\times 0^\al: A\times U^\al \to A\times \E^\al$.
\item
A {\bf topological  Kuranishi chart with collared boundary} for $(Y, \io_Y^0, \io_Y^1)$
is a tuple $\bK = (U,\E ,\s,\psi)$ as in 
Definition~\ref{def:tchart}, with the following collar form requirements:
\begin{enumerate}
\item
The footprint $F =\psi (\s ^{-1}(0))\subset  Y$ 
is collared and intersects 
at least one of the boundary components $\p^\al Y$.
\item 
The domain $U$ and obstruction bundle $\E$ 
in addition carry the structure of 
collared cobordisms $(U,\io^0_U, \io^1_U)$ resp.\ $(\E, \io^0_{\E}, \io^1_{\E})$ whose boundary components $\partial^\al U$  and $\p^\al \E$ are nonempty iff $\p^\al F= F \cap \p^\al Y\ne \emptyset$. 
\item 
If $\partial^\al F \neq\emptyset$ then for some $\eps>0$ 
there is a topological Kuranishi chart $\partial^\al\bK$ for $\p^\al Y$ 
with footprint $\p^\al F$, domain $\p^\al U$, and obstruction bundle $\p^\al\E$,
and an embedding of the product chart $A_\eps^\al\times \p^\al \bK$ into $\bK$ in the following sense:
The boundary 
 embeddings $\iota_U^\al:A_\eps^\al \times \partial^\al U \hookrightarrow  U $ 
 and  $\iota_{\E}^\al:A_\eps^\al \times \partial^\al \E \hookrightarrow  \E $ 
 intertwine the projection, section, zero and footprint maps of the charts $A_\eps^\al\times \p^\al \bK$ and $\bK$ as in Definition~\ref{def:tchange}.
In particular, the footprint map is compatible with the boundary collars on $\bK$ and $Y$ in the sense that the following diagram commutes:
$$
  \begin{array} {ccc}
(\id_{A_\eps^\al}\times  \s^\al)^{-1}(\id_{A_\eps^\al}\times 0^\al) & \stackrel{\io_U^\al} \longrightarrow &{\s^{-1}(0)} \\
 \id_{A_\eps^\al}\times \psi^\al\downarrow\;\;\;\;\;\;\;\;\;&&\downarrow{\psi}  \\
\phantom{right}{A_\eps^\al\times \p^\al Y} & \stackrel{\io_Y^\al} \longrightarrow &{Y} \; .
\end{array}
$$
\end{enumerate}
\item
For any topological Kuranishi chart with collared boundary for $(Y, \io_Y^0, \io_Y^1)$
we call the resulting uniquely determined topological Kuranishi charts $\partial^\al\bK$ (with footprints in $\p^\al Y$) the {\bf restrictions of $\bK$ to the boundary}. 
\end{itemlist}
\end{defn}

We now define a coordinate change between charts on $Y$ that may have boundary.   Because in a Kuranishi atlas there is a coordinate change $\bK_I\to \bK_J$ only when $F_I\supset F_J$, we will  restrict to this case here.  (Definition~\ref{def:tchange} considered a more general scenario.)
In other words, we need not consider coordinate changes from a chart without boundary to a chart with boundary.

\begin{defn} \label{def:Ccc}
\begin{itemlist}
\item
Let $\Hat\Phi^\al_{IJ}:\bK^\al_I\to\bK^\al_J$ be a coordinate change between topological Kuranishi charts for $\p^\al Y$, and let $A_I,A_J\subset[0,1]$ be relatively open intervals.
Then the {\bf product coordinate change}
$\id_{A_I\cap A_J} \times \Hat\Phi^\al_{IJ}  : 
(A_I\cap A_J) \times \bK^\al_I \to A_J \times  \bK^\al_J$ 
is given by
$$
\id_{A_I\cap A_J}\times \Hat\Phi^\al_{IJ}: \; (A_I\cap A_J)\times  \E^\al_{I}\big|_{U_{IJ}} \;\to\; A_J\times  \E^\al_J.
$$
In particular, the embedding of domains is of the form
$\id_{A_I\cap A_J}\times \phi^\al_{IJ} : (A_I\cap A_J)\times U_{IJ} \to A_J\times  U_J$.
\item
Let $\bK _I,\bK _J$ be topological Kuranishi charts on  $(Y, \io_Y^0, \io_Y^1)$ such that only $\bK _I$ or both
$\bK _I,\bK _J$ have collared boundary.
Then a {\bf 
topological 
coordinate change with collared boundary} $\Hat\Phi_{IJ} :\bK _I\to\bK _J$ 
with domain $U_{IJ}$ satisfies the conditions in Definition~\ref{def:tchange}, with the following boundary variations and collar form requirement:
\begin{enumerate}
\item
The domain is a 
collared subset $U _{IJ}\subset U _I$
in the sense of Definition~\ref{def:collarset},
so has ``boundary components''
$\partial^\al U _{IJ}:= U _{IJ} \cap \partial^\al U _I$.
\item
If $F_J\cap \p^\al Y \ne \emptyset$ then $F_I\cap \p^\al Y \ne \emptyset$ 
and there is 
a  topological coordinate change $\partial^\al\Hat\Phi_{IJ} : \partial^\al\bK _I \to \partial^\al\bK _J$
with domain $\partial^\al U _{IJ}$ such that the restriction of $\Hat\Phi_{IJ}$ 
to $U_{IJ} \cap \io_{U_I}^\al(A^\al_\eps\times \partial^\al U _I)$
pulls back via the collar inclusions $\io^\al_{\E_I}, \io^\al_{\E_J}$ 
to the product $ {\rm id}_{A^\al_\eps} \times \partial^\al\Hat\Phi_{IJ}$
for some $\eps>0$.
In particular we have 
\begin{align*}
(\iota_{U_I}^\al)^{-1}(U _{IJ})
\cap \bigl(A^\al_\eps\times \partial^\al U _I \bigr)
&\;=\; A^\al_\eps\times \partial^\al U _{IJ}, \\
(\iota_{U_J}^\al)^{-1}(\im\phi _{IJ})
\cap \bigl(A^\al_\eps  \times \partial^\al U _J \bigr)
&\;=\;
 \phi _{IJ}( A^\al_\eps\times \partial^\al U _{IJ}) .
\end{align*}
\item
If $F_J\cap \p^\al Y= \emptyset$ but $F_I\cap \p^\al Y\ne \emptyset$ 
then 
we have $\p^\al U_{IJ}=\emptyset$ and hence also
$U _{IJ}\cap \iota_{U_I}^\al \bigl(A^\al_\eps\times \partial^\al U _I \bigr) = \emptyset$ for some $\eps>0$.
\end{enumerate}
\item
For any topological coordinate change with collared boundary $\Hat\Phi_{IJ} $ on $(Y, \io_Y^0, \io_Y^1)$  
we call the uniquely determined topological coordinate changes $\p^\al \Hat\Phi_{IJ}$ for  $\p^\al Y$ the {\bf restrictions of $\Hat\Phi_{IJ} $ to the boundary} for $\al=0,1$.
\end{itemlist}
\end{defn}

\begin{defn}\label{def:CKS}
A {\bf (weak) topological Kuranishi cobordism} on a compact collared cobordism  $(Y, \io_Y^0, \io_Y^1)$
is a tuple
$$
\Kk  = \bigl( \bK_{I} , \Hat\Phi_{IJ} \bigr)_{I,J\in \Ii_{\Kk}}
$$
of basic charts and transition data as in Definition~\ref{def:Ku} resp.\ \ref{def:Kwk}, with the following boundary variations and collar form requirements:
\begin{itemlist}
\item
The charts of $\Kk$  are either topological Kuranishi charts with collared boundary or standard topological Kuranishi charts whose footprints are precompactly contained in $Y\less (\p^0 Y\cup \p^1 Y)$.
\item
The coordinate changes  
$\Hat\Phi_{IJ}: \bK_{I} \to \bK_{J}$ 
are either standard coordinate changes on $Y\less (\p^0 Y\cup \p^1 Y)$  
between pairs of standard charts, or coordinate changes with collared boundary between 
pairs of charts, of which at least the first has collared boundary.
\end{itemlist}

A {\bf (weak) topological Kuranishi concordance} is a (weak) topological Kuranishi cobordism on a collared cobordism of product type $Y = [0,1]\times X$ with canonical collars as in Example~\ref{ex:natcol}.
\end{defn}

\begin{rmk}\label{rmk:restrict}\rm  
Let $(Y, \io_Y^0, \io_Y^1)$ be a compact collared cobordism.
Then any (weak) topological Kuranishi cobordism $\Kk$ on $(Y, \io_Y^0, \io_Y^1)$ induces by restriction (weak) topological Kuranishi atlases $\partial^\al\Kk$ on 
each boundary component $\p^\al Y$
with
\begin{itemlist}
\item 
basic charts $\p^\al\bK_i$ given by restriction of  basic charts of $\Kk$ with $F_i\cap \p^\al Y\neq\emptyset$;
\item 
index set $\Ii_{\p^\al\Kk}=\{I\in\Ii_{\Kk}\,|\, F_I\cap  \p^\al Y\neq\emptyset\}$;
\item 
transition charts $\p^\al\bK_I$ given by restriction of transition charts of $\Kk$;
\item
coordinate changes $\p^\al\Hat\Phi_{IJ}$ given by restriction of coordinate changes of $\Kk$.
\end{itemlist}
In this case we say that {\bf $\Kk$ is a cobordism from $\p^0\Kk$ to $\p^1\Kk$}
and call $\p^\al\Kk$ the {\bf restrictions of $\Kk$ to the boundary}.
In the special case when $\Kk$ is in fact a topological Kuranishi concordance, we also say that 
{\bf $\Kk$ is a concordance from $\p^0\Kk$ to $\p^1\Kk$}.
$\hfill\er$
\end{rmk}

With this language in hand, we can now introduce the cobordism and concordance relations between topological Kuranishi atlases.

\begin{defn}\label{def:Kcobord}
Let $\Kk^0, \Kk^1$ be (weak) topological Kuranishi atlases on compact metrizable spaces $X^0,X^1$.
\begin{itemlist}
\item 
$\Kk^0$ is {\bf (weakly) cobordant} to $\Kk^1$ if there exists a (weak) topological Kuranishi cobordism $\Kk$ from $\Kk^0$ to $\Kk^1$. Equivalently, $\Kk$ is an atlas on a compact collared cobordism $(Y, \io_Y^0, \io_Y^1)$ with boundary components $\p^\al Y = X^\al$ and boundary restrictions $\p^\al\Kk=\Kk^\al$ for $\al = 0,1$.
To be more precise, there are injections $\iota^\al:\Ii_{\Kk^\al} \hookrightarrow \Ii_{\Kk}$ for $\al=0,1$ such that $\im\iota^\al=\Ii_{\partial^\al\Kk}$ and for all $I,J\in\Ii_{\Kk^\al}$ we have
$$
\bK^\al_I = \p^\al \bK_{\iota^\al(I)}, \qquad
\Hat\Phi^\al_{IJ} = \p^\al \Hat\Phi_{\iota^\al(I) \iota^\al (J)} .
$$
\item
$\Kk^0$ is {\bf (weakly) concordant} to $\Kk^1$ if there exists a (weak) topological Kuranishi concordance $\Kk$ from $\Kk^0$ to $\Kk^1$.
Equivalently, the spaces $X^0=X^1=X$ are identical and $\Kk$ is a (weak) topological Kuranishi cobordism on $Y = [0,1]\times X$ with boundary restrictions $\p^\al\Kk=\Kk^\al$ for $\al = 0,1$ as above.
\end{itemlist}
\end{defn}

In the following we will usually identify the index sets $\Ii_{\Kk^\al}$ of cobordant Kuranishi atlases with the restricted index set $\Ii_{\partial^\al\Kk}$ in the cobordism index set $\Ii_{\Kk}$, so that $\Ii_{\Kk^0}, \Ii_{\Kk^1}\subset \Ii_{\Kk}$ are the (not necessarily disjoint) subsets of charts whose footprints intersect $\p^0 Y$ resp.\ $\p^1 Y$.

\begin{example} \label{ex:triv}\rm
Let $\Kk= \bigl( \bK_I, \Hat\Phi_{IJ}\bigr)_{I,J\in\Ii_\Kk}$ be a weak topological Kuranishi atlas on $X$.
Then the {\bf product Kuranishi concordance} $[0,1]\times \Kk$ from $\Kk$ to $\Kk$ is the weak 
topological Kuranishi  cobordism on $ [0,1]\times X$ consisting of the product charts $[0,1]\times \bK_I$ and the product coordinate changes $\id_{[0,1]}\times \Hat\Phi_{IJ}$ for $I,J\in\Ii_\Kk$.

Note that in this case all index sets are the same, 
$\Ii_{\Kk^0}=\Ii_{\Kk^1}=\Ii_{[0,1]\times \Kk}= \Ii_{\Kk}$.
$\hfill\er$
\end{example}

In the following we will extend the categorical and topological notions from sections \ref{ss:Ksdef} and \ref{ss:tame} to topological Kuranishi cobordisms before proving the uniqueness statements claimed there.
For that purpose it will often be convenient to work with the following uniform collar width.

\begin{remark} \rm \label{rmk:Ceps}   Let $\Kk$ be a (weak) topological Kuranishi cobordism.
Since the index set $\Ii_{\Kk}$ in Definition~\ref{def:CKS} is finite, there exists a uniform 
{\bf collar width} $\eps>0$ such that all collar embeddings  $\io^\al_Y$, $\io^\al_U$,
$\io^\al_\E$ are defined on a neighbourhood of
$\ov{A_{\eps}^\al}$, all coordinate changes 
between charts with nonempty boundary
are of collar form on $B^\al=A_{\eps}^\al$, and all charts without 
boundary have footprint contained in 
$Y\less \bigcup_{\al=0,1} \io_Y^\al(\ov{ A_{\eps}^\al}\times \p^\al Y)$.
In particular,  the footprints of the charts with nonempty boundary cover a neighbourhood of 
$\bigcup_{\al=0,1} \io_Y^\al(\ov{ A_{\eps}^\al}\times \p^\al Y) \subset Y$.
$\hfill\er$
\end{remark}

\begin{rmk}\rm  \label{rmk:cobordreal}
Let $\Kk$ be a topological Kuranishi cobordism.
Its associated categories $\bB_{\Kk}, \bE_{\Kk}$ with projection, section, and footprint functor, as well as their realizations $|\bB_{\Kk}|, |\bE_{\Kk}|$  are defined as for topological Kuranishi atlases without boundary in Section~\ref{ss:Ksdef}, and form cobordisms in the following sense.

\begin{itemlist}
\item
We can think of the virtual neighbourhood $|\Kk|$ of $Y$ as a collared cobordism with boundary components 
$\p^0|\Kk|\cong|\p^0\Kk|$ and $\p^1|\Kk|\cong|\p^1\Kk|$ 
in sense of Definition~\ref{def:Ycob}, with the exception that $|\Kk|$ is usually not locally compact or metrizable.
More precisely, using Remark~\ref{rmk:Ceps} we have collared neighbourhoods for some $\eps>0$,
$$
\iota_{|\Kk|}^0: [0,\eps) \times  |\p^0\Kk|  \hookrightarrow |\Kk|,  
\qquad
\iota_{|\Kk|}^1: (1-\eps,1]\times  |\p^1\Kk|   \hookrightarrow |\Kk| .
$$
These are induced by the natural functors $\iota^\al_{\bB_\Kk} : A^\al_\eps\times \bB_{\p^\al\Kk} \to \bB_\Kk$
given by the inclusions 
$\iota^\al_{U_I} : A^\al_\eps\times U^\al_I \hookrightarrow U_I$ 
on objects and 
$\iota^\al_{U_{IJ}} : A^\al_\eps\times U^\al_{IJ}\hookrightarrow U_{IJ}$
on morphisms, where $A^\al_\eps$ is defined in \eqref{eq:Naleps}. The axioms on the interaction of the coordinate changes with the collar neighbourhoods imply that the functors map to full subcategories that split $\bB_\Kk$ 
in the sense that there are no morphisms between any other object and this subcategory.
Hence the functors $\iota^\al_{\bB_\Kk}$ descend to topological embeddings $\iota_{|\Kk|}^\al : \bigl| A^\al_\eps \times \bB_{\p^\al\Kk}\bigr| \to |\Kk|$, i.e.\ homeomorphisms onto open subsets of $|\Kk|$. 
Here the product topology on 
$A^\al_\eps \times\bigl| \bB_{\p^\al\Kk}\bigr| \cong \bigl| A^\al_\eps \times \bB_{\p^\al\Kk}\bigr|$ coincides with the quotient topology by 
\cite[Ex.~29.11]{Mun}: Applied to the topological space $X:=\bigsqcup_{I\in\Ii_\Kk} U_I$ with equivalence relation $Y\subset X\times X$ induced by the morphisms, and the locally compact Hausdorff space $T:=A^\al_\eps$, it asserts that the product topology on $T\times \qu{X}{Y}$ coincides with the quotient topology induced by the product relation $T\times Y \hookrightarrow (T\times X) \times (T\times X)$. 

To check that $\iota_{|\Kk|}^\al$ are collared neighbourhoods in the sense of Definition~\ref{def:Ycob}, note that $\io^\al_{|\Kk|}\bigl(\{\al\}\times |\p^\al\Kk|\bigr)$ is contained in the open image of $\iota_{|\Kk|}^\al$.
Moreover, to see that $\iota_{|\Kk|}^\al\bigl(\{\al\} \times  |\p^\al\Kk| \bigr)\subset |\Kk|$ is closed we verify that its complement has open preimage in $\bigsqcup_{I\in\Ii_\Kk} U_I$ by noting that each $\io^\al_{U^\al_I}\bigl(\{\al\}\times U^\al_I\bigr) \subset U_I$ is closed. 
\smallskip
 
\item
The ``obstruction bundle'' consists of an analogous collared cobordism $|\bE_{\Kk}|$ with boundary components $\p^\al |\bE_{\Kk}|\cong |\bE_{\p^\al\Kk}|$ and a projection $|\pr_{\Kk}|: |\bE_{\Kk}|\to |\Kk|$  that has product form on the collared boundary $|\pr_{\Kk}| \circ \iota^\al_{|\bE_\Kk|}= \iota^\al_{|\Kk|}\circ \bigl(\id_{A^\al_\eps} \times  |\pr_{\p^\al\Kk}|\bigr)$ induced by the ``obstruction bundles'' of the boundary components, $|\pr_{\p^\al\Kk}|: |\bE_{\p^\al\Kk}|\to |\p^\al\Kk|$.
\smallskip

\item
The embeddings $\io^\al_{|\Kk|}$ extend the natural map between footprints
$$
|\s_{\Kk}|^{-1}(|0_\Kk|)
\;\;
\xleftarrow{\io_\Kk}  \;\;Y \;\; \xleftarrow{\io_Y^\al}    \;\;
A_\eps^\al \times \p^\al Y
\;\;\xrightarrow{{\rm id}\times \iota_{\p^\al\Kk}}\;\;
 A_\eps^\al\times |\s_{\p^\al\Kk}|^{-1}(|0_{\p^\al\Kk}|)  .
$$

\item 
If $\Kk$ is a topological Kuranishi concordance on $Y = [0,1]\times X$ then the footprint functor to $[0,1]\times X$  induces a continuous surjection
$$
{\rm pr}_{[0,1]}\circ \psi_{\Kk}
: \; \s_{\Kk}^{-1}(0_\Kk) \;\to\; [0,1]\times X  \;\to\;  [0,1] .
$$
In general we do not assume that this extends to a functor $\bB_{\Kk}\to [0,1]$.
However, all the topological Kuranishi concordances that we construct explicitly do have this property. 
 $\hfill\er$
\end{itemlist}
\end{rmk}

In order to extend e.g.\ the Hausdorff property of the realization $|\bB_{\Kk}|$ from Theorem~\ref{thm:K} to Kuranishi cobordisms, we need the notions of filtration and tameness.
In the following we introduce them for general topological Kuranishi cobordisms with the 
understanding that they can equally be applied to the special case of topological Kuranishi concordances.

\begin{defn}\label{def:Cccf} 
A (weak) topological Kuranishi cobordism $\Kk$  is said to be {\bf filtered} if it is equipped with a filtration $(\E_{IJ})_{I\subset J\in\Ii_\Kk}$ as in Definition~\ref{def:Ku3} that is compatible with the boundary collars in the sense that 
for some $\eps>0$ 
the following holds for  $\al=0,1$:
\begin{itemlist} 
\item
The restrictions 
$\bigl(\p^\al\E_{IJ} = \E_{IJ}\cap \p^\al \E_J\bigr)$
define a filtration of the boundary atlas $\p^\al \Kk$.
\item For each $I\subset J$ we have 
$\;
\io_{\E_J}^\al(A_\eps^\al\times \p^\al\E_{IJ})  = \E_{IJ}\cap \pr^{-1}_J \bigl(\io_{U_J}^\al\bigl( A_\eps^\al\times \p^\al U_J\bigr)\bigr)$.
\end{itemlist}
A filtered cobordism is said to be {\bf tame} if it also satisfies the conditions in Definition~\ref{def:tame}.
\end{defn}

\begin{rmk} \label{rmk:restrict2}\rm  
\NI
(i)
If $\Kk$ is a (weak) topological Kuranishi atlas with filtration $(\E_{IJ})_{I\subset J\in\Ii_\Kk}$, then the product concordance $[0,1]\times \Kk$ is equipped with the canonical filtration $([0,1]\times\E_{IJ})_{I,J\in\Ii_\Kk}$.

\MS
\NI
(ii)
If $\Kk$ is a tame topological Kuranishi atlas, then $[0,1]\times \Kk$ with the canonical filtration is a tame topological Kuranishi concordance.

\MS
\NI (iii) 
If $\Kk$ is a tame topological Kuranishi cobordism, then both restrictions $\p^\al \Kk$ are also tame with the induced filtrations $(\p^\al\E_{IJ})_{I,J\in \Ii_{\p^\al\Kk}}$. $\hfill\er$  
\end{rmk}

\begin{defn} 
Let $\Kk^0, \Kk^1$ be filtered (weak) topological Kuranishi atlases with filtrations $(\E^\al_{IJ})_{I,J\in \Ii_{\Kk^\al}}$.
\begin{itemlist}
\item 
$\Kk^0$ is {\bf filtered (weak) cobordant/concordant} to $\Kk^1$ if there exists a filtered (weak) topological Kuranishi cobordism/concordance $\Kk$ from $\Kk^0$ to $\Kk^1$ such that the induced filtrations $(\p^\al\E_{IJ})_{I,J\in \Ii_{\p^\al\Kk}}$ coincide with the given $(\E^\al_{IJ})_{I,J\in \Ii_{\Kk^\al}}$.
\item
If $\Kk^0, \Kk^1$ are in addition tame, then 
$\Kk^0$ is {\bf tame cobordant/concordant} to $\Kk^1$ if there exists a filtered topological Kuranishi cobordism/concordance $\Kk$ from $\Kk^0$ to $\Kk^1$ as above that in addition is tame.
\end{itemlist}
\end{defn}

While filtered versions of topological Kuranishi cobordisms are used to define suitable equivalence relations, the tameness notion is needed because of its topological implications.

\begin{lemma}\label{le:cob0}
Let $\Kk$ be a tame topological Kuranishi cobordism
(or concordance).
Then its realization $|\Kk|$ has the Hausdorff and homeomorphism properties stated in Theorem~\ref{thm:K}.
\end{lemma}
\begin{proof}
These properties are proven by precisely the same arguments as in Proposition~\ref{prop:Khomeo}.
The fact that some charts have collared boundaries is irrelevant in this context.
\end{proof}

Finally, all reasonable flavours of cobordism and concordance form equivalence relations.
We end this section by proving this fact in the cases that will be used in applications.

\begin{lemma}\label{lem:cobord1}
Filtered (weak) cobordism is an equivalence relation between (weak) filtered topological Kuranishi atlases.
\end{lemma}

\begin{rmk} \rm 
As we saw, the existence of a filtration on a topological Kuranishi atlas resp.\ cobordism $\Kk$ 
is an essential hypothesis in Proposition~\ref{prop:Khomeo} and Lemma~\ref{le:cob0}.
However, when manipulating charts other than by shrinking the domains, one may easily destroy this property.  For example, when constructing concordances over $[0,1]\times X$ it is natural to use product charts of the form $A\times \bK_I$.  However one must guard against taking two products of the same basic chart with intersecting footprints, e.g.\ $[0,\frac 13)\times \bK_i $ and $(\frac 14,\frac 12)\times \bK_i$, since their obstruction bundles have the same fibers.
The natural transition chart for the overlap of footprints $ (\frac 14,\frac 13)\times F_i$ is
$ (\frac 14,\frac 13)\times \bK_i$ which again has an obstruction bundle with the same fiber and so, as in Example~\ref{ex:Ku3}, in general fails to be filtered. 
This means that in the following proof we have to construct cobordisms with great care. 
$\hfill\er$
\end{rmk}

\begin{proof}[Proof of Lemma~\ref{lem:cobord1}]
Filtered (weak) cobordism is reflexive by Remark~\ref{rmk:restrict2}~(i).

To check symmetry of the filtered cobordism relation,
suppose that $(Y, \io_Y^0, \io_Y^1)$ is a compact collared cobordism as in Definition~\ref{def:Ycob}, and that $\Kk$ is a filtered weak topological Kuranishi cobordism on $(Y, \io_Y^0, \io_Y^1)$
 from $\Kk^0$ to $\Kk^1$ as in Definition~\ref{def:Kcobord}.  
Then we may interchange the two boundary components of $Y$ to get a compact collared cobordism $(Y'
:=Y, \io_{Y'}^0, \io_{Y'}^1)$
from $\p^0Y':=\p^1Y$ to $\p^1Y':=\p^0Y$ 
with collars $\io_{Y'}^\al$ defined as composites with the reflection $\tau: [0,1]\to[0,1], t\mapsto 1-t$:
\begin{align*}
 \io_{Y'}^0:[0,\eps)\times \p^0Y' \;\;\stackrel{\tau\times \id} \to\;\; (1-\eps,1]\times \p^1Y \;\; \stackrel{\io_Y^1}\to \;\;Y = Y',\\
\io_{Y'}^1: (1-\eps,1]\times \p^1Y'\;\;  \stackrel{\tau\times \id} \to\;\;  [0,\eps)\times \p^0Y\;\;  \stackrel{\io_Y^0}\to \;\;Y = Y'.
 \end{align*}
If we similarly relabel and reparametrize the boundary collars of each chart in $\Kk$, we obtain a filtered cobordism $\Kk'$ over $Y'$ from $\Kk^1$ to $\Kk^0$,
which proves symmetry. 
 
The nontrivial point is transitivity. For that purpose consider two (weak) topological Kuranishi cobordisms,\footnote{
The prototypical example -- and source of our notation -- are two Kuranishi concordances which we would in a preliminary step shift to obtain cobordisms over $[0,1]\times X$ and $[1,2]\times X$. 
}
the first $\Kk^{[0,1]}$ on $(Y^{[0,1]}, \io_{Y^{[0,1]}}^0, \io_{Y^{[0,1]}}^1)$ from an atlas $\Kk^0$ on $Y^0= \p^0 Y^{[0,1]}$  to an atlas $\Kk^1$ on $Y^1= \p^1 Y^{[0,1]}$, and the second $\Kk^{[1,2]}$ on $(Y^{[1,2]}, \io_{Y^{[1,2]}}^0, \io_{Y^{[1,2]}}^1)$ from the atlas $\Kk^1$ on $Y^1 = \p^0 Y^{[1,2]}$ to an atlas $\Kk^2$ on $Y^2 = \p^1Y^{[1,2]}$.   
Our goal is to concatenate these to a (weak) topological Kuranishi cobordism
$\Kk^{[0,2]}$ on $(Y^{[0,2]}, \io_{Y^{[0,2]}}^0, \io_{Y^{[0,2]}}^1)$, which restricts to $\Kk^0$ on $Y^0 = \p^0 Y^{[0,2]}$ and to $\Kk^2$ on $Y^2 = \p^1 Y^{[0,2]}$.
As underlying space $Y^{[0,2]}$ we choose the boundary connected sum
$$
 Y^{[0,2]} = Y^{[0,1]} \underset{\scriptscriptstyle Y^1}{\cup}\, Y^{[1,2]} \; : = \;\quotient{ Y^{[0,1]}{\sqcup} Y^{[1,2]}}
{ \scriptstyle
 \iota^1_{Y^{[0,1]}}(1,y) \,\sim\, \iota^0_{Y^{[1,2]}}(1,y)\; \forall y\in Y^1}
$$
with boundary collars
$\iota^0_{Y^{[0,2]}}: = j_{[0,1]}\circ \iota^0_{Y^{[0,1]}}$, $\iota^1_{Y^{[0,2]}}: = j_{[1,2]}\circ \iota^1_{Y^{[1,2]}}$, 
where $j_{[i,i+1]}$ denotes the inclusion $Y^{[i,i+1]}\subset Y^{[0,2]}$.
We obtain an atlas $\Kk^{[0,2]}$ on $Y^{[0,2]}$ as follows:
\begin{itemlist}
\item
The index set $\displaystyle \; \Ii_{\Kk^{[0,2]}} :=\;\Ii_{[0,1)}  \;\sqcup\; \Ii_{\Kk^1} \;\sqcup\; \Ii_{(1,2]} \;$ is given by 
$$
\Ii_{[0,1)}:=\Ii_{\Kk^{[0,1]}}\less\iota_{\Kk^{[0,1]}}^1(\Ii_{\Kk^1}), \quad \Ii_{(1,2]}:=\Ii_{\Kk^{[1,2]}}\less\iota_{\Kk^{[1,2]}}^0(\Ii_{\Kk^1}).
$$
This partitions $\Ii_{\Kk^{[0,2]}}$ into those index sets in $\Ii_{\Kk^1}$ whose footprints intersect $\p^1 Y^{[0,1]}=Y^1=\p^0 Y^{[1,2]}$ and the index sets in $\Ii_{[0,1)}$ resp.\ $\Ii_{(1,2]}$ whose footprints are contained in $Y^{[0,1]}\less\io^1_{Y^{[0,1]}}((1-\eps,1]\times Y^1)\subset Y^{[0,2]}$, resp.\ 
$Y^{[1,2]}\less\io^0_{Y^{[1,2]}}([0,\eps)\times Y^1)\subset Y^{[0,2]}$.

\item
The charts are
$\bK^{[0,2]}_{I} := \bK^{[0,1]}_{I}$ for $I\in\Ii_{[0,1)}$, and $\bK^{[0,2]}_{I} := \bK^{[1,2]}_{I}$ for $I\in\Ii_{(1,2]}$.
For $I\in\Ii_{\Kk^1}$ denote by
$I^{01}=\iota_{\Kk^{[0,1]}}^1(I)\in\Ii_{\Kk^{[0,1]}}\less\Ii_{[0,1)}$, $I^{12}=\iota_{\Kk^{[1,2]}}^0(I)\in\Ii_{\Kk^{[1,2]}}\less\Ii_{(1,2]}$ the labels of the charts that restrict to $\bK_I$.
In particular this implies ${\partial^1 U^{[0,1]}_{I^{01}} = U^1_I =\partial^0 U^{[1,2]}_{I^{12}}}$.
Then define the glued chart (possibly with collared boundary at $Y^0$ or $Y^2$)
$$
\qquad
\bK^{[0,2]}_{I} :=
 \bK^{[0,1]}_{I^{01}}\underset{\scriptscriptstyle U^1_I}{\cup} \bK^{[1,2]}_{I^{12}} :=
\left( U^{[0,2]}_I \,, \ \E^{[0,2]}_I,\ 
\left\{ \begin{aligned}
\s^{[0,1]}_{I^{01}}\;,\; \psi^{[0,1]}_{I^{01}}  \quad &\text{on} \; U^{[0,1]}_{I^{01}}\\
 \s^{[1,2]}_{I^{12}}\;,\; \psi^{[1,2]}_{I^{12}}  \quad &\text{on} \;U^{[1,2]}_{I^{12}}
\end{aligned} \right\}
\right)
$$
where the domain and obstruction bundle are boundary connected sums
$$
U^{[0,2]}_{I} \;:=\; U^{[0,1]}_{I^{01}} \underset{\scriptscriptstyle U^1_I}{\cup}  U^{[1,2]}_{I^{12}}
\;:=\;
\quotient{ U^{[0,1]}_{I^{01}} \sqcup U^{[1,2]}_{I^{12}} }{ 
\iota^1_{U_{I^{01}}}(x) \sim \iota^0_{U_{I^{12}}}(x)\quad \forall x\in U^1_I
} 
$$
and similarly for $\E^{[0,2]}_I$. Further the projection $\pr_I: \E^{[0,2]}_I\to U^{[0,2]}_I$ and zero section 
$0_I: U^{[0,2]}_I\to \E^{[0,2]}_I$
are the obvious concatenation.
Note that due to the collar requirements, both domain and obstruction bundle contain an embedded
product $(1-\eps,1+\eps)\times  U^1_I$, $(1-\eps,1+\eps)\times  \E^1_I$ for some $\eps>0$.
Moreover the sections, projections, and footprint maps fit continuously since 
their pullbacks to this 
$(1-\eps,1+\eps)$-collar have product form and agree at the middle $1$-slice.
\vspace{.09in}
\item
The coordinate changes are $\Hat\Phi^{[0,2]}_{IJ} := \Hat\Phi^{[0,1]}_{IJ}$ for $I,J\in\Ii_{[0,1)}$, and $\Hat\Phi^{[0,2]}_{IJ} := \bK^{[1,2]}_{IJ}$ for $I,J\in\Ii_{(1,2]}$, and the following.
\vspace{.09in}
\begin{itemize}
\item[-]
For $I,J\in\Ii_{\Kk^1}$ the
coordinate charts corresponding to
$I^{01}, J^{01}\in\Ii_{\Kk^{[0,1]}}$, $I^{12},J^{12}\in\Ii_{\Kk^{[1,2]}}$
fit together to give a glued coordinate change
(possibly with collared boundary)
$$
\qquad
\Hat\Phi^{[0,2]}_{IJ} := \left( U^{[0,1]}_{I^{01}J^{01}} \underset{\scriptstyle U^1_{IJ}}\cup U^{[1,2]}_{I^{12}J^{12}} \; , \,
\left\{ \begin{aligned}
\Hat\Phi^{[0,1]}_{I^{01}}  \;\quad &\text{on} \; \E^{[0,1]}_{I^{01}J^{01}}\\
\Hat\Phi^{[1,2]}_{I^{12}}  \;\quad &\text{on} \; \E^{[1,2]}_{I^{12}J^{12}}
\end{aligned}
\right\}
\right) .
$$
Here the embeddings of obstruction bundles fit continuously since as before their 
restrictions to the product 
$(1-\eps,1+\eps) \times \E^1_I\big|_{U^1_{IJ}} \subset \E^{[0,2]}\big|_{U^{[0,2]}_{IJ}}$ 
have product form and agree on 
$\{1\}\times\E^1_I\big|_{U^1_{IJ}}$.
A similar remark applies to the domains.
\vspace{.09in}
\item[{-}]
For $J\in\Ii_{[0,1)}$ and $I\in\Ii_{\Kk^1}$ corresponding to $I^{01}\in\Ii_{\Kk^{[0,1]}}$ with $I^{01}\subsetneq J$ the coordinate change $\Hat\Phi^{[0,2]}_{IJ} := \Hat\Phi^{[0,1]}_{I^{01} J}$
is well defined with domain $U^{[0,1]}_{I^{01}J} \subset U^{[0,2]}_{IJ}$;
similarly for $J\in\Ii_{(1,2]}$,  $I\in\Ii_{\Kk^1}$.\vspace{.09in}
\end{itemize}
\end{itemlist}

Note that we need not construct coordinate changes from $I\in\Ii_{[0,1)}$ (or $I\in\Ii_{(1,2]}$) to $J\in\Ii_{\Kk^1}$
since in these cases $F_J$ is not a subset of  $F_I$.
Now we may define the basic charts in $\Kk^{[0,2]}$ to consist of the basic charts in $\Ii_{[0,1)}$ and $\Ii_{(1,2]}$ whose footprints are disjoint from $\{1\}\times X$, together with one glued chart for each basic chart in $\Ii_{\Kk^1}$ (which is constructed from a pair of charts in ${\Kk^{[0,1]}}$ and ${\Kk^{[1,2]}}$ with matching collared boundaries).
The further charts and coordinate changes constructed above then cover exactly the overlaps of the new basic charts, since the charts from $\Ii_{[0,1)}$ have no overlap with those arising from $\Ii_{(1,2]}$.
The weak cocycle condition  for charts or coordinate changes in $\Ii_{[0,1)}\sqcup\Ii_{(1,2]}$ then follows directly from the corresponding property of $\Kk^{[0,1]}$ and $\Kk^{[1,2]}$.
Furthermore,
for $I\in\Ii_{\Kk^{1}}$ the glued chart $\bK^{[0,2]}_{I}=\bK^{[0,1]}_{I^{01}}\underset{\bK^1_I}{\cup} \bK^{[1,2]}_{I^{12}}$ has restrictions (up to natural pullbacks)
\begin{align*}
&\bK^{[0,2]}_{I}\big|_{{\rm int}(U^{[0,1]}_{I^{01}})} = \bK^{[0,1]}_{I^{01}}\big|_{{\rm int}(U^{[0,1]}_{I^{01}})}, \qquad
\bK^{[0,2]}_{I}\big|_{{\rm int}(U^{[1,2]}_{I^{12}})} = \bK^{[1,2]}_{I^{12}}\big|_{{\rm int}(U^{[1,2]}_{I^{12}})}, \\
&\bK^{[0,2]}_{I}\big|_{(1-\eps,1+\eps)\times U^1_I } =(1-\eps,1+\eps)\times  \bK^1_{I}  .
\end{align*}
The cocycle condition for any tuple of coordinate changes can be checked separately for these restrictions (which cover the entire domain of $\bK^{[0,2]}_{I}$) and hence follows from the corresponding property of $\Kk^{[0,1]}$, $\Kk^{[1,2]}$, and $\Kk^{1}$.
Thus we have constructed a weak topological Kuranishi cobordism $\Kk^{[0,2]}$ from $\Kk^0$ to $\Kk^2$.  

Moreover, $\Kk^{[0,2]}$ satisfies the cocycle condition if both constituent Kuranishi cobordisms do. 
Indeed, the inclusion of domains $(\phi^{[0,2]}_{IJ})^{-1}(U^{[0,2]}_{JK}) \subset U^{[0,2]}_{IK}$
holds immediately if $I\subset J\subset K$ all lie in the same subset  $\Ii_{\Kk^{[0,1]}}, \Ii_{\Kk^1}, \Ii_{\Kk^{[1,2]}}$ of the index set.
The only other cases are $I\in \Ii_{\Kk^1}$ and $K\in \Ii_{[0,1)}$ resp.\ $K\in \Ii_{(1,2]}$ with the intermediate $J$ in one or the other subset.
For w.l.o.g.\ $K\in \Ii_{[0,1)}$ we have $U^{[0,2]}_{\bullet K}=U^{[0,1]}_{\bullet K}$ and 
$(\phi^{[0,2]}_{IJ})^{-1}(U^{[0,2]}_{JK})=(\phi^{[0,1]}_{IJ})^{-1}(U^{[0,1]}_{JK})$
since even in case $J\in \Ii_{\Kk^1}$ the map $\phi^{[0,2]}_{IJ}$ takes $U^{[0,2]}_{IJ}\less U^{[1,2]}_{IJ}$ to $U^{[0,2]}_{J}\less U^{[1,2]}_{J}$ due to its product form on the middle collar.
\smallskip

It remains to construct a filtration on $\Kk^{[0,2]}$ whose boundary restrictions are given filtrations $(\E_{IJ}^{[0,1]})_{I,J\in \Ii_{\Kk^{[0,1]}}}$ and $(\E_{IJ}^{[1,2]})_{I,J\in\Ii_{\Kk^{[1,2]}}}$.
For a chart 
$\bK^{[0,2]}_J$ with $J\in \Ii_{\Kk^{[0,1)}}$ we define $\E_{IJ}^{[0,2]}: = \E_{IJ}^{[0,1]}$ for all $I\subset J$.
Note here that although $J\in \Ii_{[0,1)}$, 
some of its subsets $I$ might lie in $\Ii_{\Kk^1}$.  However $\E_{IJ}^{[0,1]}$ is defined because
$I,J$ are both in the indexing set 
$\Ii_{\Kk^{[0,1]}}$.
Similarly, if  $J\in \Ii_{\Kk^{(1,2]}}$ we take  $\E_{IJ}^{[0,2]}: = \E_{IJ}^{[1,2]}$ for all  $I\subset J$.
Finally,  if $J\in \Ii_{\Kk^1}$ and $I\subset J$ we define 
$\E_{IJ}^{[0,2]}: = \E_{I^{01}J^{01}}^{[0,1]}\cup_{\E_{IJ}^1} \E_{I^{12}J^{12}}^{[1,2]}$
by concatenation. 
Note that this is well defined for all $I\subset J$ because $J\in  \Ii_{\Kk^1}$ implies that $I\in  \Ii_{\Kk^1}$
(since the footprint of $I$ contains that of $J$).
That these sets satisfy the filtration properties (i) and (iii) in Definition~\ref{def:Ku3} follow directly 
from the corresponding statements for the atlases $\Kk^{[0,1]}, \Kk^1, \Kk^{[1,2]}$. 
Conditions (ii) and (iv) concern the compatiblity with coordinate changes.  They are also immediate if they involve indices $I\subset J\subset K$ that all lie in the same subset  $\Ii_{\Kk^{[0,1]}}, \Ii_{\Kk^1}, \Ii_{\Kk^{[1,2]}}$ of the index set. The other cases in which we need to check property (ii),
\begin{equation}\label{eq:ii}
\Hat\Phi^{[0,2]}_{JK}\bigl(
(\pr^{[0,2]}_J)^{-1}(U^{[0,2]}_{JK})\cap
\E^{[0,2]}_{IJ}\bigr) = \E^{[0,2]}_{IK}\cap (\pr^{[0,2]}_K)^{-1}(\im \phi^{[0,2]}_{JK}),
\end{equation}
are as before $I\in \Ii_{\Kk^1}$ and $K\in \Ii_{[0,1)}$ resp.\ $K\in \Ii_{(1,2]}$.
Suppose w.l.o.g.\ $K\in \Ii_{[0,1)}$.  Then $I,J,K\in \Ii_{\Kk^{[0,1]}}$, and we can identify
$(\pr^{[0,2]}_J)^{-1}(U^{[0,2]}_{JK})$ with the subset $(\pr^{[0,1]}_J)^{-1}(U^{[0,1]}_{JK}) \subset \E^{[0,1]}_J$ 
on which $\Hat\Phi^{[0,2]}_{JK}=\Hat\Phi^{[0,1]}_{JK}$, so that \eqref{eq:ii} follows from the filtration property of $\Kk^{[0,1]}$, since we have $\E^{[0,2]}_{IK}=\E^{[0,1]}_{IK}\subset \E^{[0,1]}_K$ and 
$\E^{[0,1]}_K\cap (\pr^{[0,2]}_K)^{-1}(\im \phi^{[0,2]}_{JK}) = (\pr^{[0,1]}_K)^{-1}(\im \phi^{[0,1]}_{JK})$.

It remains to check the filtration property (iv) for $I\subsetneq J$ with w.l.o.g.\ $I\in \Ii_{\Kk^1}$ and $J\in  \Ii_{[0,1)}$. 
We have $U_{J}^{[0,2]}\cong U_{J}^{[0,1]}$ so that $\im \phi_{IJ}^{[0,2]} \cong \im \phi_{IJ}^{[0,1]}$ is an open subset of $(\s^{[0,1]}_J)^{-1}(\E^{[0,1]}_{IJ})$ by the filtration property of $\Kk^{[0,1]}$.
To see that $\im \phi_{IJ}^{[0,2]}$ is also an open subset of $(\s^{[0,2]}_J)^{-1}(\E^{[0,2]}_{IJ})$, observe that the latter is identified with $(\s^{[0,1]}_J)^{-1}(\E^{[0,1]}_{IJ})\subset U^{[0,2]}_J$ since $\E^{[0,2]}_{IJ}=\E^{[0,1]}_{IJ}\subset \E^{[0,1]}_J$ and $(\s^{[0,2]}_J)^{-1}(\E^{[0,1]}_J)\subset \pr^{[0,2]}_J(\E^{[0,1]}_J) =
U^{[0,1]}_J$.
This
shows transitivity of the filtered (weak) cobordism relation and thus
completes the proof.
\smallskip
\end{proof}

\subsection{Construction of metric tame cobordisms} \hspace{1mm}\\ \vspace{-3mm}

The final  task in this section is to construct tame concordances 
between different tame shrinkings in order to establish the uniqueness claimed in Theorem~\ref{thm:K}.  It will also be useful to have suitable metrics on these concordances,  
since they are used in the construction of perturbations.  
We therefore begin by discussing the notion of metric tame topological Kuranishi cobordism.
However, we prove in Proposition~\ref{prop:metcob} that this is an equivalence relation  in which the particular choice of metric is to a large extent irrelevant.
One difficulty here is that we are dealing with an arbitrary distance function, not a length metric such as a Riemannian metric. Hence we must build in resp.\ prove various elementary results that would be evident in the Riemannian case.

\begin{defn}\label{def:mCKS}
A {\bf metric tame topological Kuranishi cobordism} on $Y$ is a tame topological Kuranishi cobordism $\Kk$ equipped with a metric $d$ on $|\Kk|$ that satisfies the admissibility conditions of Definition~\ref{def:metric} and has a metric collar as follows:

There is $\eps>0$ such that for $\al=0,1$ 
the collaring maps $\io^\al_{|\Kk|}: A^\al_\eps\times |\p^\al\Kk|\to |\Kk|$ 
of Remark~\ref{rmk:cobordreal} are defined and pull back $d$  
to the product metric 
\begin{equation} \label{eq:epsprod}
(\io^\al_{|\Kk|})^* d 
\bigl((t,x),(t',x')\bigr) \;=\; |t'-t|  + d^\al(x,x') \qquad \text{on} \;\; A^\al_\eps\times |\p^\al\Kk| ,
\end{equation}
where the metric $d^\al$ on 
$|\p^\al\Kk|$ is given by pullback of the restriction of $d$ 
to $\p^\al |\Kk|  = \io^\al_{|\Kk|}\bigl(\{\al\}\times |\p^\al\Kk|\bigr)$, which we denote by
$$
d^\al \,:=\; d|_{|\p^\al \Kk|} \,:=\; \io^\al_{|\Kk|}(\al,\cdot)^* d.
$$ 
In addition, we require
for all $y\in |\Kk| \less \io^\al_{|\Kk|}  \bigl( A^\al_\eps\times  |\p^\al\Kk|\bigr)$ 
\begin{equation}\label{eq:epscoll}
d\bigl( y , \io^\al_{|\Kk|}(\al + t, x )\bigr) \;\ge\; 
\eps - |t|
\qquad \forall \;
(\al + t,x)\in  A^\al_\eps \times  |\p^\al\Kk| .
\end{equation}
We call a metric on 
$|\Kk|$ {\bf admissible} if it satisfies the conditions of Definition~\ref{def:metric}, 
{\bf $\eps$-collared} if it satisfies \eqref{eq:epsprod} and \eqref{eq:epscoll}, 
and {\bf collared} if it is $\eps$-collared for some $\eps>0$.
\end{defn}

Condition \eqref{eq:epscoll} controls the distance between points $\io^\al_{|\Kk|}(\al + t,x)$ in the collar and 
points $y$ outside of the collar.  
In particular, if $\de<\eps-|t|$, then the $\de$-ball around $\io^\al_{|\Kk|}(\al + t,x)$ is contained in the $\eps$-collar, while the $\de$-ball around 
$y\in |\Kk| \less \io^\al_{|\Kk|}  \bigl( A^\al_\eps\times  |\p^\al\Kk|\bigr)$
does not intersect the $|t|$-collar $\io^\al_{|\Kk|}\bigl(A^\al_{|t|}\times |\p^\al\Kk|\bigr)$.

\begin{example} \label{ex:mtriv}\rm
\NI (i) 
Any admissible metric $d$ on $|\Kk|$ for a topological Kuranishi atlas $\Kk$ induces an admissible 
collared metric $d_\R + d$ on $|[0,1]\times \Kk|\cong [0,1] \times |\Kk| $, given by 
$$
\bigr(d_\R + d \bigl)\bigl( (t,x) , (t',x') \bigr) = |t'-t| + d(x,x')  .
$$ 
For short, we call $d_\R + d$ a {\bf product metric.} 

\NI (ii)
Let $d$ be an admissible collared metric on $|\Kk|$ for a topological Kuranishi cobordism 
$\Kk$, and let $b$ be an upper bound of $d^\al:=d\big|_{|\p^\al\Kk|}$ for $\al=0,1$. Then we claim that for any $\ka>b$ the truncated metric $\min(d,\ka)$ given by $(x,y) \mapsto \min(d(x,y),\ka)$ is an admissible $\eps'$-collared metric for $\eps':=\min(\eps,\ka-b)$. 
Indeed, the metric $\min(d,\ka)$  is  admissible because it induces the same topology on each $U_I$ as $d$.
The product form on the collar is preserved since
\begin{equation}\label{eq:trunc}
|t-t'| < \ka-b \qquad\Longrightarrow\qquad   |t-t'| + d^\al(x,x')  \;\le\;  |t-t'|  + b \;<\; \ka ,
\end{equation}
and \eqref{eq:epscoll} holds since for $(\al + t,x)\in A^\al_{\eps'}\times  |\p^\al\Kk|$ we have 
\[
d\bigl( y , \io^\al_{|\Kk|}(\al + t,x) \bigr) \;\ge\; 
\begin{cases}
\eps - |t| \ge \eps' - |t|  &\quad\text{if} \; y\in |\Kk|\less \io^\al_{|\Kk|}\bigl( A^\al_\eps\times |\p^\al\Kk|\bigr) , \\
|t' - t| \ge \eps' - |t|  &\quad\text{if} \; y=\io^\al_{|\Kk|}(\al+t',x')\in \io^\al_{|\Kk|}\bigl(( A^\al_\eps\less A^\al_{\eps'} )\times  |\p^\al\Kk|\bigr) 
\end{cases}
\]
by \eqref{eq:epscoll} for $d$ resp.\ the product form of the metric on the $\eps$-collar, and moreover $\ka > \ka - b - t \geq \eps' - t$.
Finally, the restrictions of this truncated metric are by $\ka>b$
$$
\min(d,\ka)\big|_{|\p^\al\Kk|} \;=\; \min(d^\al,\ka) \;=\; d^\al \qquad\text{for}\; \al=0,1 .
$$

\NI (iii)  
Let $(\Kk,d)$ be a metric tame topological Kuranishi cobordism on $Y$
with collar width $\eps>0$. Then for any $0<\de<\eps$ the $\de$-neighbourhood of the inclusion of $Y$, 
\begin{equation}\label{eq:metcoll2}
\Ww_\de: = B_\de\bigl(\io_\Kk(Y)\bigr) 
= \bigl\{y\in |\Kk| \, \big| \, 
\exists z\in  \io_\Kk(Y)
 : d(y, z)<\eps\bigr\},
\end{equation}
is collared 
with collars of width $\eps-\de$, that is
$$
\Ww_\de \cap \io^\al_{|\Kk|}\bigl(A^\al_{\eps-\de}\times |\p^\al\Kk|\bigr) 
= \io^\al_{|\Kk|}\bigl(A^\al_{\eps-\de}\times B_\de(\io_{\p^\al\Kk}(\p^\al Y)) )
 \bigr).
$$
This holds because  \eqref{eq:epscoll} implies that $B_\de\bigl(\io_\Kk\bigl(Y\less \io_Y^\al(A^\al_\eps\times \p^\al Y )
 \bigr)\bigr)$ does not intersect the collar $\io^\al_{|\Kk|}\bigl(A^\al_{\eps-\de}\times \p^\al\Kk\bigr)$, and on the other hand
$B_\de\bigl(\io_\Kk\bigl(\io_Y^\al(A^\al_\eps\times \p^\al Y) 
\bigr)\bigr)\cap \io^\al_{|\Kk|}\bigl(A^\al_\eps\times |\p^\al\Kk|
\bigr)$ has product form because the metric in the collar is a product metric.
$\hfill\er$
\end{example}

The following result establishes the existence of collared metrics on Kuranishi cobordisms with given boundary restrictions. This is useful for proving invariance of homological information in applications (when e.g.\ homotopies of almost complex structures give rise to Kuranishi cobordisms) and is also used abstractly below in proving the uniqueness of metrics in Theorem~\ref{thm:K}.

\begin{prop} \label{prop:metcob}
Metric tame Kuranishi cobordism is an equivalence relation. Moreover, admissible metrics on tame Kuranishi atlases are unique in the following sense: 
\begin{enumerate}
\item
Given a tame topological Kuranishi atlas $\Kk$ and two admissible metrics $d^0, d^1$ on $|\Kk|$, there exists an admissible collared metric $D$ on $|[0,1]\times  \Kk|$ with restrictions $D|_{|\{\al\}\times \Kk|}=d^\al$ at the boundaries $\p^\al |[0,1]\times  \Kk|=\{\al\}\times |\Kk|\cong|\Kk|$ for $\al=0,1$.
\item
Suppose that $\Kk$ is a metrizable tame topological Kuranishi cobordism and $d^\al$ are admissible metrics on $|\p^\al\Kk|$ for $\al=0,1$.
Then there exists an admissible collared metric $D$ on $|\Kk|$ with $D|_{|\p^\al \Kk|}=d^\al$ for $\al=0,1$.
\end{enumerate}
\end{prop}

The proof of this proposition will be based on the following techniques for interpolating, collaring, and concatenating admissible metrics. 
Here, as in Remark~\ref{rmk:cobordreal} and Lemma~\ref{lem:cobord1}, we will consider product atlases $A\times \Kk$ for various intervals $A\subset \R$, that is with domain category $A\times \Obj_{\bB_\Kk}$. Their virtual neighbourhoods $|A\times \Kk|$ are canonically identified with $A\times |\Kk|$ 
as topological spaces by Remark~\ref{rmk:cobordreal}.
We will write $B\times |\Kk|$ for the quotients $|B\times \Kk|\subset |A\times \Kk|$ for any $B\subset A$.

\begin{lemma} \label{le:metcoll}
Let $\Kk$ be a metrizable tame topological Kuranishi atlas and
suppose that $d$ and $d'$ are admissible metrics on $|\Kk|$, where $d'$ is bounded by~$1$.
Then for any $0<\eps<\frac 14$ there is an admissible $\eps$-collared metric $D$ on  
 $|[0,1]\times  \Kk|$ 
that 
restricts to  $d_\R +  d$ on $ [0,\eps]\times  |\Kk|$ and to $d_\R + d+d' $ on $[1-\eps,1]\times  |\Kk|$.
\end{lemma}
\begin{proof}
Choose a smooth nondecreasing function $\be: [0,1]\to [0,1]$  
such that $\be|_{[0,\eps]}=0$, $\be|_{[1-\eps,1]}=1$, and the derivative is bounded by $\sup \be'<2$. 
(At this point we need to know that $(1-\eps)-\eps > \frac 12$, which holds by assumption $\eps<\frac 14$.)
For $r\in [0,1]$ we then obtain a metric $d_r$ on $|\Kk|$ by
$$
d_r(x,x'): = d(x,x') + \be(r) d'(x,x') 
$$
and note that $d_r(x,x') \le d_s(x,x')$ whenever $r\le s$.
Moreover, each $d_r$ is admissible on $|\Kk|$ since their pullback to the charts are analogous sums, and the sum of two metrics that induce the same topology also induces this topology.
Now we claim that 
$$
D\bigl( (t,x) , (t',x') \bigr) = d_{\min(t,t')}(x,x') + |t-t'|
$$
provides the required metric $D$ on $|[0,1]\times \Kk |$. This is evidently symmetric and positive definite, and by symmetry it suffices to check the following triangle inequality 
\begin{equation}\label{eq:tri}
D\bigl( (t,x) , (t'',x'') \bigr)\le D\bigl( (t,x) , (t',x') \bigr) + D\bigl( (t',x') , (t'',x'') \bigr)
\end{equation}
for $t\le t''$. 
In the case $t'< t\le t''$ we use $0\le \be(t) - \be(t') \le 2 (t-t')$ 
and $d'\le 1$
to obtain
\begin{align*}
& \; D\bigl( (t,x) , (t',x') \bigr) + D\bigl( (t',x') , (t'',x'') \bigr)\\
&\qquad \qquad  = \;
d(x,x') + d(x',x'') + \be(t') d'(x,x') +  \be(t') d'(x',x'') + |t-t'| + |t'-t''|\\
&\qquad \qquad \ge\;  d(x,x'') + \be(t') d'(x,x'') +  2|t-t'| \\
&\qquad \qquad\ge\; d(x,x'') + \be(t) d'(x,x'') 
\; =\; D\bigl( (t,x) , (t'',x'') \bigr).
\end{align*}
In the other cases $t\le t'\le t''$ resp.\ $t\le t''\le t'$, we can use the monotonicity $\be(t')\geq \be(t)$ resp.\ $\be(t'')\geq \be(t)$ to check \eqref{eq:tri}. Therefore $D$ is a metric. It is admissible because each $d_r$ is admissible on $|\Kk|$ so that the pullback metric on each set $[0,1]\times U_I$ induces the product topology. 
Finally it is $\eps$-collared by construction. In particular, it satisfies \eqref{eq:epscoll} due to the term $|t-t'|$ in its formula.
This finishes the proof.
\end{proof}

The next lemma first shows how to stretch a collar by a factor $L$, and then uses this as a tool to patch together metrics on different halves of a cobordism, using the long neck to establish the triangle inequality.
 
\begin{lemma} \label{le:metcob}
\begin{enumerate}
\item 
Let $\Kk$ be a tame topological Kuranishi cobordism with collar width $2\eps$, and equipped with an admissible metric $d$ on $|\Kk|$. 
Given a boundary component $\al\in\{0,1\}$ and a length parameter $L\ge 1$, there exists an admissible metric $D$ on $|\Kk|$ which on $|\Kk|\less \io^\al_{|\Kk|}(A^\al_{2\eps}\times|\p^\al\Kk|)$ coincides with $d$ and which is $\eps$-collared with stretching factor 
$L$ near $|\p^\al\Kk|$ in the following sense: 
We have $$
D|_{\io^\al_{|\Kk|}(A^\al_\eps\times|\p^\al\Kk|)} = (\io^\al_{|\Kk|})_*\bigl( d^\al + L d_\R \bigr),\quad\mbox{ where }\;   d^\al:= d|_{|\p^\al\Kk|},
$$
 and require
in addition that for all  $y\in |\Kk| \less \io^\al_{|\Kk|} \bigl( A^\al_\eps\times  |\p^\al \Kk|\bigr)$
\begin{equation}\label{eq:Lepscoll}
d\bigl( y , \io^\al_{|\Kk|}(\al + t, x )\bigr) \;\ge\; 
L ( \eps - |t| )
\qquad \forall \;
(\al + t,x)\in  A^\al_\eps \times  |\p^\al\Kk| .
\end{equation}
\item
Let $\Kk^{[0,1]}, \Kk^{[1,2]}$ be two tame topological Kuranishi cobordisms with $\p^1\Kk^{[0,1]}=\p^0\Kk^{[1,2]}$ and collar width $2\eps$, and equipped with admissible $2\eps$-collared metrics $d^{[0,1]},d^{[1,2]}$ whose boundary restrictions agree $d^{[0,1]}|_{|\p^1\Kk^{[0,1]}|}= d^{[1,2]}|_{|\p^0\Kk^{[1,2]}|}$.
Then for the tame topological Kuranishi cobordism $\Kk$ obtained from composing $\Kk^{[0,1]}$ and $\Kk^{[1,2]}$ as in Lemma~\ref{lem:cobord1} there exists an admissible $2\eps$-collared metric $D$ on $|\Kk|$ with $D|_{|\p^0 \Kk|}=d^{[0,1]}|_{|\p^0 \Kk^{[0,1]}|}$ and $D|_{|\p^1 \Kk|}=d^{[1,2]}|_{|\p^1 \Kk^{[1,2]}|}$. 
\end{enumerate}
\end{lemma}

\begin{proof}  
We give the proof of (i) in case $\al=0$; the construction for $\al=1$ is similar.
We begin by pushing forward the metric $d$ on $|\Kk|$ by the bijection
$F : |\Kk| \overset{\cong}{\longmapsto} |\Kk| \less \io^0_{|\Kk|}\bigl( A^0_\eps\times |\p^0\Kk|\bigr)$
given by $F : \io^0_{|\Kk|}(t,x) \mapsto \io^0_{|\Kk|}( \eps + \frac 12 t ,x)$ and the identity on the complement of the $2\eps$-collar. The pushforward $F_*d$ is admissible since $F$ pulls back to homeomorphisms supported in the collars of the boundaries $\p^0 U_I\subset U_I$ of domains in $\Kk$.
Now $d^0$ is equal to the restriction $(F_*d)|_{\io^0_{|\Kk|}(\{\eps\}\times |\p^0\Kk|)}$, pulled back via $\io^0_{|\Kk|}(\eps,\cdot)$, and we extend $F_*d$ to an admissible metric on $|\Kk|$ by symmetric extension of
\[
D\bigl( p , p' \bigr) \; :=\; 
\left\{\begin{array}{ll}
 (F_*d)\bigl( p , p' \bigr)  &\mbox{ if } p,p' \in \im F ;\\ 
 (\io^0_{|\Kk|})_*(d^0 + L d_\R ) (p,p') &\mbox{ if } p,p'\in \io^0_{|\Kk|}\bigl( [0,\eps)\times |\p^0\Kk|\bigr)  ;\\
(F_*d)\bigl( p , \io^0_{|\Kk|}(\eps,x') \bigr) + L |\eps - t' | &\mbox{ if } p\in \im F, p'=\io^0_{|\Kk|}(t',x'), t'\in[0,\eps).
\end{array}\right.
\]
This evidently restricts to $(\io^0_{|\Kk|})_*d^0$ on $\io^0_{|\Kk|}\bigl(\{0\}\times |\p^0\Kk|\bigr)$ as required,
and restricts on the $\eps$-collar $\io^0_{|\Kk|}\bigl([0,\eps)\times |\p^0\Kk|\bigr)$ to a product metric with stretching factor $L$. Moreover the collaring requirement \eqref{eq:Lepscoll} 
holds due to the $L |\eps-t'|$ summand in the construction.
Boundedness, symmetry, and positive definiteness of $D$ follow from that of $d$, and we will check the triangle inequality below.
For charts $\bK_I$ in the interior of $\Kk$ we have $\pi_\Kk(U_I)\subset\im F$, so that admissibility follows from admissibility of $F_*d$. This also implies compatibility of the pullback metric $(\pi_\Kk|_{U_I})^*D$ with the topology on $U_I\cap\pi_\Kk^{-1}(\im F)$ for charts $\bK_I$ with boundary. To check compatibility of the metric on the collars $\io^0_{U_I}([0,2\eps)\times \p^0 U_I)\subset U_I$ note that the collar pullback to $A^0_{2\eps}\times \p^0 U_I$ is the symmetric extension of
\[
(\io^0_{U_I})^*(\pi_\Kk|_{U_I})^*D \bigl( (t,x) , (t',x') \bigr) \; :=\; 
\left\{\begin{array}{ll}
d_I\bigl( (2 (t-\eps), x) ,  ( 2(t'-\eps), x') \bigr)  &\mbox{ if } t,t'\ge \eps ;\\ 
d^0(x,x') + L |t-t'| &\mbox{ if } t,t'<\eps ;\\
d_I\bigl( 2 (t-\eps), x) , (0, x') \bigr) + L |\eps - t' | &\mbox{ if }  t\ge \eps > t' .
\end{array}\right.
\]
This restricts to $d_I:=(\io^0_{U_I})^*(\pi_\Kk|_{U_I})^*d$ on $[\eps,2\eps)\times \p^0 U_I$ 
and to $d^0+L d_\R$ on $[0,\eps]\times \p^0 U_I$, both of which are admissible (i.e.\ induce the product topology) by admissibility of $d$ resp.\ $d^0=d|_{|\p^\al\Kk|}$. 
Moreover, balls around $(\eps,x)\in\{\eps\}\times \p^0 U_I$ are unions of balls in $[\eps,2\eps)\times \p^0 U_I$ and in $[0,\eps]\times \p^0 U_I$, so that the metric topology on $[0,2\eps)\times \p^0 U_I$ is the quotient topology obtained from the product topologies on the disjoint union $[0,\eps]\times \p^0 U_I \;\sqcup\;
[\eps,2\eps)\times \p^0 U_I$ by identifying the two copies of $ \{\eps\}\times \p^0 U_I$. 
Since this coincides with the product topology on $[0,2\eps)\times \p^0 U_I$, it shows that $D|_{\im\io^0_{|\Kk|}}$ is admissible. The analogous argument applies to $D|_{\im\io^1_{|\Kk|}}$, and together with admissibility of $D|_{\im F}$ shows that the metric $D$ is admissible.

Finally, we need to check the triangle inequality $D(p, p') + D(p',p'') \ge D( p , p'') $ when $p,p',p''$ do not all lie in the same components of $|\Kk| = \io^0_{|\Kk|}\bigl( A^0_\eps\times |\p^0\Kk|\bigr) \sqcup \im F$. By symmetry under exchange of $p,p''$ it suffices to go through the following cases, in which we use the triangle inequalities for $d, d^0, d_\R$ 
and the fact that e.g.\ $d\bigl( y , \io^0_{|\Kk|}(0,x') \bigr) = (F_*d)\bigl(F(y), \io^0_{|\Kk|}(\eps,x') \bigr)$.
\begin{itemlist}
\item
For $p = F(y),p''=F(y'')\in \im F$ and $p'\in 
\io^0_{|\Kk|}\bigl( A^0_\eps\times |\p^0\Kk|\bigr)$
we have
\begin{align*}
&D\bigl( F(y) , \io^0_{|\Kk|}(t',x') \bigr) + D\bigl(\io^0_{|\Kk|}(t',x') ,  F(y'') \bigr) \\
&\; = \;
d\bigl( y , \io^0_{|\Kk|}(0,x') \bigr) + L |\eps-t'|
+ d\bigl(\io^0_{|\Kk|}(0,x') ,  y'' \bigr)  + L |\eps-t'| 
\;\ge\; d(y,y'') \;=\; D\bigl( F(y), F(py') \bigr).
\end{align*}
\item
For $p,p'\in \im F$ and $p''\in \io^0_{|\Kk|}\bigl( A^0_\eps\times |\p^0\Kk|\bigr)$ we have
\begin{align*}
D\bigl( F(y) , F(y') \bigr) + D\bigl( F(y'), \io^0_{|\Kk|}(t'',x'') \bigr) 
&= 
d\bigl( y ,y' \bigr) +  d\bigl( y' , \io^0_{|\Kk|}(0,x'') \bigr)  + L |\eps-t''|  \\
&\geq \bigl( y , \io^0_{|\Kk|}(0,x'') \bigr)  + L |\eps-t''|   \;=\; D\bigl( F(y), \io^0_{|\Kk|}(t'',x'') \bigr).
\end{align*}

\item
For $p\in \im F$ and $p',p''\in \io^0_{|\Kk|}\bigl( A^0_\eps\times |\p^0\Kk|\bigr)$ we have
\begin{align*}
& D\bigl( F(y) , \io^0_{|\Kk|}(t',x') \bigr) + D\bigl( \io^0_{|\Kk|}(t',x'), \io^0_{|\Kk|}(t'',x'') \bigr)  \\
&\qquad= \; 
d\bigl( y , \io^0_{|\Kk|}(0,x') \bigr)  + L |\eps-t'| + d^0(x',x'') + L |t''-t'|  \\
&\qquad\ge\; d\bigl( y, \io^0_{|\Kk|}(0,x'') \bigr)  + L |\eps-t''|   \;=\; D\bigl( F(y), \io^0_{|\Kk|}(t'',x'') \bigr).
\end{align*}

\item
For $p' \in \im F$ and $p,p''\in \io^0_{|\Kk|}\bigl( A^0_\eps\times |\p^0\Kk|\bigr)$ we have
\begin{align*}
&D\bigl( \io^0_{|\Kk|}(t,x) , F(y') \bigr) + D\bigl(  F(y') , \io^0_{|\Kk|}(t'',x'') \bigr) \\
&\qquad= 
d\bigl( \io^0_{|\Kk|}(0,x) , y' \bigr)  + L |\eps-t| +  d\bigl( y' , \io^0_{|\Kk|}(0,x'') \bigr)  + L |\eps-t''| 
 \\
&\qquad \ge\; d\bigl( \io^0_{|\Kk|}(0,x) , \io^0_{|\Kk|}(0,x'') \bigr)  + L |t''-t'|   
 \;=\;
 D\bigl(  \io^0_{|\Kk|}(t,x), \io^0_{|\Kk|}(t'',x'') \bigr).
\end{align*}
\end{itemlist}
This proves the triangle inequality and thus finishes the proof of~(i).


To prove (ii), let us denote $\Kk^1:=\p^1\Kk^{[0,1]}=\p^0\Kk^{[1,2]}$ and 
$d^1:=d^{[0,1]}|_{|\p^1\Kk^{[0,1]}|\cong|\Kk^1|}= d^{[1,2]}|_{|\p^0\Kk^{[1,2]}|\cong|\Kk^1|}$, 
and let $B\ge 1$ be a common upper bound of $d^{[0,1]}, d^{[1,2]}$.
Then in a first step we replace the metric $d^{[0,1]}$ by the metric $D^{[0,1]}$ on $|\Kk^{[0,1]}|$ constructed  in (i) with $L:=\eps^{-1} B$ and $\al = 1$, which restricts to $d^{[0,1]}$ outside of $\io^1_{|\Kk^{[0,1]}|}(A^1_{2\eps}\times|\Kk^1|)$ and to $(\io^1_{|\Kk^{[0,1]}|})_*\bigl( d^1 + L d_\R \bigr)$ on $\io^1_{|\Kk^{[0,1]}|}(A^1_{\eps}\times|\Kk^1|)$.
Similarly, (i) provides a metric $D^{[1,2]}$ on $|\Kk^{[1,2]}|$ which restricts to $d^{[1,2]}$ outside of $\io^0_{|\Kk^{[1,2]}|}(A^0_{2\eps}\times|\Kk^1|)$ and to $(\io^0_{|\Kk^{[1,2]}|})_*\bigl( d^1 + L d_\R \bigr)$ on $\io^0_{|\Kk^{[1,2]}|}(A^0_{\eps}\times|\Kk^1|)$.
Now the virtual neighbourhood induced by the composed Kuranishi cobordism $\Kk$ is
$|\Kk|=|\Kk^{[0,1]}|\cup_{|\Kk^1|} |\Kk^{[1,2]}|$, where we identify 
$\io^1_{|\Kk^{[0,1]}|}(\{1\}\times|\Kk^1|) \cong |\Kk^1| \cong \io^0_{|\Kk^{[1,2]}|}(\{0\}\times|\Kk^1|)$.
It can also be decomposed as
$$
|\Kk| \;=\; C^{[0,1]} \underset{\io^1(\{1-\eps\}\times|\Kk^1|) \cong \{-\eps\}\times|\Kk^1|}{\cup}
[-\eps,\eps]\times |\Kk^1|
\underset{\io^0(\{\eps\}\times|\Kk^1|) \cong \{\eps\}\times|\Kk^1|}{\cup}
C^{[1,2]} 
$$
into pieces $C^A=|\Kk^A|\less \io^\al_{|\Kk^A|}(A^\al_\eps\times|\Kk^1|)$ for $A=[0,1]$, $\al=1$, resp.\ for $A=[1,2]$, $\al=0$, and a cylindrical piece 
$$
[-\eps,\eps] \times |\Kk^1| \;\cong\;
\io^1_{|\Kk^{[0,1]}|}(A^1_{\eps}\times|\Kk^1|) 
\;\underset{|\Kk^1|}\cup\;
\io^0_{|\Kk^{[1,2]}|}(A^0_{\eps}\times|\Kk^1|) \;\subset\; |\Kk| ,
$$
obtained by identifying $A^1_\eps \underset{0\sim 1}{\cup} A^0_\eps \cong [-\eps,\eps]$.
We claim that the required metric on $|\Kk|$ can be obtained by gluing the metrics $D^A$ on $C^A$ via a long neck metric $d^1+L d_\R$ on $[-\eps,\eps] \times |\Kk^1|$, that is by symmetric extension of
$$
D(p, p' ): = \left\{
\begin{array}{ll} 
\min\bigl(\, D^A(p,p') \,,\, B \,\bigr)  &\mbox{ if } p,p'\in |\Kk^A|;\\
\min\bigl( \, d^1(x,x') + L |1-t| + L |t'| \, ,\, B \,\bigr) \phantom{\int_A^B} \!\!\!\!
 &\mbox{ if } p=\io^1_{|\Kk^{[0,1]}|}(t,x) , t\in (1-\eps,1], \\
 & \quad  p'=\io^0_{|\Kk^{[1,2]}|}(t',x'), t'\in [0,\eps); \\
B &\mbox{ otherwise}.
\end{array}\right.
$$
Indeed, $D$ is well defined since $D^A$ for $A=[0,1]$ resp.\ $A=[1,2]$ pulls back via the respective collar embedding to 
$d^1 + L d_\R$ 
on $\bigl([1-\eps,1+\eps]\cap A\bigr)\times |\Kk^1|$.
It is constructed to be symmetric, positive definite, and bounded by $B$, and is $2\eps$-collared with the required restrictions since $D$ restricts to $d^A$ on $|\Kk^A|\less \io^\al_{|\Kk^A|}(A^\al_{2\eps}\times|\Kk^1|)\subset C^A$, where $\al = 1$ for $A = [0,1]$ and  $\al = 0$ for $A = [1,2]$.
Further, $D$ is admissible because each $D^A$ is admissible, and as in (i) the pullback of $D$ induces the product topology on each concatenation of collars
$(1-\eps,1]\times \p^1U^{[0,1]}_{I^{01}} \underset{\sim}{\cup} [0,\eps)\times \p^0 U^{[1,2]}_{I^{12}} \cong (-\eps,\eps)\times U^1_I$ via $\io^1_{U^{[0,1]}_{I^{01}}}(1,x)\sim \io^0_{U^{[1,2]}_{I^{12}}}(0,x) \sim x\in U^1_I$ in Lemma~\ref{lem:cobord1}.

Finally, to see that $D$ satisfies the triangle inequality $D(p, p') + D(p',p'') \ge D( p , p'')$, note that $D$ is bounded by $B$ and already restricts to metrics $\min(D^A,B)$ on the subsets $|\Kk^A|\subset|\Kk|$
as well as to $\min(d^1+L d_\R,B)$ on the middle neck $[-\eps,\eps] \times |\Kk^1|$.
So it remains to check triples with $D(p, p') + D(p',p'') < B$ 
that lie in different parts of $|\Kk|$. Up to symmetries (switching $p,p''$ or switching $[0,1]$ for $[1,2]$), the only remaining case is $p\in C^{[0,1]}$, $p'=\io^1_{|\Kk^{[0,1]}|}(t',x')$ with $t'\in (1-\eps,1]$, and $p''=\io^0_{|\Kk^{[1,2]}|}(t'',x'')$ with $t''\in [0,\eps)$, when 
\begin{align*}
& D\bigl( p , \io^1_{|\Kk^{[0,1]}|}(t',x')\bigr) + D\bigl( \io^1_{|\Kk^{[0,1]}|}(t',x') , \io^0_{|\Kk^{[1,2]}|}(t'',x'')\bigr) \\
&\qquad=  \;
\min\bigl( \, D^{[0,1]}\bigl( p , \io^1_{|\Kk^{[0,1]}|}(t',x')\bigr) \,,\, B\, \bigl) +  
\min\bigl( \, d^1(x',x'') + L |1-t'| + L |t''|  \,,\, B\, \bigl) \\
&\qquad\geq \;
\min\bigl( 
(F_*d^{[0,1]})\bigl( p , \io^1_{|\Kk^{[0,1]}|}(1-\eps,x')\bigr) + L | 1-\eps -t'| +  d^1(x',x'') + L |1-t'| + L |t''| 
 \,,\, B\, \bigl) \\
&\qquad
\geq \;
\min\bigl( L | 1-\eps -t'| + L |1-t'|  \,,\, B\, \bigl) 
\;\ge\; \min\bigl( L \eps  \,,\, B\, \bigl)  \;=\; B \;=\;   D\bigl( p, \io^0_{|\Kk^{[1,2]}|}(t'',x'') \bigr).
\end{align*}
This proves the triangle inequality, so $D$ is a metric, which finishes the proof of (ii).
\end{proof}

\begin{proof}[Proof of Proposition~\ref{prop:metcob}]
To show that metric tame cobordism is an equivalence relation, we have to extend the constructions in the proof of Lemma~\ref{lem:cobord1} by constructions of admissible metrics, and check the tameness identities \eqref{eq:tame1}, \eqref{eq:tame2} for the given filtrations.
For reflexivity one can use the product metric of Example~\ref{ex:mtriv}~(i), and tameness transfers automatically to the product atlas.
The reflection, reparametrization, and relabeling in the proof of symmetry in Lemma~\ref{lem:cobord1} does not affect the metric or tameness properties.
For transitivity note that the concatenation construction in the proof of Lemma~\ref{lem:cobord1} transfers the tameness properties of the constituent atlases directly to the concatenated atlas.
Moreover, concatenation of metrics is provided by Lemma~\ref{le:metcob}~(ii).  
This proves the first statement.

Although the uniqueness statement (i) is a special case of (ii) for the trivial Kuranishi cobordism $[0,1]\times\Kk$, 
we need to prove it first, so we can use it in the proof of (ii).
So let two admissible metrics $d^0,d^1$ on $|\Kk|$ be given, and let $c\ge 1$ be a common upper bound for $d^0$ and $d^1$.
By Lemma~\ref{le:metcoll} there is an admissible $\frac 16$-collared metric $d^{[0,1]}$ on $[0,1]\times |\Kk|$ with boundary restrictions $d_\R +\frac 1{c} d^0$ on $\{0\}\times|\Kk|$ and  $d_\R + \frac 1{c} (d^0 +d^1)$ on $\{1\}\times |\Kk|$. 
Similarly, there is an admissible $\frac 16$-collared metric $d^{[1,2]}$ on $[1,2]\times |\Kk|$ with boundary restrictions $\frac 1{c} (d^0 +d^1)$ on $\{1\}\times |\Kk|$ and  $\frac 1{c} d^1$ on $\{2\}\times |\Kk|$.
These satisfy the assumptions of Lemma~\ref{le:metcob}~(ii), so that we obtain an admissible $\frac 16$-collared metric $D'$ on $[0,2]\times |\Kk|$ with boundary restrictions $\frac 1c d^0$ on $\{0\}\times |\Kk|$ and $\frac 1c d^1$ on $\{2\}\times |\Kk|$.
Now define $D$ on $[0,1]\times |\Kk|$ to be the pullback of $ c D'$ by a rescaling map $(t,x)\mapsto (\be(t),x)$,
where $\be:[0,1]\to[0,2]$ is a   monotone smooth function  such that
$\be(t) = \frac {t}{c}$ for $0\le t \le \frac 16$ and $\be(t) = 2- \frac {1- t}{c}$ for $\frac 56 \le t \le 1$.  Then $D$ restricts to $d_\R+d^0$ on $[0,\frac 16]\times |\Kk|$ and to $d_\R + d^1$ on $[\frac 56,1]\times  |\Kk|$ and also satisfies \eqref{eq:epscoll} with $\eps=\frac 16$ since for $0\leq t <\frac 16 \le s$ 
(and similarly for $s\le \frac 56 < t < 1$) we either have $\be(s)\ge \frac 1{6}$, which by \eqref{eq:epscoll} for $D'$ implies
$$
D\bigl((s,y), (t,x) \bigr) =  c \cdot D'\bigl( (\be(s),y) , ( \tfrac {t}{c},x)   \bigr)   
\ge c \bigl( \tfrac 1{6} - \tfrac {t}{c} \bigr)
= \tfrac c6 -  t  \ge \tfrac 16 - t , 
$$ 
or we have $\frac 1{6} > \be(s) \ge \be(\frac 16)= \frac 1{6c}$ so that the collar form of $D'$ implies
$$
D\bigl((s,y), (t,x) \bigr) =  c \ D'\bigl( (\be(s),y) , ( \tfrac t{c},x)  \bigr)   
\ge c\bigl( \be(s) - \tfrac {t}{c} \bigr)
\ge c \tfrac {1}{6c} - t
= \tfrac 16 - t.
$$ 

Moreover, this rescaling does not affect the admissibility.
Hence it provides the required metric tame concordance  $([0,1]\times \Kk,D)$ from $(\Kk,d^0) $ to $(\Kk,d^1)$, which proves (i).

Now in (ii) we consider a general tame topological Kuranishi cobordism $\Kk$ and may assume that it is equipped with an admissible metric on $|\Kk|$. Let $4\eps>0$ be the collar width of $\Kk$, then we can use Lemma~\ref{le:metcob}~(i) with $L =1$ to obtain a $2\eps$-collared admissible metric $d$ on $|\Kk|$.
Next, we may view $\Kk$ as the result of concatenating the product Kuranishi atlases $[0,\eps]\times\p^0\Kk$, $[1-\eps,1]\times\p^1\Kk$ with the Kuranishi cobordism $\Kk'$ that is obtained from $\Kk$ by removing an $\eps$-collar $\io^\al_{U_I}(A^\al_\eps\times \p^\al U_I)$ from each domain $U_I$.
This $\Kk'$ still has collar width $3\eps$ 
and canonically identified boundaries $\p^\al\Kk'\cong\p^\al\Kk$, 
and the metric $d$ restricts to an admissible $\eps$-collared metric $d'$ on 
$|\Kk'|=|\Kk|\less \bigl( \io^0_{|\Kk|}([0,\eps)\times|\p^0\Kk|)\cup \io^1_{|\Kk|}((1-\eps,1]\times|\p^1\Kk|)\bigr)$.
Now we apply the construction of (i) to the above product Kuranishi atlases to obtain metrics $D^0$ and $D^1$ on  $[0,\eps]\times|\p^0\Kk|$ resp.\ $[1-\eps,1]\times|\p^1\Kk|$ with
$D^0|_{\{0\}\times|\p^0\Kk|}=d^0$, 
$D^0|_{\{\eps\}\times|\p^0\Kk|}=d'|_{|\p^0\Kk'|}$,
$D^1|_{\{1-\eps\}\times|\p^1\Kk|}= d'|_{|\p^1\Kk'|}$,
$D^1|_{\{1\}\times|\p^1\Kk|}=d^1$.
These satisfy the assumptions of Lemma~\ref{le:metcob}~(ii), so that by two consecutive concatenations we obtain an admissible collared metric $D$ on 
$$
|\Kk|= 
[0,\eps]\times|\p^0\Kk| \underset{\{\eps\}\times|\p^0\Kk|\cong |\p^0\Kk'|}{\cup} |\Kk'| 
 \underset{\{1-\eps\}\times|\p^1\Kk|\cong |\p^1\Kk'|}{\cup} [1-\eps,1]\times|\p^1\Kk|
$$ 
with the required boundary restrictions $D|_{|\p^0\Kk|\cong \{0\}\times|\p^0\Kk|}=d^0$, $D|_{|\p^1\Kk|\cong\{1\}\times|\p^1\Kk|}=d^1$. This finishes the proof.
\end{proof}

We are now in a position to prove the uniqueness part of Theorem~\ref{thm:K}, namely that different tame 
shrinkings of the same filtered weak topological Kuranishi atlas are concordant. 
This is a crucial ingredient in establishing that the virtual fundamental class associated to a given Kuranishi atlas is well defined. 
In fact, in practice one only associates a well defined concordance class of additive weak Kuranishi atlases to a given moduli space, which requires the full generality of the next result for concordances together with the uniqueness of metrics from Proposition~\ref{prop:metcob}~(i).
Moreover, different geometric choices (e.g.\ of the almost complex structure in Gromov--Witten moduli spaces) yield more general Kuranishi cobordisms as in Example~\ref{ex:natcol}, not just concordances, so that uniqueness of the resulting (e.g.\ Gromov--Witten) invariants requires the full generality of the following result.  For the proof, we must revisit the construction of shrinkings in the proof of Proposition~\ref{prop:proper} and will construct the following notions of cobordism shrinkings.

\begin{defn}\label{def:cshr}
A {\bf tame shrinking} of a weak topological Kuranishi cobordism $\Kk$ is a tame topological Kuranishi cobordism $\Kk'$ that is a shrinking of $\Kk$ in the sense of Definition~\ref{def:shr}.
That is, $\Kk'$ is a restriction of $\Kk$ to a suitable choice of collared subsets $U'_I\sqsubset U_I$ of the Kuranishi domains. 
Moreover, a tame shrinking $\Kk''$ of $\Kk$ is {\bf preshrunk} if $\Kk''$ is also a shrinking of an intermediate tame shrinking $\Kk'$ of $\Kk$.
\end{defn}

\begin{thm}\label{thm:cobord2}
Let $\Kk$ be a filtered weak topological Kuranishi cobordism on a compact collared cobordism $(Y, \io^0_Y,\io^1_Y)$, let $\Kk^0_{sh}, \Kk^1_{sh}$ be preshrunk tame shrinkings of $\p^0\Kk$ and $\p^1\Kk$, and let admissible metrics $d^\al$ on $|\Kk^\al_{sh}|$ be given (which exist by Proposition~\ref{prop:metric}).
Then there is a preshrunk tame shrinking of $\Kk$ and an admissible collared metric $d$ on $|\Kk|$ that provide a metric tame topological Kuranishi cobordism 
$\widetilde\Kk$ on $(Y, \io^0_Y,\io^1_Y)$ 
with $\p^\al \widetilde\Kk = \Kk^\al_{sh}$ and $d|_{|\p^\al \widetilde\Kk|} = d^\al$.
\end{thm}

\begin{remark}\rm
In the case of shrinkings $\Kk^0,\Kk^1$ of a fixed topological Kuranishi atlas $\Kk$ there is an easier construction of a concordance in one special case:
If the shrinkings of the footprint covers $(F^\al_I)$ are compatible in the sense that their intersection
$(F^0_I\cap F^1_I)$ is also a shrinking (i.e.\ covers $X$ and has the same index set of nonempty intersections of footprints), then by Remark~\ref{rmk:shrink} the intersection of domains $U^0_{IJ}\cap U^1_{IJ}$ defines another shrinking of~$\Kk$. 
Thus one obtains a  shrinking of the product concordance $[0,1]\times \Kk$ by
$$ 
U_{IJ}^{[0,1]} :=  \Bigl([0,\tfrac 13) \times  U^0_{IJ}  \Bigr)
\;\cup\; \Bigl([\tfrac 13,\tfrac 23] \times  \bigl( U^0_{IJ}\cap U^1_{IJ} \bigr) \Bigr)
\;\cup\; \Bigl( (\tfrac 23,1]\times U^1_{IJ}  \Bigr) . 
$$
Further it is straightforward to check that if $\Kk$ is filtered, and if the shrinkings 
$\Kk^0,\Kk^1$ are tame, then the above shrinking is also tame.  (Here, we use the canonical filtration 
that is induced from the product filtration of $[0,1]\times \Kk$ as in Lemma~\ref{le:filter}.)
\end{remark}

\begin{proof}[Proof of Theorem~\ref{thm:cobord2}]
As in the proof of Proposition~\ref{prop:metric} we first construct a tame shrinking between any pair of  
tame shrinkings  $\Kk^0, \Kk^1$ of $\p^0\Kk$ and $\p^1\Kk$.
We will use this first to obtain a tame shrinking $\Kk'$ of $\Kk$ with $\p^\al\Kk'=\Kk^\al$ 
(where by assumption we have chosen tame shrinkings $\Kk^\al$ of $\p^\al\Kk$ for $\al=0,1$ so that $\Kk^\al_{sh}$ is a shrinking of $\Kk^\al$), and second to obtain a tame shrinking $\Kk_{sh}$ of $\Kk'$ with $\p^\al\Kk_{sh}=\Kk^\al_{sh}$.
The latter topological Kuranishi cobordism $\Kk_{sh}$ supports an admissible metric $d_{sh}$ by the same argument as in Proposition~\ref{prop:metric}. 
We may arrange that the metric is collared with the given restrictions $d^\al$ to the boundary atlases by Proposition~\ref{prop:metcob}~(ii). 
So it remains to show how to construct a tame shrinking  of $\Kk$ with given boundary shrinkings.

We write the index set as the union $\Ii_{\Kk} = \Ii_0 \cup \Ii_{(0,1)}\cup \Ii_1$ of
$\Ii_\al := \Ii_ {\p^\al\Kk} \subset \Ii_{\Kk}$ and $\Ii_{(0,1)}:=\Ii_{\Kk} \less (\Ii_0\cup\Ii_1)$.
Since the footprint of a chart $\Kk$ might intersect both $\p^0 Y$ and $\p^1 Y$
the sets $\Ii_0$ and $\Ii_1$ may not be disjoint, though they are both disjoint from $\Ii_{(0,1)}$,
which indexes the charts with precompact footprint in $Y\less (\p^0 Y\sqcup \p^1 Y)$. 
We will denote the charts of the Kuranishi cobordism $\Kk$ 
by $\bK_I=(U_I,\ldots)$, while $\bK^\al_I=(U^\al_I,\ldots) = \p^\al \bK_I |_{U^\al_I}$ denotes the charts of the shrinking $\Kk^\al$ of $\p^\al\Kk$  with domains $U^\al_{IJ}\subset \partial^\al U_{IJ}$.
Recall moreover that by definition of a shrinking the index sets $\Ii_ {\Kk^\al}=\Ii_ {\p^\al\Kk}=\Ii_\al$ coincide.
We suppose that the charts and coordinate changes of $\Kk$ 
 have uniform collar width $5\eps>0$ as in Remark~\ref{rmk:Ceps}. Then the footprints have induced $5\eps$-collars
$$
F_I\cap \bigl(\io_Y^\al (A^\al_{5\eps}\times \p^\al Y)\bigr)  =  \io_Y^\al (A^\al_{5\eps}\times \partial^\al F_I)
$$
with $\partial^\al F_I = F_I\cap \p^\al Y$ as in Definition~\ref{def:Ycob}. 
It will be convenient to denote for any $0<\delta\leq 5\eps$ the complement in $Y$ of the $\delta$-collar of the boundary by the following shorthand, which denotes by $\io_\delta$ the indicated restriction of the collar embeddings for $Y$:
$$
Y\less  \im \iota_\delta: = Y\less 
{\textstyle \;\bigsqcup_{\al=0,1} \io_Y^\al(A_\delta^\al \times \p^\al Y) .}
$$
By construction of the shrinkings $\Kk^\al$ of $\p^\al\Kk$, 
we have  precompact inclusions $F^\al_I \sqsubset \partial^\al F_I$.
Now one can form a  $3\eps$-collared shrinking $(F'_i\sqsubset F_i)_{i=1,\ldots,N}$ of the cover of  $Y$  
by the footprints of the basic charts by first choosing an arbitrary shrinking $(F_i'')$ as in Definition~\ref{def:shr0}, and then adding $3\eps$-collars to the boundary charts.
Namely, for $i\in \Ii_0\cup \Ii_1$ we define
$$
F_i': = \bigl(F_i''\cup \bigl(\io_Y^\al(A^\al_{4 \eps}\times F^\al_i ) \bigr)\bigr)\;
\less \; \io_Y^\al \bigl(\, \overline{A^\al_{3\eps}}\times 
(\p^\al Y \less F^\al_i) \bigr)  .
$$
By construction, these sets still cover $\im \io_{4\eps} = \sqcup_\al\io_Y^\al(A^\al_{4\eps}\times \p^\al Y )$,
and so, because the sets $F'_i:=F''_i$ for $i\in\Ii_{(0,1)}$ cover $Y\less \im \io_{3\eps}$, their union covers all of $Y$. 
Moreover, each $F'_i$ is open with $3\eps$-collar $F_i'\cap \io_Y^\al\bigl(A^\al_{3\eps}\times \p^\al Y \bigr) = \io_Y^\al(A^\al_{3\eps}\times F^\al_i)$ and has compact closure in $F_i$ because $F_i''$ does by construction and $A^\al_{4\eps}\times F^\al_i \sqsubset A^\al_{5\eps}\times \partial^\al F_i  \subset F_i $.
Next, the induced footprints $F'_I=\bigcap_{i\in I} F'_i$ also have $3\eps$-collars, and $F_I\neq\emptyset$ implies $F'_I\neq\emptyset$ since either $F_I\cap \im \io_{5\eps} \neq \emptyset$ so that $\emptyset\neq \io_Y^\al( A^\al_{4\eps}\times F^\al_I)
 \subset F'_I$, or $F_I\subset Y\less \im \io_{5\eps}$ so that $\emptyset\neq F''_I\subset F'_I$. 
Hence $(F_i')_{i=1,\ldots,N}$ is a  shrinking of the footprint cover with $3\eps$-collars.

We now carry through the proof of Proposition~\ref{prop:proper},
in the $k$-th step choosing domains $  U^{(k)}_{IJ}\subset U^{(k-1)}_{IJ}\subset U_{IJ}$ for $I, J\in \Ii_{\Kk^{[0,1]}}$ satisfying the conditions
(i$'$), (ii$'$),(iii$'$) as well as
the following collar requirement which ensures that the resulting shrinking of $\Kk$ 
is a topological Kuranishi cobordism between the given tame atlases $\Kk^0$ and $\Kk^1$:
\begin{equation}\label{collar}
 (\iota^\al_I)^{-1} \bigl( U^{(k)}_{IJ} \bigr) \;=\; A^\al_\eps\times U^{\al}_{IJ}   
\qquad\forall \; \al \in\{0,1\}, \; I\subset J\in \Ii_\al  .
\end{equation}
For $k=0$ we first must choose precompact sets $U^{(0)}_I\sqsubset U_I$ satisfying the zero set condition \eqref{eq:U(0)}, namely $U_I^{(0)}\cap \s_I^{-1}(0_I) = \psi_I^{-1}(F_I')$.
For that purpose we apply Lemma~\ref{le:restr0} to
$$
F_I'\cap   \bigl(Y\less \im \io_{2\eps} \bigr)  
\;\sqsubset \; \psi_I \bigl(\s_I^{-1}(0_I)\less{\textstyle  \,\bigsqcup_\al \,} \io^\al_I(\p^\al U_I\times \ov{A^\al_{\eps}})\bigr)
$$
to find
$
U_I' \;\sqsubset \; U_I \;\less\;  {\textstyle  \bigsqcup_\al} \io^\al_I(\ov{A^\al_{\eps}}\times \p^\al U_I)$ with
$$
U_I'\cap \s_I^{-1}(0_I) = \psi_I^{-1}\bigl(F_I'\cap \bigl(Y\less \im \io_{2\eps} Y\bigr)\bigr).
$$
Then we add the image under $\io^\al_I$ of the precompact subsets
$A^\al_{3\eps}\times U_I^\al \sqsubset   A^\al_{4\eps}\times \partial^\al U_I \subset U_I$, 
which have footprint 
$$
\io_Y^\al(A^\al_{3\eps}\times F_I^\al) = F_I' \cap \io_Y^\al(A^\al_{3\eps}\times \p^\al Y), 
$$
 to obtain the required domains
$$
U_I^{(0)} := U_I'\;\cup\; {\textstyle \bigsqcup_{\al=0,1}} \io^\al_I(A^\al_{3\eps}\times U_I^\al) \quad\sqsubset\; U_I 
$$ 
with boundary $\p^\al U_I^{(0)} = U_I^\al$ and collar width $\eps$.
Next, the domains $U_{IJ}^{(0)}$ for $I\subsetneq J$ are determined by \eqref{eq:UIJ(0)} and satisfy \eqref{collar} since 
both $U_{IJ}$ and $U_{I}^{(0)}, U_{J}^{(0)}$ have $\eps$-collars,
and $\phi_{IJ}$ has product form on the collar.
For $I\in\Ii_{(0,1)}$ these constructions also apply, and reproduce the construction without boundary, if we denote $\im\io^\al_I:=\emptyset$ and recall that the footprints are contained in $Y\less \im \io_{2\eps} $. 
In the following we will use the same conventions and hence need not mention $\Ii_{(0,1)}$ separately.

Now in each iterative step for $k\geq 1$ there are two adjustments of the domains. First, in Step~A the domains $U^{(k)}_{IK}\subset W_{K'}$ for $|I|=k$ are chosen using Lemma~\ref{le:set}, where $W_{K'}$ is given by~\eqref{eq:WwK}.
In order to give these sets $\eps$-collars, we denote the sets provided by Lemma~\ref{le:set} by $V_{IK}^{(k)}\subset W_{K'}$ and define 
\begin{equation}\label{eq:UIJeps1}
U_{IK}^{(k)} \,:=\; V_{IK}^{(k)} \cup U_{IK}^{0,1,\eps}
\qquad \text{with}\quad
U_{IK}^{0,1,\eps} \,:=\; 
\iota^0_I\bigl( A^0_\eps\times U_{IK}^0)  \;\cup\;  \iota^1_I\bigl( A^1_\eps \times U_{IK}^1 ) .
\end{equation}
This set is open and satisfies \eqref{collar} because $V_{IK}^{(k)}$ is a subset of $U^{(k-1)}_{IK}$, which by induction hypothesis has the required $\eps$-collar.
We now show that  \eqref{eq:UIJeps1} satisfies the requirements of Step A.

\begin{itemlist}
\item[(i$'$)]
holds since $U_{IJ}^{(k-1)}\cap \bigl(\s_I\bigr)^{-1}(\E_{HI})
 \;\subset\; V_{IJ}^{(k)}\subset
U_{IJ}^{(k)}$, where the first inclusion holds by construction of $V_{IJ}^{(k)}$.
\item[(ii$'$)]
holds since $V_{IJ}^{(k)}\cap V_{IK}^{(k)}= V_{I (J\cup K)}^{(k)}$ by construction and
$U_{IJ}^{0,1,\eps}\cap U_{IK}^{0,1,\eps} =  U_{I (J\cup K)}^{0,1,\eps}$ by the tameness of the collars, so
\begin{align*}
U_{IJ}^{(k)}\cap U_{IK}^{(k)}
&\;=\;
\bigr(V_{IJ}^{(k)}\cap V_{IK}^{(k)}\bigl)
\;\cup\;
\bigr(U_{IJ}^{0,1,\eps}\cap V_{IK}^{(k)}\bigl)
\;\cup\;
\bigr(V_{IJ}^{(k)}\cap U_{IK}^{0,1,\eps} \bigl)
\;\cup\;
\bigr(U_{IJ}^{0,1,\eps} \cap U_{IK}^{0,1,\eps} \bigl)  \\
&\;=\;
V_{I(J\cup K)}^{(k)}
\;\cup\;
U_{I(J\cup K)}^{0,1,\eps}  \;=\; U_{I (J\cup K)}^{(k)} .
\end{align*}
Here the two mixed intersections are subsets of the collar $U^{0,1,\eps}_{IJ}\cap U^{0,1,\eps}_{IK}$ by ${V_{I\bullet}^{(k)}\subset U_{I\bullet}^{(k-1)}}$.
\item[(iii$''$)]
holds since $V_{IK}^{(k)} \subset (\phi_{IJ})^{-1}(V_{JK}^{(k-1)})$ by construction and
$U^{0,1,\eps}_{IK} \subset (\phi^{[0,1]}_{IJ})^{-1}(U^{0,1,\eps}_{JK})$
by the tameness of the shrinkings $U^0_\bullet, U^1_\bullet$.
\end{itemlist}

\NI
This completes Step A.
In Step B the domains $U^{(k)}_{JK}$ for $|J|>k$ are constructed by \eqref{eq:UJK(k)}, namely
$$
U_{JK}^{(k)}\, :=\; U_{JK}^{(k-1)}\less \bigcup_{I\subset J, |I|= k} 
\bigl( \s_J^{-1}(\E_{IJ}) \less \phi_{IJ}(U^{(k)}_{IJ}) \bigr) .
$$
We must check that this removes no points in the collars, i.e.\
$$
\iota^\al_J(A^\al_\eps\times U_{JK}^\al )\cap \s_J^{-1}(\E_{IJ}) \;\subset\; \phi_{IJ}(U^{(k)}_{IJ}) .
$$
But in this collar $\s_J$ and $\phi_{IJ}$ have product form induced from the corresponding maps in the atlases $\Kk^\al$, where tameness implies $U_{JK}^\al\cap (\s_J^\al)^{-1}(\E_{IJ}) =
\phi_{IJ}^\al (U^\al_{IJ})$.
Since the $U^{(k)}_{IJ}$ already have $\eps$-collars by construction, this guarantees the above inclusion. Thus, with these modifications, the $k$-th step in the proof of Proposition~\ref{prop:proper} carries through. After a finite number of iterations, we find a tame shrinking $\Kk'$ of 
$\Kk$  
with given restrictions $\p^\al\Kk 
=\Kk^\al$ for $\al=0,1$.
This completes the proof.
\end{proof}

\section{Reductions and 
perturbations}\label{s:red}

The next step in the construction of a virtual moduli cycle is to show that the canonical section $\s_\Kk$ has suitable transverse perturbations. 
The discussion at the beginning of Section~\ref{ss:red0} shows that one cannot expect to be able to define suitable perturbations as functors ${\bB_\Kk}\to \bE_\Kk$ since the category $\bB_\Kk$ has too many morphisms.  In this section we explain the construction and properties of a suitable subcategory of $\bB_\Kk$ whose coordinate changes are governed by the natural partial order $I\subset J$ on the indexing set $\Ii_\Kk$ of the charts in the atlas.  This subcategory is called a {\bf reduction}; it is the closest we come to the notion from \cite{FO,FOOO} of a \lq\lq good coordinate system".   
While transverse perturbations can only be constructed in a smooth context, e.g.\ \cite{MW1}, we introduce the notion of a precompact perturbation of the reduced section and its perturbed zero set in Definition~\ref{def:pert}.
The main results of this section are Theorem~\ref{thm:red}, which establishes the existence and uniqueness  of reductions, and Theorem~\ref{thm:zeroS0}, which establishes compactness of the perturbed zero set.

\subsection{Definitions and properties}\label{ss:red0} \hspace{1mm}\\ \vspace{-3mm}

Suppose that $\Kk$ is a smooth and tame atlas.
The cover of $X$ by the footprints $(F_I)_{I\in \Ii_\Kk}$ of all the Kuranishi charts
(both the basic charts and those that are part of the transitional data) is closed under intersection. This makes it easy to express compatibility of the charts, since the overlap of footprints of any two charts $\bK_I$ and $\bK_J$ is covered by another chart $\bK_{I\cup J}$.
However, this yields so many compatibility conditions that a construction of nonzero compatible perturbations in the Kuranishi charts may not be possible.
For example, suppose that $I\cap J=\emptyset$ but $K: = I\cup J\in \Ii_\Kk$. Then for the perturbations
$\nu_I:U_I\to\E_I$ and $\nu_J:U_J\to \E_J$  to be compatible they must induce the same perturbation 
over the intersection $\im \phi_{IK}\cap \im \phi_{JK}\supset \psi_K^{-1}(F_K)$.  
In particular, when working with an additive atlas -- i.e.\ the smooth version of a filtered atlas, in which $\E_{IJ}=\im\Hat\Phi_{IJ}$ -- we must have for all $x\in\im \phi_{IK}\cap \im \phi_{JK}$
$$
\E_{IK}\; \ni\; \Hat\Phi_{IK}\circ  \nu_I (\phi_{IK}^{-1}(x)) \; = \; \nu_K(x) \;=\; \Hat\Phi_{JK}\circ \nu_J (\phi_{JK}^{-1}(x))\;\in\; \E_{JK}.
$$
But the filtration conditions imply that $\E_{IK}\cap \E_{JK}   = \E_{\emptyset K} = \im\, 0_K$, so that $\nu_K$ must vanish at such points. 
In particular this means that $\s_K^{-1}(0)$ is contained in any perturbed zero set $(\s_K+\nu_K)^{-1}(0)$, so that compatible perturbations of all $\s_I:U_I\to \E_I$ generally cannot achieve transversality.

We will avoid these difficulties, and also make a step towards compactness, by reducing the domains of the Kuranishi charts to precompact subsets $V_I\sqsubset U_I$ such that all compatibility conditions between $\bK_I|_{V_I}$ and $\bK_J|_{V_J}$ are given by direct coordinate changes $\Hat\Phi_{IJ}$ or $\Hat\Phi_{JI}$.
The left diagram in Figure~\ref{fig:1} illustrates a typical family of sets $V_I$ for $I\subset\{1,2,3\}$ with the appropriate intersection properties.
As we explain in Remark~\ref{rmk:nerve} this reduction process is analogous to replacing the star cover of a simplicial set by the star cover of its first barycentric subdivision.
This method was introduced in the current context by \cite{LiuT}.

\begin{figure}[htbp] 
   \centering
   \includegraphics[width=4in]{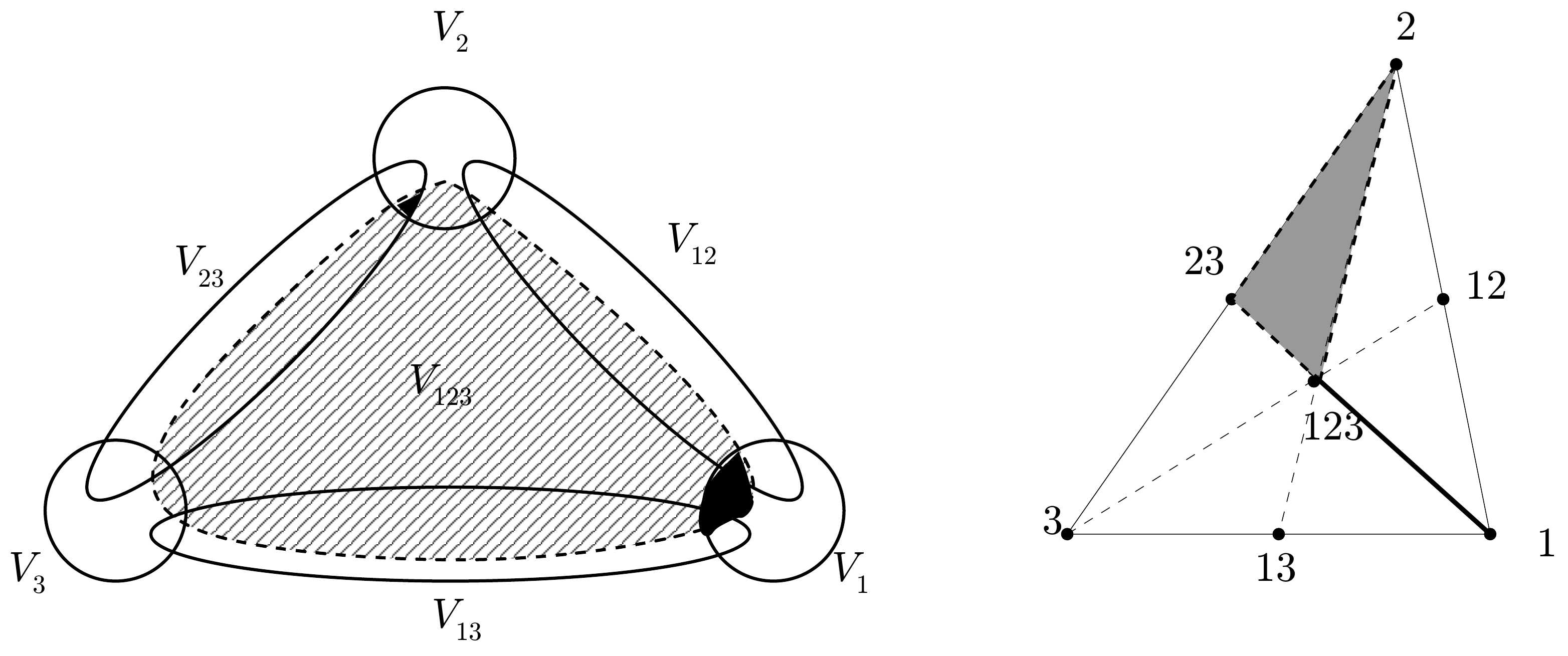}
     \caption{
The right diagram shows the
first barycentric subdivision of the triangle with vertices $1,2,3$.
It has three new vertices labelled $ij$ at the barycenters of the three edges and one vertex labelled $123$ at the barycenter of the triangle.
The left is a schematic picture of a 
reduction of the cover as in Lemma~\ref{le:cov0}.
The black sets are examples of multiple intersections of the new cover, which correspond to the simplices in the barycentric subdivision. E.g.\ $V_2\cap V_{23}\cap V_{123}$ corresponds to the triangle with vertices $2, 23, 123$, whereas $V_1\cap V_{123}$ corresponds to the edge between $1$ and $123$. }  \label{fig:1}
\end{figure}

\begin{rmk}\label{rmk:nerve}\rm
In algebraic topology it is often useful to consider the nerve $\Nn: = \Nn(\Uu)$ of an
open cover $\Uu: = (F_i)_{i=1,\ldots,N}$ of a space $X$, namely the
simplicial complex with one vertex for each open subset $F_i$ and a $k$-simplex for each nonempty intersection of $k+1$ subsets.\footnote
{
A simplicial complex is defined in \cite[\S2.1]{Hat} as a finite set of vertices (or $0$-simplices) $V$ and a subset of the power set $\Ii\subset 2^V$, whose $(k+1)$-element sets are called $k$-simplices for $k\geq 0$. The only requirements are that any subset $\tau \subsetneq \si$ of a simplex $\si\in\Ii$ is also a simplex $\tau\in\Ii$, and that each simplex is linearly ordered, compatible with a partial order on $V$.
In our case the ordering is provided by the linear order on $V=\{1,\dots,N\}$.
Then the $j$-th face of a $k$-simplex $\si: = \{i_0,\dots, i_k\}$, where $i_0<\dots<i_k$,
is given by the subset of $\si$ obtained by omitting its $j$-th vertex $i_j$. This provides the order in which faces are identified when constructing the realization of the simplicial complex.
}
We denote its set of simplices by
$$
\Ii_\Uu: = \bigl\{ I\subset \{1,\ldots,N\} \,\big|\, \cap _{i\in I} F_i\ne \emptyset \bigr\} .
$$
This combinatorial object is often identified with its realization,
the topological space
$$
|\Nn| :=\; \quotient{{\textstyle \coprod_{I\in \Ii_\Uu}} \{I\} \times \De^{|I|-1} }{\sim}
$$
where $\sim$ is the
equivalence relation under which the $|I|$ codimension $1$ faces of the simplex $\{I\}\times\De^{|I|-1}$ are identified with $\{I\less \{i\}\} \times \De^{|I|-2}$ for $i\in I$.
The realization $|\Nn|$ has a natural open cover by the stars $St(v)$ of its vertices $v$,
where $St(v)$ is the union of all (open) simplices whose closures contain $v$.
Notice that the nerve of the star cover of $|\Nn|$ can be identified with~$\Nn$.

Next, let $\Nn_1:=\Nn_1(\Uu)$ be the first barycentric subdivision of $\Nn$.
That is, $\Nn_1$ is a simplicial complex with one vertex $v_I$ at the barycenter of each simplex
$I\in\Ii_\Uu$ and a $k$ simplex for each {\it chain} $I_0\subsetneq I_1\subsetneq\ldots\subsetneq I_k$ of simplices $I_0,\ldots,I_k\in\Ii_\Uu$.
This linear order on each simplex is induced from the partial order on the set of vertices
$\Ii_\Uu$ given by the inclusion relation for subsets of $\{1,\ldots,N\}$.
Further, the star $St(v_I)$ of the vertex $v_I$ in $\Nn_1$ is the union of all simplices given by chains that contain $I$ as one of its elements.
Hence two stars $St(v_I), St(v_J)$ have nonempty intersection if and only if $I\subset J $ or $J\subset I$, because this is a necessary and sufficient condition for there to be a chain containing both $I$ and $J$.
For example in the right hand diagram in Figure~\ref{fig:1} the stars of the vertices $v_{12}$ and $v_{13}$ are disjoint, as are the stars of $v_1$ and $v_2$.
As before the nerve of the star cover of $|\Nn_1|$ can be identified with $\Nn_1$ itself.
In particular, each nonempty intersection of sets in the star cover of $|\Nn_1|$  corresponds to a
simplex in $|\Nn_1|$, namely to  a chain $I_0\subset \ldots \subset I_k$ in the poset $\Ii_\Uu$.
Therefore the indexing set for this cover is the set $\Cc$ of chains in $\Ii_\Uu$; cf.~\cite[p.119ff]{Hat}.

Now suppose that $\Uu = (F_i)_{i=1,\ldots,N}$ is the footprint cover of $X$ by the basic charts of a tame Kuranishi atlas. Then $\Ii_\Uu=\Ii_\Kk$ is the index set of the Kuranishi atlas.
Hence $\Kk$ consists of one basic chart for each vertex of the nerve $\Nn(\Uu)$ and one transition chart $\bK_I$ for each simplex in $\Nn(\Uu)$.  We are aiming to construct from the original cover
$(F_i)$  a {\it reduced} cover $(Z_I)_{I\in \Ii_\Uu}$ whose pattern of intersections mimics that of the star cover of $|\Nn_1(\Uu)|$.
In particular, we will require $\ov{Z_I}\cap \ov{Z_J}= \emptyset$ unless $I\subset J$ or $J\subset I$.
Next, we will aim to construct corresponding subsets $V_I\subset U_I$
with $V_I\cap \s_I^{-1}(0_I)=\psi_I^{-1}(Z_I)$ and $\pi_\Kk(\ov{V_I})\cap \pi_\Kk(\ov{V_J})= \emptyset$ unless $I\subset J$ or $J\subset I$.
We proved in \cite{MW0}  that such a reduction gives rise to a
Kuranishi atlas $\Kk^\Vv$ that has one basic chart for each vertex in $\Nn_1(\Uu)$, i.e.\ for
each element in $\Ii_\Kk$, and one transition chart for each simplex in $\Nn_1(\Uu)$, i.e.\ for each chain $C$ of elements in the poset $\Ii_\Kk$.
However, because such a Kuranishi atlas contains extra structure with no real purpose for us, we work below with the full subcategory of $\bB_\Kk$
with objects $\Vv$. 
$\hfill\er$
\end{rmk}

We will prove the existence of the following type of reduction in Theorem~\ref{thm:red} below.
As always, we denote the closure of a set $Z\subset X$ by $\ov Z$
and write $V\sqsubset U$ to mean that the inclusion $V\hookrightarrow U$ is precompact, i.e.\ the closure $\ov V \subset U$ is compact.
The notions of reductions make sense for general Kuranishi atlases and cobordisms, but we will throughout assume tameness since that is the context in which perturbations will be constructed.

\begin{defn}\label{def:vicin}  
Let $\Kk$ be a tame topological Kuranishi atlas. 
A {\bf reduction} of 
$\Kk$ is an open subset $\Vv=\bigsqcup_{I\in \Ii_\Kk} V_I \subset \Obj_{\bB_\Kk}$ i.e.\ a tuple of (possibly empty) open subsets $V_I\subset U_I$, satisfying the following conditions:
\begin{enumerate}
\item
$V_I\sqsubset U_I $ 
is precompact 
for all $I\in\Ii_\Kk$, and if $V_I\ne \emptyset$ then $V_I\cap \s_I^{-1}(0_I)\ne \emptyset$;
\item
if $\pi_\Kk(\ov{V_I})\cap \pi_\Kk(\ov{V_J})\ne \emptyset$ then
$I\subset J$ or $J\subset I$;
\item
the zero set $\iota_\Kk(X)= |\s_\Kk|^{-1}(|0_\Kk|)$ 
is contained in 
$
\pi_\Kk(\Vv) \;=\; {\textstyle{\bigcup}_{I\in \Ii_\Kk}  }\;\pi_\Kk(V_I).
$
\end{enumerate}
Given a reduction $\Vv$, we define the {\bf reduced domain category} $\bB_\Kk|_\Vv$ and the {\bf reduced obstruction category} $\bE_\Kk|_\Vv$ to be the full subcategories of $\bB_\Kk$ and $\bE_\Kk$ with objects $\bigsqcup_{I\in \Ii_\Kk} V_I$ resp.\ $\bigsqcup_{I\in \Ii_\Kk}  \E_I|_{V_I}$ 
and denote by $\s_\Kk|_\Vv: \bB_\Kk|_\Vv\to \bE_\Kk|_\Vv$ the section given by restriction of $\s_\Kk$. 
\end{defn}

In some ways, which we will further explore in Section~\ref{ss:pert}, the closest we come in this paper to constructing a ``good cover" in the sense of \cite{FO,J1} is the subcategory $\bB_\Kk|_\Vv$ of the category $\bB_\Kk$. Though it is not a Kuranishi atlas, Lemma~\ref{lem:full} below shows that its realization $|\bB_\Kk|_\Vv|$ injects into $|\Kk|$ with image $|\Vv|=\pi_\Kk(\Vv)$, and we proved in \cite{MW0} 
that there is a well defined Kuranishi atlas $\Kk^\Vv$ with virtual neighbourhood $|\Kk^\Vv|\cong |\bB_\Kk|_\Vv|$.

Uniqueness of the virtual fundamental class 
requires a relative notion of reduction.

\begin{defn} \label{def:cvicin}
Let $\Kk$ be a tame topological Kuranishi cobordism. 
Then a {\bf cobordism reduction} of $\Kk$ is an open subset $\Vv=\bigsqcup_{I\in\Ii_{\Kk}}V_I\subset \Obj_{\bB_{\Kk}}$ that satisfies the conditions of Definition~\ref{def:vicin}  
and
is collared in the following sense: 
\begin{enumerate}
\item[(iv)]
For each $\al\in\{0,1\}$ and $I\in 
\Ii_{\p^\al\Kk}\subset\Ii_{\Kk}$ 
there exists $\eps>0$ and a subset $\partial^\al V_I\subset \partial^\al U_I$ such that $\partial^\al V_I\ne \emptyset$ iff 
$V_I \cap \psi_I^{-1}\bigl(
\partial^\al F_I \bigr)\ne \emptyset$,
and 
$$
(\iota^\al_I)^{-1} \bigl( V_I \bigr) \cap \bigl(A^\al_\eps \times  \partial^\al U_I \bigr)
 \;=\; A^\al_\eps \times \partial^\al V_I .
$$
\end{enumerate}
We call 
$\partial^\al\Vv := \bigsqcup_{I\in\Ii_{\p^0\Kk}} \partial^\al V_I \subset \Obj_{\bB_{\p^\al\Kk}}$  
the {\bf restriction} of $\Vv$ to 
$\p^\al\Kk$.
\end{defn}

\begin{remark}\rm 
The restrictions $\partial^\al\Vv$ of a cobordism reduction $\Vv$ of a topological Kuranishi cobordism $\Kk$ are reductions of the restricted topological Kuranishi atlases $\p^\al\Kk$ for $\al=0,1$.
In particular condition (i) holds because part~(iv) of Definition~\ref{def:cvicin} implies that if $\p^\al V_I\ne \emptyset$ then $\p^\al V_I \cap \psi_I^{-1}\bigl( \partial^\al F_I\bigr)\ne \emptyset$.
This is the reason for making a stronger collaring requirement in (iv), which implies that $V_I\subset U_I$ is a collared subset in the sense of Definition~\ref{def:collarset}.
$\hfill\er$
\end{remark}

Before stating the existence and uniqueness result for reductions, we introduce the notion of a nested pair of reductions, which is extensively used both for the control of compactness of perturbed zero sets in Section~\ref{ss:pert}, and for the construction of perturbations in \cite{MW1}.

\begin{definition}\label{def:nest}
Let $\Kk$ be a topological Kuranishi atlas (or cobordism). Then we call a pair of subsets $\Cc,\Vv\subset\Obj_{\bB_\Kk}$ a {\bf nested (cobordism) reduction} if both are (cobordism) reductions of $\Kk$ and $\Cc\sqsubset \Vv$.
\end{definition}

\begin{thm} \label{thm:red}
\begin{enumerate}
\item
Any tame topological Kuranishi atlas $\Kk$ has a unique concordance class of reductions as follows:
There exists a reduction of $\Kk$ in the sense of Definition~\ref{def:vicin}, 
and given any two reductions $\Vv^0,\Vv^1$ of $\Kk$, there exists a cobordism reduction $\Vv$ of $[0,1]\times \Kk$ such that $\p^\al\Vv = \Vv^\al$ for $\al = 0,1$. 
\item
Any tame topological Kuranishi cobordism has a cobordism reduction in the sense of Definition~\ref{def:cvicin}.

\item
For any reduction $\Vv$ of a metric topological Kuranishi atlas $(\Kk,d)$ there exist smaller and larger nested reductions as follows:
\begin{enumerate}
\item
Given any open subset $\Ww\subset |\Kk|$ with respect to the metric topology such that $\io_\Kk(X)\subset\Ww$,  there is a nested reduction $\Cc_\Ww\sqsubset\Vv$ such that $\pi_\Kk(\Cc_\Ww)\subset\Ww$.
\item
There exists $\de>0$ such that $\Vv \sqsubset \bigsqcup_{I\in\Ii_\Kk} B^I_\de(V_I)$ is a nested reduction, and we moreover have
$B_{2\de}^I({V_I})\sqsubset U_I$ for all $I\in\Ii_\Kk$, and for any $I,J\in\Ii_\Kk$
$$
B_{2\de}(\pi_\Kk({V_I}))\cap B_{2\de}(\pi_\Kk({V_J}))
 \neq \emptyset \qquad \Longrightarrow \qquad I\subset J \;\text{or} \; J\subset I. 
$$
\end{enumerate}
\item
For any cobordism reduction $\Vv$ of a metric topological Kuranishi cobordism $(\Kk,d)$
there exist nested reductions with prescribed boundary as follows:
\begin{enumerate}
\item
Let $\Ww\subset |\Kk|$ be a collared subset, open with respect to metric topology, and such that $\io_\Kk(X)\subset\Ww$, and let $\Cc^\al\sqsubset \p^\al \Vv$ for $\al = 0,1$ be nested cobordism reductions of $\p^\al\Kk$ with $\pi_{\p^\al\Kk}(\Cc^\al)\subset \Ww\cap |\p^\al\Kk|$.
Then there is a nested cobordism reduction $\Cc\sqsubset\Vv$ such that $\pi_\Kk(\Cc)\subset\Ww$ and $\p^\al \Cc = \Cc^\al$ for $\al=0,1$.
\item
There exists $\de>0$ such that $\Vv \sqsubset \bigsqcup_{I\in\Ii_\Kk} B^I_\de(V_I)$ is a nested cobordism reduction, and moreover the properties of {\rm (iii)(b)} hold.
\end{enumerate}
\end{enumerate}
\end{thm}

\begin{proof}
The existence and uniqueness of reductions in (i), (ii) are proven in Proposition~\ref{prop:cov2}.
The $\de>0$ enlargement of a given reduction in (iii)(b) and (iv)(b) is proven in Lemma~\ref{le:delred}.
Moreover, Lemma~\ref{le:delred}~(b), (c) construct a nested (cobordism) reduction (with prescribed boundaries), which proves (iii)(a) and (iv)(a) in the case $\Ww=|\Kk|$. 

Given a more general open subset $\Ww\subset(|\Kk|,d)$, we start from a nested (cobordism) reduction $\Cc \sqsubset \Vv$ of $\Kk$, which is provided by Lemma~\ref{le:delred}.
Then we obtain a further precompact subset $\Cc \cap \pi_\Kk^{-1}(\Ww) \sqsubset \Vv$, which is open by Lemma~\ref{le:metric}~(i).
After discarding components $C_I \cap \pi_\Kk^{-1}(\Ww)$ that have empty intersection with $s_I^{-1}(0)$, the resulting subset $\Cc_\Ww \sqsubset \Vv$ forms another nested reduction since $\io_\Kk(X)=|\s_\Kk|^{-1}(0)\subset\pi_\Kk(\Cc)\cap \Ww$.
Finally, if $\Kk$ is a cobordism, then $\Cc_\Ww$ has product form in the boundary collars since $\Ww$ does.
\end{proof}

\subsection{Compactness of perturbed zero sets} \label{ss:pert} \hspace{1mm}\\ \vspace{-3mm}

Given a reduction $\Vv$, the reduced section $\s_\Kk|_\Vv: \bB_\Kk|_\Vv\to \bE_\Kk|_\Vv$ has local zero sets $(\s_\Kk|_{\Vv})^{-1}(0)  := {\textstyle \bigsqcup_{I\in \Ii_\Kk}}(s_I|_{V_I})^{-1}(0) \subset\Obj_{\bB_\Kk}$
which still cover $X$ by the reduction condition (iii), in the sense that 
$\pi_\Kk\bigl( (\s_\Kk|_{\Vv})^{-1}(0) \bigr)= \iota_\Kk(X) \subset |\Kk|$.
In particular that means that the zero set is contained in the image of any other reduction $\Cc$, i.e.\ $\pi_\Kk\bigl( (\s_\Kk|_{\Vv})^{-1}(0) \bigr)\subset\pi_\Kk(\Cc)$.
While $\pi_\Kk(\Cc)\subset|\Kk|$ is rarely an open neighbourhood of $\io_\Kk(X)$, it plays the role of a precompact neigbhourhood in that perturbations of the zero set are constructed in \cite{MW1} to remain contained in $\pi_\Kk(\Cc)$, yielding the following notion of desirable perturbations.

\begin{definition}\label{def:pert}
Let $\Vv$ be a (cobordism) reduction of a tame Kuranishi atlas/cobordism~$\Kk$. 
Then a  {\bf precompact perturbation} of $\s_\Kk|_\Vv$ is a continuous functor $\nu:\bB_\Kk|_\Vv\to\bE_\Kk|_\Vv$ between the reduced domain and obstruction categories such that $\pr_\Kk\circ\nu$ is the identity functor and such that $\pi_\Kk\bigl((\s_\Kk|_\Vv + \nu)^{-1}(0)\bigr)\subset\pi_\Kk(\Cc)$
for some nested reduction $\Cc\sqsubset \Vv$.
That is, $\nu=(\nu_I)_{I\in\Ii_\Kk}$ is given by a family of continuous maps $\nu_I: V_I\to \E_I$ such that \begin{align}\label{eq:comp}
\Hat\Phi_{IJ} \circ \nu_I  \; = \; \nu_J \circ\phi_{IJ} \quad\text{on}\;  V_I\cap \phi_{IJ}^{-1}(V_J) 
\qquad&\forall I,J\in\Ii_\Kk, I\subsetneq J , \\
\label{eq:C}
(s_I|_{V_I}+\nu_I)^{-1}(0) \;\subset\; \pi_\Kk^{-1}\bigl(\pi_\Kk(\Cc)\bigr) \cap V_I \qquad  &\forall I\in \Ii_\Kk.
\end{align}
Its {\bf perturbed zero set} $|\bZ_\nu|$ is the realization of the full subcategory $\bZ_\nu$ of $\bB_\Kk$ with objects
$$
(\s_\Kk|_{\Vv} + \nu)^{-1}(0)  
 := {\textstyle \bigsqcup_{I\in \Ii_\Kk}}(s_I|_{V_I}+\nu_I)^{-1}(0) \;\subset\;\Obj_{\bB_\Kk} . 
$$
That is, we equip
$
|\bZ_\nu| : = \bigl|( \s_\Kk|_{\Vv}  + \nu)^{-1}(0)  
\bigr| \,=\; \quotient{ {\textstyle\bigsqcup_{I\in\Ii_\Kk} (s_I|_{V_I}+\nu_I)^{-1}(0) }}{\!\sim} 
$
with the quotient topology generated by the morphisms of $\bB_\Kk|_\Vv$.
\end{definition}

The key topological property of precompact perturbations is the following sequential compactness of the perturbed zero set, whose use in the construction of the virtual moduli cycle we will explain in Remark~\ref{rmk:vmc} below.

\begin{thm} \label{thm:zeroS0}
Let $\Kk$ be a tame Kuranishi atlas (or cobordism) with a (cobordism) reduction $\Vv$, and 
suppose that $\nu: \bB_\Kk|_\Vv \to \bE_\Kk|_\Vv$ is a precompact perturbation.
Then the realization $|(\s_\Kk|_\Vv+ \nu)^{-1}(0)|$ is a sequentially compact Hausdorff space. 
\end{thm}

The proof of Hausdorffness in this theorem, at the end of this section, will be based on the following comparison of topologies.

\begin{lemma}\label{lem:full}  
Let $\Kk$ be a tame topological Kuranishi atlas with reduction $\Vv$, and suppose that $\bC$ is a full subcategory of the  reduced domain category $\bB_\Kk|_\Vv$. Then the map $|\bC|\to |\Kk|$, induced by the inclusion of object spaces, is
a continuous injection.
In particular, the realization $|\bC|$ is homeomorphic to its image $|\Obj_{\bC}|=\pi_\Kk(\Obj_{\bC})$ with the quotient topology in the sense of Definition~\ref{def:topologies}. 
\end{lemma}
\begin{proof}  
The map $|\bC|\to |\Kk|$ is well defined because $(I,x) \sim_{\bC} (J,y)$ implies $(I,x) \sim_{\bB_\Kk} (J,y)$ since the morphisms in $\bC$ are a subset of those in $\bB_\Kk$.
In order for $|\bC|\to |\Kk|$ to be injective we need to check the converse implication, that is we consider objects $(I,x),(J,y)\in \Obj_{\bC}$, identify them with points $x\in V_I$ and $y\in V_J$, and assume $\pi_\Kk(I,x) = \pi_\Kk(J,y)$. 
Then we have $I\subset J$ or $J\subset I$ by Definition~\ref{def:vicin}~(ii), so that Lemma~\ref{le:Ku2}~(a) implies either $y=\phi_{IJ}(x)$ or $x=\phi_{JI}(y)$.  Since $\bC$ is a full subcategory of $\bB_\Kk|_\Vv$ and hence of $\bB_\Kk$, the corresponding morphism $(I,J,x)$ (or $(J,I,y)$) belongs to $\bC$. 
Hence $(I,x) \sim_{\bB_\Kk} (J,y)$ implies $(I,x) \sim_{\bC} (J,y)$, so that $|\bC|\to |\Kk|$ is injective.
In fact, this shows that the relations $\sim_{\bB_\Kk}$ and $\sim_{\bC}$ agree on $\Obj_{\bC}\subset\Obj_{\bB_\Kk}$, 
and thus $|\bC|\to |\pi_\Kk(\Obj_{\bC})|$ is a homeomorphism with respect to the quotient topology on $\pi_\Kk(\Obj_{\bC})$. Finally, Proposition~\ref{prop:Ktopl1}~(i) asserts that the identity map $|\pi_\Kk(\Obj_{\bC})| \to \|\pi_\Kk(\Obj_{\bC})\|\subset|\Kk|$ is continuous from this quotient topology to the relative topology induced by $|\Kk|$, which finishes the proof.
\end{proof}

\begin{example}\rm 
The inclusion $|\bC| \hookrightarrow |\Kk|$ does {\it not} hold for arbitrary full subcategories of $\bB_\Kk$.  For example, the full subcategory $\bC$ with objects $\bigsqcup_{i=1,\dots, N} U_i$ (the union of the domains of the basic charts) has only identity morphisms, so that $|\bC| = \Obj_{\bC}$ equals $|\Kk|$ only if there are no transition charts.
$\hfill\er$
\end{example}

As explained in Definition~\ref{def:topologies}, the subset $\pi_\Kk(\Vv) \subset|\Kk|$ 
has two different topologies: its quotient topology and the subspace topology.  
If $(\Kk,d)$ is metric, there might conceivably be a third topology, namely that induced by 
restriction of the metric.  
Although we will not use this explicitly, let us show that the metric topology on $\pi_\Kk(\Vv)$ agrees with the subspace topology, so that we only have two different topologies in play.

\begin{lemma}\label{le:Zz}
Let $\Vv$ be a reduction of a metric tame topological Kuranishi atlas 
$(\Kk,d)$. Then the metric topology on $\pi_\Kk(\Vv)$ equals the subspace topology. 
\end{lemma}

\begin{proof}  
Since every reduction $\Vv \subset\Obj_{\bB_\Kk}$ is precompact, the continuity of $\pi_\Kk:\Obj_{\bB_\Kk}\to |\Kk|$ (to $|\Kk|$ with its quotient topology) and of $\id_{|\Kk|}: |\Kk| \to (|\Kk|,d)$ from Lemma~\ref{le:metric} imply that $\pi_\Kk(\ov\Vv)\subset|\Kk|$ is compact in both topologies. Thus the identity map $\id_{\pi_\Kk(\ov\Vv)}: |\Kk|\supset \pi_\Kk(\ov\Vv) \to \bigl(\pi_\Kk(\ov\Vv), d \bigr)$ is a continuous bijection from the compact space $\pi_\Kk(\ov\Vv)$ with the subspace topology to the Hausdorff space $\bigl(\pi_\Kk(\ov\Vv), d \bigr)$ with the induced metric. But this implies that $\id_{\pi_\Kk(\ov\Vv)}$ is a homeomorphism, see Remark~\ref{rmk:hom}, and hence restricts to a homeomorphism $\id_{\pi_\Kk(\Vv)} :  |\Kk|\supset \pi_\Kk(\Vv) \to \bigl(\pi_\Kk(\Vv), d \bigr)$.
Thus, the relative and metric topologies on $\pi_\Kk(\Vv)$ agree.
\end{proof}

Before proving the sequential compactness of precompactly perturbed zero sets, the following remark explains its use in the construction of the virtual moduli cycle.

\begin{remark}{\rm  \label{rmk:vmc}
In an orbibundle context one would expect that any precompact perturbation $\nu$ has closed zero set $\pi_\Kk\bigl((\s_\Kk|_\Vv + \nu)^{-1}(0)\bigr)\subset\pi_\Kk(\Cc)$.
Since $\ov{\pi_\Kk(\Cc)}=\pi_\Kk(\ov\Cc)\subset|\Kk|$ is compact by Proposition~\ref{prop:Ktopl1}~(ii),(iii), this would imply compactness of the zero set.
At the same time, the Hausdorffness of $|\Kk|$ transfers to its subset.
Theorem~\ref{thm:zeroS0} establishes the same for the Kuranishi context, but has to overcome a number of topological obstacles, beginning with an abundance of topologies:
As in Definition~\ref{def:topologies}, we need to differentiate between the quotient topology and the subspace topology on $\pi_\Kk\bigl((\s_\Kk|_\Vv + \nu)^{-1}(0)\bigr)$, which we denote by 
$\bigl|( \s_\Kk|_{\Vv}  + \nu)^{-1}(0)\bigr|$ and $\bigl\|( \s_\Kk|_{\Vv}  + \nu)^{-1}(0)\bigr\|\subset |\Kk|$, respectively. In fact, there are two quotient topologies induced by the morphisms of $\bB_\Kk|_\Vv$ and by $\pi_\Kk$, respectively. If $(\Kk,d)$ is metric, a fourth topology is induced by restriction of the metric.
Fortunately, Lemma~\ref{le:Zz} above shows that the metric and subspace topologies coincide, while Lemma~\ref{lem:full} identifies the quotient topologies induced by the morphisms resp.\ $\pi_\Kk$.
Moreover, the inclusion $(\s_\Kk|_{\Vv} +\nu)^{-1}(0) \subset\Vv = \Obj_{\bB_\Kk|_\Vv}$ induces a continuous injection
\begin{equation}\label{eq:Zinject} 
i_\nu \,:\;  |(\s_\Kk|_\Vv+\nu)^{-1}(0)| \;\longrightarrow\;  \bigl\|( \s_\Kk|_{\Vv}  + \nu)^{-1}(0)\bigr\| \;\subset\; |\Kk| .
\end{equation}
The Hausdorffness of $|\Kk|$ also transfers to the domain of this injection, so the perturbed zero set is equipped with two Hausdorff topologies: the quotient topology on $|(\s_\Kk|_\Vv+\nu)^{-1}(0)|=\qu{(\s_\Kk|_\Vv+\nu)^{-1}(0)}{\sim}$  and the relative topology on $\|(\s_\Kk|_\Vv+\nu)^{-1}(0)\|\subset|\Kk|$.

In a smooth context, transversality of the perturbation will imply local smoothness of the perturbed zero set, though only in the quotient topology, which may contain smaller neighbourhoods than the relative topology.
On the other hand, compactness of the perturbed zero set is easier to obtain in the relative topology than in the quotient topology, which may have more open covers:
One could use the fact that $\|(\s_\Kk|_\Vv+\nu)^{-1}(0)\|\subset\|\Vv\|$ is precompact in $|\Kk|$ by Proposition~\ref{prop:Ktopl1}~(iii), so it would suffice to deduce closedness. This would follow if the continuous map $|\s_\Kk|_\Vv+\nu|:\|\Vv\| \to |\bE_\Kk|_\Vv|$ had a continuous extension to $|\Kk|$ with no further zeros. 
However, such an extension may not exist. In fact, generally $\|\Vv\|\subset |\Kk|$ fails to be open, 
$\io_\Kk(X)\subset |\Kk|$ does not have any precompact neighbourhoods (see Example~\ref{ex:Khomeo}), and even those assumptions would not guarantee an appropriate extension.

So compactness of the perturbed zero set in either topology will not hold in general without further hypotheses on the perturbation that force its zero set to be ``away from the boundary"  of $\pi_\Kk(\Vv)$. 
This is exactly what the notion of precompactness in Definition~\ref{def:pert} accomplishes, and what guarantees the sequential compactness of $|(\s_\Kk|_\Vv+\nu)^{-1}(0)|$.

In the smooth context, $|(\s_\Kk|_\Vv+ \nu)^{-1}(0)|$ will moreover be second countable, so that compactness is equivalent to sequential compactness (see e.g.\ \cite[Theorem~5.5]{Kel}) and hence $|(\s_\Kk|_\Vv+ \nu)^{-1}(0)|$ has a fundamental cycle that represents the virtual moduli cycle.

In that case, the map \eqref{eq:Zinject} then is a continuous bijection between the compact space $|(\s_\Kk|_\Vv+ \nu)^{-1}(0)|$ and the Hausdorff space $\|(\s_\Kk|_\Vv+\nu)^{-1}(0)\|\subset|\Kk|$. By Remark~\ref{rmk:hom}, this in fact is a homeomorphism $|(\s_\Kk|_\Vv+ \nu)^{-1}(0)|\cong \|(\s_\Kk|_\Vv+\nu)^{-1}(0)\|$, so that eventually all four topologies on the perturbed zero set agree.
}
$\hfill\er$
\end{remark}

\begin{proof}[Proof of Theorem~\ref{thm:zeroS0}] 
The continuous injection $|(\s_\Kk|_\Vv+ \nu)^{-1}(0)| \to |\Kk|$ from Lemma~\ref{lem:full} transfers the Hausdorffness of $|\Kk|$ from Proposition~\ref{prop:Khomeo} to the domain $|(\s_\Kk|_\Vv+ \nu)^{-1}(0)|$.

To prove sequential compactness we consider a sequence $(p_k)_{k\in\N}\subset |(\s_\Kk|_\Vv+ \nu)^{-1}(0)|$ (whose subsequences we will index by $k\in\N$ as well).
By finiteness of $\Ii_\Kk$ there is $I\in\Ii_\Kk$ and a subsequence of $(p_k)$ that has lifts in $(s_I|_{V_I}+\nu_I)^{-1}(0)$. 
In fact, by the precompactness assumption \eqref{eq:C}, and using the language of Definition~\ref{def:preceq}, the subsequence lies in
$$
V_I \cap \pi_\Kk^{-1}\bigl(\pi_\Kk(\Cc)\bigr)   \;=\;  V_I \cap {\textstyle \bigcup_{J\in\Ii_\Kk}} \eps_I(C_J)
\;\subset\; U_I .
$$
Here we have $\eps_I(C_J)=\emptyset$ unless $I\subset J$ or $J\subset I$, due to the intersection property (ii) of Definition~\ref{def:vicin} and the inclusion $C_J\subset V_J$.
So we can choose another subsequence and lifts $(x_k)_{k\in\N}\subset V_I$ with $\pi_\Kk(x_k)=p_k$ such that either 
$$
(x_k)_{k\in\N}\subset V_I \cap \phi_{IJ}^{-1}(C_J)
\qquad\text{or}\qquad 
(x_k)_{k\in\N}\subset V_I \cap \phi_{JI}(C_J\cap U_{JI})
$$ 
for some $I\subset J$ or some $J\subset I$.
In the first case, compatibility of the perturbations \eqref{eq:comp} implies that there are other lifts
$\phi_{IJ}(x_k)\in (s_J|_{V_J}+\nu_J)^{-1}(0)\cap C_J$, which by precompactness $\ov{C_J}\sqsubset V_J$ have a convergent subsequence $\phi_{IJ}(x_k)\to y_\infty \in (s_J|_{V_J}+\nu_J)^{-1}(0)$.
Thus we have found a limit point in the perturbed zero set 
$p_k = \pi_\Kk(\phi_{IJ}(x_k)) \to \pi_\Kk(y_\infty) \in |(\s_\Kk|_\Vv+ \nu)^{-1}(0)|$,
as required for sequential compactness.

In the second case we use the relative closedness of $\phi_{JI}(U_{JI})=s_I^{-1}(E_J)\subset U_I$ 
from Lemma~\ref{le:phitrans} and the precompactness $V_I\sqsubset U_I$ to find a convergent subsequence 
$x_k\to x_\infty \in \ov{V_I} \cap \phi_{JI}(U_{JI})$.
Since $\phi_{JI}$ is a homeomorphism to its image, this implies convergence of the preimages
$y_k:= \phi_{JI}^{-1}(x_k) \to \phi_{JI}^{-1}(x_\infty) =: y_\infty \in U_{JI}$.
By construction and compatibility of the perturbations \eqref{eq:comp}, this subsequence $(y_k)$ of lifts of $\pi_\Kk(y_k)=p_k$ lies in $(s_J|_{V_J}+\nu_J)^{-1}(0)\cap C_J$.
Now precompactness $C_J\sqsubset V_J$ implies $y_\infty\in V_J$, and continuity of the section implies $y_\infty\in (s_J|_{V_J}+\nu_J)^{-1}(0)$. Thus we again have a limit point 
$p_k = \pi_\Kk(y_k) \to \pi_\Kk(y_\infty) \in |(\s_\Kk|_\Vv+ \nu)^{-1}(0)|$.
This proves the claimed sequential compactness.
\end{proof}

\subsection{Existence and uniqueness  of reductions} \label{ss:exred} \hspace{1mm}\\ \vspace{-3mm}

In order to prove the existence and uniqueness up to cobordism of (nested) reductions, we start by analyzing the induced footprint cover of $X$.
Since $\pi_\Kk(\Vv)$ contains the zero set $\io_\Kk(X)$, the further conditions on reductions imply that the reduced footprints $Z_I = \psi_I(V_I\cap \s_I^{-1}(0_I))$ form a reduction of the footprint cover $X=\bigcup_{i=1,\ldots,N} F_i$ in the sense of the following lemma.
This lemma makes the first step towards existence of reductions by showing how to reduce the footprint cover. We will use the fact that every compact Hausdorff space is a {\it shrinking space} in the sense that every open cover has a shrinking -- in the sense of Definition~\ref{def:shr0} without requiring condition \eqref{same FI}.

\begin{lemma}\label{le:cov0}
For any finite open cover of a compact Hausdorff
space $X=\bigcup_{i=1,\ldots,N} F_i$ there exists a {\bf cover reduction} $\bigl(Z_I\bigr)_{I\subset \{1,\ldots,N\}}$ in the following sense:
The $Z_I\subset X$ are (possibly empty) open subsets satisfying
\begin{enumerate}
\item
$Z_I\sqsubset F_I
:= \bigcap_{i\in I} F_i$
for all $I$;
\item
if $\ov{Z_I}\cap \ov{Z_J}\ne \emptyset$ then $I\subset J$ or $J\subset I$;
\item
$X\,=\, \bigcup_{I} Z_I$.
\end{enumerate}
\end{lemma}

\begin{proof}
Since $X$ is compact Hausdorff, we may choose precompact open subsets $F_i^0\sqsubset F_i$ that still cover $X$.
Next, any choice of precompactly nested sets
\begin{equation}\label{eq:FGI}
F_i^0\,\sqsubset\, G_i^1 \,\sqsubset\, F_i^1
\,\sqsubset\, G_i^2 \,\sqsubset\,\ldots \,\sqsubset\,
F_i^{N} = F_i
\end{equation}
yields further open covers  
$X=\bigcup_{i=1,\ldots,N} F_i^n$ and $X=\bigcup_{i=1,\ldots,N} G_i^n$ 
for $n=1,\ldots,N$.
Now we claim that the required cover reduction can be constructed by
\begin{equation}\label{eq:ZGI}
Z_I \,: =\; \Bigl( {\textstyle\bigcap_{i\in I}} G_i^{|I|} \Bigr) \;\less\; {\textstyle \bigcup_{j\notin I}} \ov{F^{|I|}_j} .
\end{equation}
To prove this we will use the following notation: Given any open cover $X=\bigcup_{i=1,\ldots,N} H_i$ of $X$, we denote the intersections of the covering sets by $H_I: = \bigcap_{i\in I} H_i$ for all $I\subset \{1,\ldots,N\}$. This convention will apply to define $F_I^k$ resp.\ $G_I^k$ from the $F_i^k$ resp.\ $G_i^k$, but it does not apply to the sets $Z_I$ constructed above, since in particular the $Z_i$ generally do not cover $X$. With this notation we have
$$
Z_I \,: =\; G_I^{|I|} \;\less\; {\textstyle \bigcup_{j\notin I}} \ov{F^{|I|}_j}
\qquad\text{for all}\;\; I\subset \{1,\ldots,N\}.
$$
These sets are open since they are the complement of a finite union of closed sets in the open set $G_I^{|I|}$. The precompact inclusion $Z_I\sqsubset F_I$ in (i) holds since $G_I^{|I|}\sqsubset F_I$.

To prove the covering in (iii)  let $x\in X$ be given. Then we claim that $x \in Z_{I_x}$ for
$$
I_x :=  \underset{I\subset\{1,\ldots,N\}, x \in G^{|I|}_I}{\textstyle \bigcup}  I  \;\;\;\subset\;\;\; \{1,\ldots, N\} .
$$
Indeed, we have $x\in G^{|I_x|}_{I_x}$ since $i\in I_x$ implies $x\in G^{|I|}_i$ for some $|I|\leq |I_x|$, and hence $x\in G^{|I_x|}_i$
since $G_i^{|I|}\subset G_i^{|I_x|}$.
On the other hand, for all $j\notin I_x$ we have $x\notin G^{|I_x|+1}_{I_x\cup j}$ by definition.
However, $x\in G^{|I_x|+1}_{I_x}$ by the nesting of the covers, so for every $j\notin I_x$ we obtain $x\in X\less G^{|I_x|+1}_j$, which is a subset of $X\less \ov{F^{|I_x|}_j}$. This proves $x \in Z_{I_x}$ and hence~(iii).

To prove the intersection property (ii), suppose to the contrary that $x\in \ov{Z_I}\cap\ov{Z_J}$ where $|I|\le |J|$ but $I\less J\ne \emptyset$.  Then given $i\in I\less J$, we have 
$x\in \ov{Z_I} \subset \ov{G^{|I|}_I}\subset F^{|J|}_i$
since $|I|\le |J|$, which contradicts $x\in \ov{Z_J} \subset X\less F^{|J|}_i$.
Thus the sets $Z_I$ form a cover reduction.
\end{proof}

To construct cobordism reductions with given boundary restrictions we need the following notion of collared concordance of cover reductions.

\begin{defn}\label{def:cobred}  
Given a finite open cover $X=\bigcup_{i=1,\ldots,N} F_i$ of a compact Hausdorff space, we say that two {\bf cover reductions $(Z_I^0)_{I\subset \{1,\ldots,N\}}, (Z_I^1)_{I\subset \{1,\ldots,N\}}$  
are collared concordant} if there exists a family of open subsets $Z_I\sqsubset [0,1]\times F_I$ satisfying conditions (i),(ii),(iii) in Lemma~\ref{le:cov0} for the cover $[0,1]\times X=\bigcup_{i=1,\ldots,N} [0,1]\times F_i $,
and in addition are collared in the sense of Definition~\ref{def:collarset}, i.e.\
\begin{enumerate}
\item[(iv)]
There is $\eps>0$ such that $Z_I\cap \bigl(A^\al_\eps\times X\bigr)= A^\al_\eps\times Z_I^\al$
for all $I\subset\{1,\ldots,N\}$ and $\al=0,1$.
\end{enumerate}
\end{defn}

\begin{lemma} \label{le:cobred1}
The collared concordance relation for cover reductions is reflexive, symmetric, and transitive.
\end{lemma}

\begin{proof}
The proof is similar to (but much easier than) that of Lemma~\ref{lem:cobord1}.
\end{proof}

With these preparations we can prove uniqueness of cover reductions up to collared 
concordance, and also provide reductions for footprint covers of Kuranishi cobordisms.

\begin{lemma}\label{le:cobred}
\begin{enumerate}
\item[(a)]
Let 
$(Y,\io^0_Y,\io^1_Y)$ be a compact collared cobordism.
Any cover $Y=\bigcup_{i=1,\ldots,N} F_i$ by collared open sets $F_i\subset  Y$
has a cover reduction $(Z_I)_{I\subset\{1,\ldots,N\}}$  by collared sets $Z_I\subset Y$.
\item[(b)]
Let $X$ be a compact Hausdorff space.
Any two cover reductions $(Z_I^0)_{I\subset\{1,\ldots,N\}}$, $(Z_I^1)_{I\subset\{1,\ldots,N\}}$ of an open cover $X=\bigcup_{i=1,\ldots,N} F_i$ are collared concordant. 
\end{enumerate}
\end{lemma}

\begin{proof}
To prove (a) first note that any finite open cover $Y =\bigcup_{i=1,\ldots,N} F_i$ by 
collared open subsets can be shrunk to sets $F_i'\sqsubset F_i$ that are also collared.
Indeed, taking a common collar width $\eps>0$, one can first choose a shrinking $F^\al_i \sqsubset\partial^\al F_i$ of the covers of the ``boundary components'' $\p^\al Y$ and a general shrinking $F''_i \sqsubset F_i$.  Then we obtain a shrinking with the required collars by setting 
$$
F'_i:=  \io^0_Y\bigl([0,\tfrac \eps 2)\times F^0_i\bigr)  \;\cup\;   \io^1_Y\bigl((1-\tfrac \eps 2,1]\times F^1_i \bigr) \;\cup\;
\bigl( F''_i \less \bigl(
 \io^0_Y\bigl([0,\tfrac \eps 4]\times \p^0 F_i\bigr) 
\cup  
\io^1_Y\bigl([1-\tfrac \eps 2,1]\times \p^1 F_i \bigr)
\bigr)
 \bigr) .
$$
Hence we may choose nested covers $F_i^0\ldots F_i^k\sqsubset G_i^{k+1}\sqsubset \ldots F_i$ as in \eqref{eq:FGI} of $Y$ that 
consist of collared open subsets.
Then the sets $(Z_I)$ defined by intersections in \eqref{eq:ZGI} 
are also collared, and the arguments of Lemma~\ref{le:cov0} prove (a).

To prove (b), we first transfer to the standard form constructed in Lemma~\ref{le:cov0}.

\MS
\noindent
{\bf Claim:} {\em
Any cover reduction $(Z_I)$ of a finite open cover $X=\bigcup_i F_i$ is collared concordant to a cover reduction constructed from nested covers $F_i^0\ldots F_i^k\sqsubset G_i^{k+1}\sqsubset \ldots F_i$
by \eqref{eq:ZGI}.
}
\MS

To prove this claim, choose a shrinking $Z_I^0\sqsubset Z_I$ such that $X = \bigcup_I Z^0_I$.  Then
these covers induce precompactly nested open covers
$$
F_i^0: = {\textstyle \bigcup_{i\in I}} Z_I^0  \quad \sqsubset \quad F_i': = {\textstyle \bigcup_{i\in I}} Z_I   \quad \sqsubset \quad F_i = {\textstyle \bigcup_{i\in I}} F_I .
$$
As in \eqref{eq:FGI} we can choose interpolating sets
$$
F_i^0\sqsubset \ldots F_i^k\sqsubset G_i^{k+1}\sqsubset \ldots \sqsubset F_i^{2N} = F_i',
$$
and let $\bigl(Z_I' := G_I^{|I|}\less \bigcup_{j\notin I} \ov{F_j^{|I|}}\bigr)$ be the resulting cover reduction of $F_i'$ and hence of $F_i$.
Then we will show
that the union $(Z_I'':= Z^0_I\cup Z'_I)$ is a cover reduction of $(F_i)$ as well.
Since $Z_I^0\sqsubset Z_I\sqsubset F_I$ we only have to check the mixed terms in the intersection axiom 
(ii)  in Lemma~\ref{le:cov0}, i.e.\ we need to verify
$$
\Bigl(J\less I\ne
\emptyset,\; I\less J\ne \emptyset\Bigr)\; \Longrightarrow\;
\Bigl((\ov{Z^0_I}\cup \ov{Z'_I})\cap (\ov{Z^0_J}\cup \ov{Z'_J}) = \emptyset\Bigr).
$$
Indeed, for $j\in J\less I$ we obtain $Z^0_J\subset F_j^0\sqsubset F_j^{|I|}\subset X\less Z'_I$
so that $\ov{Z^0_J}\cap \ov{Z_I'}=\emptyset$.  Conversely, $\ov{Z^0_I}\cap \ov{Z_J'}=\emptyset$ follows from the existence of $i\in I\less J$.
Now we have a chain of inclusions between cover reductions of $(F_i)$, namely $Z_I^0\subset Z_I$, $Z_I^0\subset Z_I''$, and  $Z_I'\subset Z_I''$. 
We will show that this induces collared concordances $(Z_I^0)\sim(Z_I)$, $(Z_I^0)\sim(Z_I'')$, and $(Z_I')\sim(Z_I'')$, so that Lemma~\ref{le:cobred1} implies that $(Z_I)$ is collared concordant to $(Z_I')$, which is constructed by \eqref{eq:ZGI}.
To show this, 
consider any two cover reductions $Z_I'\subset Z_I''$ of $(F_i)$ and note that a collared 
concordance of reductions as in Definition~\ref{def:cobred}
 is given by
$$
 \bigl( [0,\tfrac 23)\times Z_I'\bigr)\cup \bigl((\tfrac 13,1] \times Z_I''\bigr)  \;\subset\;  [0,1]  \times X.
$$
Indeed, these open subsets are collared and form a cover reduction since each of $(Z_I'),(Z_I'')$ satisfies the axioms (i),(ii), and the mixed intersection in (iii) is
$$
\Bigl([0,\tfrac 23]\times \ov{Z_I'} \Bigr)\cap \Bigl( [\tfrac 13,1]\times \ov{Z_J''}\Bigr)
\subset
\Bigl([\tfrac 13,\tfrac 23] \times \ov{Z_I'} \cap \ov{Z_J''} \Bigr) ,
$$
which is empty unless $I\subset J$ or $J\subset I$. This proves the Claim.
\MS

To prove (b), 
note that by the above Claim and Lemma~\ref{le:cobred1} it now suffices
to consider cover reductions $(Z^\al_I)$ that are constructed from nested covers $F_i^{0,\al} \ldots F_i^{k,\al} \sqsubset G_i^{k+1,\al}\sqsubset \ldots\partial^\al F_i$ as in \eqref{eq:FGI}.
We extend these to nested 
reductions 
of the product cover $[0,1]  \times X=\bigcup_i [0,1]  \times F_i$ by choosing constants
$$
\tfrac 13=\eps_{2N}<\ldots<\eps_{0} < \tfrac 12<\de_{0}<\ldots<\de_{2N} = \tfrac 23,
$$
and then setting
\begin{align*}
F_i^{k}: & = \Bigl( [0,\de_{2k})\times F_i^{k,0}\Bigr) \cup  \Bigl( (\eps_{2k}, 1]\times F_i^{k,1}\Bigr)
\;\sqsubset\; 
 [0,1] \times F_i ,\\
G_i^{k}: & = \Bigl([0,\de_{2k-1})\times G_i^{k,0}\Bigr) \cup \Bigl( (\eps_{2k-1}, 1] \times  G_i^{k,1}\Bigr)\;\sqsubset\;  [0,1]  \times F_i.
\end{align*}
Since these sets satisfy the nested property in \eqref{eq:FGI} and 
are collared subsets with boundaries 
given by the nested covers $G_i^{k,\al}\sqsubset F_i^{k,\al}$, the cover reduction $(Z_I)$ defined as in \eqref{eq:ZGI} is a collared concordance between the reductions $(Z^0_I)$ and $(Z^1_I)$.
\end{proof}

We are now in a position to prove existence and uniqueness of reductions in Theorem~\ref{thm:red}.

\begin{prop} \label{prop:cov2} \hspace{1mm}
\vspace{-1mm}
\begin{itemize}
\item[(a)]
Every tame topological Kuranishi atlas 
has a reduction.
\item[(b)] 
Every tame topological Kuranishi cobordism 
has a cobordism reduction. 
\item [(c)]  
Let $\Vv^0,\Vv^1$ be reductions of a tame topological Kuranishi atlas $\Kk$. Then there exists a cobordism reduction $\Vv$ of $[0,1]\times \Kk$ such that $\p^\al\Vv = \Vv^\al$ for $\al = 0,1$. 
\end{itemize}
\end{prop}

\begin{proof} 
For (a) we begin by using Lemma~\ref{le:cov0} to find a cover reduction 
$(Z_I)_{I\subset \{1,\ldots,N\}}$ of the footprint cover $X=\bigcup_{i=1,\ldots,N} F_i$
of a given Kuranishi atlas $\Kk$. Since $Z_I\subset F_I=\emptyset$ for $I\notin\Ii_\Kk$, we can index the potentially nonempty sets in this cover reduction by $(Z_I)_{I\in\Ii_\Kk}$.
Then Lemma~\ref{le:restr0} provides precompact open sets $W_I\sqsubset U_I$ for each $I\in\Ii_\Kk$
with $Z_I\ne \emptyset$, satisfying
\begin{equation}\label{eq:wwI}
W_I\cap \s_I^{-1}(0_I) = \psi_I^{-1}(Z_I),\qquad
\ov{W_I}\cap \s_I^{-1}(0_I) = \psi_I^{-1}(\ov{Z_I}).
\end{equation}
The set $\Vv=(W_I)_{I\in\Ii_\Kk}$ now satisfies condition (iii) in Definition~\ref{def:vicin}, namely $\bigcup_I \pi_{\Kk}(W_I)$ contains $\bigcup_I \pi_{\Kk}\bigl(\psi_I^{-1}(Z_I)\bigr)$, which covers $\iota_\Kk(X)$.
We will construct the reduction by choosing $V_I \subset W_I$ so that (ii) is satisfied, while the intersection with the zero set does not change, i.e.\
\begin{equation}\label{eq:zeroV}
V_I\cap \s_I^{-1}(0_I) = \psi_I^{-1}(Z_I),
\end{equation}
which guarantees (iii).
Further, condition (i) holds automatically since $V_I\subset W_I$,
and we define $V_I:=\emptyset$ when $Z_I=\emptyset$.
To begin the construction of the $V_I$, define 
$$
\Cc(I):=\{J\in \Ii_\Kk \,|\, I\subset J  \;\text{or}\; J\subset I \},
$$
and  for each $J\notin\Cc(I)$ define 
$$
Y_{IJ} \,:= \ov{W_I}\cap \pi_\Kk^{-1}(\pi_\Kk(\ov{W_J}))
\; = \; 
U_I \cap \pi_\Kk^{-1}\bigl(\pi_\Kk(\ov{W_I})\cap \pi_\Kk(\ov{W_J})\bigr)
$$
which is closed by the continuity of $\pi_\Kk$ and the fact that each set $\pi_\Kk(\ov{W_I}), \pi_\Kk(\ov{W_J})$ is compact and hence closed by the Hausdorff property of $|\Kk|$.
If $J\notin\Cc(I)$, then using $\psi_I^{-1}(\ov{Z_I}) =\s_I^{-1}(0_I)\cap \ov{W_I}$ and the notation $\eps_I$ from Definition~\ref{def:preceq} we obtain
\begin{align*}
\psi_I^{-1}(\ov{Z_I}) \cap Y_{IJ}&\;\subset
\;\psi_I^{-1}(\ov{Z_I})\cap \s_I^{-1}(0_I)\cap 
\eps_I(\ov{W_J}) \\
&\;=\; \psi_I^{-1}(\ov{Z_I}) \cap \eps_I\bigl(\s_J^{-1}(0_J) \cap \ov{W_J} \bigr)\\
&\;= \; \psi_I^{-1}(\ov{Z_I})\cap \eps_I\bigl(\psi_J^{-1}(\ov{Z_J})\bigr)
\;= \; \psi_I^{-1}(\ov{Z_I}) \cap \psi_I^{-1}(\ov{Z_J})\;=\; \emptyset,
\end{align*}
where the first equality holds because $\s$ is compatible with the coordinate changes.
The inclusion $Y_{IJ}\subset W_I\sqsubset U_I$ moreover ensures that $Y_{IJ}$ is compact, 
so has a nonzero Hausdorff distance from the closed set $\psi_I^{-1}(\ov{Z_I})$.
Thus we can find closed neighbourhoods 
$\Nn(Y_{IJ})\subset U_I$ of $Y_{IJ}$
for each $J\notin\Cc(I)$ such that  
$$
\Nn(Y_{IJ})\cap \psi_I^{-1}(\ov{Z_I})= \emptyset.
$$
We now claim that we obtain a reduction by removing these neighbourhoods,
\begin{equation}\label{eq:QIVI}
V_I := W_I \;\less {\textstyle\bigcup_{J\notin \Cc(I)}} \Nn(Y_{IJ}).
\end{equation}
Indeed, each $V_I\subset W_I$ is open, and 
$V_I\cap \s_I^{-1}(0_I) = \psi_I^{-1}(Z_I) \;\less \bigcup_{J\notin \Cc(I)} \Nn({Y_{IJ}}) = \psi_I^{-1}(Z_I)$ 
by construction of $\Nn({Y_{IJ}})$.
Moreover, because $\ov{V_I}\subset \ov{W_I}$ for all $I$, 
we have for all $J\notin\Cc(I)$
$$
\ov{V_I}\cap \pi_\Kk^{-1}(\pi_\Kk(\ov{V_J}))\;\subset\; \ov{V_I}\cap \ov{W_I}\cap \pi_\Kk^{-1}(\pi_\Kk(\ov{W_J}))
\;\subset\; {\ov{V_I}}\cap  Y_{IJ} \;=\; \emptyset .
$$
This completes the proof of  (a).

To prove (b) we use Lemma~\ref{le:cobred}~(a) to obtain a  collared cover reduction $(Z_I)_{I\in \Ii_{\Kk^{[0,1]}}}$ of the footprint cover of a given Kuranishi cobordism $\Kk^{[0,1]}$, which is collared by Remark~\ref{rmk:Ceps}.
Then the arguments for (a) also show that $\Kk^{[0,1]}$ has a reduction $(V_I')_{I\in \Ii_{\Kk^{[0,1]}}}$.
It remains to adjust it to achieve
the collaring property in Definition~\ref{def:cvicin}~(iv).
For that purpose note that the footprint of the reduction is given by the cover reduction, that is \eqref{eq:zeroV} provides the identity
$$
V'_I\cap \s_I^{-1}(0_I) = \psi_I^{-1}(Z_I) \qquad\forall I\in \Ii_{\Kk^{[0,1]}} .
$$
In particular, if 
$\p^\al Z_I\ne \emptyset$ then $\p^\al V'_I\ne \emptyset$,
though the converse may not hold.  
Now choose $\eps>0$ less or equal to half the collar width of $\Kk^{[0,1]}$ in Remark~\ref{rmk:Ceps} and so that condition (iv) in Definition~\ref{def:cobred} holds with $A^\al_{2\eps}$
for the footprint reduction $(Z_I)_{I\in \Ii_{\Kk^{[0,1]}} }$, 
in particular
\begin{equation}\label{eq:Zcoll}
\psi_I^{-1}(Z_I) \cap \io^\al_I\bigl(A^\al_{2\eps}\times \p^\al U_I \bigr)
\;=\;
\psi_I^{-1}( A^\al_{2\eps}\times \partial^\al Z_I  )
\qquad \forall \al\in\{0,1\}, I\in \Ii_{\Kk^\al} .
\end{equation}
Then we set $V_I:=V'_I$ for interior charts 
$I\in \Ii_{\Kk^{[0,1]}}\less(\Ii_{\Kk^0}\cup \Ii_{\Kk^1})$.
If $I\in\Ii_{\Kk^\al}$ for $\al=0$ or $\al=1$ (or both) 
we define $V^\al_I \subset \p^\al U^\al_I$ so that, with $\eps_0: = \eps$ and $\eps_1:=1-\eps$, 
$$
\io^\al_I\bigl(\{\eps_\al\}\times  V^\al_I  \bigr) 
= \left\{ \begin{array}{ll} 
V_I' \cap \io^\al_I \bigl( \{\eps_\al\}\times \p^\al U_I \bigr) & \mbox{ if } \p^\al Z_I\ne \emptyset,\\
\emptyset &  \mbox{ if } \p^\al Z_I = \emptyset.
\end{array}\right.
$$
Since 
$(\p^\al Z_I)_{I\in\Ii_{\Kk^\al}}$ is a cover reduction of the footprint cover of $\Kk^\al$, and 
$$
\psi_I^{-1}(Z_I) \cap \io^\al_I \bigl( \{\eps_\al\}\times \p^\al U_I \bigr) = \psi_I^{-1}(\{\eps_\al\}\times \partial^\al Z_I ),
$$ this defines reductions 
$(V^\al_I)_{I\in\Ii_{\Kk^\al}}$ of $\Kk^\al$.
With that, we obtain collared subsets of $U_I$ for each 
$I\in \Ii_{\Kk^0}\cup \Ii_{\Kk^1}$ by
\begin{equation}  \label{eq:Vcoll}
V_I:= \Bigl( \, V_I' \;\less  {\textstyle \bigsqcup_{\al=0,1}\,} \io^\al_I\bigl( \ov{A^\al_\eps}\times \p^\al U_I
\,\bigr) \Bigr) 
\;\cup\;
 {\textstyle \bigsqcup_{\al=0,1}\,} \io^\al_I\bigl( \ov{A^\al_\eps}\times V^\al_I\bigr) \quad\subset\; U_I .
\end{equation}
These are open subsets of $U_I$ because, firstly, each $ \ov{A^\al_\eps}\times \p^\al U_I$ is a relatively closed subset of the domain of the embedding $\io^\al_I$.
Secondly, each point in the boundary 
of $\io^\al_I\bigl(\ov{A^\al_\eps}\times V^\al_I\bigr)\subset U_I$ is of the form $\io^\al_I(x,\eps_\al)$ and, since the collar width is $2\eps$, has neighbourhoods $\io^\al_I\bigl( (\eps_\al -\eps, \eps_\al + \eps)\times \Nn_x  \bigr)\subset V_I$ for any neighbourhood $\Nn_x\subset V^\al_I$ of $x$.
Moreover, $V_I$ has the same footprint as $V_I'$, since adjustment only happens on the collars $\iota^\al_I(\ov{A^\al_\eps}\times \p^\al U_I)$, where the footprint is of product form, thus preserved by the construction.
Therefore the sets $(V_I)_{I\in \Kk^{[0,1]}}$ satisfy conditions (i) and (iii) in Definition~\ref{def:vicin}.
They also satisfy (ii) because, in the notation of Remark~\ref{rmk:cobordreal}, we can check this separately in the collars 
$\io^\al_{|\Kk|}({\ov{A^\al_\eps}}\times |\Kk^\al|)$ of $|\Kk^{[0,1]}|$ (where it holds because $(V^\al_I)$ is a reduction) and in their complements (where it holds because $(V'_I)$ is a reduction). 

Finally, condition (iv) in Definition~\ref{def:cvicin} holds by construction, namely 
$(\iota^\al_I)^{-1} \bigl( V_I \bigr) \cap \bigl(  A^\al_\eps \times \partial^\al U_I \bigr) = A^\al_\eps \times  V_I^\al $, and the condition $\bigl( \; \partial^\al V_I\ne \emptyset \; \Rightarrow \; V_I \cap \psi_I^{-1}\bigl( \{\al\} \times \partial^\al F_I \bigr)\ne \emptyset \;\bigr)$ holds since
$\partial^\al V_I = V_I^\al \ne \emptyset$ implies $\p^\al Z_I \neq \emptyset$ by definition of $V_I^\al$.
This completes (b).

To prove (c) we will use 
reflexivity and transitivity for cobordism reductions of the product cobordism $[0,1]\times \Kk$. 
More precisely, given a cobordism reduction $\Vv$ of $[0,1]\times \Kk$, we can apply the isomorphism on $[0,1]\times \Obj_{\bB_\Kk}$ that reverses the inverval $[0,1]$ to turn $\Vv$ into a cobordism reduction $\Vv'$ with $\p^0\Vv'=\p^1\Vv$ and $\p^1\Vv'=\p^0\Vv$.
Similarly, given two cobordism reductions $\Vv$ and $\Vv'$ of $[0,1]\times \Kk$ with $\p^1\Vv=\p^0\Vv'$ we obtain a reduction $\Vv''\subset [0,1]\times \Obj_{\bB_\Kk}$ with $\p^0\Vv''=\p^0\Vv$ and $\p^1\Vv''=\p^1\Vv'$ by gluing and rescaling 
$$
\Vv'' := \bigl\{ (\tfrac 12 s,x) \,\big|\, (s,x)\in \Vv \bigr\} \cup  \bigl\{ (\tfrac 12 (1+s),x) \,\big|\, (s,x)\in \Vv' \bigr\}.
$$
Just as in the proof of filtration in Lemma~\ref{lem:cobord1}, the resulting sets satisfy the separation condition (ii) of Definition~\ref{def:vicin} because the reductions $\Vv, \Vv'$, and $\partial^1\Vv=\p^0\Vv'$ do.
Based on this, we will prove (c) in several stages.

\MS\NI
{\bf Step 1:} {\it The result holds if $\Vv^0 \subset \Vv^1$ and  $V^0_I\cap \s_I^{-1}(0_I) = V^1_I\cap \s_I^{-1}(0_I)$ for all $I\in \Ii_\Kk$.
}

\smallskip\NI
In this case the footprints $Z_I := \psi_I\bigl(V^\al_I\cap \s_I^{-1}(0_I)\bigr)$ are the same for $\al=0,1$ by assumption.
Hence the sets for $I\in \Ii_\Kk$ 
$$
V_I \, 
 : =\; 
\Bigl( [0,\tfrac 23)\times  V^0_I\Bigr) \cup\Bigl((\tfrac 13,1]\times  V^1_I \Bigr)
\;\subset\;  [0,1]\times U_I 
$$
form the required cobordism reduction. 
In particular, they satisfy the intersection condition (ii) over the interval $(\tfrac 13,\tfrac 23)$ because $V^0_I\subset V^1_I$ for all $I$. 
Their footprints are $[0,1]\times Z_I$ and so cover $[0,1]\times X$, and they satisfy the collar form requirement (iv) because $V_I\ne \emptyset$ iff $V^1_I\ne \emptyset$, and the latter implies $Z_I\ne \emptyset$  by condition (i) for $\Vv^1$.

\MS\NI
{\bf Step 2:} {\it The result holds if all footprints coincide, i.e.\ $V^0_I\cap \s_I^{-1}(0_I) = V^1_I\cap \s_I^{-1}(0_I)$.
}

\smallskip\NI
Note that $\Vv^{01}:= \bigcup_{I\in\Ii_\Kk} V^0_I\cap V^1_I$ is another reduction of $\Kk$ since it has the common footprints $Z_I$ as above, thus covers $\io_\Kk(X)$.
So Step 1 for $\Vv^{01}\subset\Vv^\al$ (together with reflexivity for $\al=0$) provides cobordism reductions $\Vv$ and $\Vv'$ of $[0,1]\times \Kk$ with $\p^0\Vv=\Vv^0$, $\p^1\Vv=\Vv^{01}=\p^0\Vv'$, and $\p^1\Vv'=\Vv^1$.
Now transitivity provides the required reduction with boundaries $\Vv^0,\Vv^1$.

\MS\NI
{\bf Step 3:} {\it The result holds for all reductions $\Vv^0,\Vv^1\subset\Obj_{\bB_\Kk}$.
}

\smallskip\NI
First use Lemma~\ref{le:cobred}~(b) to obtain a family of collared sets $Z_I\sqsubset  [0,1]\times F_I$ for $I\in\Ii_\Kk$ that form a cover reduction of $[0,1]\times X=\bigcup_{i=1,\ldots,N} [0,1] \times F_i $ and restrict to the cover reductions $\partial^\al Z_I =  \psi_I(V_I^\al \cap \s_I^{-1}(0_I))$ induced by the reductions $\Vv^\al$ for $\al=0,1$.
Next, as in the proof of (b), we construct a cobordism reduction $\Vv'$ of $[0,1]\times \Kk$ with footprints $Z_I$. 
Its restrictions $\partial^\al\Vv'$ are reductions of $\Kk^\al$ with the same footprint $\partial^\al Z_I$ as $\Vv^\al$  for $\al=0,1$.
Now Step 2 provides further cobordism reductions $\Vv$ and $\Vv''$ with $\Vv^0=\p^0\Vv$, $\p^1\Vv=\p^0\Vv'$, $\p^1\Vv'=\p^0\Vv''$, and $\p^1\Vv''=\Vv^1$, so that another transitivity construction provides the required cobordism reduction with boundaries $\Vv^0,\Vv^1$.
\end{proof}

\begin{rmk}\label{rmk:disjoint}  \rm 
If $\Kk$ is also preshrunk and hence metrizable, then the above construction of a reduction $\Vv = (V_I)_{I\in \Ii_\Kk}$ with footprints equal to a given cover reduction $(Z_I)_{I\in \Ii_\Kk}$ can be improved to satisfy the additional condition 
 $$
  \ov Z_I\cap \ov Z_J = \emptyset \quad \Longrightarrow\quad \pi_\Kk(\ov{V_I}) \cap \pi_\Kk(\ov{V_J})  = \emptyset.
  $$
To see this, choose an admissible  metric on $|\Kk|$ and choose $\de>0$ so that each pair of disjoint sets
$\io_\Kk(\ov Z_I), \io_\Kk(\ov Z_J)$ in $|\Kk|$ has disjoint $\de$-neighbourhoods.  Then choose the initial neighbourhood $W_I\subset U_I$ of $\psi_I^{-1}(Z_I)$ to lie in
$B^I_{\de/2}(\psi_I^{-1}(Z_I))$ so that $\ov{W_I}\subset B^I_{\de}(\psi_I^{-1}(\ov{Z_I}))$.   Because $\pi_\Kk$ is an isometry, we have
 $$
 \pi_\Kk(\ov{W_I})\cap \pi_\Kk(\ov{W_J}) \subset B_{\de}(\io_\Kk(\ov Z_I))\cap B_{\de}(\io_\Kk(\ov Z_J)),
 $$
 which is  empty if $\ov Z_I, \ov Z_J$ are disjoint.
Since $V_I\subset W_I$, the conclusion follows.
\hfill$\er$
\end{rmk}

Our last task is to establish the relevant existence and uniqueness results for nested cobordism reductions. 
Given one reduction $\Vv$, the smaller reductions $\Cc\sqsubset\Vv$ in (b),(c) below are used to control compactness of the perturbed zero set, as described in Section~\ref{ss:pert}. The larger reductions ``$\Vv\sqsubset B_\de(\Vv)$'' are used in \cite{MW1} for the construction of transverse perturbations by iteration over a hierarchy of charts along with nested reductions 
$$
\Cc \;\sqsubset\; \Vv   \;\sqsubset\;  \ldots \;\sqsubset\; B_{2^{-k}\de}(\Vv) \;\sqsubset\;  \ldots  \;\sqsubset\;  B_{\de}(\Vv) .
$$

\begin{lemma}\label{le:delred} 
Let  $\Vv$ be a (cobordism) reduction of a metric topological Kuranishi atlas (or cobordism) $\Kk$.
\begin{enumerate}
\item[(a)]
There exists $\de>0$ such that $\Vv \sqsubset \bigsqcup_{I\in\Ii_\Kk} B^I_\de(V_I)$ is a nested (cobordism) reduction, and we moreover have
$B_{2\de}^I({V_I})\sqsubset U_I$ for all $I\in\Ii_\Kk$, and for any $I,J\in\Ii_\Kk$
$$
B_{2\de}(\pi_\Kk({V_I}))\cap B_{2\de}(\pi_\Kk({V_J}))
 \neq \emptyset \qquad \Longrightarrow \qquad I\subset J \;\text{or} \; J\subset I. 
$$
\item[(b)] 
If $\Vv$ is a reduction of a Kuranishi atlas, then there is a nested reduction $\Cc\sqsubset \Vv$.
\item[(c)]
If $\Vv$ is a cobordism reduction of the Kuranishi cobordism $\Kk$, and if $\Cc^\al\sqsubset \p^\al \Vv$ for $\al = 0,1$ are nested cobordism reductions of the boundary restrictions $\p^\al\Kk$, then there is a nested cobordism reduction $\Cc\sqsubset \Vv$ such that $\p^\al \Cc = \Cc^\al$ for $\al=0,1$.
\end{enumerate}
\end{lemma}

\begin{proof}
To prove (a) for a Kuranishi atlas $\Kk$  we need to find $\de>0$ so that
\begin{itemize} 
\item[(I)]
$\displaystyle\phantom{\int\!\!\!\!\!\!}  
B^I_{2\de}(V_I)\sqsubset U_I$ for all $I\in\Ii_\Kk$;
\item[(II)]
if $B_{2\de}(\pi_\Kk({V_I})) \; \cap \; B_{2\de}(\pi_\Kk({V_J})) \ne \emptyset$ then $I\subset J$ or $J\subset I$.
\end{itemize}
The latter implies the separation condition (ii) in Definition~\ref{def:vicin}, due to the 
inclusion $\pi_\Kk(\ov{B^I_\de(W)})\subset B_{2\de}(\pi_\Kk(W))$ by compatibility of the metrics for any $W\subset U_I$.
The other conditions $B^I_\de(V_I)\cap \s_I^{-1}(0_I)\neq\emptyset$ and $\io_\Kk(X)\subset\pi_\Kk(\Aa)$ for $\Aa := \bigsqcup_{I\in\Ii_\Kk} B^I_\de(V_I)\sqsubset\Obj_{\bB_\Kk}$ 
to be a reduction then follow directly from the inclusion $\Vv\subset\Aa$.
Finally, the inclusion $\Vv\sqsubset\Aa$ is precompact since each component $V_I\sqsubset B^I_\de(V_I)$ is precompact.

In order to obtain cobordism reductions from this construction, recall that, by definition, a metric Kuranishi cobordism carries a product metric in the collar neighbourhoods $\iota^\al_I( A^\al_\eps\times \partial^\al U_I )\subset U_I$ of the boundary, which ensures that for $\de<\frac \eps 2$ the $\de$-neighbourhood of a collared set (in this case $V_I$ with $(\iota^\al_I)^{-1}(V_I)=A^\al_\eps\times \partial^\al V_I $) is $\frac\eps 2$-collared. 
More precisely, with $B^{\p^\al I}_\de(W)\subset \p^\al U_{I}$ denoting balls in the domains of $\p^\al\Kk$ we have
$$
B^I_\de(V_I) \cap \iota^\al_I\bigl(A^\al_{\frac \eps 2}\times \partial^\al U_I \bigr) \;=\; \iota^\al_I\bigl( 
A^\al_{\frac\eps 2}\times B^{\p^\al I}_\de(\partial^\al V_I)  \bigr) .
$$
So it remains to find $\delta>0$ satisfying (I) and (II).
Property (I) for sufficiently small $\delta>0$ follows from the precompactness $V_I\sqsubset U_I$ in Definition \ref{def:vicin}~(i) and a covering argument based on the fact that, in the locally compact 
space $U_I$, every $p\in\ov{V_I}$ has a compact neighbourhood $\ov{B^I_{\delta_p}(p)}$ for some ${\delta_p>0}$.
To check (II) recall that by Definition~\ref{def:vicin} the subsets $\pi_\Kk(\ov{V_I})$ and $\pi_\Kk(\ov{V_J})$ of $|\Kk|$ are disjoint unless $I\subset J$ or $J\subset I$. Since each $\pi_\Kk|_{U_I}$ maps continuously to the quotient topology on $|\Kk|$, and the identity to the metric topology is continuous by Lemma~\ref{le:metric}, the $\pi_\Kk(\ov{V_I})$ are also compact subsets of the metric space 
$(|\Kk|,d)$. 
Hence (II) is satisfied if we choose $\de>0$ less than quarter the distance between each disjoint pair $\pi_\Kk(\ov{V_I}),\pi_\Kk(\ov{V_J})$.

For (b) choose any shrinking $(Z_I')_{I\in \Ii_\Kk}$ of the footprint cover $\bigl(Z_I = \psi_I(V_I\cap \s_I^{-1}(0_I))\bigr)_{I\in \Ii_\Kk}$ of $\Vv$ as in Definition~\ref{def:shr0}.  By Lemma~\ref{le:restr0} there are open subsets $C_I\sqsubset V_I$ such that $C_I\cap \s_I^{-1}(0_I)  = \psi_I^{-1}(Z_I')$.  
This guarantees $\io_\Kk(X)\subset \pi_\Kk(\Cc)$, and  since the $(V_I)$ satisfy the separation condition (ii) in Definition~\ref{def:vicin}, so do the sets $(C_I)$.

Note that the statement of (b) also holds if $\Kk$ is a Kuranishi cobordism and we require that $\Cc$ be collared.   
For this construction we must start with a collared shrinking $(Z_I')_{I\in \Ii_\Kk}$ of the footprint cover 
of the cobordism reduction $\Vv$, which exists by Lemma~\ref{le:cobred}~(a).
Then we choose $C_I'\sqsubset V_I$ such that $C_I'\cap \s_I^{-1}(0_I)  = \psi_I^{-1}(Z_I')$ as above.
Finally we adjust each $C_I'$ to be a product in the collar by the method described in \eqref{eq:Vcoll}.

To prove (c), we proceed as above, starting with a collared shrinking $(Z_I')_{I\in \Ii_\Kk}$ of the footprint cover $\bigl(Z_I = \psi_I(V_I\cap \s_I^{-1}(0_I)\bigr)_{I\in \Ii_\Kk}$ of the cobordism 
reduction $\Vv$ that extends the shrinkings at $\al=0,1$ determined by the reductions $\Cc^\al$. 
This exists by the construction in the proof of Lemma~\ref{le:cobred}~(a).
Then choose collared open subsets $C_I'\sqsubset V_I$ as above with $C_I'\cap \s_I^{-1}(0_I)  = \psi_I^{-1}(Z_I')$ for all $I$.   
Finally, let $2\eps>0$ be less than the collar width of the sets $V_I$, $Z_I'$, and $C_I'$, then
we adjust $C_I'$ in these collars so that they have the needed restrictions by setting
$$
C_I: = \Bigl(C_I'\less \io^\al_I\bigl( \ov{A^\al_\eps} \times \p^\al C'_I \bigr)\Bigr) \cup \io^\al_I\bigl( A^\al_{2\eps} \times C^\al_I\bigr) .
$$
This set is open by the same arguments as for \eqref{eq:Vcoll}.
To see that it yields a nested reduction note that because $\Cc^\al$ and $\p^\al \Cc'$ are both precompactly contained in $\p^\al\Vv$, their union is also a nested reduction.
\end{proof}

\bibliographystyle{alpha}

\end{document}